\documentclass[]{article}

\usepackage{amsmath}
\usepackage{amsfonts}
\usepackage{amssymb}
\usepackage{amsthm}
\usepackage{mathtools}
\usepackage{color}
\usepackage{esint}
\usepackage{hyperref}
\usepackage{enumitem}
\usepackage{appendix}
\usepackage{url}
\usepackage[margin=1.2in]{geometry}
\usepackage{graphicx}
\usepackage[nottoc]{tocbibind}

\usepackage{tikz}
\usetikzlibrary{shapes,snakes,arrows, chains, positioning, shapes.geometric, shapes.symbols,calc}
\usepackage{tikz-3dplot}
\tdplotsetmaincoords{70}{110}
\usepackage{tikz-cd}

\def\centerarc[#1](#2)(#3:#4:#5)% Syntax: [draw options] (center) (initial angle:final angle:radius)
{ \draw[#1] ($(#2)+({#5*cos(#3)},{#5*sin(#3)})$) arc (#3:#4:#5); }

\theoremstyle{plain}
\newtheorem{theorem}{Theorem}
\newtheorem{lemma}[theorem]{Lemma}
\newtheorem{proposition}[theorem]{Proposition}
\newtheorem{corollary}[theorem]{Corollary}
\newtheorem{remark}[theorem]{Remark}

\newtheorem{claim}[theorem]{Claim}
\newtheorem{fact}{Fact}

\theoremstyle{definition}
\newtheorem{definition}[theorem]{Definition}
\newtheorem{def/lem}[theorem]{Definition/Lemma}
\newtheorem{example}[theorem]{Example}
\newtheorem{assumptions}[theorem]{Assumptions}
\newtheorem{notation}[theorem]{Notation}
\newtheorem{note}[theorem]{Note}
\newtheorem{setting}[theorem]{Setting}
\numberwithin{theorem}{section}
\numberwithin{equation}{section}

\newcommand{\C}{\mathbb{C}}
\newcommand{\CP}{\mathbb{CP}}
\newcommand{\R}{\mathbb{R}}
\newcommand{\Z}{\mathbb{Z}}
\newcommand{\mH}{\mathcal{H}}
\newcommand{\mL}{\mathcal{L}}

\newcommand{\mO}{\mathcal{O}}
\newcommand{\mA}{\mathcal{A}}
\newcommand{\mB}{\mathcal{B}}
\newcommand{\mD}{\mathcal{D}}
\newcommand{\dd}{\partial\bar{\partial}}
\newcommand{\p}{\partial}

\newcommand{\tz}{\tilde{z}}
\newcommand{\tA}{\tilde{A}}
\newcommand{\tL}{\tilde{L}}
\newcommand{\tH}{\tilde{H}}
\newcommand{\tp}{\tilde{\partial}}
\newcommand{\Th}{\tilde{h}}

\newcommand{\inn}{\langle \cdot , \cdot \rangle}
\newcommand{\bmH}{\bar{\mH}}
\newcommand{\bH}{\bar{H}}
\newcommand{\bL}{\bar{L}}

\DeclareMathOperator{\tr}{tr}
\DeclareMathOperator{\End}{End}
\DeclareMathOperator{\Aut}{Aut}
\DeclareMathOperator{\Hom}{Hom}
\DeclareMathOperator{\codim}{codim}
\DeclareMathOperator{\Ann}{Ann}
\DeclareMathOperator{\Span}{span}
\DeclareMathOperator{\img}{img}
\DeclareMathOperator{\rk}{rk}
\DeclareMathOperator{\Id}{Id}
\DeclareMathOperator{\hol}{(hol.\!)}
\DeclareMathOperator{\hot}{(h.o.t.)}
\DeclareMathOperator{\Hol}{Hol}
\DeclareMathOperator{\Hess}{Hess}
\DeclareMathOperator{\grad}{grad}
\DeclareMathOperator{\RE}{Re}

\DeclareMathOperator{\RES}{Res}
\DeclareMathOperator{\Vol}{Vol}
\DeclareMathOperator{\dev}{dev}

\DeclarePairedDelimiter\floor{\lfloor}{\rfloor}

\newcommand{\Address}{{% additional braces for segregating \footnotesize
		\bigskip
		\begin{flushright}
			\textsc{King's College London, Department of Mathematics\\
				Strand, London, WC2R 2LS, United Kingdom}
		\end{flushright}
}}

\title{Polyhedral K\"ahler cone metrics on \(\C^n\) singular at hyperplane arrangements}
\date{\today}
\author{
Martin de Borbon \\ (\small{\href{mailto:martin.deborbon@kcl.ac.uk}{\nolinkurl{martin.deborbon@kcl.ac.uk}}}) \and Dmitri Panov \\ (\small{\href{mailto:dmitri.panov@kcl.ac.uk}{\nolinkurl{dmitri.panov@kcl.ac.uk}}})
}

\begin{document}

\maketitle

\begin{abstract}
	Let \(X\) be a complex manifold and
	let \(g\) be a polyhedral metric on it inducing its topology. We say that \(g\) is a polyhedral K\"ahler (PK) metric on \(X\) if it is K\"ahler outside its singular set. The local geometry of PK metrics is modelled on PK cones, and in this article we focus on an interesting class of examples of these. Following work of Couwenberg-Heckman-Looijenga \cite{CHL}, we consider a special kind of flat torsion free meromorphic connections on \(\C^n\) with simple poles at the hyperplanes of a linear arrangement. In the case of unitary holonomy we show that, under suitable numerical conditions, the metric completion is a PK cone metric on \(\C^n\). We apply our results to the essential braid arrangement, extending to higher dimensions the classical story of spherical metrics on \(\CP^1\) with three cone points.
\end{abstract}

\setcounter{tocdepth}{2}
\tableofcontents

\section{Introduction and main results}

The general theme of this article is the study of constant holomorphic sectional curvature K\"ahler manifolds - whose local geometry is modelled on Euclidean \(\C^n\), Fubini-Study \(\CP^n\) and complex hyperbolic \(\C \mathbb{H}^n\) spaces - with metric singularities of conical type. 

Thurston showed in \cite{Thu} that the moduli space of flat metrics on the \(2\)-sphere with prescribed cone angles in the interval \((0,2\pi)\) at a fixed number of distinct points has a natural complex hyperbolic metric induced by the area of the flat surfaces it parametrises. In the same paper, he introduced the notion of a \((X,G)\) cone manifold (\(X= \C\mathbb{H}^n\) and \(G=PU(1,n)\) for the case at hand) 
to construct the metric completion of the moduli space. In a different direction, in their search for lattices in \(PU(1,n)\), Deligne-Mostow \cite{DM} were led to consider monodromy representations of the Lauricella  hypergeometric differential equation. The solutions of such equations form a locally constant sheaf of vector spaces over the configuration space of a fixed number of distinct points on \(\CP^1\). The monodromy of this local system preserves a natural Hermitian form given by the cup product of an appropriate twisted cohomology group. For suitable values of parameters, the signature of the Hermitian form is \((1,n)\) and it induces a complex hyperbolic metric on the configuration space. As it turns out, the families of complex hyperbolic metrics found by Deligne-Mostow and Thurston are the same. 

Couwenberg-Heckmann-Looijenga introduced in \cite{CHL} the notion of a Dunkl system and recovered Deligne-Mostow's results into a more geometric framework.
In this setting, for given complex parameters \(a \in \C^{n+1}\) there is an explicit Lauricella connection \(\nabla^a\) on \(\C^n\). The connection \(\nabla^a\) is meromorphic and has simple poles at the hyperplanes of the essential braid arrangement \(\mA_n\). Moreover, the connection \(\nabla^a\) is flat, torsion free and its local parallel sections correspond to solutions of the Lauricella hypergeometric equation. One advantage of the approach in \cite{CHL} is that it allows for generalization to different hyperplane arrangements, such as mirrors of unitary complex reflection groups; previously considered by Hirzebruch and collaborators \cite{Hir, BHH} in relation to ball quotients constructions. In the presence of an admissible flat Hermitian form and under special rational assumptions on the weights of a Dunkl connection, known as the Schwarz conditions, the authors \cite{CHL} realize the metric completion of the projectivized arrangement complement as a global orbifold quotient of a complex space form. The Schwarz conditions are not satisfied in generic cases, for example when cone angles in transverse directions at generic points of hyperplanes are irrational numbers, and  Thurston's \((X,G)\) cone manifolds serve as natural candidates for describing the metric completion in such situations, see the forthcoming paper \cite{Shen}. 

We build on \cite{CHL} and consider a special class of flat torsion free meromorphic connections on \(\C^n\) with logarithmic poles at the hyperplanes of a linear arrangement. Our setting is a bit more general than that of Dunkl connections \cite[Definition 2.8]{CHL} in the sense that we lift the self-adjoint condition on the residues with respect to some fixed positive definite inner product. This allows us to consider angles bigger than \(2\pi\) in full generality, a main novelty in our work. When the holonomy is unitary,  we identify the metric completions with polyhedral K\"ahler cones,  introduced in \cite{Pan}. 
In the special case of the Lauricella connection, our results provide a higher dimensional version of the existence criterion for spherical metrics with three cone points on $\CP^1$, a  topic which can be traced back to Klein's work on the monodromy of the Gauss hypergeometric equation, see \cite{Ere}.  
%In the special case of the Lauricella connection, our results provide a higher dimensional version of the well-known existence criteria for spherical triangles (with angles non-integer multiples of \(\pi\)), a classical topic which goes back to Klein's work on the monodromy of the Gauss hypergeometric equation, see \cite{Ere}. 
We proceed to the statement of our main results.

\subsection{Polyhedral K\"ahler cones from unitary connections}

We work on \(\C^n\) with linear coordinates \((z_1, \ldots, z_n)\). Let \(\mH\) be a finite collection complex hyperplanes \(H \subset \C^n\) going through the origin. For every \(H \in \mH\) we choose a defining linear equation \(h=h(z_1, \ldots, z_n)\) so that \(H= \{h=0\}\). Moreover,
suppose that for each \(H \in \mH\) we are given a \(n \times n\) complex matrix \(A_H\). Consider the singular affine connection \(\nabla\)  defined by 
\begin{equation}\label{eq:thmconect} 
\nabla = d - \sum_{H \in \mH} A_H \frac{dh}{h} .
\end{equation}
The expression \eqref{eq:thmconect} for \(\nabla\) is understood
with respect to the trivialization of \(T\C^n\) given by the coordinate frame \(\p/\p z_1, \ldots, \p / \p z_n\) and the \(d\) term represents the usual connection of Euclidean space given by the exterior derivative.

We refer to meromorphic connections of the form specified by Equation \eqref{eq:thmconect} as \emph{standard connections}. The connection \(\nabla\) is holomorphic on the hyperplane complement \((\C^n)^{\circ} = \C^n \setminus \mH\) and it has simple poles at \(\mH\). We say that \(\nabla\) is flat (torsion free, unitary) if it is so on its regular part \((\C^n)^{\circ}\).
Our first main result is the next.

\begin{theorem}[Metric completion]\label{PRODTHM}
	Let \(\nabla\) be a flat torsion free standard connection  given by Equation \eqref{eq:thmconect}. Suppose that \(\nabla\) is unitary and let \(g\) be a K\"ahler metric on \((\C^n)^{\circ}\) such that \(\nabla g =0\). 
	
	Moreover, assume that for all irreducible intersections \(L \in \mL_{irr}(\mH)\) the following numerical conditions on the weights \(a_L\) of \(\nabla\)  are satisfied:
	\begin{itemize}
		\item \emph{(non-integer)} \(a_L \notin \Z \);
		\item \emph{(positivity)}  \(1 - a_L>0\).
	\end{itemize}

	Then the metric completion of the flat K\"ahler manifold \(\left((\C^n)^{\circ}, g\right)\) \emph{is}  \(\C^n\) endowed with a polyhedral cone metric. 
\end{theorem}

In the context of Theorem \ref{PRODTHM}, the torsion free property of \(\nabla\) implies that its residues \(A_H\) are rank one matrices and satisfy \(\ker A_H = H\) for all \(H \in \mH\). The flatness of \(\nabla\) results into commutation relations of the form \([A_{L_1}, A_{L_2}] =0\) whenever \(L_1 \subset L_2\) are linear subspaces given by intersections of hyperplanes in \(\mH\) and \(A_{L_i}\) denote the sum of the matrices \(A_H\) over all hyperplanes of the arrangement that contain \(L_i\). The weight of \(\nabla\) at a hyperplane \(H\) is defined as the complex number \(a_H\) given by
\begin{equation}\label{eq:aH}
a_H= \tr A_H .	
\end{equation}
In Theorem \ref{PRODTHM} we assume that the weights \(a_H\) are real and \(<1\).
More generally, if \(L\) is an intersection of hyperplanes in the arrangement the weight \(a_L\) is defined as \(a_L = (\codim L)^{-1} \tr A_L\).

Recall that a polyhedral (cone) metric is a complete length metric space that can be decomposed into Euclidean simplices (simplicial cones) glued isometrically along their faces. With this in mind,
we clarify the meaning of the word \emph{is} in last sentence of Theorem \ref{PRODTHM}. Let 
\[d^{\circ}_g: (\C^n)^{\circ} \times (\C^n)^{\circ} \to \R_{\geq 0}\] 
be the Riemannian distance function defined by \(g\). We show that \(d^{\circ}_g\) extends continuously to \(\C^n \times \C^n\) as a distance function of a complete length metric  \(d_g\) inducing the usual topology of \(\C^n\).
Moreover, the complete length space \((\C^n, d_g)\) is a polyhedral cone.
\newline

\noindent \textbf{2-cones.} We explain the elementary but nevertheless important case of complex dimension \(n=1\). Let \(a\) be a complex number and  consider the standard connection \(\nabla\) with a simple pole at \(0\) given by
\begin{equation}\label{eq:1dimnabla}
\nabla = d - a \frac{dz}{z} . 
\end{equation}
The connection \(\nabla\) is flat and torsion free. Parallel transport along a simple anti-clockwise loop encircling the origin equals multiplication by \(\exp(2\pi i a)\), so there is a non-zero parallel Hermitian form if and only if \(a \in \R\).

Let \(\alpha>0\) be a positive real number and write \(\C_{\alpha}\) for the complex line endowed with the metric of a \(2\)-cone of total angle \(2\pi\alpha\) with vertex at \(0\) defined by the line element 
\begin{equation}\label{eq:1dimmet}
|z|^{\alpha-1}|dz| .	
\end{equation}
The Levi-Civita connection of \eqref{eq:1dimmet} is given by Equation \eqref{eq:1dimnabla} with weight
\begin{equation}\label{eq:aalpha}
a = 1 -\alpha .
\end{equation}
The positivity assumption \(1-a>0\) for the connection \eqref{eq:1dimnabla} is equivalent to \(\alpha>0\). On the other hand,
if \(\alpha \leq 0\) then the origin is at infinite distance with respect to the line element \eqref{eq:1dimmet} and therefore it does not belong to the metric completion. Indeed, a neighbourhood of the origin is isometric to the end of a cone of total angle \(2\pi(-\alpha)\) when \(\alpha<0\) and to the end of a cylinder when \(\alpha=0\).

\subsection{Local product decomposition and holonomy representation}

\textbf{Local splitting.}
The core of our proof of Theorem \ref{PRODTHM} is to  show that the metric \(g\) splits locally away from the origin of \(\C^n\) as a product of a flat Euclidean part tangential to the metric singularities together with a lower dimensional transversal factor of the `same kind' as \(g\). To make a precise statement we introduce some notation.

Let \(\mL(\mH)\) be the set of linear subspaces obtained as intersections of hyperplanes in \(\mH\).  Given \(L \in \mL(\mH)\) we write \(\mH_L\) for the members of \(\mH\) that contain \(L\). 
The \emph{stratum} \(L^{\circ}\) is defined as
\[L^{\circ} = L \setminus  \bigcup_{H \in \mH \setminus \mH_L} H . \]

In the context of Theorem \ref{PRODTHM},
the connection \(\nabla\) gives rise to a standard quotient connection \(\nabla^{\C^n/L}\) with simple poles
at the quotient arrangement
\[\mH_L/L = \{H/L, \,\, H \in \mH_L \} .\] 
The standard connection \(\nabla^{\C^n/L}\) is also flat, torsion free and unitary. The key result is the following.

\begin{theorem}[Local Product Decomposition]\label{thm:locprod}
	Suppose that the hypothesis of Theorem \ref{PRODTHM}
	are satisfied
	and let \(x \in L^{\circ}\). Then
	close to \(x\) the metric \(g\) is holomorphically isometric to a direct product
	\begin{equation}\label{eq:lpt}
		\C^{d} \times \left((\C^{n-d})^{\circ}, g^{\perp} \right) 
	\end{equation}
	where \(d = \dim L\) and \(g^{\perp}\) is a flat K\"ahler metric on the complement of the quotient arrangement which is parallel for the standard quotient connection \(\nabla^{\C^n/L}\).
\end{theorem}

Theorem \ref{thm:locprod} gives detailed information of the metric at its singular locus, an immediate consequence is the following.

\begin{corollary}\label{cor:locprod}
	Suppose that the hypothesis of Theorem \ref{PRODTHM} are satisfied. Let \(H \in \mH\) and let \(x \in H^{\circ}\). Then close to \(x\) the metric \(g\) is holomorphically isometric to \(\C^{n-1} \times \C_{{\alpha_H}}\).
\end{corollary}

Recall that \(\C_{{\alpha_H}}\) denotes the \(2\)-cone with total angle \(2\pi\alpha_H\).
The number \(\alpha_H>0\) in Corollary \ref{cor:locprod} that measures the cone angle \(2\pi\alpha_H\) transverse to the hyperplane \(H\) is given by
\(\alpha_H=1-a_H\) where \(a_H\) is the weight of \(\nabla\) at \(H\) given by Equation \eqref{eq:aH}.
\newline

\noindent \textbf{Holonomy.}
A hyperplane arrangement \(\mH\) is essential if the common intersection of all its members is the origin.
We say that \(\mH\) is irreducible if it can not be expressed as a product of two non-empty arrangements.
With these notions,
our next main result is as follows.

\begin{theorem}[Irreducible holonomy]\label{thm:irredhol}
	Suppose that the hypothesis of Theorem \ref{PRODTHM}
	are satisfied. Moreover, assume that the arrangement \(\mH\) is essential and irreducible. Then the holonomy representation of \(\nabla\), or equivalently of the flat K\"ahler metric \(g\) on \((\C^n)^{\circ}\), is irreducible.
\end{theorem}

As a consequence we have the next uniqueness result.

\begin{corollary}\label{cor:uniqueinnerprod}
	Let \(\mH\) be an essential irreducible arrangement and suppose that the hypothesis of Theorem \ref{PRODTHM}
	are satisfied. Moreover, assume that \(g\) and \(g'\) are K\"ahler metrics which are parallel with respect to \(\nabla\). Then \(g'=\lambda g\) for some \(\lambda>0\).
\end{corollary}

\subsection{Properties of PK metrics on complex manifolds}

Following \cite{Pan}, we introduce the notion of a  polyhedral K\"ahler (PK) metric on a complex manifold and
provide a series of foundational results for them. The polyhedral cones provided by Theorem \ref{PRODTHM} fit into this framework. 

Let \(X\) be a complex manifold and let \(g\) be a polyhedral metric on it inducing its topology. We say that \(g\) is a \emph{PK metric on \(X\)} if \(g\) is K\"ahler on its regular part \(X^{\circ}\). A first non-trivial consequence of the definition is the following.

\begin{theorem}\label{thm:PK1}
	Let \(X\) be a complex manifold and let \(g\) be a PK metric on it. Then the set of metric singularities \(X^s\) is a complex hypersurface of \(X\). At smooth points of \(X^s\) the metric \(g\) is locally holomorphically isometric to a product \(\C^{n-1} \times \C_{\alpha} \).
\end{theorem}

The local study of PK metrics reduces to the consideration of PK cones and we provide more detailed information for these. Every cone is equipped with a canonical action of the real numbers by dilations of the metric fixing its vertex. As a mater of fact, this action can be complexified in the PK case and it generates a holomorphic vector field, we denote it by \(e\) and refer to it as the (complex) Euler vector field of the cone.

\begin{theorem}\label{thm:PK2}
	Let \(X\) be a complex manifold and let \(g\) be a PK cone metric on it with vertex at \(x\). Then \(X\) is biholomorphic to \(\C^n\). There are global complex coordinates \((z_1, \ldots, z_n)\) centred at \(x\) and positive real numbers \(c_1, \ldots, c_n > 0\) such that the complex Euler vector field of the cone is equal to
	\begin{equation}\label{eq:action}
		e = c_1 z_1\frac{\p}{\p z_1} + \ldots + c_n z_n \frac{\p}{\p z_n} .
	\end{equation}
	In particular, the metric singular set \(X^s\) is invariant under the flow generated by \eqref{eq:action}.
\end{theorem}

In our next result we identify \(X\) with \(\C^n\) using global complex coordinates \((z_1, \ldots, z_n)\) as above.

\begin{theorem}\label{thm:PK3}
	Let \(g\) be a PK cone metric on \(\C^n\) with vertex at the origin and complex Euler vector field given by Equation \eqref{eq:action}. Suppose that the set of metric singularities make a hyperplane arrangement \(\mH\). Then the following holds.
	\begin{itemize}
		\item[(i)] The stratification of \(\C^n\) given by the singularities of the polyhedral metric \(g\) agrees with the natural complex stratification of the arrangement given by the strata \(L^{\circ}\) with \(L \in \mL(\mH)\).
	
		\item[(ii)] The Levi-Civita connection of \(g\) is a standard connection of the form given by Equation \eqref{eq:thmconect}.
	\end{itemize}	
\end{theorem}

\subsection{PK cone metrics on \(\C^n\) singular at the \(\mA_n\)-arrangement}

The essential braid arrangement \(\mA_n =  \{H_{i,j}, \,\, 1 \leq i <j \leq n+1 \} \) consists of \(\binom{n+1}{2}\) hyperplanes in \(\C^n\) given as follows
\begin{equation*}
	H_{i,j} = \begin{cases}
	\{z_i = z_j\} \, &\mbox{ if } \, 1 \leq i < j \leq n ,\\
	\{z_i = 0\} \, &\mbox{ if } \, 1 \leq i \leq n \mbox{ and } j =n+1 .
	\end{cases}
\end{equation*}

Our final main result is the next.

\begin{theorem}\label{LAUTHM}
	Let \(\alpha_1, \ldots, \alpha_{n+1}\) be real numbers satisfying the following conditions:
	\begin{itemize}
		\item \emph{(non-integer)} \(\alpha_i \notin \Z\) for all \(1\leq i \leq n+1\);
		\item \emph{(positivity)} 
		\(\alpha_i + \alpha_j -1 >0\) 
		for all \(1 \leq i < j \leq n+1\);
		\item \emph{(signature)} either \(\sum_{i=1}^{n+1} \{\alpha_i\} < 1\) or \(\sum_{i=1}^{n+1} \{\alpha_i\} > n\);
		where \(0 \leq \{\alpha\} < 1\) denotes the fractional part of a real number \(\alpha\).
	\end{itemize}
	
	Then there is a unique up to scale PK cone metric 
	on \(\C^n\) singular at the \(\mA_n\)-arrangement with cone angles \(2\pi(\alpha_i + \alpha_j -1)\) at \(H_{i,j}\) for all \(1 \leq i < j \leq n+1\).
	Moreover, any PK cone metric on \(\C^n\) singular at the \(\mA_n\)-arrangement and with non-integer cone angles at every irreducible intersection must be equal to one of the above family.
\end{theorem}

The complex link of a \emph{regular} PK cone is endowed with a singular Fubini-Study metric, obtained as the quotient of its unit sphere by the isometric action generated by the imaginary part of its complex Euler vector field. The regular condition means that the action is periodic and free. The PK cones of Theorem \ref{LAUTHM} are regular and we have corresponding Fubini-Study metrics on \(\CP^{n-1}\) singular at the projectivized \(\mA_n\)-arrangement.
The \(\mA_2\)-arrangement consists of three complex lines through the origin in \(\C^2\) and we recover spherical metrics on \(\CP^1\) with three non-integer cone points, which are all obtained by  doubling of spherical triangles\footnote{A spherical triangle is disk with a metric of curvature one whose boundary is composed of three geodesic arcs meeting at some angles, possibly larger than $\pi$.}. In this sense, Theorem \ref{LAUTHM} extends to higher dimensions the well known numerical criterion for the existence of a spherical triangle in terms of its angles.

\subsection{Outline of the article}

\textbf{Section} \ref{sect:hyparr} contains background material on hyperplane arrangements, the central notion is that of irreducibility.
We begin by recalling some basic definitions in Section \ref{sect:defhyp}. An arrangement, for us, is a finite collection of linear hyperplanes in \(\C^n\). The centre of an arrangement \(\mH\), denoted by \(T(\mH)\), is the linear subspace of \(\C^n\) obtained by intersecting all its members. We say that \(\mH\) is essential if its centre equals zero.
We denote by \(\mL(\mH)\)
the set of all subspaces of \(\C^n\) obtained as intersections of hyperplanes in \(\mH\). If \(L \in \mL(\mH)\) then the localization of \(\mH\) at \(L\) is the arrangement given by the members of the arrangement that contain \(L\), namely
\[\mH_L = \{H \in \mH, \,\, L \subset H\} .\]

In Section \ref{sect:irredarr} we discuss irreducible arrangements.
An arrangement \(\mH\) is called reducible if we can write it as a disjoint union \(\mH=\mH_1 \cup \mH_2\) such that, after a linear change of coordinates, the defining linear equations for hyperplanes in \(\mH_1\) and \(\mH_2\) share no common variables. We say that \(\mH\) is irreducible if it is not reducible. We explain that essential irreducible arrangements are characterized by the property that the identity component of their linear automorphism groups are as small as they can be, namely \(\C^*\) acting by usual scalar multiplication.

An intersection of hyperplanes  \(L \in \mL(\mH)\) is irreducible if the localization \(\mH_L\) is irreducible. In particular, hyperplanes of the arrangement are irreducible subspaces.
We write \(\mL_{irr}(\mH)\) for the subset of all irreducible subspaces.  For any \(L \in \mL(\mH)\) there are uniquely determined \emph{irreducible components} \(L_1, \ldots, L_k\). These are irreducible subspaces \(L_i \in \mL_{irr}(\mH)\) for \(1 \leq i \leq k\) characterized by the following three properties: (i) \(L \subset L_i\); (ii) their common intersection is \(L\); (iii) if \(M \in \mL_{irr}(\mH)\) contains \(L\) then there is a unique \(L_i\) such that \(L_i \subset M\). 

We denote by \(\mH_i\) the irreducible subsets \(\mH_{L_i} \subset \mH_L\) and decompose \(\mH_L\) as a product of the form
\begin{equation}\label{eq:hyperplanedec}
	(\C^n, \, \mH_L) = (\C^{\dim L}, \, \emptyset) \times (\C^{n_1}, \, \bar{\mH}_1) \times \ldots \times (\C^{n_k}, \, \bar{\mH}_k) 	
\end{equation}
where \((\C^{n_i}, \, \bar{\mH}_i)\) are essential irreducible arrangements isomorphic to the quotients \((\C^n/L_i, \, {\mH}_i/L_i)\). The decomposition given by Equation \eqref{eq:hyperplanedec} is crucial in our proof of Theorem \ref{PRODTHM} and is upgraded in  later sections  at a connection, affine and metric levels.
\newline

\noindent
\textbf{Section} \ref{sect:affst} is about standard connections,  that is connections \(\nabla\) of the form given by Equation \eqref{eq:thmconect}.  These are central objects of \cite{CHL}. In Section \ref{sect:ftfstnd} we recall a simple algebraic criterion in terms of the residue matrices \(A_H\) that determines when a standard connection \(\nabla\) is flat and torsion free, see \cite[Proposition 2.3]{CHL}. 

\emph{Criterion.}
Let \(\nabla\) be a standard connection of the form given by Equation \eqref{eq:thmconect}.
We define the residue of \(\nabla\) at \(L \in \mL(\mH)\) to be the matrix \(A_L\) equal to the sum of the matrices \(A_H\) over all \(H \in \mH_L\).
The connection \(\nabla\) is flat if and only if the commutation relations
\begin{equation}\label{eq:flatbracket}
	[A_L, A_H] = 0
\end{equation}
hold for every codimension two intersection \(L\) and every \(H \in \mH_L\). The connection \(\nabla\) is torsion free if and only if
\begin{equation}\label{eq:torfree}
	H \subset \ker A_H
\end{equation}
for every hyperplane \(H\) of the arrangement.

Let \(\nabla\) be a flat torsion free standard connection given by Equation \eqref{eq:thmconect}.
In Section \ref{sect:resnorm} we explore the implications of the conditions given by Equations \eqref{eq:flatbracket} and \eqref{eq:torfree}, in particular we show how these conditions determine the residues \(A_L\) at higher codimension subspaces.

The weight of \(\nabla\) at a hyperplane \(H \in \mH\) is the complex number \(a_H\) given by
\(a_H = \tr A_H\).
More generally, 
the weight of \(\nabla\) at an irreducible intersection \(L\) is the complex number \(a_L\) given by
\begin{equation}
a_L = \frac{1}{\codim L} \sum_{H \in \mH_L} a_H . 
\end{equation}

If \(H \in \mH\) and \(a_H\) is non-zero then the image of the matrix \(A_H\) is a one dimensional subspace \(\C \cdot n_H\) complementary to the hyperplane \(H\) and we denote it by \(H^{\perp}\).
Given \(L \in \mL(\mH)\)
we define its normal subspace \(L^{\perp}\) by
\[L^{\perp} = \Span \{n_H, \,\, H \in \mH_L\} . \]
The normal subspaces and weights of irreducible intersections completely determine the residues \(A_L\).

%We show that if \(\nabla\) has non-zero weights at all irreducible intersections and \(L \in \mL(\mH)\) then we have a direct sum decomposition \(\C^n = L \oplus L^{\perp}\). Moreover, \(L\) is the kernel of \(A_L\) and \(L^{\perp}\) is the image of \(A_L\). In the case when \(L\) is irreducible \(A_L\) acts on \(L^{\perp}\) as scalar multiplication by \(a_L\).

In Section \ref{sect:nablaLsplit} we introduce the  localization of \(\nabla\) at \(L \in \mL(\mH)\), it is  the connection \(\nabla^L\) obtained from Equation \eqref{eq:thmconect} by discarding all terms \(A_H dh/h\) for hyperplanes \(H\) which do not contain \(L\). More explicitly,
\[\nabla^L = d - \sum_{H \in \mH_L} A_H \frac{dh}{h} . \]
The connection \(\nabla^L\) is also flat and torsion free.
We improve the decomposition given by Equation \eqref{eq:hyperplanedec} to an affine splitting of the form
\begin{equation}\label{eq:affdecnablaL}
	\left(\C^n, \, \mH_L, \, \nabla^L\right) = \left(\C^{\dim L}, \, \emptyset, \, d \right) \times \left(\C^{n_1}, \, \bar{\mH}_1, \, \bar{\nabla}^i\right) \times \ldots \times \left(\C^{n_k}, \, \bar{\mH}_k, \, \bar{\nabla}^k\right) 	.
\end{equation}
In Section \ref{sect:quotconect} we introduce the quotient connection \(\nabla^{\C^n/L}\)
(which is also standard)
and interpret the terms \(\bar{\nabla}^i\) in Equation \eqref{eq:affdecnablaL} as  quotient connections \(\nabla^{\C^n/L_i}\).

Section \ref{sect:eulervf} is about Euler vector fields. Under suitable numerical conditions, every flat torsion free standard connection \(\nabla\) has a unique Euler vector field \(e\) characterized by the equation \(\nabla e = \Id\), meaning that \(\nabla_v e = v\) for every tangent vector \(v\).
In the one dimensional case given by Equation \eqref{eq:1dimnabla} we have \(e=(1-a)^{-1} z \p/\p z\) whenever \(a \neq 1\).
In higher dimensions, if the arrangement \(\mH\) is essential and irreducible, then 
\begin{equation}\label{eq:e}
e = \frac{1}{\alpha_0} \sum_{i=1}^{n} z_i \frac{\p}{\p z_i} 	
\end{equation}
where \(\alpha_0=1-a_0\) and \(a_0\) is the weight of \(\nabla\) at the origin. In general, the arrangement \(\mH\) is a product of the empty arrangement \((\C^d, \emptyset)\), where \(d= \dim T(\mH)\) is the dimension of its centre, together with essential irreducible factors. The Euler vector field of \(\nabla\) is a sum of the usual Euler vector field of \(\C^d\) together with terms of the form \eqref{eq:e} for each of the irreducible factors.
These Euler vector fields play a crucial role in our arguments.
\newline

\noindent
\textbf{Section} \ref{sec:locprod} establishes the Local Product Theorem \ref{thm:locprod}.
Let \(L \in \mL(\mH)\) be a non-zero intersection of hyperplanes in \(\mH\) and fix a point \(x\) in the stratum \(L^{\circ}\). Take a sufficiently small neighbourhood \(U\) of \(x\) such that \(U \cap H = \emptyset\) for every \(H \in \mH\) with \(L \not\subset H\). We have a model connection \(\nabla^L\) (the localization of the standard connection \(\nabla\) at \(L\)) and  we can write
\[\nabla =  \nabla^L + \hol\]
where \(\hol\) denotes a holomorphic matrix-valued \(1\)-form on the neighbourhood \(U\).
The main work is to transplant the affine product decomposition \eqref{eq:affdecnablaL} from \(\nabla^L\) to \(\nabla\). This is the content of Proposition \ref{prop:locaffprod} and we prove it along the next steps.

\begin{itemize}
	\item[(1)] Section \ref{sect:step1}: \emph{Gauge transformation.} 
	We use that \(\nabla\) is unitary to show that holonomies of \(\nabla\) and the model connection \(\nabla^L\) are conjugate. As a consequence, we have a holomorphic gauge equivalence \(G\) between \(\nabla\) and \(\nabla^L\) over \(U^{\circ}= U \setminus \mH\). The non-integer condition on the weights in Theorem \ref{PRODTHM} allows us to extend \(G\) holomorphically across the hyperplanes.

	\item[(2)] Section \ref{sect:foliation}: \emph{Foliations.}
	The gauge transformation \(G\) takes
	the tangent sub-bundles of the factors in the product decomposition \eqref{eq:affdecnablaL}  to integrable distributions preserved by \(\nabla\). We show that \(\nabla\) decomposes as a product of its restrictions to the leaves of these distributions going through the point \(x\). 
	
	\item[(3)] Section \ref{sect:loceulervf}: \emph{Local dilation vector field.} 
	We introduce a local dilation field \(e_x\) that vanishes at \(x\) and whose flow preserves \(\nabla\). The vector field \(e_x\) is tangent to the leaves of the above foliations going through the point \(x\). The linearisation of \(e_x\) equals the Euler vector field \(e_L\) for the model connection \(\nabla^L\). 
	
	\item[(4)] Section \ref{sect:identification}: \emph{Identification of transversal factors.}
	We use the Poincar\'e-Dulac theorem to linearise the restriction of \(e_x\) to the leaves of the constructed foliations that go through the point \(x\). From here we conclude that the induced connections on the leaves are standard. This establishes the desired affine local product decomposition for \(\nabla\) given by Proposition \ref{prop:locaffprod}.
\end{itemize}

Finally, in Section \ref{sect:pflocprod} we use the non-integer assumptions of Theorem \ref{PRODTHM} to show that the factors in the affine product decomposition given by Proposition \ref{prop:locaffprod} are pairwise orthogonal and this finishes the proof of Theorem \ref{thm:locprod}.
\newline	

\noindent
\textbf{Section} \ref{sect:MC} settles Theorem \ref{PRODTHM}. 

In Section \ref{sect:riemcone} we show that the flat K\"ahler metric \(g\) of Theorem \ref{PRODTHM}, defined on \((\C^n)^{\circ}\), is a Riemannian cone. This fact is a direct consequence of the existence of a real Euler vector field \(e_r = (1/2) \RE (e)\) whose flow acts by dilations of the metric. 
On the complement of the hyperplanes we have a smooth function \(r\) defined as the norm of \(e_r\), that is
\[r^2 = g(e_r, e_r) . \]
Equivalently, \(r\) is the radius function of the Riemannian cone \(g\) in the sense that we can write
\(g = dr^2 + r^2 g_{S^{\circ}}\) where \(S^{\circ}=\{r=1\}\)
is the (real) link of \(g\). We note that \((S^{\circ}, g_{S^{\circ}})\) is an incomplete Sasakian manifold of constant sectional curvature \(1\).

In Section \ref{sect:r2} we use our Local Product Decomposition Theorem \ref{thm:locprod} together with an inductive argument to show that the radius function \(r\) on \((\C^n)^{\circ}\) extends continuously across the arrangement \(\mH\) by taking strictly positive values outside the origin of \(\C^n\). In a way, this relatively rough information is the key ingredient we need to prove Theorem \ref{PRODTHM}.

Sections \ref{sect:lengthspaces} and \ref{sect:LePe} contain background material on metric geometry. Section \ref{sect:lengthspaces} discusses length spaces and metric cones. Section \ref{sect:LePe} recalls a theorem of Lebedeva-Petrunin \cite{LePe} that characterizes polyhedral spaces as complete length spaces with the property that every point has a neighbourhood which embeds isometrically in a metric cone by sending the given point to the vertex of the cone. Section \ref{sect:metricexten} contains general results about metric completions,  tailored to our needs. More specifically, we consider triples \((X, X^{\circ}, d^{\circ})\) where \(X\) is a topological space, \(X^{\circ}\) and open dense subset and \(d^{\circ}\) is metric on \(X^{\circ}\) inducing its topology. We spell out conditions under which \(d^{\circ}\) extends continuously to \(X\) as a complete metric inducing its topology.

In Section \ref{sect:completion}
we identify \(\C^n\) as the metric completion of \(((\C^n)^{\circ}, g)\), this is  accomplished by verifying the conditions of Section \ref{sect:metricexten}.
The key players in our argument are the localized radius functions \(r_x\).
For each \(x \in \C^n\) we have a local dilation vector field \(e_x\) defined on a \emph{standard neighbourhood} of \(x\). The vector field \(e_x\) is holomorphic, it vanishes at \(x\) and its flow acts by dilations of the metric \(g\) on \((\C^n)^{\circ}\) fixing the point \(x\). We define the function \(r_x\) to be the norm of \(e_x\) on the complement of the hyperplanes. The \(r_x\)  are distance functions for \(g\), so integral curves of \(\nabla r_x\) (i.e reparametrized integral curves of \(e_x\)) are minimizing unit speed geodesics of length equal to the difference between the endpoint values of \(r_x\). The key fact, proved in Section \ref{sect:r2}, is that the functions \(r_x\) extend continuously across \(\mH\) by taking strictly positive values outside \(x\). This allows us to extend the  Riemannian distance \(d^{\circ}_g\) defined by \(g\) to a continuous function
\[d_g : \C^n \times \C^n \to \R_{\geq 0} .\]  
An elementary argument using standard neighbourhoods shows that \(d_g\) takes strictly positive values outside the diagonal (hence it is a metric) and it induces the usual topology. The properness of \(r = d_g( \cdot, 0)\) implies that \(d_g\) is a complete metric and therefore \((\C^n, d_g)\) is a metric completion of \(((\C^n)^{\circ}, d_g^{\circ})\). Since \(d_g\) is a metric completion of a length space, it follows that \(d_g\) is a length metric.
On the other hand \(d_g^{\circ}\) is the length distance of a Riemannian cone \((S^{\circ}, g_{S^{\circ}})\), so \(d_g\) is a metric cone over the metric completion \((S, g_S)\) of \((S^{\circ}, g_{S^{\circ}})\).

In order to finish the proof of Theorem \ref{PRODTHM} it remains to show that \(d_g\) is polyhedral. This is accomplished in Section \ref{sect:pfthm1} by applying Lebedeva-Petrunin's characterization of polyhedral spaces \cite{LePe}.
\newline 

\noindent
\textbf{Section} \ref{sect:PK} introduces and develops, following \cite{Pan}, the concept of a PK metric on a complex manifold. In Sections \ref{sect:simplicialcomplex} and \ref{sect:polyhedral cones}  we recall the definition of polyhedral spaces, polyhedral (pseudo)manifolds, their tangent cones and
the stratification given by the singularities of the metric. In Section \ref{sect:parallelvf}  we discuss vector fields \emph{parallel in codimension \(k \geq 1\)}. The natural category in which we work is the one of \emph{normal polyhedral pseudomanifolds}. In this setting, we show that a polyhedral cone admitting a non-zero vector field parallel in codimension \(2\) necessarily splits a Euclidean factor. 

In Section \ref{sect:polKahlerspaces} we introduce the notion of \emph{polyhedral K\"ahler spaces}, polyhedral pseudomanifolds with a parallel complex structure on their regular part and such that the metric singularities have \emph{complex directions}. We show that if the real codimension \(2\) strata have complex directions then all higher codimension strata do, in particular there are no metric singularities of odd real codimension. In Section \ref{sect:singularitiesPK} we analyse the singularities of a PK metric on a complex manifold. We prove existence of complex directions, identify the connected components of the open polyhedral strata with smooth complex submanifolds and show that the metric singular set \(X^s\) is a complex (possibly singular) hypersurface of \(X\).

The local study of PK metrics reduces to the one of PK cones,
which we consider in Section \ref{sect:globalcxcoord}. We show that if \((X,g)\) is a complex manifold endowed with a PK cone metric then \(X\) is biholomorphic to \(\C^n\); the main idea is to linearise the complex Euler vector field \(e\) of the cone. We establish the existence of global complex coordinates 
\[(z_1, \ldots, z_n): X \xrightarrow{\sim} \C^n\]
centred at the vertex of the cone such that the canonical dilation of the metric that scales distances by a factor of \(\lambda>0\) is given by
\begin{equation}\label{eq:actiontorus}
	(z_1, \ldots, z_n) \mapsto (\lambda^{c_1}z_1, \ldots, \lambda^{c_n}z_n)
\end{equation}
where \(c_1, \ldots, c_n\) are positive real numbers. The real dilations \eqref{eq:actiontorus} give rise to a linear action of a complex torus \((\C^*)^k\) and the metric singular set is a \((\C^*)^k\)-invariant complex hypersurface. In Section \ref{sect:sumfldandindmet} we look at complex polyhedral submanifolds of PK complex manifolds and show that the induced metric is also PK.

In Section \ref{sect:PKhyperplane} we consider PK cone metrics on \(\C^n\) whose singularities make a hyperplane arrangement \(\mH\). We prove that the polyhedral stratification given by the codimensions of metric singularities agrees with the natural complex stratification
\[\C^n = \bigcup_{L \in \mL(\mH)} L^{\circ} . \]
Moreover, we show that the Levi-Civita connection of the metric is of the type described by Equation \eqref{eq:thmconect}, i.e. it is standard.
\newline
 
\noindent
\textbf{Section} \ref{sect:FS} discusses a class of singular Fubini-Study (FS) metrics on \(\CP^{n-1}\) which arise as complex links of PK cones on \(\C^n\). 

We begin with Section \ref{sect:regPKcone} about regular PK cones.
Following common language is Sasakian geometry, we say that a PK cone is regular if its Reeb vector field generates a free \(S^1\)-action. Given a regular PK cone metric on \(\C^n\) we have linear coordinates \((z_1, \ldots, z_n)\) uniquely defined up to the action of \(GL(n, \C)\) in which the Euler vector field of the cone is given by Equation \eqref{eq:e}
where \(2\pi\alpha_0>0\) is the total angle of the \(2\)-cone obtained by restricting \(g\) to any complex line going through the origin. The distance squared to the origin is a K\"ahler potential for the metric and it is given by
\begin{equation}\label{eq:kpotpk}
	r^2 = e^{2\varphi} |z|^{2\alpha_0} 
\end{equation}
where \(\varphi\) is the pull-back of a continuous (non-explicit) function on \(\CP^{n-1}\) under the standard projection map \(\C^{n} \setminus \{0\} \to \CP^{n-1}\). 

In Section \ref{sect:VolformPKcone} we prove that the volume form of a PK cone metric on \(\C^n\) singular at a hyperplane arrangement \(\mH\) is equal to the Euclidean volume form multiplied by the factor
\begin{equation}\label{eq:voldensity}
	\prod_{H \in \mH} |h|^{2\alpha_H-2} .
\end{equation}
We show that the positivity assumption in Theorem \ref{PRODTHM} is equivalent to the volume density function
given by Equation \eqref{eq:voldensity} being locally integrable in the usual Lebesgue sense. 

In Section \ref{sect:FSquotPK} we consider singular Fubini-Study metrics \(g_{FS}\) on \(\CP^{n-1}\) obtained from the quotient of the unit sphere of a regular PK cone metric on \(\C^n\) under multiplication by complex units. If we write the K\"ahler potential of the PK cone as in Equation \eqref{eq:kpotpk} then the K\"ahler form of the Fubini-Study metric \(g_{FS}\) on its regular part is equal to
\begin{equation}\label{eq:globalpotential}
	\omega_{FS} = \alpha_0 \omega_{\CP^{n-1}} + i \dd \varphi
\end{equation} 
where \(\omega_{\CP^{n-1}}\) is the standard Fubini-Study metric with sectional curvatures in the range \(1 \leq \mbox{sec} \leq 4\). In Section \ref{sect:locmodsingFS} we provide local models for singularities of FS metrics. Tangent cones of FS metrics are PK cones, if \(C\) is the tangent cone at \(\bar{x} \in \CP^{n-1}\) then close to \(\bar{x}\) we have
\begin{equation*}
	\omega_{FS} = \frac{i}{2} \dd \log (1 + \rho^2)
\end{equation*}
where \(\rho\) is the distance to the vertex of \(C\). In Section \ref{sect:volformFS} we analyse the volume form \(\omega_{FS}^{n-1}\) close to singularities. The measure defined by the Riemannian volume form of \(g_{FS}\) on the regular part extends by zero over the projectivized arrangement to an absolutely continuous measure. 

The unit sphere of a polyhedral cone is a spherical polyhedral space, its (total) volume is the sum of the volumes of its top dimensional simplices or equivalently is the integral of the Riemannian volume form over its regular part. Let \((\C^n, g)\) be a regular PK cone and let \(g_{FS}\) be the corresponding Fubini-Study metric on \(\CP^{n-1}\). We define the \emph{volume of} \(g_{FS}\)
as the integral of its Riemannian volume form over its regular part, equivalently it is equal to the volume of the unit sphere of the cone divided by \(2\pi\alpha_0\) where \(2\pi\alpha_0\) is the length of the \(S^1\)-fibres of the projection map from the unit sphere to the complex link.
In Section \ref{sect:totalvol} we give a formula for the  volume of our Fubini-Study metrics, namely
\begin{equation}\label{eq:totvol}
	\Vol(g_{FS}) = \alpha_0^{n-1} \frac{\pi^{n-1}}{(n-1)!} 
\end{equation}
where \(\pi^{n-1}/(n-1)!\) is the volume of the standard Fubini-Study metric on \(\CP^{n-1}\).
Equation \eqref{eq:totvol} follows from Equation \eqref{eq:globalpotential} together with the fact that \(\varphi\) is continuous on the whole \(\CP^{n-1}\) and general results from pluripotential theory. In Section \ref{sect:PKliftsFS} we introduce the notion of a FS metric in \(\CP^{n-1}\) adapted to a projective arrangement and show that these metrics lift to PK cones metrics on \(\C^n\).
\newline 
 
\noindent
\textbf{Section} \ref{sect:irredhol}
establishes Theorem \ref{thm:irredhol} which asserts that,
in the context of Theorem \ref{PRODTHM}, if the arrangement \(\mH\) is essential and irreducible, then the holonomy of the connection \(\nabla\) is irreducible.
The proof goes by induction on the dimension, by looking at the \emph{induced connections} on the hyperplanes of the arrangement. 

In Section \ref{sec:triples} we show that if \(\mH\) is an essential irreducible arrangement then there is a hyperplane \(H_0 \in \mH\) such that the induced arrangement
\[\mH'' = \{H \cap H_0, \,\, H \in \mH \setminus \{H_0\} \} \]
is also essential and irreducible. This result is needed for our inductive proof of Theorem \ref{thm:irredhol}.

In Section \ref{sect:indcon} we discuss induced connections on hyperplanes \(\nabla''\) along the lines of \cite{CHL} in the Dunkl case. If \(\nabla\) is a standard connection as in Theorem \ref{PRODTHM} given by Equation \eqref{eq:thmconect} and \(H_0 \in \mH\), then \(\nabla''\) is defined as the restriction to \(TH_0\) of the connection
\[\nabla' = d - \sum_{H \in \mH \setminus \{H_0\}} A_H \frac{dh}{h} .\]
The induced connection is also flat, torsion free and unitary. Moreover, the weights of \(\nabla''\) at irreducible intersections of \(\mH''\) are a subset
of the weights of \(\nabla\) at irreducible intersections of \(\mH\). In particular \(\nabla''\) also satisfies the non-integer and positivity assumptions. 

In Section \ref{sect:pfirrhol} we finish the proof of Theorem \ref{thm:irredhol}.
By induction, we reduce to the case of \(n=2\). In the case of \(n=2\) the result follows from the correspondence between PK cones and spherical metrics (\cite{Pan}) together with the characterization of co-axial monodromy given in \cite{Dey}. 
\newline

\noindent
\textbf{Section} \ref{sect:Lau} establishes Theorem \ref{LAUTHM} which characterizes the values of the cone angles for which there is a PK cone metric on \(\C^n\) singular at the hyperplanes of the \(\mA_n\)-arrangement.
We introduce the braid arrangement and Lauricella connection in Sections \ref{sect:braidarrang} and \ref{sect:Lauconnect} following mainly the exposition in \cite{CHL}. 

In Section \ref{sect:charactLau} we characterize the Lauricella family as the unique flat torsion free connections with poles at the hyperplanes of the \(\mA_n\)-arrangement. This uniqueness result might be known to experts but we haven't found such a statement in the literature. 

In Section \ref{sect:flathermform} we recall well-known facts about the holonomy representation of the Lauricella connection and the existence of a natural flat Hermitian form \(\inn\) studied by several authors. Of particular importance is the formula for the signature of \(\inn\) which goes back at least to \cite{DM}. 
In Section \ref{sect:pfLau} we establish Theorem \ref{LAUTHM}. Our proof relies on Theorem \ref{PRODTHM} for the existence part and it uses Theorem \ref{thm:irredhol} together with the characterization of the Lauricella connection from Section \ref{sect:charactLau} for the uniqueness and non-existence statements. 

Finally, in Section \ref{sect:FSLau} we re-state Theorem \ref{LAUTHM} in terms of FS metrics and give formulas for their volumes. We believe that these FS metrics provide interesting genuine examples of K\"ahler-Einstein metrics with cone angles bigger that \(2\pi\). In the case of angles \(< 2\pi\) we relate the existence criterion to log K-stability (when angles are bigger that \(2\pi\) the divisors of the corresponding pairs are non-effective and there is no notion of stability available for these objects).

\subsection*{Acknowledgements}
\addcontentsline{toc}{subsection}{Acknowledgements}

We would like to thank Alexander Lytchak, Anton Petrunin, Dali Shen, Eduard Looijenga, Gert Heckman and Misha Verbitsky  for answering questions and discussions. D.P. is particularly grateful to Misha Vebritsky for many years of collaboration on polyhedral K\"ahler manifolds.

Thank you EPSRC Project EP/S035788/1, \emph{K\"ahler manifolds of constant curvature with conical singularities} for supporting our work.

\section{Hyperplane arrangements}\label{sect:hyparr}

We recall some elementary definitions and facts from the theory of hyperplane arrangements and at the same time set-up some notation for the rest of the paper.
A good general reference for the content of this section and more advanced aspects of the topic is Orlik-Terao's book \cite{OT1}. We present a self-contained treatment. The key concept is that of an irreducible arrangement (Definition \ref{def:irrarr}) and the main result is Lemma \ref{lem:HLdec}.

\subsection{Definitions}\label{sect:defhyp}

\begin{definition}
	An \textbf{arrangement} is a finite collection of linear hyperplanes \(H \subset V\), where \(V\) is a complex vector space of dimension \(n < \infty\).  
\end{definition}

\begin{notation}
	We use either \((V, \mH)\) or \(\mH\) to denote arrangements.
\end{notation}

\begin{remark}
	We allow \(\mH=\emptyset\). 
\end{remark}

\begin{definition}
	The \textbf{centre} of \(\mH\) is the common intersection of all its members, we denote it by 
	\[T(\mH)= \bigcap_{H \in \mH} H.\] 
\end{definition}

\begin{definition}
	The \textbf{intersection poset} is defined as
	\[\mL(\mH) = \{L, \,\ L \mbox{ is an intersection of hyperplanes in } \mH \} .\]
	\(\mL(\mH)\) is equipped with the reverse inclusion relation: \(L \leq M \iff L \supset M\).
\end{definition}

\begin{note}
The minimal element of \(\mL(\mH)\) is \(V\), corresponding to the intersection over the empty subset of \(\mH\). The maximal element is the centre \(T(\mH)\).	
\end{note}

\begin{notation}
Given a non-zero linear function \(h \in V^*\), it defines a linear hyperplane in \(V\) which we denote by 
\[H=\{h=0\} .\]
\end{notation}

\begin{definition}
	The \textbf{span of} \(\mH\) is the subspace of \(V^{*}\) generated by all linear functions \(h\) such that \(H=\{h=0\}\) belongs to \(\mH\). We  denote it by
	\[W(\mH) = \Span \{h \in V^* , \,\ H \in \mH\} . \]
\end{definition}

\begin{remark}
	The subspaces
	\(T(\mH) \subset V\) and \(W(\mH) \subset V^*\) are annihilators of each other.
\end{remark}

\begin{remark}
	If \(\mH=\emptyset\) then \(T(\mH)=V\) and \(W(\mH)=\{0\}\).
\end{remark}

\begin{definition}
The \textbf{rank} of \(\mH\) is the dimension of the span \(W(\mH)\), 
\begin{equation*}
r(\mH)= \dim W(\mH) = \codim T(\mH) . 
\end{equation*}	
\end{definition}

\begin{definition}\label{def:essarr}
	\((V,\mH)\) is \textbf{essential} if \(T(\mH)=\{0\}\), or equivalently \(r(\mH) = \dim V\).
\end{definition}

\begin{definition}
	Let \(S \subset V\) be a linear subspace such that \(S \subset T(\mH)\). The \textbf{quotient arrangement} \(\mH/S\) is defined as
	\[\mH/S = \{H/S, \,\ H \in \mH\},\]
	it is an arrangement of hyperplanes in \(V/S\).
\end{definition}

\begin{remark}
The quotient arrangement \(\mH/T(\mH)\) is always essential.	
\end{remark}

\begin{definition}
Let \(L \in \mL(\mH)\), the \textbf{localization} of \(\mH\) at \(L\) is 
\[\mH_L = \{H \in \mH, \,\ L \subset H\} .\]	
\end{definition}

\begin{note}
	\(\mL(\mH_L)\) consists of the subspaces in \(\mL(\mH)\) that contain \(L\).
\end{note}

\begin{lemma}\label{lem:HLMisHM}
	If \(M \in \mL(\mH_L)\) then \((\mH_L)_M =\mH_M\).
\end{lemma}
\begin{proof}
	Clear.
\end{proof}

\begin{definition}
Let \(L \in \mL(\mH)\), the \textbf{induced arrangement} on \(L\) is
\[\mH^L = \{H \cap L, \,\ H \in \mH \setminus \mH_L\} . \]	
\end{definition}

\begin{note}
	\(\mL(\mH_L)\) consists of the subspaces in \(\mL(\mH)\) which are contained in \(L\).
\end{note}

\begin{lemma}
	If \(M \in \mL(\mH^L)\) then \((\mH^L)^M = \mH^M\).
\end{lemma}

\begin{proof}
	Clear.
\end{proof}

\begin{definition}
	For \(L \in \mL(\mH)\) we denote its smooth part by 
	\[L^{\circ} = L \setminus \bigcup_{H \in \mH \setminus \mH_L}H.\]
	If \(L=V\) then \(V^{\circ}\) is the hyperplane complement. The sets \(L^{\circ}\) give a \textbf{stratification} of \(V\), meaning that \[V = \bigcup_{L \in \mL(\mH)} L^{\circ}\] 
	is a disjoint union of smooth submanifolds.
\end{definition}

\subsection{Irreducible arrangements}\label{sect:irredarr}

\begin{definition}
	A \textbf{u-plus decomposition} of \(\mH\), denoted by
	\begin{equation}\label{eq:udec}
	\mH = \mH_1 \uplus \ldots \uplus \mH_k ,
	\end{equation}	
	 is a collection of subsets \(\mH_i \subset \mH\) such that \(\cup_{i=1}^k\mH_i = \mH\) and
	\begin{equation}\label{eq:udec2}
	W(\mH)=W(\mH_1) \oplus \ldots \oplus W(\mH_k).
	\end{equation}
\end{definition}

\begin{remark}
	It follows from Equation \eqref{eq:udec2} that the union \(\cup_{i=1}^k\mH_i\) must be disjoint.  It also follows that the ranks satisfy
	\begin{equation}\label{eq:sumranks}
		r(\mH) = r(\mH_1) + \ldots + r(\mH_k) .
	\end{equation}
	In particular, if all \(\mH_i\) are non-empty, we must have \(k \leq r(\mH)\).
\end{remark}

\begin{note}
	Providing a u-plus decomposition of an arrangement is essentially the same thing as expressing the arrangement as a product. Indeed,
	after a suitable linear change of coordinates in \(V\), the defining equations of hyperplanes in \(\mH_i\) and \(\mH_j\) share no common variable for \(i\neq j\). We will say more about this in Section \ref{sect:productarr}.
\end{note}

\begin{lemma}\label{lem:HiisHTi}
	If \(\mH= \mH_1 \uplus \ldots \uplus \mH_k\) with centres \(T_i = T(\mH_i)\) then \(\mH_i=\mH_{T_i}\).
\end{lemma}

\begin{proof}
	By definition \(\mH_i \subset \mH_{T_i}\). Let \(H = \{h=0\} \in \mH_{T_i}\), we know that \(H \in \mH_j\) for some \(j\) and our goal is to show that \(j=i\). Note that, since \(T_i \subset H\), we have \(h|_{T_i}=0\) and \(h \in \Ann T_i = W(\mH_i)\). On the other hand,  since \(H \in \mH_j\), we also have \(h \in W(\mH_j)\). The subspaces \(W(\mH_i)\) and \(W(\mH_j)\) have zero intersection when \(i\neq j\), so it follows that \(i=j\).
\end{proof}

\begin{lemma}\label{lem:ann}
	Assume \(\mH = \mH_1 \cup \mH_2\) then
	\begin{equation*} 
		\mH = \mH_1 \uplus \mH_2 \iff  W(\mH_1) \cap W(\mH_2) = \{0\} \iff  T(\mH_1) + T(\mH_2) = V .
	\end{equation*}
\end{lemma}

\begin{proof}
	The first equivalence is a direct consequence of the definition, the second follows from the identity
	\[ \Ann (A \cap B) = \Ann (A) + \Ann (B) \]
	for annihilators of subspaces.
\end{proof}

\begin{definition}\label{def:irrarr}
	\((V, \mH)\) is called \textbf{reducible} if there are \emph{non-empty} subsets \(\mH_1, \mH_2 \subset \mH\) such that \(\mH=\mH_1 \uplus \mH_2\).
	 The arrangement is \textbf{irreducible} if it is not reducible. 
\end{definition}

It follows from Lemmas \ref{lem:HiisHTi} and \ref{lem:ann} that \(\mH\) is reducible if and only if there are two non-zero proper subspaces \(L_1, L_2 \in \mL(\mH)\) such that \(L_1+L_2=V\) and \(\mH = \mH_{L_1} \cup \mH_{L_2}\).

\begin{remark}
	An arrangement \(\mH\) is irreducible if and only if
	\(\mH/T(\mH)\) is.
\end{remark}

\begin{lemma} \label{lem:uplussubarrang}
	Assume \(\mH = \mH_1 \uplus \mH_2\) and let \(\mH' \subset \mH\). Then \(\mH' = (\mH' \cap \mH_1) \uplus (\mH' \cap \mH_2)\).
\end{lemma}

\begin{proof}
	Clear.
\end{proof}

\begin{lemma}\label{lem:uniq0}
	Assume \(\mH = \mH_1 \uplus \mH_2\) and let \(\mH' \subset \mH\) such that \(\mH'\) is irreducible. Then either \(\mH' \subset \mH_1\) or \(\mH' \subset \mH_2\).
\end{lemma}

\begin{proof}
	Suppose not. By Lemma \ref{lem:uplussubarrang} we have \(\mH' = (\mH'\cap \mH_1) \uplus (\mH' \cap \mH_2)\) with both \(\mH'\cap \mH_i\) non-empty which contradicts irreducibility.
\end{proof}

\begin{definition}
	Let \(\mH\) be a non-empty arrangement. We say that \eqref{eq:udec} is an \textbf{irreducible decomposition} of \(\mH\) if every \(\mH_i\) is non-empty and irreducible. 
\end{definition}

\begin{lemma}
	Every non-empty arrangement \(\mH\) admits an irreducible decomposition 
	\[\mH= \mH_1 \uplus \ldots \uplus \mH_k .\] 
	This decomposition is unique up to re-labelling of the subsets \(\mH_i \subset \mH\).
\end{lemma}

\begin{proof}
	Existence is clear. 
	Fix an irreducible decomposition \(\mH= \mH_1 \uplus \ldots \uplus \mH_k\). In order to show uniqueness we use the following consequence of Lemma \ref{lem:uniq0}: if \(\mH'\subset \mH\) is an irreducible subset then \(\mH' \subset \mH_i\) for some \(i\). If we have another irreducible decomposition
	\(\mH= \tilde{\mH}_1 \uplus \ldots \uplus \tilde{\mH}_l\)
	then each \(\tilde{\mH_j}\) must be contained in some \(\mH_i\) and vice-versa, so the two decompositions agree after re-labelling.
\end{proof}

\begin{lemma}\label{lem:automorphisms}
	Let \((V, \mH)\) be an essential and irreducible arrangement and let \(A \in \End V\) be such that \(A(H)\subset H\) for every \(H \in \mH\), then \(A\) is multiplication by a scalar.
\end{lemma}

\begin{proof}
	Consider the dual \(A^* \in \End V^*\). If \(h \in V^*\) is a defining equation for \(H \in \mH\) then \(A^*h=\lambda_H h\) for some \(\lambda_H \in \C\), as follows from \(A(H)\subset H\). Fix \(H_0 \in \mH\) and let \(\lambda=\lambda_{H_0}\). Since \(\mH\) is essential there is a basis of \(V^*\) made of defining linear equations of hyperplanes in the arrangement. In particular, \(A^*\) is diagonalizable and we can write \(V^*=W_1 \oplus W_2\) where \(W_1\) is the \(\lambda\)-eigenspace of \(A^*\) and \(W_2\) is the sum of all other eigenspaces. We obtain a u-plus decomposition \(\mH=\mH_1 \uplus \mH_2\) given by \(\mH_i=\{H \in \mH, \,\ h \in W_i\}\) for \(i=1,2\). Since \(\mH\) is irreducible we must have \(\mH=\mH_1\), therefore \(A^*=\lambda \cdot \Id_{V^*}\) and the statement follows.
\end{proof}

\subsubsection{{\large Irreducible intersections and components}}

\begin{lemma}\label{lem:ABirred}
	Let \(\mA \subset \mB\) be arrangements with equal centres \(T(\mA)=T(\mB)\). If \(\mA\) is irreducible then \(\mB\) is also irreducible. In particular, if \(\mA\) and \(\mB\) are both essential, then \(\mA\) irreducible implies \(\mB\) also is.
\end{lemma}

\begin{proof}
	Let \(\mB = \mH_1 \uplus \mH_2\) with \(T(\mH_1) + T(\mH_2)=V\). Since \(\mA\) is irreducible we must have \(\mA \subset \mH_1\), say. In particular the centres satisfy
	\[T(\mA) \supset T(\mH_1)\supset T(\mB) .\]
	Since \(T(\mA)=T(\mB)\) we get that \(T(\mH_1) = T(\mB) =T(\mH_1)\cap T(\mH_2)\), so \(T(\mH_1)\subset T(\mH_2)\). The condition \(T(\mH_1)+T(\mH_2)=V\) implies that \(T(\mH_2)=V\); which is to say \(\mH_2=\emptyset\).
\end{proof}

\begin{definition}\label{def:Lirr}
	\(L \in \mL(\mH)\) is an \textbf{irreducible intersection} if \(\mH_L\) is irreducible. The set of all irreducible intersections is denoted by \(\mL_{irr}(\mH)\).
\end{definition}

\begin{lemma} \label{lem:irredHL}
	Let \(L \in \mL(\mH)\), then
	\[\mL_{irr}(\mH_L) = \{M \in \mL_{irr}(\mH), \,\ L \subset M\} . \]
\end{lemma}

\begin{proof}
	If \(M \in \mL_{irr}(\mH_L)\) then \(M \in \mL(\mH)\), \(L \subset M\) and \((\mH_L)_M = \mH_M\) is irreducible, hence \(M \in \mL_{irr}(\mH)\). The reverse inclusion follows the same way. 
\end{proof}

\begin{lemma}
	Let \(\mH = \mH_1 \uplus \ldots \uplus \mH_k\) be an irreducible decomposition, then 
	\[\mL_{irr}(\mH) =  \bigcup_{i=1}^k \mL_{irr}(\mH_i) \]
	and the union is disjoint.
\end{lemma}

\begin{proof}
	Let \(L \in \mL_{irr}(\mH)\), then \(\mH_L\) is irreducible and there must be exactly one \(i\) such that \(\mH_L \subset \mH_i\). It follows that \((\mH_i)_L=\mH_L\) is irreducible and therefore \(L \in \mL_{irr}(\mH_i)\). Conversely, if \(L \in \mL_{irr}(\mH_i)\) then \(T_i=T(\mH_i)\subset L\) and \((\mH_i)_L\) is irreducible. By Lemma \ref{lem:HiisHTi} \(\mH_i=\mH_{T_i}\) so \((\mH_i)_L = (\mH_{T_i})_L\). By Lemma \ref{lem:HLMisHM} \((\mH_{T_i})_L = \mH_L\). We conclude that \(\mH_L=(\mH_i)_L\) is irreducible; hence \(L \in \mL_{irr}(\mH)\).
\end{proof}

\begin{definition}
	Let \(L \in \mL(\mH)\) and
	let \(\mH_L = \mH_1 \uplus \ldots \uplus \mH_k\) be its irreducible decomposition. The subspaces 
	\[L_i = \bigcap_{H \in \mH_i} H \in \mL_{irr}(\mH) \]
	are called the \textbf{irreducible components} of \(L\).
\end{definition}

\begin{note}
Clearly \(L \subset L_i\) for every \(i\) and \(L\)
is the common intersection of all \(L_i\).
Moreover, if \(L' \in \mL_{irr}(\mH)\) and \(L \subset L'\) then there is exactly one \(i\) such that \(L_i \subset L'\). 	
\end{note}

\subsubsection{{\large Products of arrangements}}\label{sect:productarr}

\begin{definition}\label{def:prodarr}
	The product of two arrangements \((V_1, \mH_1)\) and \((V_2, \mH_2)\) is the arrangement in the product vector space \(V_1 \times V_2\) defined by
	\[\mH_1 \times \mH_2 = \{{H}_1 \times V_2, \,\ {H}_1 \in {\mH}_1 \} \bigcup \{V_1 \times {H}_2, \,\  {H}_2 \in {\mH}_2 \} .  \]
\end{definition}

\begin{remark}
	In the context of Definition \ref{def:prodarr} the hyperplane  complements satisfy
	\[(V_1\times V_2)^{\circ} = V_1^{\circ} \times V_2^{\circ} .\]
	In words, the complement of a product arrangement is the product of the complements of its factors.
\end{remark}

\begin{lemma}
	The map \((L_1, L_2) \mapsto L_1 \times L_2\) defines a natural bijection between \(\mL(\mH_1) \times \mL(\mH_2)\) and \(\mL(\mH_1 \times \mH_2)\). Moreover, \(L \in \mL_{irr}(\mH_1 \times \mH_2)\) if and only if \(L=L_1 \times V_2\) with \(L_1 \in \mL_{irr}(\mH_1)\) or \(L=V_1 \times L_2\) with
	\(L_2 \in \mL_{irr}(\mH_2)\).
\end{lemma}

\begin{proof}
	Easy, see \cite[Proposition 2.14]{OT1}.
\end{proof}

\begin{definition}\label{def:isoarrang}
	Two arrangements \((V_1, \mH_1)\) and \((V_2, \mH_2)\) are isomorphic if there is a linear isomorphism \(\Phi: V_1 \to V_2\) with \(\Phi(\mH_1)=\mH_2\). We write it as \((V_1, \mH_1) \equiv (V_2, \mH_2)\).
\end{definition}

\begin{lemma}\label{lem:HLdec}
	Take \(L \in \mL(\mH)\). Let \(\mH_L= \mH_1 \uplus \ldots \uplus \mH_k\) be the irreducible decomposition and \(L_i=T(\mH_i)\) its irreducible components. Then we can choose linear coordinates
	\begin{equation}
	(x_1, \ldots, x_n) \,:\, V \xrightarrow{\sim} \C^n	
	\end{equation}
	such that 
	\begin{equation}\label{eq:HLdec}
	(\C^n, \, \mH_L) = (\C^{\dim L}, \, \emptyset) \times (\C^{n_1}, \, \bar{\mH}_1) \times \ldots \times (\C^{n_k}, \, \bar{\mH}_k) 	
	\end{equation}
	where the factors
	\begin{equation}
		(\C^{n_j}, \bar{\mH}_j) \equiv (V/L_j, \, {\mH}_j/L_j)
	\end{equation}
	are essential and irreducible.
\end{lemma}

\begin{proof}
	By definition of a u-plus decomposition, we have 
	\begin{equation}\label{eq:Wdec}
	W(\mH_L) = W(\mH_1) \oplus \ldots \oplus W(\mH_k) .	
	\end{equation}
	Write \(d = \dim L\) and \(n_i=\dim W(\mH_i) = \codim L_i\).
	Equation \eqref{eq:Wdec} implies that
	\(n-d=n_1+\ldots+n_k\).
	
	Introduce index sets 
	\begin{equation}\label{eq:indexset}
	\begin{gathered}
	I_1=\{1, \ldots, n_1\}, \,\ I_2 = \{n_1+1, \ldots, n_1+n_2\}, \ldots, \\
	I_k=\{n_1+\ldots+n_{k-1}+1, \ldots, n_1+\ldots+n_k\} 
	\end{gathered}
	\end{equation}
	and
	\begin{equation}\label{eq:I0}
	I_0 = \{n-d+1, \ldots, n\} .
	\end{equation}
	Note that \(|I_0|=d\) and \(|I_j|=n_j\) for \(1\leq j \leq k\). We have a disjoint union
	\[I_0 \cup I_1 \cup \ldots \cup I_k = \{1, \ldots, n\} .\]

	Take a basis \(\{x_1, \ldots, x_n\}\) of \(V^*\) such that
	\begin{equation}\label{eq:Wspan}
		\Span \{x_i, \,\ i \in I_j\} = W(\mH_j) \,\ \mbox{ for } \,\ 1 \leq j \leq k .
	\end{equation}
	We use the basis \(\{x_1, \ldots, x_n\}\) to identify \(V\cong \C^n\).
	If \(H \in \mH_j\) then 
	\[H = \{h_j(x_i, \,\ i \in I_j)=0\} = \C^{n_1} \times \ldots \times \bar{H} \times \ldots \times \C^{n_k} \times \C^d \]
	with \(\bar{H} \subset \C^{n_j}\).
	Define
	\[\bar{\mH}_j = \{\bar{H} \subset \C^{n_j}, \,\ H \in \mH_j\} \,\ \mbox{ for } 1\leq j \leq k . \]
	We have isomorphisms
	\begin{equation}\label{eq:quotiso}
	(V/L_j, \,\, \mH_j/L_j) \equiv (\C^{n_j}, \,\, \bar{\mH}_j) \,\, \mbox{ for }	1 \leq j \leq k
	\end{equation}
	given by the coordinates \(\{x_i, \,\ i \in I_j\}\). On the other hand, the coordinates \(\{x_i, \,\ i \in I_0\}\) identify the centre of the arrangement \(L=T(\mH_L)\) with \(\C^d\). Altogether, we get that
	\[(\C^n, \mH_L) = (\C^d, \emptyset) \times (\C^{n_1}, \bar{\mH}_1) \times \ldots \times (\C^{n_k}, \bar{\mH}_k) .\]
	By construction, the arrangements \((\C^{n_j}, \bar{\mH}_j)\equiv (V/L_j, \mH_j/L_j)\) are essential and irreducible.
\end{proof}

\begin{corollary}\label{lem:irrsameprod}
	\(\mH\) is reducible if and only if it is isomorphic to the product of two non-empty arrangements.
\end{corollary}

\begin{proof}
	This is a direct consequence of Lemma \ref{lem:HLdec} applied to \(L=T(\mH)\), so \(\mH_L=\mH\).
\end{proof}

\begin{note}
	Our notion of irreducible arrangement differs slightly from \cite[Definition 2.15]{OT1}. According to us \((V, \mH)\) is irreducible if and only if \((V, \mH) \times (V', \emptyset)\) is. In particular, if \(\mH\) is irreducible  then it is not necessarily essential, e.g. take \(\mH\) to be a single hyperplane or three hyperplanes intersecting along a codimension two subspace. Irreducible arrangements (in our sense) are called indecomposable in \cite{OT2}. Our terminology agrees with one used in \cite{CHL}. 
\end{note}

\begin{definition}\label{def:automorphimsH}
	The collection of all isomorphisms of an arrangement \((V, \mH)\) with itself form a subgroup of \(GL(V)\) which we call the automorphism group of \(\mH\) and denote it by \(\Aut(\mH)\). We denote by \(\Aut_0(\mH)\) the subgroup made of automorphisms which preserve each of the hyperplanes in the arrangement.
\end{definition}

\begin{corollary}
	An arrangement \(\mH\) is essential and irreducible if and only if \(\Aut_0(\mH)=\C^*\).
\end{corollary}

\begin{proof}
	One direction is established in Lemma \ref{lem:automorphisms}. Conversely, if \(\mH\) is either non-essential or reducible then the product decomposition \eqref{eq:HLdec} applied to \(\mH=\mH_{T(\mH)}\) implies that 
	\((\C^*)^2 \subset \Aut_0(\mH)\) as we can act by scalar multiplication separately on different factors.
\end{proof}

\section{Affine structures on arrangement complements}\label{sect:affst}

\begin{definition}\label{def:affine}
	An affine structure on a complex manifold \(X\) consists of a maximal atlas made of holomorphic coordinate charts with transition functions given by complex affine linear transformations.
\end{definition}

\begin{lemma}[{\cite[Lemma 1.1]{CHL}}]\label{lem:affineconnectioncorresp}
	Given a complex manifold \(X\), there is a natural correspondence between affine structures on \(X\) and flat torsion free connections on \(TX\).
\end{lemma}

\begin{proof}
	In one direction, given a torsion free flat connection \(\nabla\) we consider local coordinate charts such that its coordinate vector fields form a parallel frame; transition functions among such coordinate systems are affine linear. In the other direction, given an affine structure, there is a unique connection defined by requiring the coordinate vector fields of affine charts to be parallel.
\end{proof}

In this section we study affine structures on the complement of a hyperplane arrangement given by a particular type of meromorphic connections with simple poles at the hyperplanes of the arrangement; we call these `standard connections'. These connections  play a key role in \cite[Section 2]{CHL}.
Definition \ref{def:localizationconection} introduces the standard connection \(\nabla^L\), the localization of the connection \(\nabla\) at a subspace \(L\).
Our main result is Proposition \ref{prop:nablaLsplit} which gives an affine splitting for the connection \(\nabla^L\). The main technical work is done in Proposition \ref{prop:AL}, which follows from elementary linear algebra and exploits the Basic Assumptions \ref{ass:basic} satisfied by every flat torsion free standard connection. We introduce Euler vector fields for standard connections in Section \ref{sect:eulervf} and show that under suitable numerical conditions these vector fields do exist \ref{lem:eulstndcon} and are unique \ref{lem:eulvfuniq}. Euler vector fields play a crucial role in forthcoming parts of our article.

\subsection{Standard connections}\label{sect:ftfstnd}

Recall that \(V\) is a finite dimensional complex vector space and \(\mH\) is a finite collection of hyperplanes \(H\) with defining linear equations \(h\).

\begin{definition}\label{def:stnd}
	A \textbf{standard connection} on the trivial vector bundle over \(V\) with fibre a vector space \(E\) is a connection of the form
	\begin{equation}\label{eq:stcon}
		\nabla = d - \sum_{H \in \mH} A_H \frac{dh}{h}
	\end{equation}
	where \(A_H \in \End E\) and \(d\) is the trivial connection given by the exterior derivative.
\end{definition}

\begin{notation}
	We write Equation \eqref{eq:stcon} in the form \(\nabla = d - \Omega\) where
	\begin{equation}
		\Omega = \sum_{H \in \mH} A_H \frac{dh}{h} 
	\end{equation}
	is the connection form.
	Note that \(d\Omega=0\), therefore a standard connection \(\nabla\) is flat if and only if \(\Omega \wedge \Omega =0\).
\end{notation}

\begin{definition} \label{def:residue}
	Suppose that \(\nabla\) is a standard connection given by Equation \eqref{eq:stcon}.
	For \(L \in \mL(\mH)\) we define the \textbf{residue} of \(\nabla\) at \(L\) to be the element of \(\End E\) given by
	\begin{equation} \label{eq:residue}
	A_L = \sum_{H \in \mH_L} A_H . 	
	\end{equation}
\end{definition}

\begin{assumptions}\label{ass:basic}
	A collection of endomorphisms \(A_H \in \End(V)\) indexed by \(H \in \mH\) satisfies the \textbf{Basic Assumptions} if
	\begin{description}
		\item[(F)] \([A_L, A_M]=0\) whenever \(L, M \in \mL(\mH)\) and \(L \subset M\);
		\item[(T)] \(H \subset \ker A_H\) for all \(H \in \mH\).
	\end{description}
\end{assumptions}

The letters \textbf{(F)} and \textbf{(T)} in Assumptions \ref{ass:basic} stand for flat and torsion free.
The next result is taken from {\cite[Proposition 2.3]{CHL}}. Because of its relevance, we give a complete proof. 

\begin{proposition}\label{prop:ftfstcon}
	Let \(\nabla\) be a standard connection on \(TV\) given by Equation \eqref{eq:stcon}. Then \(\nabla\) is flat torsion free if and only if the Basic Assumptions \ref{ass:basic} are satisfied.
\end{proposition}

\begin{proof}
	Follows from Proposition \ref{prop:comm} and Lemma \ref{lem:torfree} in the next sections.
\end{proof}

\subsubsection{{\large Flatness \(\iff\) Commutation relations}}

\begin{proposition}\label{prop:comm}
	Let \(\nabla\) be a standard connection given by Equation \eqref{eq:stcon}, then the following conditions are equivalent:
	\begin{itemize}
		\item[(i)] \(\Omega \wedge \Omega =0\);
		\item[(ii)] \([A_L, A_H]=0\) for every \(L \in \mL(H)\) of \(\codim L =2\) and every \(H \in \mH_L\);
		\item[(iii)] \([A_M, A_N]=0\) for every \(M, N \in \mL(\mH)\) with \(M \subset N\).
	\end{itemize}
\end{proposition}
\begin{proof}
	This is a consequence of Lemmas \ref{lem:comm0} and \ref{lem:comm1}.
\end{proof}

\begin{lemma}\label{lem:comm0}
	Conditions (ii) and (iii) in Proposition \ref{prop:comm} are equivalent.
\end{lemma}

\begin{proof}
	It is enough to show that (ii) implies (iii). Moreover, from the definition of \(A_N\) in Equation \ref{eq:residue}, we can clearly assume that \(N=H_0\) is a hyperplane. Define an equivalence relation on \(\mH'_M=\mH_M - \{H_0\}\) by \(H \sim H'\) whenever \(H \cap H_0 = H' \cap H_0\), it gives a partition of \(\mH'_M\) into disjoint subsets \(\mH'_M=\sqcup_i \mH_i\).
	Let \(L_i\) be the common intersection of all hyperplanes in \(\mH_i\) together with \(H_0\). By construction \(\codim L_i = 2\) and \(L_i \subset H_0\). Moreover, we see that
	\[[A_M, A_{H_0}] = \sum_i [A_{L_i}, A_{H_0}] = 0 . \qedhere\] 
\end{proof}

For \(L \in \mL(\mH)\) we set 
\[\Omega^L = \sum_{H \in \mH_L} A_H \frac{dh}{h} .\]

\begin{lemma}\label{lem:cod2i}
	The following identity holds
	\begin{equation*}
		\Omega \wedge \Omega = \sum_{\codim L=2} \Omega^L \wedge \Omega^L
	\end{equation*}
	where the sum on the right hand side is over all codimension two intersections \(L \in \mL(\mH)\).
\end{lemma}

\begin{proof}
	Clearly
	\[\Omega \wedge \Omega = \sum_{H, H'} A_H A_{H'} \frac{dh}{h} \wedge \frac{d h'}{h'} \]
	where the sum runs over all ordered tuples of \(\mH \times \mH \setminus \{\mbox{diagonal}\}\). On the other hand we have a disjoint union 
	\[\mH \times \mH \setminus \{\mbox{diagonal}\} = \bigcup_{\codim L =2} \{(H, H'), \,\ L = H \cap H' \} \]
	and we get that
	\begin{align*}
		\sum_{H, H'} A_H A_{H'} \frac{dh}{h} \wedge \frac{d h'}{h'} &=
		\sum_{\codim L=2} \left( \sum_{L=H\cap H'} A_H A_{H'} \frac{dh}{h} \wedge \frac{d h'}{h'} \right) \\
		&= \sum_{\codim L=2} \Omega^L \wedge \Omega^L.
		\qedhere
	\end{align*}
\end{proof}

\begin{lemma} \label{lem:cod2ii}
	Fix \(L \in \mL(\mH)\) with \(\codim L=2\). The following two conditions are equivalent:
	\begin{itemize}
		\item[(i)] \([A_L, A_H]=0\) for every \(H \in \mH_L\);
		\item[(ii)] \(\Omega^L \wedge \Omega^L = 0\).
	\end{itemize}
\end{lemma}

\begin{proof}
	This is essentially a two dimensional case. We can simplify notation to \(\mH=\mH_L\) and \(\Omega=\Omega^L\). Take linear coordinates so that \(L=\{x=y=0\}\) and \(h_i= y -a_ix\) are the equations of the hyperplanes \(\mH=\{H_1, \ldots, H_N\}\). We introduce variables \(u, v\) so that \(x=u\), \(y=uv\). This gives us
	\[\Omega = A_L \frac{du}{u} + \sum_i A_{H_i} \frac{dv}{v-a_i} \]
	and
	\[\Omega \wedge \Omega = \left(\sum_i [A_L, A_{H_i}] \frac{1}{v-a_i} \right) \frac{du\wedge dv}{u}, \]
	which vanishes if and only if the meromorphic function inside parenthesis does; equivalently \([A_L, A_{H_i}]=0\) for all \(i\).
\end{proof}

\begin{lemma} \label{lem:comm1}
	The following two conditions are equivalent:
	\begin{itemize}
		\item[(i)] \([A_L, A_H]=0\) for every \(L \in \mL(H)\) of \(\codim L =2\) and every \(H \in \mH_L\);
		\item[(ii)] \(\Omega \wedge \Omega = 0\).
	\end{itemize}
\end{lemma}

\begin{proof}
	The direction (i) implies (ii) follows from Lemmas \ref{lem:cod2ii} and  \ref{lem:cod2i}. To show the converse, we fix \(L \in \mL(\mH)\) and a smooth point \(p \in L^{\circ}\). Take coordinates centred at \(p\) so that \(L=\{x=y=0\}\) and proceed as in the proof of Lemma \ref{lem:cod2ii} to get
	\[\Omega \wedge \Omega = \left(\sum_i [A_L, A_{H_i}] \frac{1}{v-a_i} + \mbox{(hol.)} \right) \frac{du\wedge dv}{u}, \]
	where (hol.) denotes some holomorphic matrix-valued function and \(H_i\) are the hyperplanes that contain \(L\). The same reasoning gives us that \(\Omega \wedge \Omega = 0\) implies \([A_L, A_{H_i}]=0\) for every \(i\).
\end{proof}

\subsubsection{{\large Torsion free \(\iff\) \(H \subset \ker A_H\) }}

\begin{definition}
	The torsion of an affine connection is the \((2,1)\)-tensor given by
	\[T(X, Y) = \nabla_X Y - \nabla_Y X - [X, Y] . \]
\end{definition}

\begin{lemma}[cf. Lemma \ref{lem:apendixTDflat}]\label{lem:torfree}
	Let \(\nabla\) be a standard connection on \(TV\) given by Equation \eqref{eq:stcon} with \(d = \) the Euclidean connection and \(A_H \in \End V\). Then
	\(\nabla\) is torsion free (i.e. \(T\equiv 0\)) if and only if \(H \subset \ker A_H\) for every \(H \in \mH\).
\end{lemma} 

\begin{proof}
	Plugging \(\nabla=d-\Omega\) into the formula for \(T(X,Y)\) and using the fact that \(d\) is torsion free, we get that
	\begin{equation} \label{eq:tor}
		T(X, Y) = \sum_{H \in \mH} \left( A_H(X) dh(Y) - A_H(Y) dh(X) \right) \frac{1}{h} .
	\end{equation}
	Suppose that \(H \subset \ker A_H\) for every \(H \in \mH\). This implies that \(A_H(X) = dh(X) \cdot n_H\) for some \(n_H \in V\), see Lemma \ref{lem:rk1} that follows. In particular, the expression
	\[A_H(X) dh(Y) = dh(X) dh(Y) \cdot n_H \]
	is symmetric in \(X, Y\) and each of the terms on the right hand side of Equation \eqref{eq:tor} vanishes, so \(T \equiv 0\). Conversely, if \(T\equiv 0\) then we use Equation \eqref{eq:tor} for \(X, Y\) constant vector fields. Since the rational functions \(1/h\) are linearly independent over \(\C\) we get that
	\[A_H(X)dh(Y)=A_H(Y)dh(X)\]
	for every \(X, Y \in V\). Take \(X \in H\) and \(Y \notin H\) to conclude \(A_H(X)=0\); hence \(H \subset \ker A_H\).
\end{proof}

\subsection{Residues and normal subspaces}\label{sect:resnorm}

Let \(\nabla\) be a flat torsion free standard connection with residues at hyperplanes \(\{A_H, \,\ H \in \mH\}\).
In this section we introduce the notion of `normal subspace' \(L^{\perp}\) (depending on \(\nabla\)) to a given \(L \in \mL(\mH)\)
and prove a direct sum decomposition result (Proposition \ref{prop:AL}) which will be used later in Section \ref{sec:locprod}.

\subsubsection{{\large Rank one endomorphisms}}

\begin{lemma}\label{lem:rk1}
	Let \(A_H \in \End V\) be a non-zero endomorphism and let \(H \subset V\) be a hyperplane with defining linear equation \(H=\{h=0\}\). If \(\ker A_H = H\) then 
	\begin{equation} \label{eq:Amat}
		A_H(X)=dh(X) \cdot n_H \hspace{2mm} \mbox{ for all } X \in V,
	\end{equation}
	where \(n_H \in V\) is a non-zero vector such that \(\img A = \C \cdot n_H\). Moreover, we have the following two cases:
	 \begin{itemize}
	 	\item Semi-simple: if \(n_H \notin H\) then \(A_H(n_H)= a_H n_H\) and \(a_H= \tr A_H \neq 0\).
	 	\item Nilpotent: if \(n_H \in H\) then \(A_H^2=0\) and \(\tr A_H=0\).
	 \end{itemize}
\end{lemma}

%\begin{notation}
%	We can write Equation \eqref{eq:Amat} as \(A_H=dh \otimes n_H\), which is consistent with the natural isomorphism between \(V^*\otimes V\) and \(\End V\).
%\end{notation}

\begin{proof}
	Take \(u \in V\) with \(dh(u)=1\). Define \(n_H= A_H(u)\)
	and \(\tilde{A} (X) = dh(X) \cdot n_H\) for all \(X \in V\). Comparing the values of \(A_H\) and  \(\tilde{A}\) on a basis of \(V\) given by \(\{u, e_2, \ldots, e_n\}\) where \(H = \Span \{e_2, \ldots, e_{n}\}\), we see that \(A_H= \tilde{A}\) and Equation \eqref{eq:Amat} follows. If \(n_H \notin H\) then \(A_H n_H = a_H n_H\) with \(a_H = dh(n_H) \neq 0\) and the matrix representing \(A_H\) with respect to the basis \(\{n_H, e_2, \ldots, e_n\}\) is diagonal with entries \((a_H, 0, \ldots, 0)\); which proves the semi-simple case. On the other hand, if \(n_H \in H\) then \(A_H^2(u)=A_H(n_H)=0\) and the nilpotent case holds.
\end{proof}

\begin{lemma} \label{lem:rk1invsub}
	Let \(A_H \in \End V\) be a rank one endomorphism as in Lemma \ref{lem:rk1}. Suppose \(E \subset V\) is an invariant subspace of \(A_H\), this is \(A_H (E) \subset E\). Then either \(E \subset H\) or \(n_H \in E\). 
\end{lemma}

\begin{proof}
	If \(E \not\subset H\) then \(\C \cdot n_H \subset A_H(E) \subset E\).
\end{proof}

\subsubsection{{\large Reducible intersections}}

\begin{definition}
	We say that two distinct hyperplanes \(H_1, H_2\) of the arrangement \(\mH\) have \textbf{reducible intersection}, and denote it by \(H_1 \pitchfork H_2\),
	if the codimension two subspace \(L=H_1 \cap H_2\) is reducible. Equivalently, 
	\[H_1 \pitchfork H_2 \iff \mH_L=\{H_1, H_2\} .\]
\end{definition}

\begin{lemma}\label{lem:transvinter}
	If \(\mH = \mH_1 \uplus \mH_2\) and \(H_i \in \mH_i\) for \(i=1,2\) then \(H_1 \pitchfork H_2\).
\end{lemma}

\begin{proof}
	Let \(L=H_1 \cap H_2\). By Lemma \ref{lem:uplussubarrang}
	\[\mH_L= (\mH_L \cap \mH_1) \uplus (\mH_L \cap \mH_2) .\]
	Note that \(H_i \in \mH_L \cap \mH_i \neq \emptyset\) for \(i=1,2\) and therefore \(\mH_L\) is reducible.
\end{proof}

\begin{lemma}\label{lem:normalintersect}
	Let \(A_i \in \End V\) for \(i=1,2\) be rank one endomorphisms with \(\img A_i = \C \cdot n_i\). Assume that the hyperplanes  \(H_i = \ker A_i\) are distinct. If \([A_1, A_2]=0\) then \(n_{1} \in H_2\) and \(n_{2} \in H_1\).
\end{lemma}

\begin{proof}
	Since the endomorphisms commute,
	the subspace \(H_2 = \ker A_2\) is invariant under \(A_{H_1}\). It follows by Lemma \ref{lem:rk1invsub} that either \(H_2 \subset H_1\) or \(n_{1} \in H_2\). The first option can not happen because \(H_1 \neq H_2\), so \(n_{1} \in H_2\); similarly \(n_{2 } \in H_1\).
\end{proof}

\begin{lemma} \label{lem:trint}
	Suppose that the Basic Assumptions \ref{ass:basic} are satisfied. If \(H_1 \pitchfork H_2\) then \([A_{H_1}, A_{H_2}]=0\). Moreover, if \(A_{H_1}\) and \(A_{H_2}\) are non-zero then  \(n_{H_1} \in H_2\), so \(A_{H_1}\) preserves \(H_2\), and vice-versa.
\end{lemma}

\begin{proof}
	Let \(L=H_1 \cap H_2\), so  \(\mH_L=\{H_1, H_2\}\) and \(A_L=A_{H_1}+A_{H_2}\). The flatness condition \textbf{(F)} implies that \([A_L, A_{H_2}] = [A_{H_1}, A_{H_2}]\) vanishes and the conclusion follows from Lemma \ref{lem:normalintersect}.
\end{proof}

\begin{remark}
	Suppose that \(\{A_H, \,\, H \in \mH\}\) satisfies the Basic Assumptions \ref{ass:basic} and the residues \(A_H\) are non-zero.
	Lemma \ref{lem:trint} imposes a strong restriction on the arrangement: fix \(H_0 \in \mH\), then the common intersection of all \(H \in \mH\) such that \(H \pitchfork H_0\) contains the line \(\C \cdot n_{H_0}\).
\end{remark}

\subsubsection{{\large Normal subspaces}}

The main result of this section is Proposition \ref{prop:AL}. It extends parts of \cite[Lemmas 2.11 and 2.13]{CHL} by lifting the existence of a Hermitian inner product in the Dunkl condition. 

\begin{definition}\label{def:weights}
	Let \((V, \mH)\) be an arrangement and let \(\nabla\) be a standard connection with residues \(\{A_H, \,\ H \in \mH\} \subset \End V\). 
	The \textbf{weight} of \(\nabla\) at an irreducible intersection \(L \in \mL_{irr}(\mH)\) is the complex number \(a_L\) defined as
	\[ a_L= (\codim L)^{-1} \tr A_L \]
	where \(A_L\) is the residue of \(\nabla\) at \(L\) given by Equation \ref{eq:residue}.
\end{definition}

\begin{note}
	The weights of hyperplanes are given by
	\(a_H= \tr A_H\). In general, we have  
	\[a_L = \frac{1}{\codim L} \sum_{H \in \mH_L} a_H .\]
\end{note} 

\begin{definition}
	The \textbf{normal subspace} to \(L \in \mL(\mH)\) is defined by
	\[L^{\perp} = \Span \{n_H, \,\ H \in \mH_L \}  .\]
\end{definition}

\begin{assumptions}\label{ass:nz}
	Let \((V, \mH)\) be an arrangement.
	A collection \(\{A_H, \,\ H \in \mH\} \subset \End V\) satisfies the \textbf{Non-Zero Weights Assumptions} if:
	\begin{description}
		\item[(F)] \([A_L, A_M]=0\) whenever \(L, M \in \mL(\mH)\) and \(L \subset M\);
		\item[(T)] \(H \subset \ker A_H\) for all \(H \in \mH\);
		\item[(NZ)] \(a_L \neq 0\) whenever \(L \in \mL_{irr}(\mH)\).
	\end{description}
\end{assumptions}

\begin{remark}
	If \emph{\textbf{(NZ)}} holds then the residues \(A_H\) are non-zero and therefore condition \emph{\textbf{(T)}} is equivalent to \(H=\ker A_H\).
\end{remark}

\begin{lemma}
	If the Non-Zero Weights Assumptions \ref{ass:nz} are satisfied then \(n_H \notin H\) for every \(H \in \mH\).
\end{lemma}

\begin{proof}
	By definition \(\mH \subset \mL_{irr}(\mH)\), so \(a_H \neq 0\) and the statement follows from Lemma \ref{lem:rk1}.
\end{proof}

\begin{lemma}\label{lem:Lidec}
	If the Non-Zero Weights Assumptions \ref{ass:nz} are satisfied and \(L \in \mL_{irr}(\mH)\), then
	\[V = L \oplus L^{\perp} .\]
	Moreover, \(\ker A_L = L\) and \(L^{\perp}\) is the \(a_L\)-eigenspace of \(A_L\).
\end{lemma}

\begin{proof}
	Clearly \(L \subset \ker A_L\), so \(A_L\) defines an endomorphism \(\bar{A}_L\) of \(V / L\). We prove first that \(\ker A_L \subset L\) by showing that \(\bar{A}_L\) has no kernel. 
	
	Let \(H \in \mH_L\) then \(A_H\) defines  \(\bar{A}_H \in \End(V/L)\) and \(\bar{A}_L = \sum_{L \subset H} \bar{A}_H\). Given a vector \(v \in V\), we denote by \(\bar{v} \in V/L\) its image under the quotient map. Our assumption that \(n_H \notin H\) implies that \(\bar{n}_H \neq 0\). It follows from the equation \(\bar{A}_H \bar{n}_H = a_H \bar{n}_H\) that \(\bar{A}_H\) is a rank one endomorphisms and 
	\[\tr \bar{A}_H = a_H .\] 
	On the other hand,  since \([A_L, A_H]=0\) for every \(H \in \mH_L\), it follows that \(\bar{A}_L\) preserves each of the hyperplanes of the essential irreducible arrangement \(\mH_L / L\), hence \(\bar{A}_L\) is a constant multiple \(\lambda\) of the identity. We can determine \(\lambda\) from the equation 
	\[ \lambda \dim(V/L) = \sum_{H \in \mH_L} \mbox{tr}(\bar{A}_H), \]
	we get that \(\lambda=a_L\). By assumption \(a_L \neq 0\) and therefore \(\ker \bar{A}_L = \{0\}\). If there was some \(v \notin L\) with \(A_L v = 0\) then we would have \(\bar{A}_L \bar{v} =  0\) with \(\bar{v} \neq 0\) which would contradict \(\ker \bar{A}_L = \{0\}\). We have shown that \(\ker A_L = L\).
	
	Note that \(L^{\perp} \subset \) \(a_L\)-eigenspace of \(A_L\). Indeed,  \(\C \cdot n_H\) is the \(a_H\)-eigenspace of \(A_H\) and \([A_L, A_H] = 0\) imply that \(A_L n_H = \lambda n_H\) for some \(\lambda \in \C\); taking the quotient we get \(\bar{A}_L \bar{n}_H = \lambda \bar{n}_H\) so \(\lambda=a_L\) because \(\bar{n}_H \neq 0\) and \(\bar{A}_L = a_L \mbox{Id}_{V/L} \). On the other hand, \(a_L \neq 0\) implies that \(a_L\)-eigenspace of \(A_L\) \(\subset \img A_L\). Note that \(\img A_L \subset L^{\perp}\) by definition, so we get a chain of inclusions
	\[L^{\perp} \subset a_L\mbox{-eigenspace} \subset \img A_L \subset L^{\perp} \]
	which must be equalities. This finishes the proof of the lemma.
\end{proof}

\begin{corollary}\label{cor:essirrA}
	Suppose \((V, \mH)\) is essential, irreducible and \(\{A_H, \,\  H \in \mH\}\) satisfy the Non-Zero Weights Assumptions \ref{ass:nz}. Then \(\sum_{H \in \mH}A_H = a_{0} \Id_V\) where \(a_{0} = (\dim V)^{-1} \sum_{H \in \mH} a_H \neq 0\) is the weight of \(\nabla\) at \(\{0\} \in \mL_{irr}(\mH)\). In particular, \(V= \Span \{n_H, \,\ H \in \mH\}\).
\end{corollary}

The main result of this section is the next.

\begin{proposition}\label{prop:AL}
	Suppose that the Non-Zero Weights Assumptions \ref{ass:nz}  are satisfied and let \(L \in \mL(\mH)\). Consider the irreducible decomposition 
	\[\mH_L = \mH_1 \uplus \ldots \uplus \mH_k , \]
	so \(L_i= T(\mH_i)\) are the irreducible components of \(L\). Then the following holds
	\begin{enumerate}
		\item[(i)] \(V = L \oplus L^{\perp}\), \(L=\ker A_L\) and \(L^{\perp} = \img A_L\);
		\item[(ii)] \(L^{\perp} = L_1^{\perp} \oplus \ldots \oplus L_k^{\perp}\) and \(A_L\) acts on \(L_j^{\perp}\) by scalar multiplication by \(a_{L_j}\);
		\item[(iii)] \(L_i = L  \oplus \left(\oplus_{j \neq i} L_j^{\perp}\right)\).
	\end{enumerate}
\end{proposition}

The proof of Proposition \ref{prop:AL} is divided into several lemmas. We will simplify notation by letting
\[A_i = A_{L_i} \mbox{ and } a_i = a_{L_i} . \]

\begin{lemma} \label{lem:imageaj}
	If \(i \neq j\) then
	\begin{equation*}
	(L_j)^{\perp} \subset L_i .
	\end{equation*}
\end{lemma}

\begin{proof}
	Let \(H_i, H_j \in \mH_L\) with \(H_i\in \mH_i\) and \(H_j \in \mH_j\), by Lemma \ref{lem:transvinter} \(H_i \pitchfork H_j\) and by Lemma \ref{lem:trint}  \(n_{H_j} \in H_i\). 
\end{proof}

\begin{lemma} \label{lem:dirsum}
	For every \(1 \leq i \leq k\) we have
	\[  L_i^{\perp} \cap \sum_{j \neq i} L_j^{\perp} = \{0\} .  \]
\end{lemma}

\begin{proof}
	By Lemma \ref{lem:imageaj} we have \(\sum_{j\neq i} L_j^{\perp} \subset L_i\) and by Lemma \ref{lem:Lidec} \(L_i \cap L_i^{\perp} = \{0\}\). 
\end{proof}

Lemmas \ref{lem:oplus}, \ref{lem:ALagrees} and \ref{lem:scalmult} that follow prove Proposition \ref{prop:AL}-\((ii)\).

\begin{lemma}\label{lem:oplus}
	We have 
	\[L^{\perp} = L_1^{\perp}\oplus \ldots \oplus L_k^{\perp} . \]
\end{lemma}

\begin{proof}
	By definition \(L^{\perp} = \sum_{i=1}^k L_i^{\perp}\). The sum is direct because of Lemma \ref{lem:dirsum}.
\end{proof}

\begin{lemma} \label{lem:ALagrees}
	\(A_L\) agrees with \(A_j\) on \(L_j^{\perp}\).
\end{lemma}

\begin{proof}
	By definition \(A_L = A_1 + \ldots + A_k\). By Lemma \ref{lem:imageaj} we have \(L_j^{\perp} \subset L_i\) for \(i \neq j\). Since \(L_i \subset \ker A_i\), the statement follows.
\end{proof}

\begin{lemma} \label{lem:scalmult}
	\(A_L\) acts on \(L_j^{\perp}\) by scalar multiplication by \(a_j\).
\end{lemma}

\begin{proof}
	Follows from Lemmas \ref{lem:ALagrees} and \ref{lem:Lidec}.
\end{proof}

Lemmas \ref{lem:kerAL}, \ref{lem:imgAL}, \ref{lem:0int} and \ref{lem:VisLoplusLper} that follow prove Proposition \ref{prop:AL}-\((i)\).

\begin{lemma}\label{lem:kerAL}
	\(\ker A_L = L \).
\end{lemma}

\begin{proof}
	Clearly \(L \subset \ker A_L\). To show the reverse inclusion let us take some \(v \in \ker A_L\), so
	\[A_L v = \sum_{i=1}^{k} A_i v = 0 .\]
	By Lemma \ref{lem:Lidec} \(A_i v \in \img A_i = L_i^{\perp}\). By Lemma \ref{lem:dirsum}  \(A_i v = 0\) for all \(i\). By Lemma \ref{lem:Lidec} \(v \in \ker A_i = L_i\) for all \(i\). We conclude that 
	\[v \in \bigcap_{i=1}^k L_i = L . \qedhere \]
\end{proof}

\begin{lemma}\label{lem:imgAL}
	\(\img A_L = L^{\perp} \).
\end{lemma}

\begin{proof}
	By definition \(\img A_L \subset L^{\perp}\). To show the reverse inclusion let us take some \(v \in L^{\perp}\). By Lemma \ref{lem:oplus} we have \(v = v_1 + \ldots + v_k\) with \(v_i \in L_i^{\perp}\). If we set \(u= \sum_{i=1}^{k}a_i^{-1}v_i\) then it follows from Lemma \ref{lem:scalmult} that \(A_L u = v \in \img A_L\).
\end{proof}

\begin{lemma}\label{lem:0int}
	\(L \cap L^{\perp} = \{0\}\).
\end{lemma}

\begin{proof}
	Let \(v \in L^{\perp}\) and write \(v = v_1 + \ldots + v_k\) with \(v_i \in L_i^{\perp}\). By Lemma \ref{lem:scalmult} \(A_L v = \sum_{i=1}^{k}a_i v_i\). If \(v\) also lies in \(L \subset \ker A_L\) then \(\sum_{i=1}^{k} a_i v_i = 0\) and  Lemma \ref{lem:dirsum} implies that \(a_i v_i = 0\) for all \(i\). Since \(a_i \neq 0\) we get that \(v_i =0\) for all \(i\) hence \(v=0\).
\end{proof}

The next result generalizes Lemma \ref{lem:Lidec}.

\begin{lemma}\label{lem:VisLoplusLper}
	\(V = L \oplus L^{\perp}\).
\end{lemma}

\begin{proof}
	By Lemma \ref{lem:0int} it is enough to show that \(\dim L + \dim L^{\perp} = \dim V\).
	By Lemma \ref{lem:oplus} \(\dim L^{\perp} = \sum_{i=1}^{k} \dim L_i^{\perp}\). By Lemma \ref{lem:Lidec} \(\dim L_i^{\perp} = \codim L_i\). On the other hand, since \(\mH_L= \uplus_{i=1}^k \mH_i\) is a u-plus decomposition, \(\codim L = \sum_{i=1}^{k} \codim L_i\) (see Equation \eqref{eq:sumranks}). We conclude that \(\dim L^{\perp} = \codim L\) and the lemma follows.
\end{proof}

Next Lemma \ref{lem:linearirreddecCn} proves Proposition \ref{prop:AL}-\((iii)\).

\begin{lemma} \label{lem:linearirreddecCn}
	For every \(1\leq i \leq k\) we have
	\begin{equation*} \label{eq:irredidec}
	L_i = L \oplus \left(\oplus_{j \neq i} (L_j)^{\perp} \right) .
	\end{equation*}
\end{lemma}

\begin{proof}
	By definition \(L \subset L_i\) and by Lemma \ref{lem:imageaj} \(\sum_{j \neq i} L_j^{\perp} \subset L_i\), therefore
	\begin{equation} \label{eq:sumaux}
		L + \sum_{j \neq i} L_j^{\perp} \subset L_i . 
	\end{equation}
	Lemmas \ref{lem:dirsum} and \ref{lem:0int} imply that the left hand side of Equation \eqref{eq:sumaux} is a direct sum. On the other hand, since \(\dim L_j^{\perp} = \codim L_j\) and \(\dim L + \sum_{i=1}^{k} \dim L_i^{\perp} = \dim V \), we have that
	\[\dim L + \sum_{j \neq i} \dim L_j^{\perp} = \dim V - \dim L_i^{\perp} = \dim L_i .\]
	We conclude that equality holds in Equation \eqref{eq:sumaux} 
	and the lemma follows.
\end{proof}

\subsection{The connection \(\nabla^L\) and its product decomposition}\label{sect:nablaLsplit}

\begin{definition} \label{def:localizationconection}
	Let \(\nabla\) be a flat torsion free standard connection given by \(\{A_H, \,\ H \in \mH\}\). The \textbf{localization of \(\nabla\) at} \(L \in \mL(\mH)\) is defined as \(\nabla^L= d- \Omega^L\) with
	\[\Omega^L = \sum_{H \in \mH_L} A_H \frac{dh}{h} . \]
\end{definition}

\begin{remark} \label{rmk:localizationconection}
	It follows from Proposition \ref{prop:ftfstcon} that \(\nabla^L\) is flat and torsion free. 
\end{remark}

In this section we upgrade the product decomposition \eqref{eq:HLdec} to an affine splitting for \(\nabla^L\). Our main tool for doing so is Proposition \ref{prop:AL}.
We begin by recalling the definition of a product of two affine connections on a product manifold.

\subsubsection{{\large Products of affine connections}}

\begin{notation}
	Let \(X=X_1 \times X_2\) be a product manifold with projection maps \(\pi_i: X \to X_i\) to its factors.
	If \(v\) is a vector field on \(X_1\), say, then there is a unique vector field \(\tilde{v}\) on \(X\) such that
	\( (D\pi_1) \tilde{v} = v\) and \((D\pi_2)\tilde{v}=0\); and similarly for vector fields on \(X_2\). We say that \(\tilde{v}\) is the \textbf{lift} of \(v\) via the projection maps \(\pi_i\).
\end{notation}
	 
\begin{definition}\label{def:prodconect}
Let \((X_i, \nabla^i)\) for \(i=1,2\) be two (complex) manifolds endowed with (holomorphic) connections \(\nabla^i\) on \(TX_i\). The \textbf{product connection} \(\nabla\) on \(T(X_1 \times X_2)\)
is defined by the requirement
\begin{equation}\label{eq:defnablaprod}
	\nabla_{\tilde{v}_1+\tilde{v}_2}(\tilde{w}_1+\tilde{w}_2) = \widetilde{\nabla^1_{v_1} w_1} + \widetilde{\nabla^2_{v_2}w_2} 
\end{equation}
where \(v_i,w_i\) are vector fields in \(X_i\) and \(\tilde{v}_i, \tilde{w}_i\) are their lifts to \(X = X_1 \times X_2\) via the projection maps \(\pi_i: X_1 \times X_2 \to X_i\). We write this as
\[(X, \nabla) = (X_1, \nabla^1) \times (X_2, \nabla^2) . \]	
\end{definition}

Alternatively, the product connection \(\nabla\) is the direct sum of the pull-back connections \(\pi_1^*\nabla^1 \oplus \pi_2^*\nabla^2\) under the natural isomorphism between \(TX\) and \(\pi_1^*(TX_1) \oplus \pi_2^*(TX_2)\).

\begin{remark}\label{rmk:prodconect}
	Some observations are in order.
	
	\begin{itemize}
		\item[(i)] Every vector field on \(X_1 \times X_2\) can be locally written as a sum of vector fields \(f_1 \tilde{v}_1 + f_2 \tilde{v}_2\) for functions \(f_i\). It follows that Equation \eqref{eq:defnablaprod} together with Leibniz's rule uniquely determine \(\nabla\).
		
		\item[(ii)] The product connection \(\nabla\) is flat (torsion free) if and only if both \(\nabla^i\) are flat (torsion free).
		
		Under the natural bijection between flat torsion free connections and affine structures given by Lemma \ref{lem:affineconnectioncorresp}, the product of connections in Definition \ref{def:prodconect} corresponds to the obvious product of affine manifolds given by taking product affine atlases.

		\item[(iii)]  In the particular case where \((V_i, \mH_i)\) are arrangements and \(X_i=V_i^{\circ}\) are  endowed with standard connections \(\nabla^i\), the product connection \(\nabla\) on the product hyperplane complement \((V_1 \times V_2)^{\circ}\) is also standard. In an obvious notation 
		\[
		A_{H_1 \times V_2} = 
		\begin{pmatrix}
		A_{H_1} & 0 \\
		0 & 0
		\end{pmatrix}, \qquad 
		A_{V_1 \times H_2} = 
		\begin{pmatrix}
		0 & 0 \\
		0 & A_{H_2}
		\end{pmatrix} . 
		 \]
	\end{itemize}
\end{remark}

\subsubsection{{\large Affine splitting of \(\nabla^L\)}}\label{sect:nablaL}

Fix \(L \in \mL(\mH)\) and let
\[\mH_L = \mH_1 \uplus \ldots \uplus \mH_k \]
be the irreducible decomposition of the localized arrangement. The irreducible components of \(L\) are
\[L_i = \bigcap_{H \in \mH_i} H .\]

\begin{proposition}\label{prop:nablaLsplit}
	Suppose that the Non-Zero Weights Assumptions \ref{ass:nz} are satisfied. Then we can choose linear coordinates \(x_1, \ldots, x_n\) such that
	\begin{equation}\label{eq:nablaLaffprod}
	\left(\C^n, \, \mH_L, \, \nabla^L\right) = \left(\C^{\dim L}, \, \emptyset, \, d \right) \times \left(\C^{n_1}, \, \bar{\mH}_1, \, \bar{\nabla}^1\right) \times \ldots \times \left(\C^{n_k}, \, \bar{\mH}_k, \, \bar{\nabla}^k\right)	.
	\end{equation}
	The factors \((\C^{n_i}, \bar{\mH}_i)\) are isomorphic to the essential irreducible arrangements \((V/L_i,\, \mH_i/L_i)\) and \(\bar{\nabla}^i\) are  standard connections.	
\end{proposition}

\begin{notation}
	Equation \eqref{eq:nablaLaffprod} means that we have a product of arrangements and also that the connection on the hyperplane complement is a product.
	More precisely, Equation \eqref{eq:nablaLaffprod} means that 
	\begin{equation}\label{eq:HLprod}
	\left(\C^n, \, \mH_L\right) = \left(\C^{\dim L}, \emptyset \right) \times \left(\C^{n_1}, \, \bar{\mH}_1\right) \times \ldots \times \left(\C^{n_k}, \, \bar{\mH}_k\right)
	\end{equation}
	and
	\begin{equation}\label{eq:HLprod2}
	\left((\C^n)^{\circ}, \, \nabla^L\right) = \left(\C^{\dim L}, d \right) \times \left((\C^{n_1})^{\circ}, \, \bar{\nabla}_1\right) \times \ldots \times \left((\C^{n_k})^{\circ}, \, \bar{\nabla}_k\right)
	\end{equation}
	hold. The products \eqref{eq:HLprod} and \eqref{eq:HLprod2} are understood according to Definitions \ref{def:prodarr} and \ref{def:prodconect}.
\end{notation}

\begin{proof}[Proof of Proposition \ref{prop:nablaLsplit}]
	By Proposition \ref{prop:AL} we have direct sum decompositions
	\begin{equation}\label{eq:dirsumProp}
	V = L \oplus L_1^{\perp} \oplus \ldots \oplus L_k^{\perp} 	
	\end{equation}
	\begin{equation}\label{eq:dirsumProp2}
	L_i = L \oplus \left(\oplus_{j \neq i} L_j^{\perp} \right) \, \mbox{ for } 1 \leq i \leq k . 	
	\end{equation}
	Let \(n_i= r(\mH_i) = \codim L_i\)
	so that \(n_1+\ldots+n_k=n-d\) where \(d=\dim L\).
	
	Fix a decomposition of the index set 
	\[\{1, \ldots, n\} = I_0 \cup I_1 \cup \ldots \cup I_k\] 
	as given by Equations \eqref{eq:indexset} and \eqref{eq:I0}. Let \((x_1, \ldots, x_n)\) 
	be linear coordinates on \(V\) such that the coordinate vector fields \(\p_{x_j}\) span the subspaces \(L\) and \(L_i^{\perp}\) for all \(1\leq i\leq k\). More precisely,
	\begin{equation}\label{eq:xcoord}
	L = \Span \{\p_{x_j}, \,\, j \in I_0 \} \, \mbox{ and }
	L^{\perp}_i = \Span \{\p_{x_j}, \,\, j \in I_i\} \,\ \mbox{ for } 1 \leq i \leq k .
	\end{equation}
	In particular, it follows that \(L\) is the common zero set of  \(x_1, \ldots, x_{n-d}\) and if \(h\) is a defining linear equation for \(H \in \mH_L\) then it depends only on the first \(n-d\) coordinates \(h=h(x_1, \ldots, x_{n-d})\). 
	
	If \(H \in \mH_i\) then \(L_i \subset H\). Equation \eqref{eq:dirsumProp2} implies that
	\(h=h(x_j, \, j \in I_i )\) and we write \(\bar{H}\) for the hyperplane in \(\C^{n_i}\) defined by \(h\).  Same as in the proof of Lemma \ref{lem:HLdec}, we have arrangements
	\begin{equation}
	\bar{\mH}_i = \{\bar{H}, \, H \in \mH_i \}
	\end{equation}
	such that Equation \eqref{eq:HLprod} holds
	and we have isomorphisms 
	\begin{equation}\label{eq:isos}
	(\C^{n_i}, \, \bar{\mH}_i) \equiv (V/L_i, \, \mH_i/L_i)	
	\end{equation}
	given by the coordinates \(x_j\) with \(j \in I_i\).

	We have to show that the connection \(\nabla^L\) splits accordingly.
	If \(H \in \mH_i \subset \mH_L\) then \(A_H\) vanishes on \(L_i\), because \(L_i \subset H= \ker A_H\), and it preserves \(L_i^{\perp}\), because \( \img A_H = \C \cdot n_H \subset L_i^{\perp}\). We conclude that \(A_H \in \End \C^n\) corresponds under the direct sum decomposition \eqref{eq:dirsumProp} to an endomorphism \(\bar{A}_{\bar{H}}\) of \(L_i^{\perp} \cong \C^{n_i}\).
	We define
	\begin{equation}\label{eq:barnablai}
	\bar{\nabla}^i = d - \sum_{\bar{H} \in \bar{\mH}_i} \bar{A}_{\bar{H}} \frac{d \bar{h}}{\bar{h}} 
	\end{equation}
	where \(\bar{A}_{\bar{H}}\) is the restriction of \(A_H\) to \(L_i^{\perp} \cong \C^{n_i}\) and
	\(\bar{h}\) denotes \(h\) thought as a linear function on \(\C^{n_i}\).
	Equation \eqref{eq:HLprod2} follows from the definitions and Remark \ref{rmk:prodconect}.
\end{proof}

\begin{remark}
	Note that if \(L=T(\mH)\) then \(\nabla^L = \nabla\). Proposition \ref{prop:nablaLsplit} asserts that flat torsion free standard connections respect the decomposition of arrangements into irreducible factors, in the sense that in suitable linear coordinates the connection splits accordingly. 
\end{remark}

\begin{remark}\label{rmk:orth}
	If \(\nabla^L\) preserves a Hermitian form \(\inn\) and the weights \(a_{L_i}\) are non-integer then the factors in Equation \eqref{eq:nablaLaffprod} are pairwise orthogonal. This can be seen by restricting the connection to complex lines that go through the origin of a \(\C^{n_i}\) factor but keep the other components fixed. Holonomy around a standard loop encircling the origin in such a line acts by scalar multiplication by \(\exp(2\pi i a_{L_i})\) on the \(\C^{n_i}\) factor and by the identity on the other factors. 
\end{remark}

\subsection{The quotient connection \(\nabla^{V/L}\)}\label{sect:quotconect}

Suppose that \(\{A_H, \, H \in \mH\} \subset \End V\) satisfy the Non-Zero Weights Assumptions \ref{ass:nz}. Since \(a_H = \tr A_H\) is non-zero, the image of \(A_H\) is a one-dimensional subspace \(\C \cdot n_H\) complementary to \(H\). We say that \(n_H\) is a normal vector to \(H\). Given \(L \in \mL(\mH)\) we have defined \(L^{\perp}\) as the subspace of \(V\) spanned by all normals \(n_H\) to hyperplanes in \(\mH_L\). By definition, \(L^{\perp}\) is preserved by all \(A_H\) for \(H \in \mH_L\).

By Proposition \ref{prop:AL} we have a direct sum decomposition \(L \oplus L^{\perp} = V\). The quotient projection map \(V \to V/L\) restricted to \(L^{\perp}\) gives us a natural isomorphism
\begin{equation}\label{eq:quotientiso}
	(L^{\perp}, \, \bar{\mH}_L) \xrightarrow{\sim} (V/L, \, \mH_L/L) .
\end{equation}
where \(\bar{\mH}_L = \{H \cap L^{\perp}, \, H \in \mH_L\}\).
Given a hyperplane \(H/L\) of \(\mH_L/L\) we define \(A_{H/L}\) as the endomorphism of \(V/L\) obtained by restricting \(A_H\) to \(L^{\perp}\) together with the identification \(L^{\perp} \cong V/L\) given by the quotient projection map.

\begin{definition}\label{def:quotconnect}
	Suppose that the Non-Zero Weight Assumptions \ref{ass:nz} are satisfied and let \(L \in \mL(\mH)\).
	We define the \textbf{quotient connection} \(\nabla^{V/L}\) to be the standard connection on \(V/L\) with simple poles at \(\mH_L/L\) determined by the residues \(\{A_{H/L}, \, H \in \mH_L\} \subset \End (V/L)\).
\end{definition}

\begin{remark}
	Note that \(\tr A_{H/L} = \tr A_H\), so the weight of \(\nabla^{V/L}\) at \(H/L\) is equal to the weight of \(\nabla\) at \(H\).
\end{remark}

\begin{lemma}\label{lem:quotconnect}
	The quotient connection \(\nabla^{V/L}\) is flat and torsion free. If \(M/L\) is an irreducible subspace of \(\mH_L/L\) then \(M \in \mL_{irr}(\mH)\) and the weight of \(\nabla^{V/L}\) at \(M/L\) is equal to the weight of \(\nabla\) at \(M\). In particular,  \(\{A_{H/L}, \, H \in \mH_L\}\) satisfies the Non-Zero Weights Assumptions \ref{ass:nz}. 
\end{lemma}

\begin{proof}
	It is clear that \(\{A_{H/L}, \, H \in \mH_L\}\) satisfies the Basic Assumptions \ref{ass:basic} because \(\{A_H, \, H \in \mH_L\}\) does. Proposition \ref{prop:ftfstcon} implies that \(\nabla^{V/L}\) is flat and torsion free. If \(M/L \in \mL_{irr}(\mH_L/L)\) then Lemma \ref{lem:irredHL} implies that \(M \in \mL_{irr}(\mH)\). Since \(\codim_V M = \codim_{V/L} M/L\) and \(\tr A_{H/L} = \tr A_H\) for every \(H \in \mH_L\), the statement follows.
\end{proof}

\begin{remark}\label{rmk:quotconnec}
	We denote by \(\bar{\nabla}\) the connection on \(L^{\perp}\) with simple poles at \(\bar{\mH}_L\) and residues \((A_H)|_{L^{\perp}}\) at \(\bar{H}=H \cap L^{\perp}\). In other words, the connection \(\bar{\nabla}\) corresponds to \(\nabla^{V/L}\) under the isomorphism \eqref{eq:quotientiso}. 
	With this notation, Proposition \ref{prop:nablaLsplit} shows that
	\[(V, \mH_L, \nabla^L) = (L, \emptyset, d) \times (L^{\perp}, \bar{\mH}_L, \bar{\nabla}) \]
	and
	\[(L^{\perp}, \bar{\mH}_L, \bar{\nabla}) = (L_1^{\perp}, \bar{\mH}_{1}, \bar{\nabla}^1) \times \ldots \times (L^{\perp}_k, \bar{\mH}_{k}, \bar{\nabla}^k)\]
	where \(\bar{\nabla}^i\) correspond to the quotient connections \(\nabla^{V/L_i}\) under projection.
\end{remark}

\subsection{Euler vector fields}\label{sect:eulervf}

\begin{notation}
	Let \(\xi\) be a vector field on a manifold \(M\) and let \(\nabla\) be a connection on \(TM\). The covariant derivative \(\nabla \xi\) can be interpreted as a section of \(\End TM\). Given \(v \in TM\), we have
	\[(\nabla {\xi}) (v) = \nabla_v \xi .\]
\end{notation}

\begin{notation}
	Let \(V\) be a complex vector space.
	There is a unique vector field \(e_V\) such that \(e_V(0)=0\) and \(d e_V\) is the identity section of \(\End TV\), where \(d\) is the Euclidean connection. In linear coordinates we have
	\[ e_V = \sum_{i=1}^{n} x_i \frac{\p}{\p x_i} . \]
	We refer to \(e_V\) as the usual Euler vector field of \(V\). 
\end{notation}

\begin{definition}\label{def:angle}
	Let \(L \in \mL_{irr}(\mH)\). The \textbf{angle} at \(L\) is defined as
	\[\alpha_L = 1-a_L . \]
\end{definition}

\begin{lemma}\label{lem:eulstndcon}
	Let \((V, \mH)\) be an arrangement equipped with a flat torsion free standard connection such that the Non-Zero Weight Assumptions \ref{ass:nz} are satisfied. Let
	\[\mH = \mH_1 \uplus \ldots \uplus \mH_k\]
	be its irreducible decomposition. Write \(T=T(\mH)\), \(T_i=T(\mH_i)\), \(a_i=a_{T_i}\) and \(\alpha_i=\alpha_{T_i}\).
	Suppose that \(\alpha_i \neq 0\) for all \(i\). Then the vector field
	\begin{equation}\label{eq:euler}
	e = e_T + \alpha_1^{-1} e_{T_1^{\perp}} + \ldots + \alpha_k^{-1} e_{T_k^{\perp}} 
	\end{equation}
	satisfies \(\nabla e = \Id_V\).
\end{lemma}

\begin{proof}
	It follows from Proposition \ref{prop:nablaLsplit} applied to \(L=T\) that it is enough to consider the case when \((V, \mH)\) is essential and irreducible.
		
	We compute \(\nabla e_{V}\) following \cite[Proof of Proposition 2.2]{CHL}. Clearly \(d e_{V} = \Id_{V}\). On the other hand
	\[A_H (e_{V}) = dh(e_V) \cdot n_H = h \cdot n_H ,\]
	so \(A_H(e_{V})dh/h = A_H\) and
	\[\nabla e_{V} = \Id_{V} - \sum_H A_H . \]
	Since \(\mH\) is irreducible and essential, \(\sum_{H\in \mH} A_H = a_0 \Id_{V}\) where 
	\[a_0=(1/\dim V) \sum_{H\in \mH} a_H \neq 1\] 
	by assumption. It follows that \(\nabla e_{V} = (1-a_0) \Id_{V}\) and therefore
	\[\nabla \left( (1-a_0)^{-1} e_{V} \right) = \Id_{ V} . \qedhere \]
\end{proof}

\begin{remark}\label{rmk:etangent}
	It follows from Equation \eqref{eq:euler} that \(e\) is tangent to all hyperplanes in the arrangement. By taking intersections, we see that if \(x\) belongs to \(L \in \mL(\mH)\) then \(e(x) \in L\).
\end{remark}

\begin{definition}
	We call the vector field \(e\) of Lemma \ref{lem:eulstndcon} the Euler vector field for \(\nabla\). We will prove next that it is uniquely characterized by the properties \(e(0)=0\) and \(\nabla e =\) identity.
\end{definition}

\begin{lemma} \label{lem:eulvfuniq}
	Suppose that the hypothesis of Lemma \ref{lem:eulstndcon} hold. Furthermore, assume that \(a_i \notin \Z\) for all \(i=1, \ldots, k\).
	If \(s\) is a holomorphic vector field with \(\nabla s=0\), i.e. \(s\) is parallel, then \(s\) is a constant vector field and \(s\equiv s_0 \in T\).
\end{lemma}

\begin{proof}
	We decompose \(s= s_T + s_1+\ldots+s_k\) according to the splitting \(V= T \oplus T_1^{\perp} \oplus \ldots \oplus T_k^{\perp}\). Each of the components \(s_i\) must be parallel, so it is enough to show that in the essential and irreducible case parallel fields must vanish. Same as before, we consider an irreducible and essential arrangement endowed with a flat standard connection \((V, \mH, \nabla)\).
	
	We compute
	\begin{equation*}
	[e, s] = \nabla_e s - \nabla_s e 
	= - s .
	\end{equation*}
	Equivalently, \([e_{V} , s ] = (a_0-1) s\). Since \(a_0 \notin \Z\), the statement follows from Lemma \ref{lem:homvf}. 
\end{proof}

\begin{lemma} \label{lem:homvf}
	Let \(s\) be a non-zero holomorphic vector field on \(V\) that satisfies \([e_V, s] = \lambda s\) where \(e_V\) is the  usual Euler vector field of \(V\), then \(\lambda \in \Z\) and the components of \(s\) are homogeneous polynomials of degree \(\lambda+1\).
\end{lemma}

\begin{proof}
	Write \( s = \sum_i f_i \p_{x_i}\) with \(f_i\) holomorphic functions. We compute \([e_V, s]\) using the Leibniz rule \([X, fY] = X(f)Y + f [X, Y]\) together with \([e_V, \p_{x_i}] = - \p_{x_i}\), we obtain
	\[[e_V, s] = \sum_i \left(e_V(f_i) -f_i \right) \p_{x_i} .  \]
	The equation \([e_V, s] = \lambda s\) implies \(e_V(f_i)=(\lambda+1)f_i\) and it follows that \(f_i\) are homogeneous polynomials of degree \(\lambda+1\). 
\end{proof}

\begin{corollary}\label{cor:euniqueness}
	In the setting of Lemma \ref{lem:eulvfuniq}, if \(\tilde{e}\) is another holomorphic vector field on \(V\) that satisfies \(\nabla \tilde{e} = \Id_{V}\) then \(\tilde{e} = e + s\) where \(s\) is a constant vector field tangent to the centre \(T\). In particular, if \(\tilde{e}(0)=0\) (or if \(T=0\)) then \(\tilde{e}=e\).
\end{corollary}

\begin{proof}
	The difference \(\tilde{e}-e\) is parallel, apply Lemma \ref{lem:eulvfuniq}.
\end{proof}

\section{Local product decomposition}\label{sec:locprod}

Throughout this section we work under the hypothesis of Theorems \ref{PRODTHM} and \ref{thm:locprod}. In particular, there is a  positive definite Hermitian form \(\inn\) which is preserved by \(\nabla\). Our goal is to establish the Local Product Decomposition Theorem \ref{thm:locprod}. The set-up is as follows.

\begin{notation}
	We fix a non-zero intersection \(L \in \mL(\mH)\)
	and \(x \in L^{\circ}\).
	Let \(U\) be a neighbourhood of \(x\) such that \(U \cap H = \emptyset\) for all \(H \in \mH \setminus \mH_L\) and write \(U^{\circ} = U \setminus \mH_L\).
	We reset our original coordinates by making a translation
	\begin{equation}\label{eq:translation}
	(z_1, \ldots, z_n) \mapsto (z_1, \ldots, z_n) - x	
	\end{equation}
	and identify \(U\) with a neighbourhood of the origin.	
\end{notation}

\begin{remark}
	If \(H \in \mH_L\) then its defining equation \(h\) is a linear function of \((z_1, \ldots, z_n)\). If \(H \in \mH \setminus \mH_L\) then its defining equation \(h\) is an affine linear function of \((z_1, \ldots, z_n)\) which is non-vanishing over \(U\). 
\end{remark}

\begin{definition}
	The \textbf{model connection} at \(x\) is defined as \(\nabla^L = d- \Omega^L\) with 
	\begin{equation}
	\Omega^L= \sum_{H \in \mH_L} A_H \frac{dh}{h} .
	\end{equation}
	In other words, \(\nabla^L\) is the localization of \(\nabla\) at \(L\) introduced in Definition \ref{def:localizationconection}. 
\end{definition}

\begin{remark}
	We have two flat connections: the original \(\nabla = d- \Omega\) and the model \(\nabla^L = d - \Omega^L\). Their connection matrices satisfy
	\begin{equation}\label{eq:omegaL}
	\Omega = \Omega^L + \hol,
	\end{equation}
	where \(\hol\) denotes a matrix of holomorphic \(1\)-forms on \(U\). The connection \(\nabla^L\) is standard with respect to our newly defined coordinates \eqref{eq:translation}. However, our original connection \(\nabla\) is not. 
\end{remark}

By Proposition \ref{prop:nablaLsplit} we have an affine splitting of \(\nabla^L\) given by Equation \eqref{eq:nablaLaffprod}. The main work in this section is to carry this affine splitting from \(\nabla^L\) to \(\nabla\) over \(U^{\circ}\). The statement is as follows.

\begin{proposition}[Local affine splitting]\label{prop:locaffprod}
	There are holomorphic coordinates \((y_1, \ldots, y_n)\) centred at \(x \in L^{\circ}\) in which the connection \(\nabla\) is standard. Furthermore, the
	affine structure defined by \(\nabla\) over \(U^{\circ}\) splits as a product of the form 
	\begin{equation}\label{eq:affinesplit}
	\left(U^{\circ}, \, \nabla \right) \subset \left(\C^{\dim L}, d \right) \times \left((\C^{n_1})^{\circ}, \bar{\nabla}^1\right) \times \ldots \times \left((\C^{n_k})^{\circ}, \bar{\nabla}^k\right)
	\end{equation}
	where the terms \(\left((\C^{n_i})^{\circ}, \bar{\nabla}^i\right)\) are as in Proposition \ref{prop:nablaLsplit}.
\end{proposition}

\begin{note}
	Similar local affine splitting results as in Proposition \ref{prop:locaffprod} are proved in \cite[Section 2.5]{CHL} for Dunkl (but not necessarily unitary) connections. 
\end{note}

The proof of of Proposition \ref{prop:locaffprod} is as follows. First we use the unitary assumption of Theorem \ref{PRODTHM} to construct a gauge equivalence \(G\) between \(\nabla\) and \(\nabla^L\) on \(U^{\circ}\). We use the non-integer condition on the weights of \(\nabla\) at hyperplanes (\(a_H \notin \Z\) for all \(H \in \mH\)) to show that \(G\) extends holomorphically across \(\mH\). This together with the affine product decomposition of \(\nabla^L\) given by Proposition \ref{prop:nablaLsplit} give us a direct sum decomposition of \(TU\) into \(\nabla\)-parallel distributions. We integrate these distributions and show that \(\nabla\) is a product of its restrictions to the leaves through \(x\). The main point is to identify the affine structures on these leaves with standard ones. The key for doing so is to introduce a local dilation vector field \(e_x\) which preserves \(\nabla\) and vanishes at the point \(x\). After linearising \(e_x\), the connection \(\nabla\) becomes standard and the affine splitting expressed by Equation \eqref{eq:affinesplit} follows. The orthogonality of the factors in Equation \eqref{eq:affinesplit} with respect to the parallel inner product is a consequence of the non-integer assumptions of Theorem \ref{PRODTHM}. 

\begin{remark}
	The first and only place where the existence of a parallel positive definite Hermitian form \(\inn\) is used occurs in showing that the holonomies of \(\nabla\) and \(\nabla^L\) are conjugate close to \(x\) (Lemma \ref{lem:sameholonomy}). Our proof only requires that the holonomy of \(\nabla\) on the neighbourhood \(U^{\circ}\) is semi-simple.
\end{remark}

\subsection{Step 1: Gauge Equivalence}\label{sect:step1}

The first main result of this section is Lemma \ref{lem:sameholonomy} which guarantees that the holonomies of \(\nabla\) and the model connection \(\nabla^L\) are conjugate close to \(x\). As a result, we obtain a gauge transformation \(G:U^{\circ} \to GL(n, \C)\) between the two connections, see Definition \ref{def:G}. The non-integer assumption of Theorem \ref{PRODTHM} in the weaker form \(a_H \notin \Z\) for every \(H \in \mH\) together with well-known results for flat logarithmic connections from Appendix \ref{app:logconn} allow us to extend \(G\) holomorphically across the hyperplanes, see Lemma \ref{gauge}.

\subsubsection{{\large Local holonomy}}

Fix affine linear coordinates \((x_1, \ldots, x_n)\) centred at \(x \in L^{\circ}\) and trivialize the tangent bundle \(TU\) by means of the coordinate frame \(\p_{x_1}, \ldots, \p_{x_n}\).

\begin{notation}
	Fix \(p \in U^{\circ}\). We denote by
	\[ \Hol_p, \,\ \Hol^L_p : \pi_1(U^{\circ}, p) \to GL(n, \C)  \]
	the holonomy representations of the two flat connections \(\nabla, \nabla^L\). The main goal of this section is to prove Lemma \ref{lem:sameholonomy}, which shows that \(\Hol_p^L\) and \(\Hol_p\) are conjugate.
\end{notation}

\begin{definition}
	For \(0< \lambda \leq 1\) we let \(\varphi_{\lambda} : U \to \lambda U\) be the diffeomorphism given by scalar multiplication by \(\lambda\), that is \(\varphi_{\lambda}(q) = \lambda q\). We define a one-parameter family of connections given by
	\begin{equation}\label{eq:nablalambda}
		\nabla^{\lambda} = d - \varphi_{\lambda}^* \Omega .
	\end{equation}
	Equation \eqref{eq:nablalambda} is meant with respect to the fixed frame \(\p_{x_1}, \ldots, \p_{x_n}\).
\end{definition}

\begin{lemma}
	The connections \(\nabla^{\lambda}\) given by Equation \eqref{eq:nablalambda} are flat.
\end{lemma}

\begin{proof}
	Pull-back commutes with the exterior derivative and wedge product, so the flat condition \(d\Omega=\Omega\wedge\Omega\) is preserved. Indeed, both terms \(d\Omega\) and \(\Omega\wedge\Omega\) vanish in our particular case.
\end{proof}

\begin{lemma}\label{lem:pathconnections}
	As \(\lambda \to 0\), the connections
	\(\varphi^*_{\lambda}\Omega\) converge to \(\Omega^L\) smoothly on compact subsets of \(U^{\circ}\). 
\end{lemma}

\begin{proof}
	Note that \(\varphi_{\lambda}^* \Omega^L = \Omega^L\), so by Equation \eqref{eq:omegaL} we have \(\varphi_{\lambda}^* \Omega - \Omega^L = \varphi_{\lambda}^* \hol\) which converges smoothly to zero as \(\lambda \to 0\).
\end{proof}

Lemma \ref{lem:pathconnections} implies that 
\[ \{\nabla^{\lambda}, \hspace{2mm} 0 \leq \lambda \leq 1 \}\] 
is a smooth path of connections  with endpoints \(\nabla^0=\nabla^L\) and \(\nabla^1=\nabla\).
We denote the holonomy representation of \(\nabla^{\lambda}\) by
\begin{equation}\label{eq:hollambda}
	\Hol^{\lambda}_p: \pi_1(U^{\circ}, p) \to GL(n, \C)	.
\end{equation}

\begin{lemma}\label{lem:pathreps}
	The holonomies \eqref{eq:hollambda} give a smooth path of representations
	with \(\Hol_p^1 = \Hol_p\) and \(\Hol_p^0=\Hol^L_p\).
\end{lemma}

\begin{proof}
	Follows from Lemma \ref{lem:pathconnections} after fixing a set of generators for \(\pi_1(U^{\circ}, p)\).
\end{proof}

\begin{lemma}\label{lem:conjugate}
	If \(\lambda>0\) then the representation \(\Hol_p^{\lambda}\) is conjugate to  \(\Hol_p\).
\end{lemma}

\begin{proof}
	Let \(c\) be a loop in \(U^{\circ}\) based at \(p\). For \(\lambda>0\) write \(\lambda c\) for the loop \(\varphi_{\lambda} \circ c\) based at \(\lambda p\). Let \(\ell_{\lambda}\) be the straight segment from \(p\) to \(\lambda p\) and write
	\[T_{\lambda} \in GL(n, \C)\]
	for the parallel transport with respect to \(\nabla\) along \(\ell_{\lambda}\).
	The curve \(c\) is homotopic to \(\ell_{\lambda}^{-1} \cdot (\lambda c) \cdot \ell_{\lambda}\),  therefore
	\begin{equation} \label{eq:hol1}
	\Hol_p(c) =  T_{\lambda}^{-1} \cdot \Hol_{\lambda p}(\lambda c) \cdot T_{\lambda} .
	\end{equation}
	
	On the other hand, it follows from the definition of \(\nabla^{\lambda}\) that
	\begin{equation} \label{eq:hol2}
	\Hol_{\lambda p}(\lambda c) = \Hol^{\lambda}_{p}(c) .
	\end{equation}
	The lemma follows from Equations \eqref{eq:hol1} and \eqref{eq:hol2}. 
\end{proof}

\begin{lemma}\label{lem:closedorbit}
	The orbit of the representation \(\Hol_p: \pi_1(U^{\circ}, p) \to GL(n, \C)\) under the action of \(GL(n, \C)\) by conjugation is closed.
\end{lemma}

\begin{proof}
	Since \(\Hol_p\) preserves a positive definite Hermitian inner product on \(T_p\C^n\), it is semi-simple. It is a general fact that semi-simple representations of a finitely generated group \(\Gamma\) have closed orbits in \(\Hom(\Gamma, GL(n, \C))\) under the action by conjugation, see \cite[Theorem 1.7]{LM}.
\end{proof}

\begin{lemma} \label{lem:sameholonomy}
	Fix \(p \in U^{\circ}\). The holonomy representations 
	\[ \Hol_p, \,\ \Hol^L_p : \pi_1(U^{\circ}, p) \to GL(n, \C)  \]
	of the two flat connections \(\nabla, \nabla^L\) are conjugate.
\end{lemma}

\begin{proof}	
	Follows from Lemmas \ref{lem:pathreps}, \ref{lem:conjugate} and \ref{lem:closedorbit}.
\end{proof}

\begin{remark}\label{rmk:localizationunitary}
	Lemma \ref{lem:sameholonomy} implies that the connection \(\nabla^L\) admits a parallel positive definite Hermitian inner product.
\end{remark}

\subsubsection{{\large The gauge transformation}}

Let \(\Omega\) and \(\Omega^L\) be the connection matrices of \(\nabla\) and \(\nabla^L\) with respect to the coordinate frame \(\p_{x_1}, \ldots, \p_{x_n}\).

\begin{notation}
	We denote the entries of \(\Omega\) by \(\omega_{ij}\), thus  \(\omega_{ij}\) are holomorphic \(1\)-forms on \(U^{\circ}\). The matrix \(\Omega\) defines a Pfaffian system of linear differential equations
	\begin{equation}\label{eq:pfaffianeq}
	df_i = \sum_{j=1}^{n} \omega_{ij} f_j .
	\end{equation}
	A solution of \eqref{eq:pfaffianeq} is an \(n\)-tuple of holomorphic functions \(f_i\). We arrange them as a column vector \(f=(f_1, \ldots, f_n)^t\) and write \eqref{eq:pfaffianeq} more shortly as \(df=\Omega f\). 
\end{notation}

\begin{definition}
	A \textbf{fundamental solution} is an invertible matrix-valued function \(X\) which solves
	\begin{equation}\label{fundsol}
	dX = \Omega X .
	\end{equation}
	The columns of \(X\) are pointwise linearly independent vector-valued holomorphic functions whose components solve
	Equation \eqref{eq:pfaffianeq}. In other words, the columns of \(X\) form a frame of parallel vector fields for \(\nabla=d-\Omega\).
\end{definition}

\begin{lemma}
	Let \(p \in U^{\circ}\) and let \(W \subset U^{\circ}\) be a simply connected neighbourhood of \(p\). Then there is a fundamental solution of Equation \eqref{fundsol} defined on \(W\).
\end{lemma}

\begin{proof}
	Follows from the integrability condition \(d\Omega=\Omega\wedge\Omega\) together with Frobenius Theorem and analytic continuation, see \cite[Theorem 1.1]{NY}
\end{proof}

\begin{note}
	If \(X\) is a fundamental solution of Equation \eqref{fundsol} then any other fundamental solution of the system \eqref{fundsol} is given by \(XC\) where \(C\) is a constant invertible matrix, and any vector solution of Equation \eqref{eq:pfaffianeq} equals \(X v\) for some \(v \in \C^n\). Solutions of Equations \eqref{eq:pfaffianeq} and \eqref{fundsol} are uniquely determined by their value at a point, i.e. the initial condition \(f(p)\) or \(X(p)\).
\end{note}
 
\begin{notation}
	Let 
	\[\gamma: [0,1] \to U^{\circ}\] 
	be a path. We can pull-back the connection form \(\Omega\) to obtain \(\gamma^*\Omega\), a matrix of one forms on the interval. We slightly abuse notation and denote by \(X(t)\) a matrix valued solution of the linear system of ODEs
	\begin{equation}\label{eq:ODE}
		dX = (\gamma^*\Omega) X
	\end{equation}
	in the interval. It is a basic fact from the theory of Ordinary Differential Equations that given any \(X(0) \in GL(n, \C)\) there is a unique solution \(X:[0,1] \to GL(n, \C)\) to equation \eqref{eq:ODE} with initial value \(X(0)\). Alternatively, we can think of \(X(t)\) as the analytic continuation along \(\gamma\) of the local fundamental solution of Equation \eqref{fundsol} with value \(X(0)\) at \(p\), see \cite[Section 1.2]{NY}.
\end{notation}

\begin{definition}
	A fundamental solution of \(dX=\Omega X\) along a path \(\gamma:[0,1] \to U^{\circ}\) is a matrix-valued function \(X:[0,1] \to GL(n, \C)\) which solves Equation \eqref{eq:ODE}. 
\end{definition}

\begin{note}
	The column vectors of a fundamental solution \(X\!(t)\) of \(dX=\Omega X\) along \(\gamma\) are parallel vector fields along \(\gamma\). If \(\{v_1, \ldots, v_n\}\) form a basis of \(T_p\C^n\) given by the columns of \(X(0)\) then we can use \(\nabla\) to parallel transport the basis along \(\gamma(t)\) to get vector-valued functions \(v_i(t): [0,1] \to \C^n\) giving the columns of \(X(t)=(v_1(t), \ldots, v_n(t))\). Moreover, if \(\gamma\) is a closed loop \(c\) based at \(p \in U^{\circ}\) then we have the relation
	\[v_i(1) = \Hol_p(c) v_i(0) \]
	or equivalently \(\Hol_p(c) = X(1) \cdot X(0)^{-1}\).
\end{note}

\begin{definition}\label{def:G}
	Let \(q \in U^{\circ}\) and let \(\gamma: [0,1] \to U^{\circ}\) be a smooth path with \(\gamma(0)=p\) and \(\gamma(1)=q\). Let \(X, X^L\) be the fundamental solutions of
	\begin{equation}\label{eq:localfundsol}
		dX = \Omega X, \qquad dX^L = \Omega^L X^L
	\end{equation}
	analytically continued along \(\gamma\) with initial conditions \(X(0) = \Id \) and \(X^L(0)=A\). Here  \(A \in GL(n, \C)\) is such that 
	\begin{equation}\label{eq:conjugacy}
		\Hol_p(c) = A^{-1} \cdot \Hol^L_p(c) \cdot A 
	\end{equation}
	for any loop \(c\) based at \(p\), its existence is guaranteed by Lemma \ref{lem:sameholonomy}.  
	We define the \textbf{gauge transformation} \(G(q)\) by the relation
	\begin{equation}\label{eq:G}
		X^L(q) = G(q) \cdot X(q) .
	\end{equation}
\end{definition}

\begin{lemma}
	The above definition of \(G(q)\) does not depend on the choice of path \(\gamma\).	
\end{lemma}

\begin{proof}
	Let \(\gamma'\) be another path connecting \(p\) to \(q\) and let \(X', (X^L)'\) be the fundamental solutions of Equations \eqref{eq:localfundsol} along \(\gamma'\). We want to show that
	\begin{equation}\label{eq:XX'}
		(X^L)'(q) \cdot (X')^{-1}(q) = X^L(q) \cdot X^{-1}(q) .
	\end{equation}
	
	Set \(c= \gamma^{-1} \gamma'\), so that \(c\) is a loop based at \(p\) and \(\gamma'\) is homotopic to \(\gamma \cdot c\). The values of \(X'(q)\) and \((X^L)'(q)\) are easily deduced as follows:
	\begin{itemize}
		\item[(i)] After we go along \(c\) the solution \(X'\) changes from \(\Id\) to \(\Hol_p(c)\). Since \(X\) is a fundamental solution of \(dX=\Omega X\) along \(\gamma\) with initial value \(\Id\) and any other solution is obtained by right multiplication by a constant matrix, we conclude that
		\[X'(q) = X(q) \cdot \Hol_p(c) .\]
		\item[(ii)] 	After we go along \(c\) the solution \((X^L)'\) changes from \(A\) to \(\Hol_p(c)\cdot A\). Since \(X^L\) is a fundamental solution of \(dX^L=\Omega^L X^L\) along \(\gamma\) with initial value \(A\) and any other solution is obtained by right multiplication by a constant matrix, we conclude that
		\begin{align*}
			(X^L)'(q) &= X^L(q) \cdot A^{-1} \cdot \Hol^L_p(c) \cdot A \\
			&= X^L(q) \cdot  \Hol_p(c)
		\end{align*}
		where the second equality follows from Equation \eqref{eq:conjugacy}.
	\end{itemize}
	
	From (i) and (ii) we deduce that \(X'(q)\) and \((X^L)'(q)\) are obtained from \(X(q)\) and \(X^L(q)\) by right multiplication by the same matrix \(\Hol_p(c)\); therefore Equation \eqref{eq:XX'} holds.
\end{proof}

\begin{note}
	The \emph{monodromy matrix} of a fundamental solution of \(dX=\Omega X\) along a loop \(c\) is defined by
	\[M_X(c)= X(0)^{-1} \cdot X(1) .\]
	The monodromy matrices do depend on the initial condition. It is easy to check the following properties. 
	\begin{itemize}
		\item[(i)] If \(X(0)=\Id\) then \(M_X(c)=\Hol_p(c)\).
		\item[(ii)] If \(X\) is a fundamental solution of \(dX=\Omega X\) along \(c\) then any other fundamental solution is of the form \(X B\) where \(B\) is a constant matrix and 
		\[M_{XB}(c) = B^{-1} M_X(c) B .\]
	\end{itemize}
	We can restate Definition \ref{def:G} by saying that \(G\) is the ratio of two fundamental solutions \(G= X^L \cdot X^{-1}\) with initial conditions chosen so that the monodromies are the same.
\end{note}

\begin{lemma}
	The function \(G: U^{\circ} \to GL(n , \C)\) given by \eqref{eq:G} is holomorphic and satisfies the linear equation
	\begin{equation} \label{eq:dG}
		dG = \Omega^L G - G \Omega .
	\end{equation} 
\end{lemma}
	
\begin{proof}
	Locally on \(U^{\circ}\) we can write \(G = X^L \cdot X^{-1}\)
	where \(X, X^L\) are holomorphic solutions of \eqref{eq:localfundsol} and
	\begin{align*}
	dG &= (dX^L) X^{-1} - X^L X^{-1}(dX)X^{-1} \\
	&= \Omega^L G - G \Omega . \qedhere
	\end{align*}
\end{proof}

\begin{note}
	Multiplying Equation \eqref{eq:dG} on the right by \(G^{-1}\) we obtain the standard formula
	\begin{equation}\label{eq:gaugechange}
	\Omega^L = G \Omega G^{-1} + (dG )G^{-1}
	\end{equation}
	cf. \cite[Equation (1.6)]{NY}.
\end{note}

\begin{lemma}\label{gauge}
	The map \(G\) extends holomorphically over the whole neighbourhood \(U\) with values in \(GL(n, \C)\). In other words,
	\(\nabla\) and \(\nabla^L\) are holomorphically gauge equivalent.
\end{lemma}

\begin{proof}
	It suffices to show that both \(G\) and \(G^{-1}\) extend over \(U\) as holomorphic matrix-valued functions. Indeed, if this is established, then the extensions must satisfy \(G^{-1} \cdot G = \Id\) everywhere - because the identity holds on the dense set \(U^{\circ}\) - and therefore the extension of \(G\) must take values in \(GL(n, \C)\). We will prove that \(G\) extends holomorphically over \(U\), the argument for \(G^{-1}\) is the same.
	
	By Hartogs, it is enough to check that \(G\) extends in complex codimension one. Let \(H \in \mH_L\) and fix \(x \in H^{\circ}\).
	Let us consider a small polydisc \(D^n\) centred at \(x\) which intersects the hyperplanes of the arrangement only at \(H\). Take complex coordinates \((z_1, \ldots, z_n)\) on \(D^n\) with \(H=\{z_1=0\}\). We have three connections on \(D^n\) which we can write with respect to our fixed trivialization of \(TU\) as 
	\(\nabla= d- \Omega\), \(\nabla^L=d-\Omega^L\) and \(\nabla^0=d-\Omega^0\) with
	\[\Omega = A_H \frac{dz_1}{z_1} + \hol, \hspace{2mm} \Omega^L = A_H \frac{dz_1}{z_1} + \hol', \hspace{2mm} \Omega^0= A_H \frac{dz_1}{z_1} .\]
	The eigenvalues of the residue matrix \(A_H\) are \(a_H\) (with multiplicity \(1\)) and \(0\) (with multiplicity \(n-1\)). Since \(a_H\) is non-integer the non-resonance condition of the Normal Form Theorem \ref{thm:normalform} is satisfied, so we have a holomorphic gauge transformation \(G_1\) between \(\nabla\) and \(\nabla^0\) on \(D^n\). The composition (pointwise product) \(\tilde{G}= G_1 \cdot G\) defined on \(D^* \times D^{n-1}\) is a  gauge equivalence between \(\nabla^L\) and \(\nabla^0\). By Lemma \ref{lem:Ghol} the gauge transformation \(\tilde{G}\) extends holomorphically over \(D^n\). Since \(G_1\) has a holomorphic inverse over \(D^n\), we conclude that \(G=G_1^{-1}\tilde{G}\) is holomorphic on \(D^n\).	
\end{proof}

\subsection{Step 2: Foliations}\label{sect:foliation}

We use the gauge transformation \(G\) to define a frame of holomorphic vector fields close to \(x=0\). The members of the frame span integrable \(\nabla\)-parallel distributions. The main result is Proposition \ref{prodaffine}, which shows that \(\nabla\) is a product of the induced connections on the leaves through the origin.

\begin{definition}
	 We choose linear coordinates \((x_1, \ldots, x_n)\)  as provided by  Proposition \ref{prop:nablaLsplit} so that \(\nabla^L\) is a product. We refer to \((x_1, \ldots, x_n)\) as
	\textbf{adapted linear coordinates}.
\end{definition}

\subsubsection{{\large Parallel distributions and the \(s\)-frame}}

\begin{definition}\label{def:sframe}
	Let \((x_1, \ldots, x_n)\) be adapted linear coordinates and let \(G\) be the gauge transformation \eqref{eq:G}. Since \(G\) is holomorphic and invertible, we define the \textbf{\(s\)-frame} of holomorphic vector fields \((s_1, \ldots, s_n)\) on \(U\) by the equation
	\begin{equation}\label{eq:frame}
	\p_{x_i} = \sum_{j=1}^{n}G_{ji} s_j .
	\end{equation}
\end{definition}

\begin{note}
	The map \(G\) defines a section of \(\End T\C^n\) over \(U\) and its matrix representation \(G=(G_{ij})\) is taken with respect to the trivialization given by \(\p_{x_1}, \ldots, \p_{x_n}\), this is \(G\p_{x_i}=\sum_j G_{ji}\p_{x_j}\). Unwinding Equation \eqref{eq:frame} we get that \(G s_i = \p_{x_i}\) for all \(1 \leq i \leq n\). In other words, the \(s\)-frame is mapped under \(G\) to the coordinate frame of the adapted linear coordinates.
\end{note}

\begin{lemma}\label{lem:sframeconnection}
	The covariant derivatives \(\nabla s_i\) are given as follows.
	\begin{itemize}
		\item If \(i \in I_j\) for some \(1\leq j \leq k\) then \(\nabla s_i = -\sum_{m \in I_k} \Omega^L_{mi} s_m\).
		\item If \(i > n-d\) then  \(\nabla s_i = 0\).  
	\end{itemize}
\end{lemma}

\begin{proof}
	Recall that the columns of \(\Omega\) are given by \(\nabla \p_{x_i} = - \sum_{j=1}^{n} \Omega_{ji} \p_{x_j}\). If we write \(\mathbf{f}\) for the coordinate fields arranged in a row \(\mathbf{f} = (\p_{x_1}, \ldots, \p_{x_n})\) then we can write more shortly \(\nabla \mathbf{f} = - \mathbf{f} \Omega\). 
	
	We claim that \(\Omega^L\) is the connection matrix of \(\nabla\) with respect to the frame \((s_1, \ldots, s_n)\). Indeed,
	if we write \(\mathbf{s}=(s_1, \ldots, s_n)\) then \(\mathbf{s}= \mathbf{f} G^{-1}\) and 
	\begin{align*}
	\nabla \mathbf{s} &= (\nabla \mathbf{f}) G^{-1} + \mathbf{f} (dG^{-1}) \\
	&= - \mathbf{f} \Omega G^{-1} - \mathbf{f} G^{-1}(dG)G^{-1}  .
	\end{align*}
	Using Equation \eqref{eq:dG} we replace \(dG=\Omega^L G - G \Omega\) to obtain
	\begin{align*}
		\nabla \mathbf{s} &= - \mathbf{f} \Omega G^{-1} - \mathbf{f} G^{-1} \Omega^L + \mathbf{f} \Omega G^{-1} \\
		&= - \mathbf{s} \Omega^L 
	\end{align*}
	and our claim follows. More explicitly, 
	\begin{equation}\label{eq:nablasi}
	\nabla s_i = - \sum_{j=1}^{n} \Omega^L_{ji} s_j .
	\end{equation}

	On the other hand, recall that
	\[ \Omega^L = \sum_{H \in \mH_L} A_H \frac{dh}{h} . \]
	By our choice of linear coordinates \(x=(x_1, \ldots, x_n)\) and by Proposition \ref{prop:nablaLsplit} we have
	\begin{equation*}
	\Omega^L = \begin{pmatrix}
	*_{(n-d)\times (n-d)} & 0_{(n-d)\times d} \\
	0_{d\times (n-d)} & 0_{d\times d}
	\end{pmatrix} .
	\end{equation*}
	Moreover, we can further decompose \(*_{(n-d)\times (n-d)}\) into \(n_i \times n_i\) blocks: the \((i,j)\)-entry of \(*_{(n-d) \times (n-d)}\) vanishes unless \(i,j\) belong to the same set \(I_m\) for some \(1\leq m \leq k\). The lemma follows because the \(i\)-column of \(\Omega^L\) determines \(\nabla s_i\) by Equation \eqref{eq:nablasi}.
\end{proof}

\begin{definition}\label{def:dist}
	We set \(\mL\) to be the distribution spanned by \(s_{n-d+1}, \ldots, s_n\). For \(i=1, \ldots, k\) we let \(\mL^{\perp}_i\) be the distribution spanned by \(s_j\) with \(j \in I_i \subset \{1, \ldots, n-d\}\).
\end{definition}

Since \(s_1, \ldots, s_n\) is a frame of holomorphic vector fields, the distributions \(\mL\) and \(\mL^{\perp}_i\) are smooth (holomorphic) on our neighbourhood \(U\). The decomposition
\begin{equation*}
T\C^n|_U = \mL \oplus \mL^{\perp}_1 \oplus \ldots \oplus \mL^{\perp}_k
\end{equation*}
follows automatically from the definition of the distributions in terms of a frame.

\begin{lemma}
	The distributions \(\mL\) and \(\mL^{\perp}_i\) are parallel with respect to \(\nabla\) on \(U^{\circ}\). The distributions \(\mL\) and \(\mL^{\perp}_i\) are integrable on the whole neighbourhood \(U\). 
\end{lemma}

\begin{proof}
	We have \(\nabla s_i = - \sum_{j=1}^{n} \Omega^L_{ji} s_j\) with:
	\begin{itemize}
		\item \(\Omega^L_{ji} = 0\) for \(i=n-d+1, \ldots, n\);
		\item if \(i \in I_j \subset \{1, \ldots, n-d\}\) then \(\Omega^L_{mi}=0\) for \(m \notin I_j\).
	\end{itemize}
	These imply that \(\nabla(\mL) \subset \mL\) and \(\nabla(\mL^{\perp}_i)\subset \mL^{\perp}_i\), i.e. the distributions are parallel.
	
	Let \(X, Y\) be holomorphic vector fields tangent to \(\mL\), we want to show that \([X, Y]\) is also tangent. The connection \(\nabla\) is a smooth torsion free connection on its regular part \(U^{\circ}\), so \([X,Y] = \nabla_X Y - \nabla_Y X\) and each term \(\nabla_X Y\) and \(\nabla_Y X\) is tangent to \(\mL\) because \(\mL\) is parallel. Therefore \([X,Y]\) is tangent to \(\mL\) on \(U^{\circ}\) and by continuity it is tangent to \(\mL\) over all \(U\). The same argument applies to \(\mL^{\perp}_i\).
\end{proof}

\subsubsection{{\large Holonomy at central loops}}

\begin{definition}
	Let \((x_1, \ldots, x_n)\) be adapted linear coordinates and let \(p \in U^{\circ}\). We can write \(p\) according to the direct sum decomposition \eqref{eq:dirsumProp} as \(p = p_1+\ldots+p_k + p_L\) with \(p_i \in L_i^{\perp}\) and \(p_L \in L\). For \(1\leq j \leq k\) we define the \textbf{central loop} \(c_{p,j}: [0,1] \to U^{\circ}\) based at \(p\) as follows:
	\begin{equation*}
		c_{p,j}(t) = e^{2\pi i t}p_j + \sum_{k \neq j} p_k + p_L .
	\end{equation*}
\end{definition}

\begin{remark}
	We assume that \(U\) is invariant under scalar multiplication by \(\lambda\) for all \(0\leq |\lambda|\leq1\).
\end{remark}

\begin{note}
	\(c_{p,j}\) is the standard loop around the origin contained in the complex line 
	\[C_{p,j} = \{(p_1, \ldots \lambda p_j, \ldots, p_k, p_L), \,\ \lambda \in \C \} .\]
	Note that the line \(C_{p,j}\) intersects \(\cup_{H \in \mH_L}H\) only when \(\lambda=0\) and the intersection point lies in \(L_j\).
\end{note}

\begin{lemma}
	The homotopy classes \([c_{p,j}]\) belong to the centre of \(\pi_1(U^{\circ}, p)\).
\end{lemma}

\begin{proof}
	Let \(c\) be a loop based at \(p\). To prove that \(c\) commutes with \(c_{p,j}\) note that both lie in a \(2\)-torus in \(U^{\circ}\) which is obtained by an \(S^1\) rotation of \(c\) around \(L_j\).
\end{proof}

\begin{notation}
	Let \(M_{p,j}\) be the holonomy of \(\nabla\) along \(c_{p,j}\).
\end{notation}

\begin{lemma} \label{lem:holonomycentralloops1}
	At the point \(p\),
	the distribution \(\mL^{\perp}_j\) is equal to the \(\exp(2\pi i a_j)\)-eigenspace of \(M_{p,j}\)  and \(\mL \oplus_{k \neq j} \mL^{\perp}_k\)  agrees with the \(1\)-eigenspace of \(M_{p,j}\).
\end{lemma}

\begin{proof}
	The pull-back of \(\nabla\) to \(C_{p,j}\) in the trivialization given by the frame \(\{s_1, \ldots, s_n\}\) is
	\[d - A_j \frac{d \lambda}{\lambda} ,\]
	where \(A_j = \sum_{H \in \mH_j} A_H\). The matrix representation of \(M_{p,j}\) w.r.t. the basis \(\{s_1(u), \ldots, s_n(u)\}\)  equals \(\exp(2\pi i A_j)\).  The matrix \(A_j\) acts by scalar multiplication by \(a_j\) on \(L_j^{\perp}\) and \(\ker A_j = L_j\). By our choice of adapted linear coordinates  the matrix \(A_j\) is diagonal and the statement follows. 
\end{proof}

\begin{notation}
	We define
	\[c_{p,0}(t) = e^{2\pi i t} p, \qquad C_{p,0}=\{\lambda p, \,\ \lambda \in \C\} . \]
	Clearly \(c_{p,0}\) is homotopic to the product of all \(c_{p,j}\) for \(1\leq j \leq k\).
\end{notation}

\begin{note}
	The \(1\)-eigenspace of holonomy around \(c_{p,0}\) equals \(\mL\) at the point \(p\), it is the common fixed locus fo all \(M_{p,j}\).
\end{note}

The next lemma identifies the restrictions of the distributions \(\mL^{\perp}_i\) to the irreducible components \(L_i\).

\begin{lemma}\label{lem:restrictiondist}
	\(\mL_i^{\perp}|_{L_i} = L_i^{\perp}\) and \( \left( \mL \oplus_{j \neq i} \mL^{\perp}_j \right)|_{L_i} = L_i\). In particular,
	\(\mL|_{L} = L\) and \(\mL^{\perp}_i|_{L} = L^{\perp}_i\).
\end{lemma}

\begin{proof}
	Let \(p \in U^{\circ}\). We write \(p = p_{L_i} + p_{L_i^{\perp}} \) with respect to the decomposition \(\C^n=L_i \oplus L_i^{\perp}\). We let \(p(\lambda)= p_{L_i} + \lambda \cdot p_{L_j^{\perp}}\) for \(\lambda \in \C\), so \(p(\lambda) \in C_{p,j}\).
	The pull-back of \(\nabla\) to \(C_{p,i}\) in the trivialization given by the frame \(\{\p_{x_1}, \ldots, \p_{x_n}\}\) equals
	\[d - \left( A_i \frac{d \lambda}{\lambda} + \hol \right) .\]
	Let \(M(\lambda)\) be the matrix representation of \(M_{p(\lambda),i}\) in the frame given by the coordinate vector fields \(\p_{x_i}\).
	We use that \(A_i\) is diagonal with \(L_i=\ker A_i\) and \(L_i^{\perp} =\) \(a_i\)-eigenspace of \(A_i\) together with the identity
	\[\lim_{\lambda \to 0} M(\lambda) = \exp(2\pi i A_i) , \]
	see \cite[Th\'eor\`eme 1.17]{Deligne}. The result follows from Lemma \ref{lem:holonomycentralloops1}.
\end{proof}

\subsubsection{{\large Leaves and product structure}}

We integrate the parallel distributions introduced in Definition \ref{def:dist}. Our goal is to show that \(\nabla\) is a product of its restriction to the leaves through \(x\), see  Note \ref{not:prodaffine} and Proposition \ref{prodaffine}.

Recall the following standard result.

\begin{lemma}[Frobenius Theorem]\label{productcoord}
	Let \( \mD_1 , \ldots ,\mD_{q}\) be integrable holomorphic distributions on a neighbourhood \(0 \in U \subset \C^n\) such that \(T\C^n = \mD_1 \oplus \ldots \oplus \mD_q\). Fix index sets with \(|I_i| = \dim \mD_i\) and \( \sqcup_i I_i = \{1, \ldots, n\}\). Then there is a holomorphic change of coordinates \(y=y(x)\), defined on a possibly smaller neighbourhood of the origin, such that the leaves of \(\mD_i\) are the level sets \(\{y_j=const. \,\ j \notin I_i \}\). 
	
	Moreover, any two coordinates \(y, z\) as above differ by composing with biholomorphisms on each of the factors. This is \(z_j = z_j(y_{k}, \,\ k \in  I_i)\) whenever \(j \in I_i\). 
\end{lemma}

\begin{notation}
	Fix index sets \(I_0 \cup I_1 \cup \ldots \cup I_k = \{1, \ldots, n\}\) given by Equations  \eqref{eq:indexset} and \eqref{eq:I0}. 
\end{notation}

\begin{lemma}\label{coordinatesFrob}
	There are complex coordinates \(y=y(x)\) such that the leaves of \(\mL\) are level sets of \(\{y_j=const. \,\ j \notin I_0\}\) and the leaves of \(\mL^{\perp}_i\) are levels sets of \(\{y_j=const. \,\ j \notin I_i \}\). Equivalently,
	\begin{equation}\label{eq:frob}
		\mL = \Span \{\p_{y_j}, \,\, j \in I_0\}, \qquad \mL^{\perp}_i = \Span \{\p_{y_j}, \,\, j \in I_i \} \,\, \mbox{ for } \,\, 1 \leq i \leq k .
	\end{equation}
	Moreover, if \(y'\) is another coordinate with this property then \(y_i'=y_i'(y_j, \,\ j \in I_m)\) for \(i \in I_m\).
\end{lemma}

\begin{proof}
	We apply lemma \ref{productcoord} with \(q=k+1\), \(\mD_1 = \mL^{\perp}_1, \ldots, \mD_k = \mL^{\perp}_k\) and \(\mD_{k+1} = \mL\).
\end{proof}

\begin{note}\label{not:prodaffine}
	If \(y=(y_1, \ldots, y_n)\) are coordinates as in Lemma \ref{coordinatesFrob} then we have
	projection maps
	\begin{align*}
		\pi_i : U &\to \C^{|I_i|} \\
		y &\mapsto (y_j, \,\ j \in I_i)
	\end{align*}
	which identify \(U\) with a product
	\[U \cong \mL_x \times (\mL^{\perp}_1)_x \times \ldots \times (\mL^{\perp}_k)_x \]
	where
	\[\mL_x = \{y_1=\ldots=y_{n-d} =0\}, \qquad (\mL^{\perp}_i)_x = \{y_j=0, \,\, j \notin I_i \} \]
	are identified with \(\C^d\) and \(\C^{n_i}\) via the projections \(\pi_0\) and \(\pi_i\).
	Since the distributions are parallel, its leaves inherit restricted connections on its regular part, we denote them by
	\[\left( (\mL^{\perp}_i)^{\circ}_x, \nabla|_{\mL^{\perp}_i} \right) . \]
\end{note}

The goal of this section is to show the next.

\begin{proposition}\label{prodaffine}
	We can choose coordinates \((y_1, \ldots, y_n)\) close to \(x\) such that Equation \eqref{eq:frob} is satisfied and moreover the following holds. 
	\begin{itemize}
		\item[(i)] The coordinate vector fields that span \(\mL\) are parallel, that is 
		\begin{equation}\label{eq:prodaff1}
			\nabla \p_{y_j} =0 \,\, \mbox{ for all } \,\,
			j \in I_0 .
		\end{equation}
		\item[(ii)] The connection \(\nabla\) on \(U^{\circ}\) equals the product of the induced connections on the leaves through \(x\). More explicitly,
		\begin{equation}\label{eq:prodaff2}
			\left(U^{\circ}, \, \nabla\right) = \left(\mL_x, \, d\right) \times \left((\mL^{\perp}_1)_x^{\circ}, \, \nabla|_{\mL^{\perp}_1} \right) 
			 \times \ldots \times \left((\mL^{\perp}_k)_x^{\circ}, \, \nabla|_{\mL^{\perp}_k} \right) 
		\end{equation}
		where \(d\) is the trivial connection on the first factor.
	\end{itemize}
\end{proposition}

Before proving Proposition \ref{prodaffine} we make a slight digression on products of affine connections.

\begin{definition}\label{def:basic}
	Let \(X=X_1 \times X_2\) be a product manifold with projection maps \(\pi_i: X \to X_i\) to its factors and let \(v\) be a vector field on \(X\). We say that \(v\) is \textbf{basic} if we can write
	\[v= \tilde{v}_1+\tilde{v}_2\]
	where \(\tilde{v}_i\) are lifts of vector fields \(v_i\) on \(X_i\) via the projection maps \(\pi_i\), that is \((D\pi_1)\tilde{v}_1\) is a vector field on \(X_1\) and \((D\pi_2)\tilde{v}_1=0\); and similarly for \(\tilde{v}_2\).
\end{definition}

\begin{lemma}\label{lem:prodconect}
	Let \(X=X_1 \times X_2\) be a product manifold with projection maps \(\pi_i: X \to X_i\) to its factors and let \(\nabla\) be a torsion free connection on \(TX\). Suppose that the bundles \(\pi_i^*TX_i\) are parallel and that \(\nabla_v w\) is basic whenever \(v\) and \(w\) are basic. Then there are unique connections \(\nabla^i\) on \(TX_i\) such that \(\nabla\) is equal to their product.
\end{lemma}

\begin{proof}
	Let \(v_i, w_i\) be vector fields on \(X_i\) for \(i=1,2\) with lifts \(\tilde{v}_i, \tilde{w}_i\) to \(X\). The vector field \(\nabla_{\tilde{v}_i}\tilde{w}_i\) is basic and tangent to \(\pi_i^*TX_i\). We define \(\nabla^i\) by the relation
	\begin{equation}\label{eq:nablafactors}
	\nabla_{\tilde{v_i}}\tilde{w}_i = \widetilde{\nabla^i_{v_i}w_i} . 
	\end{equation}
	It is straightforward to verify that Equation \eqref{eq:nablafactors} defines a connection \(\nabla^i\) on \(TX_i\).
	On the other hand, since the vector fields \(\tilde{v}_1\) and \(\tilde{w}_2\) commute with each other,  the parallel bundles  \(\pi_i^*TX_i\) have zero intersection and \(\nabla\) is torsion free, we have \(\nabla_{\tilde{v}_1}\tilde{w}_2=0\) and  \(\nabla_{\tilde{v}_2}\tilde{w}_1 =0\). We conclude that Equation \eqref{eq:defnablaprod} is satisfied and therefore \(\nabla\) equals the product of \(\nabla^1\) and \(\nabla^2\). The uniqueness of the connections \(\nabla^i\) is also clear from Equation \eqref{eq:nablafactors}.
\end{proof}

\begin{example}[{\cite[pg. 209]{KN}}]
	Let \(\C^2 = \C \times \C\) with standard complex coordinates \(x_1, x_2\). 
	Consider the affine connection given by
	\[\nabla_{\p_{x_1}} \p_{x_1} = x_2 \p_{x_1}, \,\ \nabla_{\p_{x_i}} \p_{x_j} =0 \,\ \mbox{otherwise} .\]
	This connection is torsion free. The distributions spanned by \(\p_{x_1}\) and \(\p_{x_2}\) are parallel. However, \(\nabla\) is not a product. If it were then its curvature would vanish, but \(R(\p_{x_2},\p_{x_1})\p_{x_1}=\p_{x_1}\).  Note that the vector field \(x_2\p_{x_1}\) is not basic.
\end{example}

Going back to the setting of Proposition \ref{prodaffine}, we have the following.

\begin{definition}
	Let \(y\) be coordinates which integrate the parallel distributions as stated in Equation \eqref{eq:frob}. A vector field \(Y\) can be uniquely decomposed into components 
	\(Y=Y_1+\ldots+Y_k+Y_{0}\) where
	\[Y_i = \sum_{j \in I_i} c_j(y) \frac{\p}{\p y_j}\]
	is tangent to either \(\mL\) if \(i=0\) or \(\mL^{\perp}_i\) if \(1\leq i \leq k\). The vector field \(Y\) is basic if for every component \(Y_i\) its coefficients satisfy \(c_j = c_j(y_{m}, \,\ m \in I_i) \).
\end{definition}

\begin{note}
	The coordinates \(y\) give us projections \(\pi_i: U \to \C^{|I_i|}\). Every vector field on \(\C^{|I_i|}\) has a unique lift under these projections to a vector field tangent to \(\mL\) if \(i=0\) or \(\mL^{\perp}_i\) if \(1 \leq i \leq k\). These lifts are precisely the basic vector fields tangent to the parallel distributions. Since any other coordinate system \(y'\) which simultaneously integrates \(\mL\) and all \(\mL^{\perp}_i\) differs from \(y\) by biholomorphisms on the factors, see Lemma \ref{coordinatesFrob}, the notion of a basic vector field is independent of the coordinates chosen.
\end{note}

\begin{lemma}
	Let \(y\) be coordinates which satisfy Equation \eqref{eq:frob}. If \(i, j\) belong to different index sets \(I_m\) then
	\begin{equation} \label{vanishmixed}
	\nabla_{\p_{y_j}} \p_{y_i} = \nabla_{\p_{y_i}} \p_{y_j} =0 .
	\end{equation}
\end{lemma}

\begin{proof}
	Since \([\p_{y_i}, \p_{y_j}]=0\) and the connection is torsion free, we have \(\nabla_{\p_{y_i}} \p_{y_j} = \nabla_{\p_{y_j}} \p_{y_i}\). On the other hand, the covariant derivatives \(\nabla_{\p_{y_i}} \p_{y_j}\) and \(\nabla_{\p_{y_j}} \p_{y_i}\) must belong to different distributions \(\mL_j^{\perp}\) (or \(\mL\)) and \(\mL^{\perp}_i\) respectively which have zero intersection.
\end{proof}

\begin{lemma}\label{independence}
	Let \(y\) be coordinates which satisfy Equation \eqref{eq:frob}. Let \(i,j \in I_m \subset \{1, \ldots, n-d\}\) for some \(1 \leq m \leq k\). Write
	\[\nabla_{\p_{y_i}} \p_{y_j} = \sum_{l \in I_m} \Gamma^l_{ij} \p_{y_l} .\]
	Then \(\p_{y_s} \Gamma^l_{ij}=0\) for every \(s \notin I_m\). In other words, the vector field \(\nabla_{\p_{y_i}}\p_{y_j}\) is basic.
\end{lemma}

\begin{proof}
	Let \(s \notin I_m\) then
	\begin{align*}
	0 &= R(\p_{y_s}, \p_{y_i}) \p_{y_j} \\
	&= \sum_{l \in I_m} \p_{y_s} (\Gamma^l_{ij}) \p_{y_l} .
	\end{align*}
	The first equality uses that \(\nabla\) is flat, the second equality uses Equation \eqref{vanishmixed}.
\end{proof}

Recall that the frame field \((s_1, \ldots, s_n)\) is defined  by the gauge transformation \(G\), see Equation \eqref{eq:frame}. The next lemma asserts that the parallel fields \(s_{n-d+1}, \ldots, s_n\) which span \(\mL\) are basic.

\begin{lemma} \label{basicfields}
	Let \(y\) be coordinates which satisfy Equation \eqref{eq:frob}. For \(i>n-d\) write
	\[s_i = \sum_{j>n-d} c_{ij}(y) \frac{\p}{\p y_j} . \]
	Then \(\p_{y_m} c_{ij} =0\) for \(1\leq m \leq n-d\).
\end{lemma}

\begin{proof}
	Let \(1\leq m\leq n-d\) then, using that \(\nabla s_i =0\),
	\begin{align*}
	0 &= \nabla_{\p_{y_m}} s_i \\
	&= \sum_{j>n-d}  (\p_{y_m} c_{ij}) \frac{\p}{\p y_j} ,
	\end{align*}
	where the second equality uses Equation \eqref{vanishmixed}.
\end{proof}

\begin{lemma}\label{flatfactor}
	We can choose coordinates \(y\) that satisfy Equation \eqref{eq:frob} and moreover
	\[ s_i = \frac{\p}{\p y_i}\]
	for all \(i>n-d\).
\end{lemma}

\begin{proof}
	Lemma \ref{basicfields} guarantees that
	\[s_i = \sum_{j>n-d} c_{ij}(y_{n-d+1}, \ldots, y_n) \frac{\p}{\p y_j} \,\ \mbox{ for } i=n-d+1, \ldots, n . \]
	If we equip \(\C^d\) with coordinates \((y_{n-d+1}, \ldots, y_n)\) then \(s_{n-d+1}, \ldots, s_n\) make a frame of holomorphic vector fields on a neighbourhood of the origin in \(\C^{d}\). Moreover, the vector fields commute \([s_i,s_j]=0\) as we now explain. Take a leaf of \(\mL\) distinct from \(\mL_0=L\). The vector fields \(s_i\) are tangent to this leaf and \([s_i, s_j]= \nabla_{s_i} s_j - \nabla_{s_j}s_i =0\) because \(s_i, s_j\) are parallel. 
	
	Since we have a frame of mutually commuting vector fields on \(\C^{d}\), we can change coordinates \((y_{n-d+1}, \ldots, y_n) \mapsto (y'_{n-d+1}, \ldots, y'_n)\) so that the frame is made of coordinate vector fields.
\end{proof}

We can now establish Proposition \ref{prodaffine}.

\begin{proof}[Proof of Proposition \ref{prodaffine}]
	We take coordinates \(y\) as provided by Lemma \ref{flatfactor}.
	In particular, since the vector fields \(s_j\) for \(j>n-d\) are parallel, Equation \eqref{eq:prodaff1} is satisfied.
	
	Let us write \(\nabla_{\p_{y_i}} \p_{y_j} = \sum_{m=1}^{n} \Gamma^m_{ij} \p_{y_m}\) for \(1\leq i,j \leq n\). 
	It follows from Lemmas \ref{flatfactor} and \ref{independence} that \(\Gamma^m_{ij}=0\) unless \(i,j,m\) belong to the same index set \(I_l \subset \{1, \ldots, n-d\}\) in which case  \(\Gamma^m_{ij}=\Gamma^m_{ij}(y_s, \,\ s \in I_l)\).
	In particular the vector fields \(\nabla_{\p_{y_i}} \p_{y_j}\) are basic, so Lemma \ref{lem:prodconect} implies that \(\nabla\) is a product and Equation \eqref{eq:prodaff2} is satisfied.
\end{proof}

\subsection{Step 3: Local dilation vector field}\label{sect:loceulervf}

Recall that \(e\) is the Euler vector field for \(\nabla\) given by Equation \eqref{eq:euler}.

\begin{lemma}\label{lem:parallelvf}
	There is a holomorphic vector field \(s\) on \(U\) such that \(\nabla s = 0\) and \(s(x)=e(x)\).
\end{lemma}

\begin{proof}
	Note that \(e(x)\) belongs to \(L\), see Remark \ref{rmk:etangent}. The distribution \(\mL\) spanned by the parallel sections \(s_{n-d+1}, \ldots s_n\) is equal \(TL\) along \(L\) (see Lemma \ref{lem:restrictiondist}) so we can take \(s= \sum_{j>n-d} c_j s_j\) for suitable \(c_j \in \C\). 
\end{proof}

\begin{definition}\label{def:localdilation}
The \textbf{local dilation vector field}  at \(x \in L^{\circ}\) is defined by
\begin{equation*}
e_x = e - s
\end{equation*}
where \(e\) is the Euler field for the standard connection \(\nabla\) given by Lemma \ref{lem:eulstndcon} and \(s\) is the local parallel vector field given by Lemma \ref{lem:parallelvf}. 
\end{definition}

\begin{remark}
	By its definition, it is clear that
	 \(\nabla e_x = \Id\)  and \(e_x(x)=0\).
\end{remark}

\begin{notation}
	Fix index sets \(I_0 \cup I_1 \cup \ldots \cup I_k = \{1, \ldots, n\}\) given by Equations \eqref{eq:indexset} and \eqref{eq:I0}. 
	We use coordinates \(y\) centred at \(x\) given by Proposition \ref{prodaffine}:
	\begin{itemize}
		\item \(\mL\) is spanned by \(\p_{y_{n-d+1}}, \ldots, \p_{y_n}\) and \(\nabla \p_{y_i}=0\) for \(i>n-d\);
		\item  \(\mL^{\perp}_i\) is spanned by \(\{\p_{y_j}, \, j \in I_i\}\) for \(i=1, \ldots, k\).
	\end{itemize}
\end{notation}

\begin{lemma}\label{lem:eulerbasic}
	The local dilation vector field
	\(e_x\) is basic and tangent to the leaves \((\mL)_{x}\), \((\mL^{\perp}_i)_{x}\).
\end{lemma}

\begin{proof}
	Write 
	\begin{equation}\label{expansione}
	e_x = \sum_{j=1}^{n} c_j \p_{y_j}
	\end{equation}
	where \(c_j \in \mO(U)\) are holomorphic functions vanishing at the origin. For \(0 \leq i \leq k\) we set
	\begin{equation*}
	e_i = \sum_{j \in I_i} c_j \p_{y_j}
	\end{equation*}
	This way
	\begin{equation}\label{eq:edec}
	e_x = e_1 + \ldots + e_k + e_{0}
	\end{equation}
	is the decomposition of \(e_x\) into components of the parallel bundles \(\mL^{\perp}_1, \ldots, \mL^{\perp}_k, \mL\).
	The statement that \(e_x\) is basic means that:
	\begin{itemize}
		\item[(i)] if \(j>n-d\) and \(i \leq n-d\) then \(\p c_j/ \p y_i =0\);
		\item[(ii)] if \(j \in I_m \subset \{1, \ldots, n-d\}\) and \(i \notin I_m\) then \(\p c_j /\p y_i = 0\). 
	\end{itemize}
	Let us first take \(i>n-d\), so \(\nabla_{\p_{y_i}}\p_{y_j} =0\) for every \(j\) (see Equation \eqref{vanishmixed}). Since \(\nabla e_x\) is the identity endomorphism of the tangent bundle, we have \(\p_{y_i} = \nabla_{\p_{y_i}} e_x\). Replacing \(e_x\) with Equation \eqref{expansione}  and using that \(\nabla_{\p_{y_i}}\p_{y_j} =0\) for every \(j\), we obtain
	\begin{equation*}
	\p_{y_i} = \nabla_{\p_{y_i}} \left(\sum_{j=1}^{n} c_j \p_{y_j}\right) = \sum_{j=1}^r \frac{\p c_j}{\p y_i} \p_{y_j} .
	\end{equation*}
	It follows that
	\begin{equation}\label{eq:e1}
	\frac{\p c_j}{\p y_i} = 0 \,\ \mbox{ when } j \leq n-d \mbox{ and } i> n-d ;	
	\end{equation}
	\begin{equation} \label{eq:e2}
	\frac{\p c_j}{\p y_i} = \delta_{ij} \,\ \mbox{ when } i, j > n-d .
	\end{equation}
	Consider now \(i \in I_m\), so \(\nabla_{\p_{y_i}} \p_{y_j} = 0\) for every \(j \notin I_m\) (see Equation \eqref{vanishmixed}). Same as before, we obtain
	\begin{equation*}
	\p_{y_i} = \nabla_{\p_{y_i}} e_x = \nabla_{\p_{y_i}} e_m +  \left(\sum_{j \notin I_m} \frac{\p c_j}{\p y_i} \p_{y_j} \right).
	\end{equation*}
	It follows that term in parenthesis must vanish, so
	\begin{equation} \label{eq:e3}
	\frac{\p c_j}{\p y_i} = 0 \,\ \mbox{ when } j \notin I_m \mbox{ and } i \in I_m .
	\end{equation}
	Items (i) and (ii) follow from equations \eqref{eq:e1}, \eqref{eq:e2} and \eqref{eq:e3}.
	
	The fact that \(e_x\) is tangent to the leaves through the origin can be proved as follows. The transversal components of \(e_x\) restricted to a leaf are independent of the coordinates that parametrize the leaf -because of the basic property just proved- and these components all vanish at the origin because \(c_j(0)=0\) for all \(j\) as \(e_x\) vanishes at \(x\).
\end{proof}

\begin{remark}
	It follows from Equations \eqref{eq:e2} and \eqref{eq:e3} in the proof of the Lemma \ref{lem:eulerbasic}  that the component of \(e_x\) tangent to \(\mL\) in the decomposition \ref{eq:edec} is given by
	\[e_{0} = \sum_{j>n-d} y_j \p_{y_j} . \]
\end{remark}

\subsubsection{{\large Comparison between the vector fields \(e_x\) and \(e_L\)}}

\begin{notation}
	Let \((x_1, \ldots, x_n)\) be adapted linear coordinates centred at \(x\). We write \(e_L\) for the Euler vector field of the model connection \(\nabla^L\) given by Lemma \ref{lem:eulstndcon}.
\end{notation}

\begin{note}
	In our neighbourhood \(U\) we have two vector fields:
	\begin{itemize}
		\item \(e_L\) is a linear vector field in the coordinates \((x_1, \ldots, x_n)\), it satisfies \(\nabla^L e_L =\) identity and \(e_L(0)=0\).
		\item \(e_x\) is a non-explicit holomorphic vector field defined only on the neighbourhood \(U\), it satisfies \(\nabla e_x =\) identity and \(e_x(0)=0\).
	\end{itemize}
\end{note}

The goal of this section is to prove Proposition \ref{prop:eigenvalues} which shows that \(e_x\) and \(e_L\) have the same linear part at \(0\). 

\begin{note}
It well known and easy to check that if \(X\) is a vector field on a manifold \(M\) vanishing at some point \(p \in M\) and \(\nabla\) is a connection on \(TM\), then the endomorphism of \(T_pM\) defined by \(\nabla X\) is independent of the choice \(\nabla\). Our connections \(\nabla, \nabla^L\) aren't smooth, still they differ by a holomorphic term so the same argument applies to yield the following.	
\end{note}

\begin{lemma} \label{lem:e0}
	\((\nabla^L e_x)|_0 = \) identity on \(T_0\C^n\).
\end{lemma}

\begin{proof}
	Since \(e_x(0)=0\) we have
	\[\Id = (\nabla e_x)|_0 = (\nabla^L e_x)|_0 + \hol e_x(0) = (\nabla^L e_x)|_0 . \qedhere \]
\end{proof}

\begin{lemma}\label{lem:e1}
	Let \(X\) be a holomorphic vector field on \(U\), then \(\nabla^L X\) is holomorphic - regarded as a section of \(\End(T\C^n)\) over \(U\) - if and only if \(X\) is tangent to all hyperplanes in \(\mH_L\).
\end{lemma}

\begin{proof}
	Suppose \(X\) is tangent to all hyperplanes, then \(dh(X) = h f_{X, h}\) with \(f_{X, h} \in \mO(U)\). It follows that \(A_H(X) = dh(X) \cdot n_H = h f_{X, h} n_H\), so \(A_H(X)dh/h = f_{X, h} A_H\) and
	\begin{equation}\label{eq:holder}
	\nabla^L X = dX - \sum_{H \in \mH_L} f_{X,h} A_H
	\end{equation}
	is holomorphic.
	
	Conversely, if there is \(H \in \mH_L\) such that \(X\) is not tangent to \(H\) then we can find a point \(p \in H \setminus \cup_{H' \neq H} H'\) such that \(X_p \not \in T_p H\). In a small neighbourhood \(V\) of \(p\) all terms \(A_{H'}(X)dh'/h'\) are holomorphic for \(H'\neq H\) but the term \(A_H(X)dh/h\) has a simple pole at \(p\) since \(A_H(X_p)\neq 0\).
\end{proof}

\begin{lemma}\label{lem:e2}
	Let \(X\) be tangent to all hyperplanes and vanishing to second order at the origin, then \((\nabla^L X)|_0 = 0\).
\end{lemma}

\begin{proof}
	With the notation above, we have \(dh(X) = f_{X, h} h\) with \(f_{x, h} \in \mO(U)\). Since \(X\) vanishes to second order at the origin we must have  \(f_{X, h}(0)=0 \) and Equation \eqref{eq:holder} vanishes when evaluated at zero. 
\end{proof}

\begin{remark}
	Lemmas \ref{lem:e1} and \ref{lem:e2} are stated for \(\nabla^L\). Since \(\nabla\) differs from \(\nabla^L\) by a holomorphic term, both lemmas also hold for \(\nabla\).
\end{remark}

\begin{notation}
	In adapted linear coordinates \((x_1, \ldots, x_n)\) we write
	\[ e_x = \sum_{i=1}^{n} f_i \frac{\p}{\p x_i} \]
	with holomorphic functions \(f_i \in \mO(U)\) vanishing at the origin. We write \(f_i = \ell_i + \hot\) where \(\ell_i\) is linear and \(\hot\) denotes higher order terms. We have
	\begin{equation} \label{eq:eorder}
	e_x = e_{\ell} + e_{\hot}, \qquad e_{\ell}= \sum_i \ell_i \frac{\p}{\p x_i} 
	\end{equation}
	and refer to \(e_{\ell}\) as the linear part of \(e_x\).
\end{notation}

\begin{lemma} \label{lem:eltangenthyper}
	\(e_{\ell}\) is tangent to all hyperplanes.
\end{lemma}

\begin{proof}
	Let \(H=\{h=0\} \in \mH_L\), since \(e_x\) is tangent to \(H\) we have 
	\[e_x(h) = dh(e_x) = c_{X, h} h + \hot\] 
	with \(c_{X, h} \in \C\). On the other hand
	\[e_x(h) = \sum_i f_i \frac{\p h}{\p x_i} = \sum_i \frac{\p h}{\p x_i} \ell_i + \hot . \]
	Equating linear terms we get
	\[e_{\ell}(h) = \sum_i \frac{\p h}{\p x_i} \ell_i = c_{X, h} h . \]
	Thus \(e_{\ell}\) is tangent to \(H\).
\end{proof}

\begin{proposition}\label{prop:eigenvalues}
	The linear part of \(e_x\) is equal to \(e_L\), that is \(e_{\ell} = e_L\).
\end{proposition}

\begin{proof}
	The covariant derivative \(\nabla^L e_{\ell}\) is a constant endomorphism with respect to the coordinate frame \(\p_{x_1}, \ldots, \p_{x_n}\) given by
	\[\nabla^L e_{\ell} = d(e_{\ell}) - \sum c_h A_H \]
	where \(dh(e_{\ell}) = c_h h\) with \(c_h \in \C\).
	
	By Lemma \ref{lem:e0} \( (\nabla^L e)|_0 = \) identity and \(e_x = e_{\ell} + e_{\hot}\). On the other hand, by Lemma \ref{lem:eltangenthyper} \(e_{\hot}\) is tangent to all hyperplanes since both \(e_x\) and \(e_{\ell}\) are. By Lemma \ref{lem:e2} \((\nabla^L e_{\hot})|_0 = 0\). We conclude that \(\nabla^L e_{\ell} =\) identity at the origin, but since \(\nabla^L e_{\ell}\) is a constant endomorphism, it is the identity everywhere. Since \(e_{\ell}\) and \(e_L\) both vanish at the origin, Corollary \ref{cor:euniqueness} implies that \(e_{\ell}=e_L\).
\end{proof}

\subsection{Step 4: Identification of the transversal leaves}\label{sect:identification}

It this section we prove that the induced connections on the transversal leaves are standard with respect to suitable coordinates, see Proposition \ref{prop:transvstnd}. The key idea is to linearise the local dilation vector field (see Definition \ref{def:localdilation}).

\subsubsection{{\large Linearisation}}

\begin{notation}
Let \(y'=(y'_1, \ldots, y'_n)\) be coordinates as in Proposition \ref{prodaffine}  centred at \(x \in L^{\circ}\). This means that \(\mL^{\perp}_i\) is the span of \(\p_{y'_j}\) for \(j \in I_i \subset \{1, \ldots, n-d\}\), \(\mL\) is the span of \(\p_{y'_j}\) for \(j>n-d\) and \(\nabla \p_{y'_j}=0\) for \(j>n-d\). Let \(e_x\) be the local dilation vector field vanishing at \(x\).
\end{notation}

\begin{notation}
	Recall that we write \(a_i=a_{L_i}\) for the weights of \(\nabla\) at the irreducible components \(L_i\) of \(L\).
\end{notation}

\begin{proposition}\label{prop:linvf}
	We can change coordinates on the \(\C^{n_i}\) factors, this is \(y' \mapsto y\) with \(y_j = y_j(y'_m, \,\ m \in I_i)\) when \(j \in I_i\), in such a way that
	\begin{equation}
	e_x = (1-a_1)^{-1} \sum_{j \in I_1}y_j \frac{\p}{\p y_j} + \ldots + (1-a_k)^{-1} \sum_{j \in I_k} y_j \frac{\p}{\p y_j} + \sum_{i>n-d} y_j \frac{\p}{\p y_j} .
	\end{equation}
\end{proposition}

\begin{proof}
	From Lemma \ref{lem:eulerbasic} it is enough to consider \(e_m =\) the restriction of \(e\) to \(\mL^{\perp}_m\) and show that there is a holomorphic change of coordinates in \(\C^{n_m}\) in which \(e_m\) is equal to \((1-a_m)^{-1}\) times the usual Euler vector field of \(\C^{n_m}\).
	
	By Proposition \ref{prop:eigenvalues} the linear part of \(e_m\) on \(T_0(\mL^{\perp}_m) = L^{\perp}_m\) is the same as that of \(e_L\) restricted to this subspace and that is \((1-a_m)^{-1}\) times the identity. By  Proposition \ref{poincaredulac}) we can linearise \(e_m\) and the statement follows.
\end{proof}

\begin{proposition}\label{poincaredulac}
	Let \(v\) be a germ of a holomorphic vector field on \((\C^n,0)\) with an isolated zero at the origin. If all the eigenvalues of the linear part of \(v\) are equal and non-zero then \(v\) is linearisable. The coordinates are unique up to linear transformation.
\end{proposition}

\begin{proof}
	This follows from the Poincar\'e-Dulac theorem, see \cite[Theorem 5.5]{Yak} . The eigenvalues belong to the Poincar\'e domain (i.e. their convex hull in \(\C\) does not contain zero) hence after a change of coordinates we can write \(v = v_S + v_N \) as the sum of a linear diagonal vector field \(v_S\) (the semisimple part) with a polynomial vector field \(v_N\) (the nilpotent part). The components of \(v_N\) are sums of resonant monomials \(x_1^{d_1} \ldots x_n^{d_n}\) where \(d_j \in \Z_{\geq 0}\) solve
	\[\lambda_i = \sum \lambda_j d_j . \]
	In particular, if all \(\lambda_i=\lambda \neq 0\) we must have \(\sum_j d_j = 1\) which has \(d_i=1\) and \(d_j=0\) for \(j \neq i\) as the only solutions; the corresponding monomials are linear.  Finally, we recall that the coordinates linearising a non-resonant vector field are unique up to linear transformations, see \cite[Remark 4.6]{Yak}.
\end{proof}

\subsubsection{{\large Transversal leaves are standard}}

We begin with the following general statement, see Appendix \ref{app:logconn} for the definitions of a meromorphic connection with a logarithmic singularity and of residue.

\begin{lemma}\label{lem:looij}
	Let \(\nabla\) be a (flat, torsion free) meromorphic connection on \(T\C^n\) with logarithmic singularities at the smooth points of a hyperplane arrangement \(\mH\). Suppose that the residues of \(\nabla\) extend holomorphically over the hyperplanes of \(\mH\). If \(\nabla\) is invariant by scalar multiplication then it is standard.
\end{lemma}

\begin{proof}
	Let \(\Omega\) be the connection matrix of \(\nabla\) with respect to the coordinate frame \(\p/\p z_1, \ldots, \p/\p z_n\). By the hypothesis on the residues we can write
	\[\Omega = \sum_{H \in \mH} A_H \frac{dh}{h} + \hol \]
	where \(A_H\) are holomorphic matrix-valued functions and \(\hol\) is a matrix of holomorphic \(1\)-forms.
	Let \(\varphi_{\lambda}(z)=\lambda z\) be scalar multiplication by \(\lambda\). The invariance of \(\nabla\) by scalar multiplication implies that \(\varphi_{\lambda}^* \Omega = \Omega\). Since \(\lim_{\lambda \to 0} \varphi_{\lambda}^* \hol =0\), by letting \(\lambda \to 0\) in the identity \(\Omega = \varphi_{\lambda}^* \Omega\) we deduce that
	\[\Omega = \sum_{H \in \mH} A_H(0) \frac{dh}{h} .\]
	Hence \(\nabla\) is standard.
\end{proof}

\begin{note}
	The statement of Lemma \ref{lem:looij} is contained in \cite[Proposition 2.2]{CHL}, except that we add the extra hypothesis of holomorphic residues.
\end{note}

Going back to our setting we have the next.

\begin{proposition} \label{prop:transvstnd}
	Let \((y_1, \ldots, y_n)\) be coordinates that linearise \(e_x\) as provided by Proposition \ref{prop:linvf} and let \(1 \leq m \leq k\). Then the induced connection on the leaf \((\mL^{\perp}_m)_{x}\) is standard with respect to the coordinates \((y_i, \,\ i \in I_m)\).
\end{proposition}

\begin{proof}
	On the tangent bundle of the leaf we have two frames: the \(s\)-frame
	\(\{s_j, \,\ j \in I_m\}\) and the \(y\)-frame \(\{\p_{y_j}, \, j \in I_m\}\). The connection on the tangent bundle of the leaf written with respect to the  \(s\)-frame is \(\bar{\nabla}^m = d - \Omega_s\). By Lemma \ref{lem:sframeconnection}
	\begin{equation} \label{eq:indlog}
	\Omega_s =
	\sum_{H \in \mH_m} \bar{A}_H \frac{dh}{h} , 
	\end{equation}
	where \(\bar{A}_H\) are constant \(n_m \times n_m\) matrices and \(h\) is a linear function of \(x_j\) for \(j \in I_m\). On the other hand \((x_j, \,\ j \in I_m)\) are coordinates on a neighbourhood of the origin in the leaf, as follows from the fact that \(T_0(\mL^{\perp}_m)_0=L^{\perp}_m\). So we can write \(x_i=x_i(y_j, \,\ j \in I_m)\) for \(i \in I_m\) and \(h=h(y_j, \,\ j \in I_m)\) are defining equations of the hypersurfaces \(H \cap (\mL^{\perp}_m)_0\). The intersections \(H \cap (\mL^{\perp}_m)_0\) are smooth hypersurfaces of the leaf invariant under scalar multiplication \(y \mapsto \lambda y\), hence must be hyperplanes in \(y\)-coordinates. We can write \(h(y) = \ell(y) e^u\) for some holomorphic \(u(y)\) and linear \(\ell(y)\), so
	\begin{equation}\label{eq:indlog2}
	\frac{dh}{h} = \frac{d \ell}{\ell} + du .
	\end{equation}
	On the other hand, with respect to the  \(y\)-frame we have \(\nabla = d - \Omega_y\)
	where
	\begin{equation} \label{eq:indlog3}
	\Omega_y = \Phi \Omega_s \Phi^{-1} + (d\Phi)\Phi^{-1}
	\end{equation}
	and \(\Phi\) is the holomorphic gauge transformation from the \(y\)-frame to the \(s\)-frame. It follows from Equation \eqref{eq:indlog}, \eqref{eq:indlog2} and \eqref{eq:indlog3} that
	\begin{equation}\label{eq:indlog4}
		\Omega_y = \sum_{H \in \mH_m} \tilde{A}_H \frac{d\ell}{\ell} + \hol 
	\end{equation}
	where (hol.) is a matrix of holomorphic \(1\)-forms and \(\tilde{A}_H = \Phi \bar{A} \Phi^{-1}\). 
	The connection \(\bar{\nabla}^m = d - \Omega_y\) is invariant under scalar multiplication \(y \mapsto \lambda y\) and Equation \eqref{eq:indlog4} implies that it has holomorphic residues. Hence, by Lemma \ref{lem:looij} we conclude that
	\begin{equation} \label{eq:Omegay}
		\Omega_y = \sum_{H \in \mH_m} \tilde{A}_H(0) \frac{d\ell}{\ell}
	\end{equation}
	which shows that the connection is standard.
\end{proof}

\begin{remark}
	It follows from Equation \eqref{eq:Omegay} that the induced connection on the leaf \((\mL^{\perp}_m)_0\) agrees with the quotient connection \((\C^n/L_m, \mH_{L_m}/L_m, \nabla^{\C^n/L_m})\).
\end{remark}

We can now finish the proof of Proposition \ref{prop:locaffprod}.

\begin{proof}[Proof of Proposition \ref{prop:locaffprod}]
	We take coordinates that linearise \(e_x\) as provided by Proposition \ref{prop:linvf}.
	From Proposition \ref{prodaffine}
	we know that the affine structure is a product \((\mL^{\perp}_1)_0 \times \ldots \times (\mL^{\perp}_k)_0 \times \mL_0\). The affine structure on \(\mL_0\) is the one of \(\C^d\) and the transverse structures \((\mL^{\perp}_i)_0\) are standard as follows from Proposition \ref{prop:transvstnd}.
\end{proof}

\subsection{Proof of Theorem \ref{thm:locprod}}\label{sect:pflocprod}

\begin{lemma}\label{lem:orthogfact}
	The factors in the product decomposition \eqref{eq:affinesplit} are pairwise orthogonal with respect to the parallel Hermitian form \(\inn\).
\end{lemma}

\begin{proof}
	Let \(\langle \cdot, \cdot \rangle\) be the parallel Hermitian inner product. We claim that for every \(1\leq m \leq k\) the distribution \(\mL^{\perp}_m\) is orthogonal to \(\mL\oplus \sum_{i \neq m} \mL^{\perp}_i\), as follows. Let \(p \in U^{\circ}\), take a complex affine line  in the coordinates \((y_1, \ldots, y_n)\) that goes through \(p\) and has all coordinates \(y_j\) constant except when \(j\in I_m\) and  projects to the \(\C^{n_m}\) factor as a complex line through the origin. This complex line \(\ell\) intersects the hyperplanes only at its origin. By Lemma \ref{lem:holonomycentralloops1}, the holonomy around the central loop (based at \(p\)) in the line \(\ell\) is equal to the identity on \(\left( \mL\oplus \sum_{i \neq m} \mL^{\perp}_i \right)(p)\) and multiplication by \(\exp(2\pi i a_m)\) on \(\mL^{\perp}_m (p)\). Take vectors \(v \in \mL^{\perp}_m (p)\) and \(w \in \left( \mL\oplus \sum_{i \neq m} \mL^{\perp}_i \right)(p)\), since holonomy preserves the inner product we have
	\begin{equation*}
	\langle v, w \rangle = \exp(2\pi i a_m) \langle v, w \rangle .
	\end{equation*}
	Because \(a_m \notin \Z\), we must have \(\langle v, w \rangle =0 \). The claim follows and we conclude that the direct sum
	\[\mL \oplus \mL^{\perp}_1\oplus \ldots \oplus \mL^{\perp}_k\]
	is orthogonal.
\end{proof}

We establish the main result of the section,  a cornerstone in our proof of Theorem \ref{PRODTHM}.

\begin{proof}[Proof of Theorem \ref{thm:locprod}]
	We take complex coordinates \((y_1, \ldots, y_n)\) as in Proposition \ref{prop:locaffprod}, so we have an affine product as given by Equation \eqref{eq:affinesplit}. By Lemma \ref{lem:orthogfact} the factors are pairwise orthogonal. 
	We conclude that in the \((y_1, \ldots, y_n)\) coordinates the metric \(g\) is holomorphically isometric to the product
	\begin{equation}\label{eq:metlocprod}
	\C^d \times ((\C^{n_1})^{\circ}, g^{\perp}_1) \times \ldots \times ((\C^{n_k})^{\circ}, g^{\perp}_k)	
	\end{equation}
	where \(g^{\perp}_i\) are flat K\"ahler metrics on the arrangement complements \((\C^{n_i})^{\circ}\) which are parallel for the connections \(\bar{\nabla}^i\).
	Finally, we recall that the product 
	\[((\C^{n-d})^{\circ}, \bar{\nabla}) = \left((\C^{n_1})^{\circ}, \bar{\nabla}^1\right) \times \ldots \times \left((\C^{n_k})^{\circ}, \bar{\nabla}^k\right)\]
	is naturally identified with the quotient connection \(\nabla^{\C^n/L}\) on the complement of the quotient arrangement \(\mH_L/L\) (see Remark \ref{rmk:quotconnec}). Equation \eqref{eq:lpt} follows and so does the theorem.
\end{proof}

\begin{remark}
	In our proof of Theorem \ref{PRODTHM} we have established Equation \eqref{eq:metlocprod} which is more refined than the product decomposition expressed by Equation \eqref{eq:lpt}. The factors \(((\C^{n_i})^{\circ}, g^{\perp}_i)\) can not be further decomposed into products, as follows from our Irreducible Holonomy Theorem \ref{thm:irredhol}.
\end{remark}

We finish by proving Corollary \ref{cor:locprod}.

\begin{proof}[Proof of Corollary \ref{cor:locprod}]
	In this case \(\nabla^{\C^n/H}\) is a one-dimensional standard connection so it is of the form
	\[\nabla = d - a_H\frac{dz}{z} \]
	with \(\alpha_H=1-a_H>0\) and it is equal to the Levi-Civita connection of the \(2\)-cone \(\C_{\alpha_H}\). The corollary follows from Theorem \ref{thm:locprod}.
\end{proof}

\section{Metric completion}\label{sect:MC}

During this whole section we work under the hypothesis and setting of Theorem \ref{PRODTHM}. Moreover, from now on we shall assume that \(\mH\) is essential and irreducible. There is no loss of generality in doing so because of the next.

\begin{lemma}
	In order to prove Theorem \ref{PRODTHM} it is enough to assume that \(\mH\) is essential and irreducible.
\end{lemma}

\begin{proof}
	Proposition \ref{prop:nablaLsplit} applied to \(\nabla=\nabla^{T(\mH)}\) shows that \(\nabla\) is a product connection and the factors are pairwise orthogonal because of the non-integer assumptions, see Remark \ref{rmk:orth}. Further, we note that the metric completion of a product is the product of the metric completions of its factors and, by definition of polyhedral cone, the product of polyhedral cones is a polyhedral cone.
\end{proof}

\subsection{The flat K\"ahler cone metric \(g\) on \((\C^n)^{\circ}\)}\label{sect:riemcone}

We show that the flat K\"ahler metric \(g\) on \((\C^n)^{\circ}\) is a Riemannian cone. As we will see, this is a direct consequence of the existence of an Euler vector field and does not rely on any of the results from Section \ref{sec:locprod}. We begin by recalling some standard conventions regarding real and complex tangent bundles.

\begin{notation}
	We identify \((T\R^{2n}, J)\) with \(T^{1,0}\C^n\) via
	\begin{equation}\label{eq:identificationT}
		\begin{aligned}
		T\R^{2n} &\xrightarrow{\sim} T^{1,0}\C^n \\
		v \,\, &\mapsto \,\, \frac{1}{2} \left(v - iJv \right)= v^{1,0} .
		\end{aligned}
	\end{equation}
	Note that \(v = 2 \RE(v^{1,0})\) and \((Jv)^{1,0}=i \cdot v^{1,0}\).
\end{notation}

\begin{notation}
	On \((\C^n)^{\circ}\) we have a torsion free flat holomorphic connection \(\nabla\) and a parallel positive definite Hermitian inner product \(\inn\) in \(T^{1,0}\C^n\).
	Using the identification \eqref{eq:identificationT}, we consider \(\nabla\) and \(\inn\) as structures on \((T\R^{2n}, J)\) and we denote them by the same symbol.
	We write
	\begin{equation}
		2\inn = g - i \omega
	\end{equation}
	where \(g\) is a flat K\"ahler metric on \((\C^n)^{\circ}\) and
	\begin{equation}
		\omega(\cdot, \cdot) = g(J \cdot, \cdot)
	\end{equation}
	is its K\"ahler form. We identify \(\nabla\) with the Levi-Civita connection of \(g\). The K\"ahler form \(\omega\) and the complex structure \(J\) are parallel with respect to \(\nabla\).
\end{notation}

\begin{notation}
	Since we are assuming that \(\mH\) is essential and irreducible, the Euler vector field for \(\nabla\) is given by
	\begin{equation}\label{eq:erecall}
		e = \frac{1}{\alpha_0} \sum_{i=1}^{n} z_i \frac{\p}{\p z_i} 	
	\end{equation}
	and \(\nabla e\) is the identity endomorphism of \(T^{1,0}\C^n\). Here, we recall that \(\alpha_0 = 1- a_0\) where \(a_0\) is the weight at the origin of \(\nabla\).
\end{notation}

\begin{definition}
	The real Euler vector field \(e_r\) is defined as \(e_r^{1,0} = e\), equivalently \(e_r = 2 \RE(e)\).
	The Reeb vector field \(e_s\) is defined as
	\(e_s=Je_r\).
\end{definition}

\begin{note}
	It follows that
	\begin{equation}
		e = \frac{1}{2} \left(e_r - ie_s\right)
	\end{equation}
	and 
	\begin{equation}
		\nabla e_r = \Id, \hspace{2mm} \nabla e_s = J .
	\end{equation}
	Equation \eqref{eq:erecall} implies that \(e_r\) and \(e_s\) are equal to the standard real Euler and Reeb vector fields of \(\C^n\) multiplied by \((1/\alpha_0)\).
\end{note}

\begin{definition}
	We denote the norm squared of the real Euler vector field of \(\nabla\) by
	\begin{equation}\label{eq:norme2}
	r^2= 2 \langle e, e \rangle = g(e_r, e_r) .	
	\end{equation}
	By definition, it is a smooth positive function on \((\C^n)^{\circ}\)
\end{definition}

\begin{lemma}
	The function \(r\) is \(\alpha_0\)-homogeneous with respect to standard scalar multiplication, that is
	\begin{equation}\label{eq:rhomog}
		r(\lambda z) = |\lambda|^{\alpha_0} r(z)
	\end{equation}
	for all \(\lambda \in \C^*\) and \(z \in (\C^n)^{\circ}\).
\end{lemma}

\begin{proof}
	Note that by Equation \eqref{eq:norme2} we have
	\[e_r (r^2) = 2 g(\nabla_{e_r}e_r, e_r) = 2r^2\]
	and
	\[e_s (r^2) = 2 g(\nabla_{e_s}e_r, e_r) = g (e_s, e_r) = 0 . \]
	The first equation implies that \(r^2(\lambda z)=\lambda^{2\alpha_0} r^2(z)\) for every \(\lambda>0\). The second implies that \(r^2\) is invariant under scalar multiplication by complex units. The statement follows.
\end{proof}

\begin{lemma}\label{lem:gradr}
	\(e_r = \grad (r^2/2)\) and \(\grad r\) has unit norm.
\end{lemma}

\begin{proof}
	Let \(v\) be a real tangent vector, the identity
	\[ (d(r^2/2))(v) = v (r^2/2) = g (\nabla_v e_r, e_r) = g(v, e_r) \]
	implies that \(e_r=\grad (r^2/2)\). We conclude that \(\grad r = r^{-1} e_r\) and therefore \(|\grad r| =1\).
\end{proof}

\begin{remark}
	It follows that \(r\) is a `distance function' on the Riemannian manifold \(((\C^n)^{\circ}, g)\). In particular, the integral curves of \(\nabla r\) are unit speed parametrized geodesics. See \cite{Pet}.
\end{remark}

\begin{definition}
	We let
	\[S^{\circ}=\{r=1\}\cap (\C^n)^{\circ} .\] 
\end{definition}

\begin{remark}
	Lemma \ref{lem:gradr} implies that \(S^{\circ}\) is a smooth real hypersurface of \((\C^n)^{\circ}\).
\end{remark}

\begin{definition}
	We introduce a diffeormorphism
	\begin{equation}\label{eq:polarcoord}
	\begin{aligned}
	F \, : \, (0, + \infty) \times S^{\circ} &\to (\C^n)^{\circ} \\
	(t, z) \,\, &\mapsto \,\, t^{1/\alpha_0}z . 
	\end{aligned} 
	\end{equation}
		
\end{definition}

\begin{note}\label{not:polcoord}
	By its definition, the map \(F\)  pulls-back the function \(r\) to the coordinate \(t\).
	We can compare the map defined by Equation \eqref{eq:polarcoord} with the standard
	one when \(n=1\) 
	\[(0, + \infty) \times S^1 \to \C^*\]
	where \(S^1=\R/2\pi\Z\) given by \(	(t, \theta) \,\, \mapsto \,\, z =t^{1/\alpha}e^{i\theta} \) which
	 pulls-back the metric of a \(2\)-cone defined in complex coordinates by \((i/2)\dd |z|^{2\alpha}\) into its polar form \(dt^2 + \alpha^2 t^2 d\theta^2\).
\end{note}

\begin{lemma} \label{lem:kahlercone2}
	The pull-back of \(g\) under the map \eqref{eq:polarcoord} equals  
	\begin{equation}\label{eq:gcone}
	dt^2 + t^2 g_{S^{\circ}}	
	\end{equation}
	where \(g_{S^{\circ}}\) is the pull-back of \(g\) under the inclusion map \(S^{\circ} \subset (\C^n)^{\circ}\). 
	In other words, the metric \(g\) is a Riemannian cone.
\end{lemma}

\begin{proof}
	Note that \(DF(\p/\p t) = \grad r\).
	Since \(\grad r\) is orthogonal to \(TS^{\circ}\), we conclude that  \(dt^2+t^2g_{S^{\circ}}\) agrees with \(F^*g\) when restricted to \(\{1\}\times S^{\circ}\). On the other hand, both metrics have homogeneous degree \(2\) with respect to \(t\p_t\), hence they must agree everywhere. 
\end{proof}

\begin{remark}
	It follows from Equation \eqref{eq:gcone} that the rays \(t \mapsto t^{1/\alpha_0} z\), where \(z \in S^{\circ}\) is fixed, are unit speed minimizing geodesics.
\end{remark}

\begin{lemma}
	\(g=\Hess (r^2/2)\) and 
	\begin{equation}\label{eq:omega}
		\omega = \frac{i}{2} \dd r^2 .
	\end{equation}
\end{lemma}

\begin{note}
These formulas hold in great generality for Riemannian and K\"ahler cones.	
\end{note}

\begin{proof}
	By definition, \(\mbox{Hess}(f)(X,Y ) = (\nabla df) (X,Y) = XY(f) - \nabla_XY(f)\). In our case
	\[\mbox{Hess}(r^2/2)(X, Y) = X(g(Y, e_r)) - g (\nabla_XY, e_r) = g(X, Y) . \]
	
	Let \(J\) act one one-forms by \((J\alpha)(X)=\alpha(JX)\) and set \(\eta= Jd(r^2/2)\), so \(\eta(X)=g(JX, e_r)\). This way
	\[(\nabla_X \eta)(Y) = X(g(JY, e_r)) - \eta(\nabla_XY) = g(JY, X) = \omega(Y, X) \]
	Recall that if \(\eta\) is a one-form then \(d\eta (X, Y) = (\nabla_X \eta)(Y) - (\nabla_Y \eta)(X) \). We get that
	\[dJd(r^2/2) (X, Y) =2\omega(Y, X) .\]
	Since \(-2i\dd=dJd\), we conclude that Equation \eqref{eq:omega} holds.
\end{proof}

\begin{lemma}\label{lem:lines}
	Let \(\ell \subset (\C^n)^{\circ}\) be a complex line through the origin. Then the restriction of \(g\) to \(\ell\) equals the metric of a \(2\)-cone of total angle \(2\pi\alpha_0\) with vertex at the origin. Moreover, \(\ell\) is totally geodesic
\end{lemma}

\begin{proof}
	The function \(r^2|_{\ell}\) is a potential for the restricted metric. The statement that \(g|_{\ell}\) is a \(2\)-cone of angle \(2\pi\alpha_0\) follows from Equation \eqref{eq:rhomog} (see also Note \ref{not:polcoord}). 
	
	Up to translations by constants, the standard Euler vector field \(e_{\C^n}\) is the only whose derivative equals the identity endomorphism. 
	We conclude that in flat coordinates, the Euler vector field \(e\) is mapped to \(e_{\C^n} + c\), where \(c\) is a constant vector field. Since \(\ell\) is invariant by \(e\), the image of \(\ell\) is part of a complex line.
\end{proof}

\begin{note}
	In particular, it follows from Lemma \ref{lem:lines} that the circles
	\[\{e^{2\pi i t} z \,\, \mbox{with} \,\, z \in S^{\circ}, \,\, 0 \leq t \leq 1 \}\]
	have length \(2\pi\alpha_0\).	
\end{note}

\subsection{Extension of \(r^2\) to \(\C^n\)}\label{sect:r2}

We begin with a preliminary lemma which is needed for our inductive argument.

\begin{lemma}\label{lem:induction}
	Suppose that \(\nabla\) satisfies the hypothesis of Theorem \ref{PRODTHM} and let \(L \in \mL(\mH)\). Then the quotient connection \(\nabla^{\C^n/L}\) also satisfies the hypothesis of Theorem \ref{PRODTHM}, i.e. it is flat torsion free unitary and satisfies the positivity and non-integer conditions. 
\end{lemma}

\begin{proof}
	By Lemma \ref{lem:quotconnect} the connection \(\nabla^{\C^n/L}\) is flat and torsion free. It follows from Lemma \ref{lem:sameholonomy} that \(\nabla^{\C^n/L}\) is unitary. Lemma \ref{lem:quotconnect} implies that the set of weights of \(\nabla^{\C^n/L}\) at irreducible intersections of \(\mH_L/L\) is a subset of the weights of \(\nabla\), i.e. \(\{a_M, \, M \in \mL_{irr}(\mH)\}\). In particular, since \(\nabla\) satisfies the positivity and non-integer conditions then so does the quotient connection \(\nabla^{\C^n/L}\).
\end{proof}

Our main application of the Local Product Theorem \ref{thm:locprod} and its proof is the next.

\begin{lemma}\label{lem:extensionr2}
	The function
	\(r^2\) extends continuously over the whole \(\C^n\) by taking strictly positive values outside the origin.
\end{lemma}

\begin{proof}
	We proceed by induction on the dimension, the case \(n=1\) is clear.
	
	Let \(0 \neq x \in L^{\circ}\). Recall the local dilation vector field
	\(e_x = e - s \) defined in Section \ref{sect:loceulervf}.
	Clearly,
	\begin{equation}\label{eq:extensionr}
	\frac{1}{2}r^2 = \langle e, e \rangle = \langle e_x, e_x \rangle + \langle e_x, s \rangle + \langle s, e_x \rangle + \langle s, s \rangle .
	\end{equation}
	Take local coordinates \((y_1, \ldots, y_n)\) as in the proof of Theorem \ref{thm:locprod}. In these coordinates the vector field \(e_x\) is linear and
	we can write
	\begin{equation}\label{eq:orthsum}
	e_x = e^{\parallel} + e^{\perp}	
	\end{equation}
	with
	\[e^{\parallel} = \sum_{j>\codim L} y_j \p_{y_j}\]
	and
	\[e^{\perp} = e_1 + \ldots + e_k\,\, \mbox{ where } \,\, e_i = \frac{1}{\alpha_{L_i}} \sum_{j \in I_i} y_j \frac{\p}{\p y_j} ,\]
	see Proposition \ref{prop:linvf}. Note the following:
	\begin{itemize}
		\item The sum \eqref{eq:orthsum} is orthogonal and \(e^{\perp}\) is the Euler vector field for \(\nabla^{\C^n/L}\).  By Lemma \ref{lem:induction} the inductive hypothesis apply  on \(\nabla^{\C^n/L}\), hence \(\langle e_x, e_x \rangle\) is a sum of continuous functions. 
		\item \(s = \sum_{j>\codim L} c_j \p_{y_j}\) with \(c_j \in \C\) and \(\langle s, s \rangle\) equals a constant \(c>0\).
		\item The term
		\[\langle e_x, s \rangle + \langle s, e_x \rangle = \langle e^{\parallel}, s \rangle_{\C^d} + \langle s, e^{\parallel} \rangle_{\C^d}\]
		is a linear function of the variables \(y_j, \bar{y}_j\) for \(j>\codim L\).
	\end{itemize}
	It follows from Equation \eqref{eq:extensionr} together with these three observations that \(r^2\) is continuous close to \(x\) and that \(r(x)=c>0\). This shows that \(r^2\) is continuous and positive outside the origin. From the homogeneity Equation \eqref{eq:rhomog} we deduce that
	\[r^2 = e^u |z|^{2\alpha_0}\]
	where \(u\) is a continuous function on the Euclidean sphere. Since \(\alpha_0>0\), we conclude that \(r^2\) is continuous at \(0\)
	 and the lemma follows.
\end{proof}

\begin{remark}
	For \(n=1\) we have \(r^2=|z|^{2\alpha}\), which is locally H\"older continuous. Our proof of Lemma \ref{lem:extensionr2} shows that \(r^2\) is H\"older continuous on bounded subsets of \(\C^n\).
\end{remark}

\subsection{Length spaces and metric cones}\label{sect:lengthspaces}

We present some background material on metric geometry, our main references are \cite{BBI} and \cite{BH}.
\subsubsection{{\large Length spaces}}

An introduction to length spaces can be found in \cite[Chapter 2]{BBI}  and  \cite[Section 1.3]{BH}. We recall the definition. 

\begin{definition}[{\cite[Definition 3.1]{BH}}]
	Let $(X,d)$ be a metric space. We say that $d$ is  a \emph{length metric} or \emph{intrinsic metric} if the distance between every pair of points $x,y\in X$ is equal to the infimum of the lengths of rectifiable curves\footnote{The definition of rectifiable curves is given in \cite[Definition 1.18]{BH}.} joining them. If $d$ is a length metric then $(X,d)$ is called a \emph{length space}. 
\end{definition}

We have the following useful statement.
\begin{lemma}[{\cite[Exercise 3.6(3)]{BH}}]\label{pathcompl} 
	Let $(X,d)$ be a length space, then its metric completion is also a length space.
\end{lemma}

A length space is called a \emph{geodesic space} if for any two points there exists a shortest path the joins them. 
\begin{lemma}[{\cite[Theorem 2.5.23]{BBI}}]\label{geodesicspace} 
	Any complete locally compact length space is a geodesic space.
\end{lemma}

\begin{remark} 
	If $(X,d)$ is any metric space and $Y\subset X$ is a subspace, then $d$ defines a metric on $Y$ which we call \emph{restricted metric}. In case $(X,d)$ is a length space, the restricted metric is often not a length metric. However, in case any two points $y_1,y_2\in Y$ can be joined by a rectifiable path lying in $Y$ we get a new length metric on $Y$. It is called the \emph{induced metric}. For more details see \cite[Section 2.3]{BBI}.
\end{remark}

For every Riemannian manifold $(X,g)$ one defines a length metric on it.
\begin{proposition}\label{lengthfromflat} 
	Let $(X,g)$ be a connected Riemannian manifold. Given $x,y\in X$ let $d(x,y)$ be the infimum of the Riemannian length of piecewise continuously differentiable paths $c:[0,1]\to X$ such that $c(0)=x$ and $c(1)=y$. Then
	\begin{enumerate}
		\item $d$ is a metric on $X$.
		\item The topology on $X$ defined by this distance is the same as the given manifold topology of $X$.
		\item $(X,d)$ is a length space.
		\item The Riemannian length of any piecewise continuously differentiable curve in $(X,g)$ is equal to the length in the metric $(X,d)$.
	\end{enumerate}
\end{proposition}

\begin{proof} 
	The first three items is \cite[Proposition 3.18]{BH}, the fourth item is \cite[Remark 3.21]{BH}.
\end{proof}

\subsubsection{{\large Metric cones}}

In this subsection we recall standard facts concerning  metric cones, following \cite[Section 3.6]{BBI} and \cite[Chapter 1.5]{BH}.

\begin{definition}[{\cite[Definition 3.6.16]{BBI}}]\label{euclideancones} 
	Let $\Sigma$ be a metric space. Consider the topological cone $C=[0,\infty)\times \Sigma / \sim$ where $(0, x)\sim(0, y)$ for every $x, y\in \Sigma$. Let us equip $C$ with the metric defined by the rule of cosines; i.e., for
	any $a, b\in [0,\infty)$ and $x,y\in \Sigma$  we have
	$$|(a,x),(b,y)|^2_C=a^2+b^2-2ab\cdot\cos|x,y|_\Sigma\;\; \mathrm{if}\,\, |x,y|_{\Sigma}\le \pi,$$
	$$|(a,x),(b,y)|_C=a+b\;\; \mathrm{if}\,\, |x,y|_{\Sigma}\ge \pi.$$
	The obtained space $C$ is called the \emph{(Euclidean) metric cone over $\Sigma$}. All the pairs of the type $(0, x)$ correspond to one point in $C$ which is called the tip (or vertex) of the cone. A metric space which can be obtained in this way is called a \emph{(Euclidean) metric cone}.
\end{definition}

\begin{remark}\label{conecomple} 
	Let $\Sigma$ be a metric space and let $C_\Sigma$ be the metric cone over $\Sigma$. Then the metric cone $C_{\overline \Sigma}$ over the metric completion $\overline \Sigma$ of $\Sigma$ is isometric to the metric completion  $\overline C_\Sigma$ of $C_\Sigma$.
\end{remark}

\begin{lemma}\label{lenghcone} 
	Let $\Sigma$ be a length space and $C$ be the metric cone over $\Sigma$ with vertex $0$. 
	\begin{enumerate}
		\item $C$ is a length space. 
		\item For any $r>0$ let $B(r)\subset C$ be the metric ball of radius $r$ centred at $0$. Then the restriction of the metric from $C$ to $B(r)$ is an intrinsic metric on $B(r)$.
		\item If $C$ is a geodesic space, then any geodesic that joins $(a,x)$ with $(b,y)$ lies in the ball $B(\max(a,b))$.
	\end{enumerate}
\end{lemma}
\begin{proof} 1. This is a sub-statement of \cite[Theorem 3.6.17]{BBI}.
	
	2. Let $(a,x), (b,y)$ be points in $C$ with $a,b\le r$. We need to find paths in $B(r)$ of length arbitrary close to $d_C((a,x), (b,y))$. In case $d_{\Sigma}(x,y)\geq \pi$ we choose the path composed of two radial segments $[(a,x), 0]$ and $[0, (b,y)]$. In case  $d_{\Sigma}(x,y)< \pi$ take a path $\gamma_{\varepsilon}\subset\Sigma$ of length $d_{\Sigma}(x,y)+\varepsilon<\pi$. The metric cone $C(\gamma_{\varepsilon})$ over $\gamma$ embeds in $C$ by a distance non-increasing embedding. Since $C(\gamma)$ is a plane sector with angle $<\pi$ we can join points $(a,x), (b,y)$ in it by a geodesic segment  $\widetilde\gamma_{\varepsilon}$ that lies on distance $\le r$ from $0$. The images of $\widetilde\gamma_{\varepsilon}$ in $C$ give us the desired paths.  
	
	3. This follows from \cite[Corollary 5.11]{BH}.  
\end{proof}

%\begin{remark} Suppose that the Euclidean cone $C$ over $\Sigma$ is geodesic (for example, if $C$ is complete and locally compact) all balls $B_r$ are convex in $C$. This means that any two points in $B_r$ can be joined in $B_r$ by a length minimizing geodesic, and any geodesic joining two points lie in the ball. 
%\end{remark}
The following remark is a consequence of the definitions.
\begin{remark}\label{addonepoint} 
	Let $(M,g)$ be a connected Riemannian manifold. Consider the Riemannian cone $(0, \infty) \times M$ with Riemannian metric $dt^2 + t^2 g$. Let $d_g$ be the  corresponding length metric on $(0, \infty) \times M$. Then the resulting metric space is isometric to the complement to the vertex of the metric cone over $(M,g)$.   
\end{remark}

\subsection{Characterization of polyhedral spaces}\label{sect:LePe}

It turns out that compact (and locally compact) polyhedral spaces (see Definition \ref{def:polspace}) have a useful characterisation, based on the notion of \emph{(Euclidean) metric cones}. 

\begin{theorem}[{\cite[Theorem 1.1]{LePe}}]\label{LebedevaPetrunin} 
	A compact length space $X$ is polyhedral if and only if a neighbourhood of each point $x\in X$ with the restricted metric admits an open isometric embedding into a metric cone which sends $x$ to the vertex of the cone.
\end{theorem}

\begin{note}
	We note that in \cite{LePe} metric cones are called Euclidean cones.
\end{note}

\begin{remark}\label{locompact} 
	As is pointed out in \cite[Section 5]{LePe}, the theorem holds for complete locally compact length spaces as well.
\end{remark}

For the time being Theorem \ref{LebedevaPetrunin} and Remark \ref{locompact} will be the only things that we need to know about polyhedral spaces. Indeed, our approach to proving that the flat metric on $\mathbb C^n\setminus \mathcal H$ extends to a polyhedral metric on $\mathbb C^n$ is as follows. First, we extend the metric to a complete length metric on $\mathbb C^n$, and then we prove that it satisfies the condition of Theorem \ref{LebedevaPetrunin} (using local dilations and distance functions $r_x$).

\subsection{General facts on metric extensions}\label{sect:metricexten}

In this section we consider triples $(X,X^\circ, d^\circ)$, where
$X$ is a locally compact topological space, $X^\circ$ is an open dense subset and $d^\circ$ is a metric on $X^\circ$ that induces on $X^\circ$ the same topology as the restricted topology of $X$. Lemma \ref{metricextension} spells out conditions on the triple that guarantee that $d^\circ$ extends to a complete metric on $X$.

\begin{lemma}\label{continuity} 
	Let $(X,X^\circ, d^\circ)$ be as above. The function $d^\circ: X^\circ \times X^\circ\to [0,\infty)$ extends as a continuous function $d$ to $X\times X$ if and only if for every $x\in X\setminus X^\circ$ and every $\varepsilon>0$ there is a neighbourhood $U_{x,\varepsilon}$ of $x$ in $X$ such that $\mathrm{diam}(U_{x,\varepsilon}\cap X^\circ, d^\circ)<\varepsilon$.
\end{lemma}

\begin{proof}
	The only if direction is clear, let's prove the if direction. 
	Take a point $(x,y)\in X\times X$. We need to show that for any $\varepsilon>0$ there is a neighbourhood $U_{(x,y),\varepsilon}\in X\times X$ such that for any two pairs of points $(x_1,y_1),\, (x_2,y_2)\in U_{(x,y),\varepsilon}\cap X^\circ\times X^\circ$ we have $|d^\circ(x_1,y_1)-d^\circ(x_2,y_2)|<\varepsilon$. We can assume that $x,y\in X\setminus X^\circ$. 
	
	By assumptions of the lemma, for any $\varepsilon>0$ we can choose neighbourhoods $U_{x,\varepsilon}, U_{y,\varepsilon}$ whose intersections with $X^\circ$ have diameter less than $\varepsilon/2$. Then for any two pairs of points $(x_1,y_1),\, (x_2,y_2)$ as above, using the triangle inequality,  we get $|d^\circ(x_1,y_1)-d^\circ(x_2,y_2)|<\varepsilon$. Hence, we can choose $U_{x,\varepsilon}\times U_{y,\varepsilon}$ as the desired neighbourhood of $(x,y)$. 
\end{proof}

\begin{lemma}\label{metricextension}
	Let $(X,X^\circ, d^\circ)$ be as above. Suppose that the following conditions are satisfied.
	\begin{enumerate}
		\item The condition of Lemma \ref{continuity} holds, and so the function $d^\circ: X^\circ \times X^\circ\to [0,\infty)$ extends as a continuous function $d$ to $X\times X$.
		\item For any $x,y\in X\setminus X^\circ$ with $x\ne y$ we have $d(x,y)>0$.
		\item For some point $o\in X$ the function $d_o(y)=d(o,y)$ is proper on $X$.
	\end{enumerate}
	Then $d$ is a metric on $X$ inducing the original topology. Moreover, $(X,d)$ is complete, and so isometric to the metric completion of $(X^\circ, d^\circ)$.
\end{lemma}

\begin{proof} 
	Since $d$ is a continuous function on $X\times X$, $X^\circ$ is dense in $X$ and $d^\circ$ satisfies the triangle inequality on $X^\circ$, it follows that $d$ satisfies the triangle inequality on $X$. For the same reason $d$ is symmetric. To see that $d(x,y)>0$ for any $x\ne y\in X$ it is enough to consider the case $x\in X\setminus X^\circ, y\in  X^\circ$. Take an open neighbourhood $U_y$ of $y$ in $X$ whose closure is disjoint from $X\setminus X^\circ$. For some $\varepsilon>0$ a ball $B_\varepsilon(y,d^\circ)$ lies in $U_y$. Since $d$ is continuous, the function $d_y$ is greater than $\varepsilon$ on $X\setminus B_\varepsilon(y,d^\circ)$ and we deduce that $d(x,y)>\varepsilon$. Hence $d$ is a metric on $X$. 
	
	Let's show that $d$ induces the original topology on $X$. In one direction, since $d$ is continuous, balls $B_x(r)$ are open in the original topology of $X$. In the opposite direction let $U$ be open in $X$ and take $x\in U$. We need to show that $B_x(\varepsilon)\subset U$ for some $\varepsilon>0$. Since $d_o$ is proper by condition 3, the closed ball $\overline B_o(2d(x,o))$ is compact, so the set $\overline B_o(2d(x,o))\setminus U$ is also compact. Since the function $d_x$ is continuous it attains minimum $m$ on this set, which is different from $0$ by condition 2. Hence we can take $\varepsilon=\min(d(x,o),m)$.

	%Suppose by contradiction there is a sequence $x_i\in X\setminus $ such that $\lim_{i\to \infty}d(x_i,x)=0$. 
	%Note that 
	%Since the closed ball $\overline B_o(2d(x,o))$ is compact, $x_i$ has a subsequence $x_i'$ converging to $x'$. By continuity $d(x,x')=0$, which contradicts condition 2.
	%Let $s=d(x,o)$. Take the closed ball $\overline B_o(2s)$
	
	To see that $(X,d)$ is a complete metric space, take any Cauchy sequence $y_i$ in it. Since $y_i$ is Cauchy, for some $r$ it lies in the closed ball $\overline B_o(r)$. Since $d_o$ is proper, $\overline B_o(r)$ is compact and so $y_i$ has a convergent subsequence that converges to some $y$. But since $d$ induces the original topology on $X$, $y$ is the limit of $y_i$. 
	
	Finally $(X,d)$ is a metric completion of $(X^\circ, d^\circ)$ since $X^\circ$ is dense in $X$.
\end{proof}

\subsection{Metric completion of \(((\C^n)^{\circ}, d_g^{\circ})\)}\label{sect:completion}

Recall that we are in the setting of Theorem \ref{PRODTHM}. The flat K\"ahler metric on the arrangement complement \(\C^n \setminus \mH =(\C^n)^{\circ}\) is denoted by $g$  and $d_g^\circ$ is the induced Riemannian length metric. 

\begin{definition}[Standard $r$-neighbourhood] 
	By Theorem \ref{thm:locprod}, for any point $x\in \mathbb C^n$ there exists an open neighbourhood $U$ with local coordinates $(y_1,\ldots,y_n)$ for which the connection is standard. We have a smooth function on \(U^{\circ}\) given by 
	\begin{equation}\label{eq:defrx}
	r_x^2=g(e_x, e_x)	
	\end{equation}
	where \(e_x\) is the local dilation vector field. The function \(r_x\) extends continuously over \(U\) by Lemma \ref{lem:extensionr2}. Let $r(U)>0$ be the minimum of $r_x$ on $\partial U$. Then for any $r\in (0,r(U))$ we denote by $\mathcal B_x(r)$ the subset of $U$ where $r_x<r(U)$. We call $\mathcal B_x(r)$ the \emph{standard $r$-neighbourhood} of $x$. 
	
	For a point $y\in \mathcal B_x(r)$ and $t\in (0,1)$ we denote by $t\cdot y$ the point obtained by the corresponding homothety of $\mathcal B_x(r)$.
\end{definition}

\begin{lemma}\label{extendingB} 
	For any $x\in \mathbb C^n$ consider its standard neighbourhood $\mathcal B_x(r)$.  
	\begin{itemize}
		\item[(1)] Let $y\in \mathcal B_x(r)\setminus \mathcal H$, $t\in (0,1)$. Then $d_g^\circ(y,t\cdot y)=(1-t)r_x(y)$.
		
		\item[(2)] $\mathrm {diam}(\mathcal B_x(r)\setminus \mathcal H, d_g^\circ)\le 2r$. 
		\item[(3)] The metric $d_g^\circ$ extends as a continuous function $d_g$ to $\mathbb C^n\times \mathbb C^n$, hence defining on $\mathbb C^n$ a pseudometric.
		
		\item[(4)] For any point $y\in \mathcal B_x(r)$ we have $ d_g(x,y)=r_x(y)$. Furthermore, for any $y\notin \mathcal B_x(r)$ we have $d_{g}(x,y)\ge r$.
	\end{itemize}
\end{lemma}
\begin{proof} (1) The distance $d_g^\circ$ between $y$ and $t\cdot y$ is the infimum of lengths of all paths in $\mathbb C^n\setminus \mathcal H$ that join $y$ with $t\cdot y$. Each such path should intersect completely the annulus $(B_x(r_x(y))\setminus B_x(t\,r_x(y)))\setminus \mathcal H$. Recalling now that (by Lemma \ref{lem:kahlercone2}) the annulus is a part of a Riemannian cone, we obtain the statement.
	
	(2) Let $y,z\in \mathcal B_x(r) \setminus \mathcal H$ be two points. Consider any path $\gamma\subset B_x(r) \setminus \mH$ that joins $y$ and $z$. For $t\in (0,1)$ we can take the path $\gamma_t$ composed of $3$ parts, $[y,t\cdot y]$, $t\cdot \gamma$, $[t\cdot z,z]$, where $t\cdot$ denotes the action by homotheties on $\mathcal B_x(r)$. Using claim (1) we get $|\gamma_t|<2r+t|\gamma|$. This implies   claim (2).
	
	(3) It suffices to apply Lemma \ref{continuity} to the triple $(\mathbb C^n,\mathbb C^n\setminus \mathcal H, d_g^\circ)$, using  (2).
	
	(4) By claim (1), for $t\in (0,1)$ and $y\in \mathcal B_r(x)\setminus \mathcal H$ we have $d_g^\circ(t\cdot y,y)=r_x(y)(1-t)$. Hence the equality $ d_g(x,y)=r_x(y)$ holds, since $d_g$ is a continuous pseudometric on $\mathbb C^n$ and $r_x$ is continuous in $\mathcal B_r(x)$.
	
	If $y\in \mathbb C^n\setminus (\mathcal H \cup \mathcal B_r(x))$ we use again that fact that $d_g^\circ$ is a length metric and any path that joins $y$ with a point in $\mathcal B_x(\varepsilon)$ intersects the annulus $\mathcal B_x(r)\setminus\mathcal B_x(\varepsilon)$ and so has length at least $r-\varepsilon$. 
\end{proof}

\begin{corollary}\label{completedmetric} 
	The triple $(\mathbb C^n,\mathbb C^n\setminus \mathcal H, d_g^\circ)$ satisfies the conditions of Lemma \ref{metricextension}. In particular, the continuous extension $d_g$ of $d_g^\circ$ to $\mathbb C^n$  is a complete metric on $\mathbb C^n$ inducing its usual topology. 
\end{corollary}

\begin{proof} 
	First, we  need to show that conditions 1, 2 and 3 of Lemma \ref{metricextension} hold.
	
	1. This holds by Lemma \ref{extendingB} (2). 
	
	2. Let us show that for any $x,\, y\in \mathbb C^n$ we have $d_{g}(x,y)>0$. It suffices to consider the case $x,y\in \mathcal H$.
	Choose $r>0$ such that the standard ball $\mathcal B_r(x)$ doesn't contain $y$. Then by Lemma \ref{extendingB} (4) $d_{g}(x,y)\ge r$.
	
	3. We choose $o$ as $0$. The function $d_g(0,.)$ is equal  $r_0$ on $\mathbb C^n$ by Lemma \ref{extendingB} (4), since $\mathcal B_0(r)$ is a standard ball for all $r$. At the same time $r_0$ is proper on $\mathbb C^n$. 
	%We conclude that $d_g$ is a complete metric on $\mathbb C^n$. It induces on $\mathbb C^n$ the correct topology since for any Euclidean ball $B$ and $x\in B$ there is $\varepsilon$ such that $\mathcal B_x(\varepsilon)\subset B$.
\end{proof}

\begin{corollary}\label{HEucone} 
	The metric space $(\mathbb C^n,d_g)$ is a metric cone and a geodesic space.
\end{corollary}

\begin{proof} 
	Consider $X'=(\mathbb C^n\setminus \mathcal H)\cup 0$ and extend the metric $d_g^\circ$ to this space by declaring $d_g^\circ(0,x)=r_0(x)$. By Lemma \ref{lem:kahlercone2} the metric space  $(\mathbb C^n\setminus \mathcal H, d_g^\circ)$ is a Riemannian cone, and so $X'$ is a metric cone by Remark \ref{addonepoint}. Now, the metric completions of $X'$ and $\mathbb C^n\setminus \mathcal H$ coincide and both are isometric to $(\mathbb C^n,d_g)$ by Corollary \ref{completedmetric}. 
	At the same time, since $X'$ is a metric cone, its metric completion is a metric cone as well by Remark \ref{conecomple}.
	
	$(\mathbb C^n,d_g)$ is a completion of a length space, so it is a length space by Lemma \ref{pathcompl}. Since $d_g$ defines on $\mathbb C^n$ its usual topology, \((\C^n, d_g)\) is locally compact, and so we can apply Lemma \ref{geodesicspace} to conclude that $(\mathbb C^n,d_g)$ is a geodesic space.
\end{proof}

\subsection{Proof of Theorem \ref{PRODTHM}}\label{sect:pfthm1}

The next result finishes the proof of Theorem \ref{PRODTHM}.

\begin{theorem}\label{polyhedralast} 
	The metric $d_g$ of $\mathbb C^n$ is polyhedral. 
\end{theorem}

\begin{proof} 
	According to Theorem \ref{LebedevaPetrunin} and Remark \ref{locompact} we need to show that each point in $(\mathbb C^n,d_g)$ has a neighbourhood that admits an open isometric embedding to a metric cone which sends $x$ to the vertex of the cone. Let $\mathcal B_x(r)$ be a standard neighbourhood of $x$. We will prove that $(\mathcal B_x(r/2),d_g)$ admits an isometric embedding into a metric cone. For $x=0$ this is just Corollary 
	\ref{HEucone} since $(\mathbb C^n,d_g)$ is a metric cone. So we can take $x\ne 0$.
	
	Let $L \in \mL(\mH)$ be the subspace of smallest dimension on which $x$ lies and $\mathcal H_L$ be the corresponding arrangement.
	Recall that by definition of  $\mathcal B_x(r)$, the standard local coordinates $y_i$ defined  in a neighbourhood of $x$ give us a holomorphic embedding $\varphi: \mathcal B_x(r)\to (\mathbb C^n,\mathcal H_L)$, where $\varphi$ sends $x$ to $0$ and $\mathcal{B}_x(r)\cap \mathcal H$ to $\mathcal H_L$. Furthermore, $\varphi$ is a Riemannian isometry on $\mathcal B_x(r)\setminus \mathcal H$, with respect to the flat K\"ahler metrics $g$ on $\mathcal B_x(r)\setminus \mathcal H$ and $g_L$ on $\mathbb C^n\setminus \mathcal H_L$. Note also that $(\mathbb C^n,g_L)$ is a metric cone and a geodesic space by Corollary \ref{HEucone}.  
	
	Let us show that $\varphi$ defines an isometric embedding $\varphi: (\mathcal B_x(r/2), d_g)\to (\mathbb C^n, d_{g_L})$. Take $y,z\in \mathcal B_x(r/2)$. Since $(\mathbb C^n, d_{g_L})$ is a geodesic space and a metric cone, by Lemma \ref{lenghcone} (3) a shortest geodesic $\gamma$ that joins $\varphi(y)$ and $\varphi(z)$ lies in $\varphi(\mathcal B_x(r/2))$. We claim that $\varphi^{-1}(\gamma)$ is a shortest geodesic that joins $y$ and $z$. Indeed, suppose there is a shorter geodesic $\gamma'$. Since $\varphi$ preserves the lengths of paths, $|\varphi^{-1}(\gamma)|=|\gamma|$. Furthermore, $|\gamma|<r$ since the $r/2$-ball in $(\mathbb C^n, d_{g_L})$ has diameter $\le r$. We see that $|\gamma'|<r$ and so $\gamma'$  can not leave $\mathcal B_x(r)$. But then $\varphi(\gamma')$ is a geodesic in $(\mathbb C^n,d_{g_L})$ shorter than $\gamma$,  a contradiction. We conclude that $d_g(y,z)=|\gamma|=d_{g_L}(\varphi(y),\varphi(z))$. It follows that $\varphi$ is an isometric embedding.
\end{proof}

\begin{remark} Let's comment on our usage of Theorem \ref{LebedevaPetrunin} in the proof of Theorem \ref{polyhedralast}. Note first that Theorem \ref{polyhedralast} can be proven by  induction on dimension $n$ as follows. Suppose the theorem is proven for dimension $\le n-1$. Then every point $x\in \mathbb C^n\setminus 0$ has a neighbourhood that can be isometrically embedded into $\mathbb C^l\times (\mathbb C^{n-l}, \mathcal{H}_L/L)$ with $l\ge 1$, where the first factor is Euclidean and the second is polyhedral by induction. So $x$ has a polyhedral neighbourhood in $\mathbb C^n$. Hence, to make the induction step we only need to prove that a neighbourhood of $0$ can be triangulated as well. This is where we appeal to Theorem \ref{LebedevaPetrunin} in an essential way. Such approach also shows that we are not bound to use Remark \ref{locompact} and can work directly with compact spaces. Namely, to perform the induction step we can cut out from $\mathbb C^n$ a neighbourhood of $0$ with polyhedral boundary and apply  Theorem \ref{LebedevaPetrunin} directly to it. 
\end{remark}

\section{PK metrics on complex manifolds}\label{sect:PK}

We recall the definition from the introduction.

\begin{definition}\label{def:PK}
	Let \(X\) be a complex manifold.
	We say that \(g\) is a \textbf{PK  metric on} \(X\) if the following holds:
	\begin{itemize}
		\item[(i)] \(g\) is a polyhedral  metric on \(X\) inducing its topology;
		\item[(ii)] \(g\) is a K\"ahler metric on the regular part \(X^{\circ}\) of the polyhedral manifold $(X,g)$.
	\end{itemize}
	In case $(X,g)$ is a polyhedral cone we say that $g$ is a \textbf{PK cone metric}.
\end{definition}

\begin{remark} 
	We note that the set of metric singularities of the polyhedral manifold $(\mathbb C^n,d_g)$ from Theorem \ref{PRODTHM} coincides with the hyperplane arrangement $\mathcal H$. Indeed, the set of singularities is closed and by Corollary \ref{cor:locprod} at a generic point of $\mathcal H$ the metric is locally a product of a $2$-cone with $\mathbb C^{n-1}$, so it is singular. For this reason this metric on $\mathbb C^n$ is polyhedral K\"ahler according to the above definition. 
\end{remark}

%\begin{remark}
%	The polyhedral cones from Theorem \ref{PRODTHM} are examples of PK cone metrics on \(\C^n\), as we now explain. By Corollary \ref{cor:locprod} the polyhedral cone is isometric to \(\C^{n-1} \times \C_{\alpha_H}\) at \(H^{\circ}\). Hence the metric \(g\) is singular at the smooth points of the arrangement
%	\[\bigcup_{H \in \mH}H^{\circ} ,\]
%	 whose closure is the whole \(\mH\). Since the metric singular set is closed, we conclude that it is equal to \(\mH\). On the other hand, the metric is K\"ahler on \(\C^n \setminus \mH\).
%\end{remark}

In this section we establish a number of foundational results for these objects which lead into Theorems \ref{thm:PK1}, \ref{thm:PK2} and \ref{thm:PK3}. More precisely,  these three theorems are a direct consequence of our next results as follows.

\begin{proof}[Proof of Theorem \ref{thm:PK1}]
	The metric singular set \(X^s\) is a complex hypersurface because of Theorem \ref{singanalitic} item (4). The tangent cone of the metric at smooth points of \(X^s\) is holomorphically isometric to \(\C^{n-1} \times \C_{\alpha}\) because of Corollary \ref{smoothpoint}.
\end{proof}

\begin{proof}[Proof of Theorem \ref{thm:PK2}]
	The existence of global complex coordinates linearising the Euler vector field follows directly
	from Proposition \ref{prop:PKcones}.
\end{proof}

\begin{proof}[Proof of Theorem \ref{thm:PK3}]
	The fact that the polyhedral stratification given by the singularities of the metric agrees with the ordinary complex stratification
	\[\C^n = \bigcup_{L \in \mL(\mH)} L^{\circ} \]
	follows from Proposition \ref{prop:PKstrat} item (3).
	The Levi-Civita is a standard connection because of Proposition \ref{prop:PKstandard}.
\end{proof}

{\bf PK manifolds} vs {\bf PK metrics on complex manifolds.} We would like to comment on a difference between the two notions. According to the definition given in \cite{Pan} a PK manifold is a piecewise linear manifold with a compatible choice of Euclidean metric on each simplex, so that the holonomy of the metric is unitary, and all faces of real codimension 2 have complex direction. An important motivation of such a definition was a result of Cheeger \cite[Theorem 3]{Cheeger}. Thanks to his theorem, an interesting class of PK manifolds is given by (Euclidean) polyhedral metrics with conical angles $\le 2\pi$ on $\CP^n$ (see \cite[Proposition 2.3]{Pan} and \cite{Petrunin}). Such polyhedral metrics on $\CP^n$ always have unitary holonomy, because in polyhedral world the above curvature condition is very strong and similar to non-negativity of curvature operator, thus all harmonic $L^2$-forms are parallel.

By definition, a PK manifold $X$ has a complex structure in real codimension $2$ (outside metric singularities $X^s$). However, a priory, just by definition, a PK manifold  doesn't need to be a complex manifold, and it is a fairly non-trivial task to prove that the complex structure on $X^\circ$ extends to the subset of metric singularities. In \cite[Theorem 1.5]{Pan} this is done for real dimension $\le 4$. The forthcoming paper \cite{PanVer} will prove this extension result in all dimensions for a more general class of polyhedral K\"ahler spaces (see Definition \ref{def:PKspaces}). Thanks to adopting Definition \ref{def:PK} here, we land ourselves directly into complex geometry and don't need to deal with extending complex structure to the whole polyhedral space. We focus more closely on the holomorphic aspect of the story.

\subsection{Types of simplicial complexes and polyhedral spaces}\label{sect:simplicialcomplex}

\subsubsection{{\large Pseudomanifolds}}

In what follows we will need to consider some natural subclasses of polyhedral spaces. For this reason we recall some standard definition concerning simplicial complexes. 

We recall first that the \emph{link} of a simplex $\Delta$ in a simplicial complex $\cal K$ is the union of all the simplexes in $\Delta'\subset \cal K$ disjoint from $\Delta$, such that both $\Delta$ and $\Delta'$ are faces of some simplex of $\cal K$.

\begin{definition}
	A locally finite simplicial complex $\cal K$ is called a   pseudomanifold of dimension $n$ if the following holds.
	
	\begin{enumerate}	
		\item $\cal K$ is {\bf pure}, i.e., $|{\cal K}|$ is the union of all $n$-simplices.
		
		\item {\bf No branching:} every $(n-1)$-simplex is a face of exactly two $n$-simplices (for $n > 1$).
		
		\item $\cal K$ {\bf is strongly connected}: any two $n$-dimensional simplices can be joined by a chain of $n$-dimensional simplices in which each pair of neighbouring simplices have a common $(n-1)$-dimensional face.
	\end{enumerate}
	
	A pseudomanifold is called \emph{normal} if the link of each simplex of codimension at least two is a pseudomanifold.
\end{definition}

\begin{remark} A pseudomanifold is connected by definition. For this reason the link of each simplex in a normal pseudomanifold is connected as well.
\end{remark}

\subsubsection{{\large Polyhedral spaces}}

There exist several equivalent definitions of polyhedral spaces, for a thorough treatment see \cite[Section 3.2 and Definition 3.2.4]{BBI}. We will use the next definition, following \cite[Section 3.4]{AKP}, with a difference that we consider not only compact spaces but locally compact as well.  

\begin{definition}[{\cite[Definition 3.4.1]{AKP}}]\label{def:polspace}
	A complete length space $X$ is called a \emph{Euclidean (spherical) polyhedral space} if it admits a locally finite triangulation such that each simplex is isometric in restricted metric to a simplex in Euclidean space (a unit sphere). 	
\end{definition}

We expand the above terminology for simplicial complexes to polyhedral spaces.

\begin{definition}
	Let $X$ be a polyhedral space with a triangulation $\tau$. We call $X$ a topological (or PL) manifold or (normal)
	\textbf{polyhedral pseudomanifold} if the corresponding simplicial  complex is a toplological (or PL) manifold or (normal) pseudomanifold.
\end{definition}

\begin{remark} 
	Euclidean polyhedral and spherical PL manifolds are also called  Euclidean and spherical cone manifolds or just cone manifolds, e.g. see \cite[Chapter 3]{CHK}.
\end{remark}

We will often call Euclidean polyhedral spaces just polyhedral spaces. Any Euclidean polyhedral space can be constructed as follows. One takes a locally finite simplicial complex $X$ and prescribes the lengths of all the edges, in such a way that for any simplex $\Delta$ of the complex there exists a Euclidean simplex with edges of the prescribed lengths. In other words, $X$ is glued from a collection of Euclidean simplexes by identifying their faces isomterically. This gives rise to a length metric on the whole complex; for the full definition see  \cite[Definition 3.1.12]{BBI}.  The distance between $x,y\in X$ is  the infimum of the lengths of all broken geodesics in $X$ that join $x$ and $y$.  

\begin{remark}\label{regularpoints} 
	For polyhedral spaces that are topological manifolds (or pseudomanifolds) of dimension $n$ we define the set of \emph{regular points} \(X^{\circ}\) as points that have a neighbourhood isometric to a Euclidean $n$-ball. The remaining points \(X^s\) are called metric singularities, and they form a closed subset of dimension at most $n-2$.
\end{remark} 

\subsection{Polyhedral cones and stratifications}\label{sect:polyhedral cones}

\subsubsection{{\large Polyhedral cones}}

\begin{definition} We say that a metric space $X$ is a (Euclidean) polyhedral cone if it is isometric to a metric cone over a compact spherical polyhedral space $S$.
\end{definition}

One can think of this definition as follows. We take a finite simplicial decomposition of a spherical polyhedral space $S$ into geodesic spherical simplices, associate to each such simplex a unique Euclidean simplicial cone, and identify the faces of these cones in the same way as the faces of spherical simplices are identified in $S$.  

A polyhedral cone $C$ is called \emph{essential (or prime)} if it is not isometric to a direct product $\mathbb R^k\times C'$, where $C'$ is a polyhedral cone. We have the following \emph{unique factorisation} result for polyhedral cones. 

\begin{theorem}\label{undec} Let $C$ be a polyhedral cone. Then there exists a unique maximal factorisation of $C$ into direct product $\mathbb R^k\times C'$, where $C'$ is an essential polyhedral cone. This decomposition is maximal in the sense that for any other decomposition $C=\mathbb R^l\times C''$, we have $\mathbb R^l\subset \mathbb R^k$ and $C''=C'\times \mathbb R^{k-l}$.
\end{theorem}

\begin{proof} The statement follows directly from the general de Rham decomposition theorem \cite[Theorem 1.1]{FL} which holds for all finite dimensional geodesic metric spaces. The statement is also equivalent to theorem \cite[Theorem 5.1]{McMullen} on spherical cone manifolds, where the equivalence is obtained by taking the unit sphere in the polyhedral cone.
\end{proof}  

{\bf Tangent cones.} Let $X$ be a polyhedral space. For any point $x\in X$ there exists a canonical polyhedral cone $C_x$ in which a small neighbourhood of $X$ embeds isometrically. 
To construct $C_x$  take a polyhedral decomposition of  $X$ such that $x$ is a vertex, take the simplicial star of $x$ and extend each simplex in the star to a simplicial cone. This cone is called the {\it tangent cone} of $X$ at $x$. An $\varepsilon$-neighbourhood of $x$ that admits an isometric embedding into $C_x$ is called a {\bf conic neighbourhood}.

{\bf Cylindrical $\varepsilon$-neighbourhoods.} Let $X$ be a polyhedral space and $x\in X$ be a point. Consider a conic $2\varepsilon$-neighbourhood of $x$. Consider the unique direct product decomposition $C_x=\mathbb R^m\times C'$, where $C'$ is an essential cone. Take a direct product of an $\varepsilon$-ball in $\mathbb R^m$ with  an $\varepsilon$-ball in $C'$. This product is isometric to a neighbourhood of $x$ in $X$ that we will call a \emph{cylindrical $\varepsilon$-neighbourhood}.

\subsubsection{{\large Stratifications}}

In this section we restrict to \emph{pure} polyhedral spaces i.e. spaces that are unions of simplexes of fixed dimension $n$. Thanks to uniqueness Theorem \ref{undec} the following is well defined.

\begin{definition} Let $X$ be a pure polyhedral space of dimension $n$. We say that a point $x\in X$ is a \emph{metric singularity of codimension} $k\ge 1$ if the tangent cone $C_x$ of $X$ is a product of $\mathbb R^{n-k}\times C_x'$ where $C_x'$ is an essential polyhedral cone of dimension $k$. A point $x$ is called \emph{regular} if it has a neighbourhood isometric to a Euclidean $n$-ball. The set of regular points is denoted $X^\circ$.
\end{definition}

\begin{definition} Let $X$ be a pure polyhedral $n$-space. 
	We denote by $X^s$ the complement to the set of regular points. We denote by $X_k^s$ the subspace of all metric singularities of codimension at least $k$. 
	An \emph{open stratum of singularities} is a connected component of $X_k^s\setminus X_{k+1}^s$ for some $k$.
\end{definition}

\begin{remark}\label{closedstarta} Note that $X_k^s$ is a closed subset of $X$. Indeed any metric singularity of codimension $m$ has a neighbourhood such that all points in the neighbourhood have codimension at most $m$. Hence a sequence from $X_k^s$ converging in $X$ can't converge to a point in $X\setminus X_k^s$.
\end{remark}

\begin{lemma}\label{subcomplex} Let $X$ be a pure polyhedral space of dimension $n$. Let $X_k^s$ be the subset of metric singularities of codimension at least $k$. Then the following holds.
	
	\begin{enumerate}
		\item $X_k^s$ is a subcomplex of any polyhedral triangulation $\tau$ of $X$. In particular it is a polyhedral space of dimension $\le n-k$.
		
		\item If $X$ is an essential polyhedral cone then its only metric singularity of codimension $n$ is the vertex.

		\item Each connected component of $X_k^s\setminus X_{k+1}^s$ (i.e. an open stratum of singularities) is a flat Riemannian manifold with respect to the induced metric.
		
		\item In case $X$ is a pseudomanifold, $X^s$ has dimension at most $n-2$ and $X^\circ=X\setminus X^s$ is a connected flat Riemannian $n$-manifold.
	\end{enumerate}
	
\end{lemma}
\begin{proof} 
	1. Let $\Delta^\circ$ be an open $(n-l)$-simplex of $\tau$ that intersects $X_k^s$. The tangent cones of all points  of $\Delta^\circ$ are isomorphic and are isometric to $\mathbb R^{n-m}\times C_{\Delta}^{m}$, where $m\le l$ and $C_{\Delta}^{m}$ is an essential polyhedral cone. For this reason
	$k\le m\le l$ and $\Delta^{\circ}$ belongs entirely to $X_k^s$. Since $X_k^s$ is closed by Remark \ref{closedstarta}, it is a union of closures of all such simplexes $\Delta^\circ$.
	
	2. Let $o$ be the vertex of the cone $X$ and take a decomposition $\tau$ of \(X\) into simplicial cones (centred at $o$). Any point $y\ne o$ lies in the interior of a simplicial cone of dimension $l>0$. Hence it is a metric singularity of codimension at most $n-l$.
	
	3. Take a point $x$ on a connected component of $X_k^s\setminus X_{k+1}^s$. We will show that the intersection of $X_k^s\setminus X_{k+1}^s$ with a small $\varepsilon$-neighbourhood of $x$ in $X$ is a flat open $(n-k)$-ball. The claim will clearly follow.

	Since $x$ is a metric singularity of codimension $k$ its tangent cone $C_x$ decomposes as a product $\mathbb R^{n-k}\times C_x'$, where $C_x'$ is an essential cone. By claim 2 all points in $C_x'$ apart from $x$ are metric singularities of codimension at most $k-1$. It follows that $\mathbb R^{n-k}$ coincides with all metric singularities of codimension $n-k$ in $C_x$. Hence it suffice to take an $\varepsilon$-neighbourhood of $x$ in $X$ that embeds isometrically into $C_x$.
	
	4. By definition of a pseudomanifold each $(n-1)$-simplex of a triangulation of $X$ is a face of two $n$-simplexes. Hence all metric singularities of $X$ lie in the union of simplexes of dimension at most $n-2$. It follows that \(X^s\) has dimension at most \(n-2\) and $X^\circ$ is connected.
\end{proof}

\begin{definition}[Conical anlges]\label{integerconeangle} Let $X$ be a normal polyhedral pseudomanifold of dimension $n$ and let $Y^\circ$ be an open stratum of singularities of codimension $2$, i.e., a connected component of $X^s\setminus X^s_3$. Since $X$ is normal, the link of each $n-2$-simplex is a circle. For this reason, for each point $y\in Y^\circ$ the tangent cone is isometric to a product $\mathbb R^{n-2}\times C_\alpha$, where $C_\alpha$ is a two-dimensional cone with angle $2\pi\alpha$. The angle $2\pi\alpha$ is independent of the choice of $y\in Y^\circ$ and we call it \emph{the conical angle of $X$ at $Y^\circ$}. In case $\alpha$ is in integer, we say that the conical angle at $Y^\circ$ is integer.
\end{definition}

\subsection{Parallel vector fields on polyhedral pseudomanifolds}\label{sect:parallelvf}

\begin{definition} Let $X$ be a polyhedral pseudomanifold of dimension $n$. We say that a vector field $v$ defined on $X^\circ$ is \emph{parallel in codimension $k\ge 1$} if the following holds:
	\begin{itemize}
		
		\item[(i)] $v$ is parallel on $X^\circ$;
		
		\item[(ii)]  for any point $x\in X^s\setminus X^s_{k+1}$ and a standard cylindrical neighbourhood $B_\varepsilon\times C_\varepsilon$ of $x$, the field $v$ is tangent to the $B_\varepsilon$ fibres on $B_\varepsilon\times C_\varepsilon^\circ$.
	\end{itemize} 
	
	Let $Y^\circ$ be an open stratum of singularities. The field is called \emph{parallel to $Y^\circ$} if condition (ii) holds for one (and hence for all) point $y\in Y^\circ$.
	
	The field is called \emph{fully parallel} if is it parallel in codimension $n$, i.e. it is parallel to all open strata. 
\end{definition}

\begin{note} According to this definition a fully parallel vector field on a polyhedral pseudomanifold of dimension $n$ with a singularity of codimension $n$ has to vanish. Indeed, such a field vanishes in a neighbourhood of the codimension \(n\) singularity. Also for $k=1$ condition (ii) is empty since $X^s_{2}=X^s$. 
\end{note}

The main claim of this section is the following result.

\begin{proposition}\label{fullyparallel} Let $X$ be a normal polyhedral pseudomanifold with a vector field $v$ parallel in codimension 2. Then $v$ is fully parallel. 
\end{proposition}
\begin{proof} Let's assume that $v$ is parallel in codimension $k-1\ge 2$, and  deduce that $v$ is  parallel in codimension $k$. Then the proposition follows by induction on $k$.
	
	For any point $x$ in $X^s_{k}\setminus X^s_{k+1}$ we need to show that condition (ii) holds. Let $C$ be the tangent cone of $X$ at $x$. The  vector field $v$ on $X^\circ$ gives rise to a vector field $u$ on $C$, and by our assumption $u$ is parallel on $C$ in codimension $k-1$. We will show that $u$ is fully parallel on $C$.
	
	Consider the maximal decomposition $C=\mathbb R^k\times C'$ such that $C'$ is an essential cone. Let us decompose $u$ (on $C^\circ$) into the sum of two parallel fields 
	$$u=u_{\mathbb R^k}+u_{C'},$$ 
	where $u_{\mathbb R^k}$ is tangent in $C^\circ$ to $\mathbb R^k$-fibres and $u_{C'},$ is tangent to $C'$-fibres. By definition $u_{\mathbb R^k}$ is fully parallel on $C$, so in order to show that $u$ is fully parallel, it's enough to show that $u_{C'}=0$. This follows from Lemma \ref{nofield} applied to the restriction of $u_{C'}$ to the essential $n-k$ dimensional polyhedral cone $(0,C')$.
\end{proof}

Before stating Lemma \ref{nofield} we prove a simple statement explaining the importance of the condition of normality.

\begin{lemma}\label{intrinsicrestriction} Suppose $X$ is a polyhedral normal pseudomanifold with metric $d_g$. Then the restriction of $d_g$ to $X^\circ$ is an intrinsic metric. 
\end{lemma}

\begin{proof} Take any two points $x,y\in X^\circ$ and join them by a shortest geodesic $\gamma$ such that $l(\gamma)=d_g(x,y)$. To prove the lemma, we will show that there is a path $\gamma'\in X^\circ$ that joins $x$ with $y$ and has length arbitrary close to that of $\gamma$.
	
	For each point $z\in \gamma$ let's find a conic neighbourhood of $z$ in $X$. Take a finite subcover of $\gamma$ by such neighbourhoods $U_i$ and let $x=z_1\ldots z_k=y\subset \gamma$ be the centres of the corresponding neighbourhoods. Then each segment $z_iz_{i+1}$ is geodesic and since $X$ is pure we can perturb it slightly to a geodesic segment $z_iz_{i+1}'$ that lies entirely in $X^\circ$ with exception of $z_i$, keeping $d_g(z_{i+1}', z_{i+1})<\varepsilon/2k$. Next, take a point $z_i''$ on  $z_iz_{i+1}'$  on distance $\varepsilon/2k$ from $z_i$. Since $X$ is normal, the link of each $z_i$ is a spherical polyhedral pseudomanifold $S_i$, in particular $S_i^\circ$ is connected. Since $U_i$ is conic, it follows that points $z_i'$ and $z_i''$ can be connected by a short curve of length at most $\varepsilon/k$ lying in $U_i^\circ$. By combining together the constructed short curves and segments $z_i''z_{i+1}'$ we get the desired curve $\gamma'$ of length at most $l(\gamma)+\varepsilon$. \end{proof}

\begin{lemma}\label{nofield} Let $C$ be a normal polyhedral pseudomanifold of dimension $n\ge 3$. Suppose $C$ is an essential polyhedral cone. Then any vector field $u$ on $C^\circ$, parallel to all strata of positive dimension is $0$.
\end{lemma}

\begin{proof} We will assume by contradiction that $u$ is non-zero.
	
	Recall that $C$ can be glued from a finite number of simplicial cones, which we denote by $C_1,\ldots, C_m$. On the interior of each simplicial cone $u$ restricts as a parallel vector field. Hence, each $C_i$ contains at most one radial ray $l_i$ passing through $0$ and tangent to $u$. Let us denote the (possibly empty) union of such $l_i$'s that don't lie entirely in the singular locus $C^s$ as $\overline l$. Our first claim is that the flow $\varphi^t$ of the vector field $u$ on the flat manifold $C^\circ\setminus \overline l$ is defined for all time $t$.
	
	{\it Proof of the claim.} Take any point $x\in C^\circ\setminus \overline l$ and assume by contradiction that the geodesic $\varphi^{t}(x)$ hits $C^s\cup \overline l$ at some time $t_0>0$. Let $y\in C^s\cup \overline l$ be the corresponding limit point $\varphi^{t_0}(x)$. 
	
	Let's show first that $y$ has to be the vertex $0$ of $C$. Indeed, if $y\ne 0$ it can not lie in $\overline l$, since by definition $\overline l$ is the union of rays whose non-zero points lie in $C^\circ$, and so a trajectory that hits such a ray at some point $y$ has to coincide with the ray.  At the same time $y$ can't lie on $C^s\setminus 0$ since $u$ is parallel to all positive-dimensional strata.
	
	We conclude that $y=0$. To rule this out, it suffices to consider the distance function $d(0, \varphi^t(x))$. Let $l_x$ be the ray of $C$ passing through $x$. Since $l_x$ is not tangent to $u$, it intersects $l_x$ at non-zero angle $\alpha$. 
	Then 
	$$d(0,\varphi^t(x))^2=(r(x)\sin(\alpha))^2+(r(x)\cos(\alpha)-t)^2\ge (r(x)\sin(\alpha))^2>0.$$ 
	Hence the flow $\varphi^t$ is defined for all times. 
	
	Let $d_g$ be the polyhedral metric on $C$. Since $C$ is a normal pseudomanifold and $n\ge 3$, by exactly the same argument as in the proof of Lemma \ref{intrinsicrestriction} we prove that the restriction of $d_g$ to $C^\circ \setminus \overline l$ is an intrinsic metric\footnote{This is the moment when we use condition $n\ge 3$, indeed a segment in a Euclidean space of dimension $n\ge 3$ doesn't separate points in its neighbourhood.}. %Indeed, any path in $C$ that joins two points $x,y\in C^\circ \setminus \overline l $ can be  perturbed to lie in $C^\circ \setminus \overline l$ increasing its length only slightly; 
 Clearly, $\varphi^t$ is an isometry on $(C^\circ \setminus \overline l, d_g)$, and since this subspace is dense in $(C,d_g)$, we see that $\varphi^t$ extends to $C$ as a distance preserving map. The same holds for $\varphi^{-t}$. Since $\varphi^t\circ\varphi^{-t}$ is the identical map, we conclude that $\varphi^t$ is an isometry of $C$. 
	
	Since  $\varphi^t$ is an isometry for all $t$, and $C$ is essential, it follows that  $\varphi^t$ fixes $0$ and so preserves distance to $0$. Let's deduce that $u=0$. Restrict $u$ to any simplicial cone $C_i\subset C$. We get a constant vector field on $C_i$ whose flow preserves distance to the vertex of the cone. Hence $u$ is zero, a contradiction.
\end{proof}

\begin{corollary}\label{coneJrotation} Let $X$ be a normal polyhedral pseudomanifold that is a polyhedral cone. Suppose $v$ is a non-zero vector field  parallel in codimension 2. Then $X$ decomposes as a direct product $\mathbb R\times X'$ so that $v$ is tangent to the $\mathbb R$-direction.
\end{corollary}
\begin{proof} $X$ can't be an essential cone since $v\ne 0$. Consider the maximal decomposition of $X$, $X=\mathbb R^k\times Y$, where $Y$ is an essential cone. Take $x\in (\mathbb R^k,0)$. By Proposition \ref{fullyparallel} in a cylindrical neighbourhood of $x$ the vector $v$ is tangent to $\mathbb R^k$-fibres. The statement clearly follows.
\end{proof}

%\subsubsection{Extending Killing vector fields to codimension 3} 
%
%The references for this section is Burago-Burago-Ivanov.
%
%\begin{remark} Recall that manifolds are length spaces and the distance between any two points is realized by the length of a piece-wise geodesic curve.
%\end{remark}
%
%\begin{lemma} Let $X$ be a complete metric space and $X'\subset X$ be a dense subset. 
%\begin{enumerate}
%\item Suppose  $\varphi': X'\to X$ is a 
%short (i.e. distance non-increasing) map. Then it extends to a short map $\varphi: X\to X$. 
%\item Let $\varphi_1, \varphi_2: X\to X$ be two short maps. Suppose the compositions $\varphi_1\circ \varphi_2$ and $\varphi_2\circ \varphi_1$ are identical on $X'$. Then $\varphi_1$ and $\varphi_2$ are inverse isomteries.
%\end{enumerate}
%\end{lemma}
%\begin{proof} The first claim is a partial case of BBI, Proposition 1.5.9. As for the second claim, note that both maps $\varphi_1\circ \varphi_2$ and $\varphi_2\circ \varphi_1$ are identity isometries since $X'$ is dense in $X$. And since $\varphi_1$ and $\varphi_2$ are short, they have to be isometries as well.  
%\end{proof}
%
%
%\begin{lemma} Let $P$ be a polyhedral manifold and let $Y\subset P$ be a piecewise analytic subset of codimension at least two. Then the length metric on $P\setminus Y$ coincides with the metric induced from $P$.
%\end{lemma}
%
%
%\begin{proof} Consider two points $x_1,x_2\in P\setminus Y$ and let $\gamma$ be a shortest path in $P$ that joins $x_1$ with $x_2$. Since $Y$ has codimension at least $2$ one can slightly perturb $\gamma$ so that it doesn't intersect $Y$. 
%\end{proof}
%

\subsection{Polyhedral K\"ahler spaces}\label{sect:polKahlerspaces}

In this section we discuss the notion of  polyhedral K\"ahler spaces, which is a natural generalisation of the notion of PK manifolds \cite{Pan}. Such spaces are the central object of the forthcoming paper \cite{PanVer}.

\begin{definition}\label{def:Complexdirections} 
	Let $X$ be a polyhedral pseudomanifold of complex dimension $n$ and let $J$ be a flat orthogonal complex structure defined on $X^\circ$. We say that the metric singularities of $X$ have \emph{complex directions} if the following conditions hold.  
	
	\item[(a)] Metric singularities of $X$ have even codimension.
	\item[(b)] Let $\tau$ be any polyhedral triangulation of $X$.
	Let $\Delta^{2n}$ be any simplex of dimension $2n$ in $\tau$ and $\Delta^{2n-2k}\subset \Delta^{2n}$ be a sub-simplex. Suppose that $\Delta^{2n-2k}$ lies in the subset $X^s_{2k}$ of singularities of codimension at least $2k$. Then $\Delta^{2n-2k}$ is a part of a complex subspace $\mathbb C^{n-k}$ with respect to the flat complex structure induced on $\Delta^{2n}$ by $g$.  
\end{definition}

\begin{remark}  
	Definition \ref{def:Complexdirections} can be rephrased in the following way. Let $x\in X^s$ be a metric singularity lying on an open stratum of singularities $Y^\circ$. Take a cylindrical neighbourhood $U$ of $x$ that embeds in the tangent cone of $X$ at $x$ and take any parallel vector field on $U^\circ$ that is parallel to the stratum $Y^\circ$. Then $Jv$ must be parallel to $Y^\circ$  as well.  
\end{remark}

Now we can define polyhedral K\"ahler spaces. 

\begin{definition}\label{def:PKspaces} 
	Let $X$ be an orientable polyhedral pseudomanifold of complex dimension $n$ that admits a parallel orthogonal complex structure $J$ on its regular locus $X^\circ$. The pair $(X,J)$ is called a \textbf{polyhedral K\"ahler space} if all singularities of $X$ have complex directions. 
\end{definition}

\begin{remark} For an orientable polyhedral pseudomanifold $X$ existence of a parallel orthogonal complex structure $J$ on $X^\circ$ depends only on the image of the holonomy representation $\mathrm{Hol}:\pi_1(X^\circ)\to SO(2n)$. Such a parallel $J$ exists if and only if the image $\mathrm{Hol}(\pi_1(X^\circ))$ is contained in a subgroup of $SO(2n)$ conjugate to $U(n)$.
\end{remark}

Even though Definition \ref{def:PKspaces} requires all metric singularities to have even codimension and have complex directions, in practice the following weaker condition is sufficient.

\begin{lemma}\label{onlyinteger} Let $X$ be a normal polyhedral pseudomanifold and $J$ be a parallel orthogonal complex structure on $X^\circ$. Let $Y^\circ_1,\ldots, Y^\circ_m$ be open strata of metric singularities of complex codimension $1$ with integer conical angle\footnote{See Definition \ref{integerconeangle}}. If these strata have complex direction with respect to $J$ then all singularities of $X$ have complex direction and $(X,J)$ is a polyhedral K\"ahler space.
\end{lemma}

\begin{proof} By \cite[Lemma 3.2]{Pan} all open strata of codimension $2$ singularities with non-integer conical angles have complex direction. It follows that \emph{all} open strata of codimension $2$ in $(X,J)$ have complex direction.
	
	Let $x\in X^s$ be a metric singularity of codimension larger than $2$ lying on an open stratum of singularities $Y^\circ$. Let $U$ be a small neighbourhood of $x$ and $v$ be a parallel vector field in $U^\circ$ that is parallel to $Y^\circ$. We need to show that $Jv$ is parallel to $Y^\circ$ as well. 
	
	Take the tangent cone $C$ to $X$ at $x$. This is a  normal polyhedral manifold  with $J$ on $C^\circ$. Consider the maximal decomposition $C=\mathbb R^m\times C'$ where $m$ is the real dimension of $Y^\circ$. Then 
	the field $v$ on $U$ gives rise to a parallel vector field on $C$ that is parallel to $(\mathbb R^m,0)$ and is in particular parallel to all strata of $C$.
	
	All open strata of metric singularities of codimension $2$ on $C$ have complex direction, since this is so for $(X,J)$. It follows that the field $Ju$ is parallel to them. Hence we can apply to $Ju$ Corollary \ref{coneJrotation} to conclude that $Ju$ is also parallel to $(\mathbb R^m,0)$. This proves that $Jv$ is parallel to $Y^\circ$.
\end{proof}

\subsection{Singularities of PK metrics are analytic subsets}\label{sect:singularitiesPK}

The goal of this section is to prove the next.

\begin{theorem}[Analyticity of singularities]\label{singanalitic} Let $(X, g)$ be a complex manifold with a PK metric as in Definition \ref{def:PK}. Then the following statements hold.
	
	\item[(1)] For each $k$ the set $X^s_{2k-1}\setminus X^s_{2k}$ of singularities of codimension $2k-1$ is empty.  
	\item[(2)] For each $k$ the set of singularities $X^s_{2k}$
	of codimension at least $2k$ is a complex analytic subset of codimension at least $k$. Moreover all connected components of $X^s_{2k-2}\setminus X^s_{2k}$ are analytic $n-k+1$ dimensional subsets of $X\setminus X^s_{2k}$.
	
	\item[(3)] Connected components of $X^s_{2k-2}\setminus X^s_{2k}$ are smooth.  
	
	\item[(4)] $X^s$ is a complex hypersurface in $X$.
\end{theorem}

 We will first prove claims (1), (2). Then in Proposition  \ref{stratprop} we conclude the proof of claims (3) and (4). First, we need a collection of preliminary results.

\begin{lemma}\label{lemma:complexdirections} Let $X$ be a complex manifold with a PK metric $g$. Then $(X,g)$ is a polyhedral K\"ahler space, i.e., all metric singularities of $X$ have complex directions.   
\end{lemma}

\begin{proof} 
	Since the flat metric $g$ is K\"ahler on the subspace $X^\circ$ of regular points, the complex structure $J$ is parallel on $X^\circ$. So according to Lemma \ref{onlyinteger} we need to show that each open stratum of singularities of complex codimension $1$ with integer conical angle has complex direction. Let $Y^\circ$ be such a stratum and let $x\in Y^\circ$ be a point. Consider the tangent cone $C$ to $X$ at $x$. It is isometric to the direct product $\mathbb R^{2n-2}\times C'$, where $C'$ is a $2$-cone with conical angle $2\pi m$ for some integer $m$. 
	
	Consider the developing map $\mathrm{dev}:C\to \mathbb R^{2n}=\mathbb R^{2n-2}\times \mathbb R^2$ (of degree $m$ on the second factor) to the Euclidean space $\mathbb R^{2n}$. The map $\mathrm{dev}$ ramifies along an Euclidean subspace $H\subset \mathbb R^{2n}$ of dimension $2n-2$. The complex structure on $X$ induces a parallel complex structure $J_C$ on $C^\circ$. Therefore, we can choose a parallel complex structure $J$ on $\mathbb R^{2n}$ so that $(\mathbb R^{2n},J)=\mathbb C^n$ and the map $\mathrm{dev}: C^\circ\to \mathbb C^n$ is holomorphic. 
	
	Let's show that the map $\mathrm{dev}: C\to \mathbb C^n$ is holomorphic on the whole $C$. Since the map $\mathrm{dev}$ is continuous, by the Riemann extension Theorem \ref{Riemann}  it is enough to show that $C^s$ is contained in a proper analytic subset of $C$. Indeed, in such case the pull-back functions $\mathrm{dev}^*z_i$ will be holomorphic on $C$ and the subspace $H$ will be complex.  Let's take a point $y\in C^s$, and take its open neighbourhood $U$ in $C$, invariant under the action of the group $\mathbb Z_m=\mathbb Z/m\mathbb Z$ of deck transformations. Take any holomorphic function $f$ in $U$ that is not invariant under $\mathbb Z_m$ (such functions exist, since we can take a function that takes different values at two points of an orbit of the $\mathbb Z_m$ action). Consider now a new function 
	$$\widetilde f=f-\frac{1}{m}\sum_{h\in \mathbb Z_m}h^*(f).$$
	Since $\mathbb Z_m$ is acting holomorphically on $U^{\circ}$, this function is holomorphic on $U^\circ$ and continuous on $U$. Moreover, by construction $\widetilde f$ vanishes on $C^s\cap U$. It follows by Rado's Theorem \ref{Rado} that $\widetilde f$ is analytic, and so $C^s\cap U$ is contained in a proper analytic subset of $U$.
\end{proof}

\begin{theorem}[{Riemann extension, \cite[pg. 38]{FG}}]\label{Riemann}
	Let \(D\) be a domain in \(\C^n\) and let \(A \subset D\) be a proper analytic subset. If \(f\) is a holomorphic function on \(D \setminus A\) that is bounded in a neighbourhood of every point in \(A\), then \(f\) can be holomorphically extended to \(D\).
\end{theorem}

\begin{theorem}[{Rado, \cite[pg. 302]{Chirka}}]\label{Rado} Let $D$ be a domain in $\mathbb  C^n$, and let $f$ be a function continuous in $D$ and holomorphic everywhere outside its zero set. Then f is holomorphic  in $D$.
\end{theorem}

Here is the remaining ingredient for proving Theorem \ref{singanalitic} (1-2).

\begin{theorem}[{Remmert-Stein, \cite[pg. 150]{FG}}]\label{Remmert-Stein}  Assume that $D\subset \mathbb C^n$ is a domain, $K\subset D$ is  an $n-d$-dimensional analytic subset, and $A$ an analytic subset of $D\setminus K$ all components of which have dimension $>n-d$. Then the closure $\overline A$ in $D$ is an analytic set in $D$.
\end{theorem}

\begin{proof}[Proof of Theorem \ref{singanalitic}: (1), (2)] 
	
	(1) This follows from Lemma \ref{lemma:complexdirections}.
	
	(2) We prove the statement by induction. For $2k=2n+2$ the  subset $X^s_{2k}\subset X$ is empty, $X^s_{2k-2}\setminus X^s_{2k}=X^s_{2n}$ is a discreet union of points, so it is an analytic subset of complex codimension $n$. Suppose now the statement is proven for $2k\le 2n+2$ and let us prove it for $2k-2$. We know by induction that $X^s_{2k}$ is an analytic subset of $X$ complex codimension at least $k$. We will prove the same claim for $X^s_{2k-2}$ by applying Theorem  \ref{Remmert-Stein}. For this it is enough to show that each connected component of $X^s_{2k-2}\setminus X^s_{2k}$ in $X\setminus X^s_{2k}$ is an analytic subset of complex dimension $n-k+1$. 
	
	Let $Y^\circ$ be a connected component of $X^s_{2k-2}\setminus X^s_{2k}$, i.e. an open stratum of codimension $2k-2$ singularities. Let $x\in Y^\circ$ be a point. It has an open neighbourhood $U\subset X\setminus X_{2k}^s$ which is a metric product $B_\varepsilon\times C_\varepsilon$,  where $B_\varepsilon$ is a radius $\varepsilon$ ball in $\mathbb R^{2n-2k+2}$ and $C_\varepsilon$ is a radius $\varepsilon$-ball centred at the vertex of a polyhedral cone. We denote by $\varphi: B_\varepsilon\times C_\varepsilon\to X$ the isometric embedding. Our goal is to show that $\varphi(B_\varepsilon,0)$ is a complex analytic subset of $U$ of dimension $n-k+1$.

	Since $Y^\circ$ has complex direction, for any regular point $z\in C_\varepsilon$ the ball $\varphi(B_{\varepsilon},z)$ is a complex submanifold in $U$. Consider the path $tz\in C_\varepsilon$ for $t\in [0,1]$. All the balls $(B_{\varepsilon},zt)$ are naturally identified with $B_{\varepsilon}$ and for $t>0$ the complex structure on $\varphi(B_{\varepsilon},zt)$ pulls back to the same linear complex structure $J$ in  $B_{\varepsilon}$. Now, we have a family of maps $\varphi: ((B_{\varepsilon},J),zt)\to U$ which are holomorphic for $t\in (0,1]$ and which depends continuously on $t\in [0,1]$. It follows that the map $\varphi: ((B_{\varepsilon},J),0)\to U$ is also holomorphic. So $\varphi(B_\varepsilon,0)$ is a complex analytic subset of $U$ of dimension $n-k+1$. We proved that $Y^\circ$ is complex analytic in $X\setminus X_{2k}^s$. This concludes the proof of (2).
\end{proof}

The first application of Theorem \ref{singanalitic} (2) is the following lemma.

\begin{lemma}\label{conemb}
	Let $(X,g)$ be a complex manifold with a $PK$ metric. For any point $x\in X$ the polyhedral tangent cone $C_x$ has a natural structure of complex manifold with a PK cone metric.
\end{lemma}

\begin{proof}
	By definition $x$ has an $\varepsilon$-neighbourhood $U_x(\varepsilon)\subset X$ that embeds isometrically to $C_x$. For each $c\in (0,1)$ there is a canonical homothetic contraction  $\varphi_c: U_x(\varepsilon)\to U_x(c\varepsilon)$ that restricts to a contraction by factor $c$ on each ray of the cone. Note that contractions $\varphi_c$ are holomorphic on $U_x(\varepsilon)^\circ$, and since they are continuous, and $U_x(\varepsilon)^s$ is analytic by Theorem \ref{singanalitic} (2),  they are holomorphic on the whole $U_x(\varepsilon)$. 
	
	Consider now an isometric embedding from $U_x(\varepsilon)$ to $C_x$. This isometry induces a complex structure on an $\varepsilon$-neighbourhood of the apex $o$ of $C_x$.  We can extend this complex structure to the whole cone by pulling it back from $U_o(\varepsilon)$ to $C_x$ by homothetic contractions with various contraction constants $c<1$. The resulting complex structure  on $C_x$ is independent of a choice of $c$ since the homotheties $\varphi_c$ are holomorphic.
\end{proof}

\begin{proposition}\label{PKconeproduct} 
	Let $(C,g)$ be a complex $n$-manifold with a PK cone metric $g$. Then $C$ admits a unique orthogonal and holomorphic direct product decomposition as  $(C,g)=\mathbb C^{n-k}\times (C',g')$, where $(C',g')$ is a complex manifold with a PK metric that is an essential polyhedral cone. 
\end{proposition}

\begin{proof} 
	Since $(C,g)$ is a polyhedral cone, it admits a unique orthogonal product decomposition $\mathbb R^{2(n-k)}\times (C',g')$ where $(C',g')$ is an essential polyhedral cone (see Theorem \ref{undec}). We will show that this decomposition satisfies the necessary properties starting with constructing a holomorphic $\mathbb C^{n-k}$ action on $C$.
	
	First, we observe that $\mathbb R^{2(n-k)}$ is acting by isometries of $(C, g)$ by parallel translations in the first factor of $\mathbb R^{2(n-k)}\times (C',g')$. Let's show that this is a holomorphic action by $\mathbb C^{n-k}$.  For any $v\in \mathbb R^{2(n-k)}$ let $\varphi_v: C\to C$ be the translation by vector $v$. Clearly $\varphi_v$ sends $C^\circ$ to $C^\circ$ and is holomorphic on it. By Theorem \ref{singanalitic} (2) the complement $C\setminus C^\circ$ is an analytic subset, and since $\varphi_v$ is continuous on $C$, it is holomorphic on $C$ as well by the Riemann extension Theorem \ref{Riemann}. So, $\mathbb R^{2(n-k)}$ is acting on $C$ by biholomorphisms. Let now $d\varphi$ be the homomorphism from the Lie algebra of $\mathbb R^{2(n-k)}$ to the Lie algebra of holomorphic vector fields on $C$. By Lemma \ref{lemma:complexdirections}, the stratum $(\mathbb R^{2(n-k)},0)$ of $C$ has complex directions. It follows that the image $d\varphi(\mathbb R^{2(n-k)})$ is invariant under multiplication by $i$. 
	%This can be checked locally in a neighbourhood of any point $x\in C^\circ$. In such a neighbourhood the fields  $d\varphi(\mathbb R^{2(n-k)})$ are parallel and tangent to the $\mathbb R^{2(n-k)}$ fibres that are (flat) complex submanifolds by Lemma \ref{lemma:complexdirections}. 
	Hence a holomorphic action of $\mathbb C^{n-k}$ is constructed and now we can identify $C$ with $\mathbb C^{n-k}\times C'$.  
	
	Consider the projection $p_1$ from $C=\mathbb C^{n-k}\times C'$ to $(\mathbb C^{n-k},0)$. This is clearly a continuous map, and we will show that $p_1$ is holomorphic. To do this, it is enough to show that $p_1$ is holomorphic on $C^\circ$ (again we use that $C^s$ is analytic). 
	
	We have seen that the fibres $\mathbb C^{n-k}\cdot x'$ in the product $\mathbb C^{n-k}\times C'$ are holomorphic submanifolds of $C$. These submanifolds are flat in $C^{\circ}$ and since the $C'$-fibres are orthogonal to $\mathbb C^{n-k}$-fibres, in local flat holomorphic coordinates on $C^\circ$ the map $p_1$ is a (linear)  projection on a complex vector subspace along another complex vector subspace. Hence $p_1$ is a holomorphic map.
	
	Next, we observe that $p_1$ is a holomorphic submersion because each $\mathbb C^{n-k}$-fibre is projected by $p_1$ isomorphically to $\mathbb C^{n-k}$ (and so the differential of $p_1$ has rank $n-k$).
	
	We conclude that each $C'$-fibre is a smooth complex submanifold of $C'$, and $g$ restricts to a polyhedral cone metric on such fibres. Finally, all $C'$-fibres are biholomorphic to each other thanks to the $\mathbb C^{n-k}$ action and the product decomposition $C=\mathbb C^{n-k}\times C'$ is clearly holomorphic.
\end{proof}

The following  proposition finishes the proof of Theorem \ref{singanalitic}.

\begin{proposition}\label{stratprop} 
	Suppose that $(X,g)$ is a complex $n$-manifold with a $PK$  metric. Then the following holds.
	
	\begin{itemize}
		\item[(1)] Each open stratum of singularities $Y^\circ$ of $X^s$ is a smooth complex submanifold of $X$. 
		
		\item [(2)] Let $y\in Y^\circ$ be a point on a open stratum of singularities of complex codimension $k$. Then the tangent cone to $y$ is a direct metric and holomorphic product $\mathbb C^{n-k}\times X'$ where $X'$ is a complex manifold with an essential $PK$ cone metric.
		\item [(3)] $X^s$ is a complex hypersurface in $X$.
		
	\end{itemize} 
\end{proposition}

\begin{proof} (1, 2) Both statements follow directly from Proposition 
	\ref{PKconeproduct} and Lemma \ref{conemb}.

	(3) By Theorem \ref{singanalitic} (2) we know that $X^s$ is an analytic  subset of $X$. Let us prove that all irreducible components of $X^s$ have dimension $n-1$. We can assume that $X$ is a cone and so it has only finite number of open strata of singularities. 
	
	 Assume by contradiction that $X^s$ contains an irreducible component $Z$ of codimension $k>1$. Let $Z'\subset Z$ be the subset of points that don't lie in other irreducible components of $X^s$. Note that $Z'$ doesn't contain singular points of codimension $l$ less than $k$. Indeed, by Theorem \ref{singanalitic} (2) in a neighbourhood of such a point $X^s$ has dimension at least $n-l$. Furthermore, open strata of metric singularities of complex coidmension larger than $k$ can't cover the entire  $Z'$. Hence there is a point $z\in Z'$ that is a metric singularity of complex codimension exactly $k$. 
	 
	 The tangent cone to $X$ at $z$ is a product $\mathbb C^{n-k}\times X'$. Note now that since $z$ lies in only one irreducible component of $X^s$, all the metric singularities of $\mathbb C^{n-k}\times X'$ lie in $\mathbb C^{n-k}\times 0$. So $X'$ is a polyhedral cone with only singularity at its vertex. The link of its vertex as a space form (i.e., a manifold of sectional curvature $1$), and so the cone $X'$ is isometric to $\mathbb C^k/\Gamma$ where $\Gamma$ is a finite group acting freely on $\mathbb C^k\setminus 0$. This contradicts to the fact that $X$ is a smooth complex manifold.
	
%	(3) By Theorem \ref{singanalitic} (2) we know that $X^s$ is an analytic  subset of $X$. Let us prove that all irreducible components of $X^s$ have dimension $n-1$. Assume by contradiction that this is not the case and let $Z\subset X^s$ be an irreducible component of codimension $k>1$. Let $Z'\subset Z$ be the subset of points that don't lie in other irreducible components of $X^s$. Note that $Z'$ intersects only a finite number of open strata of metric singularities. Indeed, by (1) an open stratum is smooth, and so if it contains a point of  $Z'$, it is contained in $Z$.
%	The open strata of metric singularities of complex coidmension larger than $k$ can't cover the entire  $Z'$. Hence there is a point $z\in Z'$ that is a metric singularity of complex codimension exactly $k$. The tangent cone to $X$ at $z$ is a product $\mathbb C^{n-k}\times X'$. Note now that since $z$ lies in only one irreducible component of $X^s$, all the metric singularities of $\mathbb C^{n-k}\times X'$ lie in $\mathbb C^{n-k}\times 0$. So $X'$ is a polyhedral cone with only singularity at its vertex. The link of its vertex as a space form (i.e., a manifold of sectional curvature $1$), and so the cone $X'$ is isometric to $\mathbb C^k/\Gamma$ where $\Gamma$ is a finite group acting freely on $\mathbb C^k\setminus 0$. This contradicts to the fact that $X$ is a smooth complex manifold.
\end{proof}

\subsection{Global complex coordinates on cones}\label{sect:globalcxcoord}

The main result of this section is the following.

\begin{proposition}\label{prop:PKcones}
	Let \(X\) be a complex manifold and \(g\) a PK cone metric on it with vertex at \(x\). Then there are holomorphic coordinate functions \((z_1, \ldots, z_n)\) centred at \(x\) and positive numbers \(c_1, \ldots, c_n\) such that the Euler vector field of the cone is given by
	\begin{equation}\label{elinear}
	e = \sum_{i=1}^{n} c_i \frac{\p}{\p z_i} .
	\end{equation}
	Moreover, these coordinates are globally defined and give a biholomorphism  
	\begin{equation}\label{bihol}
	(z_1, \ldots, z_n): X \xrightarrow{\sim} \C^n 
	\end{equation}
	which is equivariant with respect to the dilation action \(p \mapsto \lambda \cdot p\) on \(X\) and 
	\begin{equation}\label{action}
	\lambda \cdot (z_1, \ldots, z_n) = (\lambda^{c_1}z_1, \ldots, \lambda^{c_n}z_n )
	\end{equation}
	on \(\C^n\). In particular, the complex hypersurface \(X^s \subset \C^n \) must be invariant under the action \eqref{action}.
\end{proposition}

Throughout the following subsections we fix a complex manifold \(X\) and a PK cone metric \(g\) on it with vertex at \(x\).

\subsubsection{{\large Radius function, K\"ahler form and unit sphere}}\label{sect:kahlerpotential}

\begin{definition}
	We define the \textbf{radius function}
	\begin{equation}
	\begin{aligned}
	r \,: X \,  &\to \, \, \R_{\geq 0} \\
	p \,\ &\mapsto \, r(p)=d(p,x)	
	\end{aligned}
	\end{equation}
	where \(d(p,x)\) denotes the distance to the vertex of the cone.
\end{definition}

\begin{notation}
	Let \(\tilde{X}^{\circ}\) be the universal cover of \(X^{\circ}\) endowed with the structure of a flat K\"ahler manifold given by the pull-back metric and complex structure.
\end{notation}

Fix \(p \in X^{\circ}\) and \(\tilde{p} \in \tilde{X}^{\circ}\) over it. We can map a neighbourhood \(\tilde{p} \in \tilde{U} \subset \tilde{X}\) holomorphically isometrically into an open subset of Euclidean \(\C^n\), say \(F: \tilde{U} \to \C^n\). Since \(\tilde{X}\) is simply connected, the map \(F\) can be \emph{uniquely} extended to a map
\[\hat{F} : \tilde{X}^{\circ} \to \C^n \]
which is holomorphic and locally isometric. The map \(F\) is uniquely determined
up to the action of \(\C^n \ltimes U(n)\) on \(\C^n\) by unitary affine transformations. We normalize the action by translations by requiring that
\begin{equation}\label{eq:normalizedevelop}
|F(\tilde{p})| = r(p)
\end{equation}
where \(|\cdot|\) denotes the Euclidean norm.

\begin{definition}\label{developingmap}
	The \textbf{developing map} is a holomorphic local isometry
	\begin{equation} \label{eq:devmap}
	\dev: \tilde{X}^{\circ} \to \C^n
	\end{equation}
	obtained by the above construction with the normalization given by Equation \eqref{eq:normalizedevelop}.
\end{definition}

\begin{remark}
	The map \eqref{eq:devmap} is defined up to compositions with  unitary linear transformations. We could kill this ambiguity by fixing the derivative of \(\dev\) at \(\tilde{p}\).
\end{remark}

\begin{remark}
	With our definitions the following holds.
	\begin{itemize}
		\item[(i)] The radius function \(r\) is continuous on \(X\) and smooth on the regular set \(X^{\circ}\).
		\item[(ii)] The map \eqref{eq:devmap} is equivariant with respect to the \(\R_{>0}\)-actions on \(\tilde{X}^{\circ}\) by lifted dilations of the cone and on \(\C^n\) by standard scalar multiplication.
		\item[(ii)] The Euclidean norm of the developing map \(|\dev(p)|\) is invariant under the action of \(\pi_1(X^{\circ})\) by deck transformations, therefore it descends to a function on \(X^{\circ}\) and it satisfies
		\begin{equation}\label{eq:moddevequalsr}
		\dev^* |z| = r .
		\end{equation}
	\end{itemize} 
\end{remark}

\begin{example}
	Consider the metric of a \(2\)-cone of total angle \(2\pi\alpha>0\) on \(\C\) with vertex at \(0\) given by the line element
	\(\alpha |z|^{\alpha-1}|dz|\). If we take the universal covering map \(\C \to \C^*\) given by \(\xi \mapsto z=\exp(\xi)\), then \(\dev(\xi) = \exp(\alpha \xi)\). More shortly, we can write \(\dev\) as a multivalued function on \(\C^*\) given by
	\(\dev = z^{\alpha}\). In most higher dimensional cases neither the developing map nor the radius function \(r\) are explicit in terms of complex coordinates.
\end{example}

If \(U \subset X^{\circ}\) is simply connected, then we can invert the covering map \(\tilde{X}^{\circ} \to X^{\circ}\) over \(U\) and compose with \(\dev\) to obtain a \textbf{branch of the developing map}
\[F: U \to \C^n\]
which is a holomorphic local isometry. Since \(\dev\) is equivariant with respect to the actions of \(\pi_1(X^{\circ})\) by deck transformations on \(\tilde{X}^{\circ}\) and by the holonomy representation \(\Hol(g): \pi_1(X^{\circ}) \to U(n)\) on \(\C^n\), it follows that any two branches of \(\dev\) differ by composition with a unitary linear transformation.

Recall that we write \(J\) for the complex structure of \(X\).

\begin{definition}
	On the regular part \(X^{\circ}\) we define the \textbf{K\"ahler form} of \((X, g)\) as
	\begin{equation}\label{kahlerform}
	\omega (\cdot, \cdot) = g (J\cdot, \cdot) .
	\end{equation}
\end{definition}

\begin{lemma}\label{lem:kahlerform}
	On \(X^{\circ}\) we have the identity
	\begin{equation}\label{eq:potential}
	\omega = \frac{i}{2} \dd r^2 .
	\end{equation}
	In other words, \(r^2\) is a \textbf{K\"ahler potential} for \(\omega\).
\end{lemma}

\begin{proof}
	The pull-back of \(\omega\) to \(\tilde{X}^{\circ}\) is equal to the pull-back by \(\dev\) of the Euclidean K\"ahler form \(\omega_E\) in \(\C^n\). On the other hand, the squared Euclidean norm is a K\"ahler potential for \(\omega_E\). Since \(\dev\) is holomorphic, the pull-back \(\dev^*\) commutes with \(\dd\) and Equation \eqref{eq:potential} follows from Equation \eqref{eq:moddevequalsr}.
\end{proof}

\begin{remark}
	The function \(r^2\) is continuous on \(X\), and smooth plurisubharmonic (psh) on \(X^{\circ}\). Since the singular set \(X^s\) is a complex hypersurface (in particular it is pluripolar),  the function \(r^2\) is psh on the whole \(X\), see \cite[Theorem 5.24]{Demailly}. In particular, this implies that we can consider \(\omega\) as a closed positive current defined on all of \(X\), see \cite[Lemma 7.1]{KozNg}.
\end{remark}

\begin{definition}
	The unit sphere around the origin of \((X, g)\) is
	\begin{equation*}
	S = \{r=1\} \subset X  .
	\end{equation*}
\end{definition}

Since \((X,g)\) is a polyhedral cone, the restriction of \(g\) to \(S\) is a spherical polyhedral metric.

\begin{notation}
	We denote the regular part of \(S\) by
	\[S^{\circ}=S\cap X^{\circ} .\] 
	It is a smooth submanifold of \(X^{\circ}\) equipped with a smooth Riemannian metric \(g_{S^{\circ}}\) locally isometric to the round unit sphere \(S^{2n-1}(1)\).
\end{notation}

\subsubsection{{\large Euler vector field}}

\begin{definition}
	On the regular locus \(X^{\circ}\) we have a smooth (real) Euler vector field given by
	\begin{equation}
	e_r = r \p_r
	\end{equation}
	where \(\p_r\) is the gradient of \(r\). 
\end{definition}

\begin{note}
	Equivalently, \(e_r\) is the vector field generated by the one-parameter subgroup of diffeomorphisms \(\varphi_t\) of \(X^{\circ}\) given by \(\varphi_t(q) = e^{t} \cdot q\).
\end{note}

\begin{remark}
	Note that \(\p_r\) is a unit vector field on \(X^{\circ}\), so we have the identity
	\begin{equation}
	r^2 = g(e_r, e_r) .
	\end{equation}
\end{remark}

\begin{lemma}\label{lem:dever}
	Let \(U \subset X^{\circ}\) be simply connected and let \(F:U \to \C^n\) be a branch of the developing map. Then \(F\) maps
	\(U\cap S^{\circ}\) into \(S^{2n-1}(1)\) and the derivative of \(F\) maps the vector field \(e_r\) to the standard real Euler vector field of \(\C^n\).
\end{lemma}

\begin{proof}
	Follows from the fact that the developing map commutes the action by dilations of the cone with usual scalar multiplication on \(\C^n\).
\end{proof}

\begin{remark}
	A Sasakian manifold is a Riemannian manifold whose metric cone is K\"ahler, see \cite{Sparks}. By definition, \((S^{\circ}, g_{S^{\circ}})\) is Sasakian.
\end{remark}

\begin{definition}
	Recall that \(J\) denotes the complex structure of \(X\). We define the \textbf{Reeb vector field} on \(X^{\circ}\) to be given by
	\begin{equation}
	e_s = J e_r .
	\end{equation}
\end{definition}

\begin{lemma}\label{lem:deves}
	Let \(U \subset X^{\circ}\) be a simply connected open set and let \(F\) be a branch of \(\dev\) over \(U\). Then the derivative of \(F\) maps \(e_s\) to the standard Reeb vector field of \(\C^n\). In particular, \(e_s\)  is tangent to \(S^{\circ}\). 
\end{lemma}

\begin{proof}
	This follows from Lemma \ref{lem:dever} together with the fact that the derivative of \(F\) exchanges \(J\) with the standard complex structure of \(\C^n\).
\end{proof}

%\begin{notation}
%	 and \(\nabla\) for the Levi-Civita connection of the flat K\"ahler manifold \((X^{\circ}, g)\); so \(\nabla J = 0\). 
%\end{notation}
%We say that \(g\) is a PK metric on \(X\) if \(g\) is K\"ahler on \(X^{reg}\) and \(X^{sing}\) is a complex hypersurface.

\begin{definition}
	The (complex) Euler vector field is defined as 
	\[e= \frac{1}{2} ( e_r - i e_s ) . \]
\end{definition}

\begin{lemma}\label{holomorphic} 
	The real Euler vector field \(e_r\) extends across \(X^s\) as a real holomorphic vector field on \(X\) with a single zero at the vertex of the cone.
\end{lemma}

\begin{proof}
	The flow \(\{\varphi_t, \,\ t \in \R\}\)  of \(e_r\) is given by \(\varphi_t(p)=e^t \cdot p\) restricted to \(X^{\circ}\). Condition (ii) of Definition \ref{def:PK} implies that \(\varphi_t\) (for fixed \(t\)) is a biholomorphism of \(X^0\) (see Lemma \ref{lem:dever}). 
	On the other hand, \(\varphi_t\) is a homeomorphism of \(X\) because of Condition (i) of Definition \ref{def:PK}. It follows that \(\varphi_t\) is holomorphic on the whole \(X\) because it is continuous and holomorphic outside the complex hypersurface \(X^s\). 
\end{proof}

\begin{corollary}
	The Reeb vector field \(e_s\) and the complex Euler vector field extend across \(X^s\) as holomorphic vector fields with a single zero at \(x\). 
\end{corollary}

\begin{proof}
	Immediate from Lemma \ref{holomorphic} together with the fact that \(J\) is a smooth complex structure.
\end{proof}

\begin{lemma}\label{lem:isometric}
	The flow of \(e_s\) generates an action of \(\R\) on \(X\) by holomorphic isometries fixing the point \(x\).
\end{lemma}

\begin{proof}
	The vector field \(e_s\) is complete because its orbits lie on compact subsets \(r \equiv\) constant. Lemma \ref{holomorphic} implies that the flow \(\{\psi_t, \,\ t \in \R\}\)  of \(e_s\) is made of biholomorphisms fixing \(x\). The singular set \(X^s\) is invariant by \(e_s\) because \(X^s\) is a complex hypersurface, hence \(\psi_t\) preserves \(X^{\circ}\). Condition (ii) of Definition \ref{def:PK} implies that the \(\psi_t\) are isometries of \(X^{\circ}\) (see Lemma \ref{lem:deves}). By continuity, \(\psi_t\) are isometries of \((X, g)\).
\end{proof}

\subsubsection{{\large Proof of Proposition \ref{prop:PKcones}}}

We begin with two preliminary lemmas.

\begin{lemma}\label{isomtery} 
	Let $(X,g)$ be a complex manifold and $g$ a polyhedral  metric on it inducing its topology. Suppose we have a holomorphic isometric action of $\mathbb R$ on $(X,g)$ that fixes a point $x$. Then $\mathbb R$ is a dense subgroup of a torus $T^k$ that acts on a metric ball centred at \(x\) by isometries and biholomorphisms.
\end{lemma}

\begin{proof}
	Let \(B\) be a metric ball centred at \(x\), if the diameter of \(B\) is sufficiently small then \(B\) is biholomorphic to a bounded domain in \(\C^n\).
	Consider the closure of $\mathbb R$ in the group of holomorphic isometries of $(B,g)$ that fix $x$.  This is a torus because the group of biholomorphisms of a bounded domain in \(\C^n\) is a Lie group and its isotropy subgroups are compact, see \cite[Theorem 1.2 in Chapter III]{Kobayashi}. 
	%Since isometries and any isometry in the closure of $\mathbb R$ can be continuously approximated by holomorphic maps, the closure of $\mathbb R$ is also acting holomorphically. 
\end{proof}

\begin{lemma}\label{linearization} 
	Let $X$ be a complex manifold with a holomorphic $T^k$ action that has an isolated fixed point \(x\). Then a neighbourhood of \(x\) in $X$ is equivariantly conjugate to a ball in $\mathbb C^n$ equipped a standard action.  
\end{lemma}

\begin{proof}
	Taking the derivative at the fixed point, we obtain a linear action of \(T^k\) on \(T_x X\) and we can identify \(T_xX\) with \(\C^n\) in such a way that this action is standard. Let \(\Phi: U \to \C^n\) be an (arbitrary) holomorphic map defined  in a neighbourhood \(U \subset X\) of \(x\) that sends \(x\) to \(0\) and induces the above identification between \(T_xX\) and \(T_0\C^n\cong \C^n\). Equip \(T^k\) with its invariant measure and define
	\begin{equation*}
	F(y) = \int_{T^k} t^{-1} \Phi(ty) dt .
	\end{equation*}
	It is easy to check that \(F\) is \(T^k\)-equivariant and its derivative at \(x\) is non-singular.
\end{proof}

\begin{proof}[Proof of Proposition \ref{prop:PKcones}]
	It follows from Lemma \ref{lem:isometric} together with Lemmas \ref{isomtery} and \ref{linearization} that we have local coordinates \((z_1, \ldots, z_n)\) centred at \(x\) which satisfy Equation \eqref{action}, in particular Equation \eqref{elinear} must hold. We have \(c_i>0\) for all \(i\) because 
	\[\lim_{\lambda \to 0} \lambda \cdot p = x\]
	for every \(p \in X\). At the same time, Equation \eqref{action} implies that each coordinate function \(z_i\) is homogeneous of degree \(c_i\) with respect to the dilation action of the cone. Therefore we can extend \((z_1, \ldots, z_n)\) by homogeneity to the whole \(X\) and the statement of the proposition follows. 
\end{proof}

\begin{remark}
	The coordinates \((z_1, \ldots, z_n)\) in Proposition \ref{prop:PKcones} are unique up to the action of a finite dimensional Lie group of biholomophisms of \(\C^n\) which preserve the vector field \(e\) given by Equation \eqref{elinear}, see \cite[Proposition 9]{Bryant}. 
\end{remark}

\begin{remark}
	The numbers \(c_i>0\) are uniquely defined, hence can be considered as invariants of the metric singularity at \(x\) given by the eigenvalues of the linearisation of \(e\) at \(x\).
\end{remark}

\begin{definition}
	A collection of smooth complex hypersurfaces \(D_i \subset X\) is said to be \textbf{arrangement-like} if for every point in \(X\) there are local complex coordinates which take the hypersurfaces \(D_i\) to hyperplanes.
\end{definition}

\begin{corollary}\label{cor:arrlike}
	Let \((X,g)\) be a complex manifold with a PK cone metric on it. Suppose that the singular set \(X^s\) is arrangement-like. Then we can take global complex coordinates \((z_1, \ldots, z_n)\) on \(X\) as in Proposition \ref{prop:PKcones} which map \(X^s\) to a hyperplane arrangement.
\end{corollary}

\begin{proof}
	It suffices to take \(\Phi\) as in the proof of Lemma \ref{linearization} such that it maps \(X^s\) to a hyperplane arrangement, its symmetrization \(F\) will also have this property and the proof of Proposition \ref{prop:PKcones} goes through.
\end{proof}

\subsection{Submanifolds and induced metrics}\label{sect:sumfldandindmet}

In this section we analyse the situation when $(X,g)$ is a complex manifold with a PK metric and $X'\subset X$ is smooth complex submanifold such that the restriction $g|_{X'}$ is polyhedral as well. We start with the following case.

\begin{lemma}\label{smoothpoints} 
	Let $(X,g)$ be a complex manifold with a PK cone metric. Suppose $X^s$ is a smooth hypersurface in $X$. Then $X$ is a product of a Euclidean vector space with a $1$-dimensional cone.
\end{lemma}

\begin{proof} 
	We will prove this statement by induction. For $n=1$ the statement is clear and for $n=2$ it follows from the classification result \cite[Theorem 1.7]{Pan}.  Assume that the statement holds for all such $(X,g)$ of dimension at most $n-1\ge 2$ and let's prove it in the case $\dim X=n$. 
	
	First, we analyse points on $X^s$ different from the vertex $o$ of $X$. Let $x$ be such a point. Since $x\ne o$,  $x$ is a metric singularity of codimension at most $n-1$. In particular, the tangent cone of $X$ at $x$ is isometric to a product $\mathbb C^l\times X'$ where $l>0$ and $X'$ is a complex manifold with a PK cone metric. Since locally $X^s=\mathbb C^l\times X'^s$, we see that $X'^s$ is a smooth hypersurface in $X'$, so $(X',g)$ satisfies the conditions of the lemma, and by induction it is a product of $\mathbb C^{n-l-1}$ with a one-dimensional cone.  We conclude that $x$ is a metric singularity of complex codimension $1$ in $X$.

	From what we proved it follows that $(X^s,g)$ is a flat manifold non-singular at all points except, perhaps, the vertex. However, $n\ge 3$ and $\dim X^s\ge 2$. So  $(X^s,g)$ is isometric to $\mathbb C^{n-1}$ by the same reasoning as in  the proof of  Proposition \ref{stratprop} (3).  
	
	It remains to show that the vertex $o$ of $X$ is a metric singularity of complex codimension $1$. Take on $(X^s,g)=\mathbb C^{n-1}$ a parallel vector field.  Extend it holomorphically to a flat vector field defined in  an open neighbourhood of a non-vertex point $x\in X^s$. Further, extend this field by parallel translation to $X\setminus o$. This can be done since $v$ is invariant under the holonomy around any loop encircling $X^s$ in $X\setminus X^s$. Hence, $v$ extends holomorphically to the whole $X$. In particular the flow of $v$ exists in a neighbourhood of $o$. The flow  is isometric and it moves $o$ to points of $X^s$ that are metric singularities of complex codimension $1$. So, $o$ is a metric singularity of complex codimension $1$, which implies the lemma.
\end{proof}

\begin{corollary}\label{smoothpoint} 
	Let $X$ be a complex manifold with a PK metric $g$. Any smooth point $x$ of  the complex hypersurface $X^{s}$ is a metric singularity of complex codimension $1$, i.e. its tangent cone is a product of a flat space with a $1$-dimensional cone. In particular $g|_{X^s}$ is flat in a small neighbourhood of $x$ in $X^s$.
\end{corollary}

\begin{proof} 
	It suffices to apply Lemma \ref{smoothpoints} to the tangent cone of $x$.
\end{proof}

\begin{lemma}\label{holball} Consider a complex domain $B\subset\mathbb C^n$ with an analytic subset $Y\subset B$. Suppose there is a metric $d$ on $B$ with respect to which it is isometric to the unit Euclidean ball $B_1\subset \mathbb C^{n}$. Suppose moreover that the restriction of $d$ to $B\setminus Y$ is K\"ahler. Then $(B, d)$ is holomorphically isometric to a unit ball in $\mathbb C^n$. 
\end{lemma}
\begin{proof} Choose a point $x\in B\setminus Y$ and choose an isometry $\varphi: B\to B_1$ such that $d\varphi$ is complex linear at $T_x$. Then by assumptions of the lemma  $\varphi$ is holomorphic on $B\setminus Y$. Since $\varphi$ is continuous on $B$ and $Y$ is an analytic subset, it follows that $\varphi$ is holomorphic on $B$.
\end{proof}

\begin{lemma}\label{pksubmanifold} Let $(X,g)$ be a complex  $n$-manifold with a PK metric and let $Y\subset X$ be a smooth complex $n-l$-submanifold. Suppose $Y$ is a polyhedral submanifold of $X$, i.e, for some triangulation $\tau$ of $X$, $Y$ is a subcomplex of $\tau$. Let $k\le l$ be minimal such that $Y\subset X^s_{2k}$. Then
	\begin{itemize}
		
		\item[(1)] If $y\in Y$ is a metric singularity of $g|_Y$ then $y\in X^s_{2k+2}$.
		
		\item[(2)] $g|_Y$ is a PK metric on $Y$.
	\end{itemize}
	
\end{lemma}
\begin{proof} (1) Let $C_y$ be the tangent cone of $y$ in $X$ and let $\varphi: U\to C_y$ be an isometric embedding of a neighbourhood $U$ of $y$ into $C_y$ with $\varphi(y)=0$.  Assume by contradiction that $y$ is a singularity of complex codimension $k$ in $X$. Then $C_y=\mathbb C^{n-k}\times C'$, where $C'$ is an essential cone. Note that $\varphi(Y\cap U)$ is a smooth complex submanifold contained in $(\mathbb C^{n-k},0)$ because this subspace coincides with the set of all metric singularities of codimension $k$ of $C$. Hence the restriction $g|_Y$ can't have a singularity at $y$, a contradiction.
	
	(2) We need to verify condition (ii)  of Definition \ref{def:PK}, i.e., for any point $y\in Y$ at which $g|_{Y}$ is flat it is also K\"ahler. By (1) all metric singularities of $g|_Y$ lie in $Z=X^s_{2k+2}\cap Y$. This is an analytic subset of $Y$ of codimension at least $1$. Now, take a flat $\varepsilon$-ball $B_y(\varepsilon)\subset Y$. Since $B_y(\varepsilon)\cap Z$ is an analytic subset, $g|_Y$ is flat on $B_y(\varepsilon)$ and K\"ahler on $B_y(\varepsilon)\setminus Z$, we can apply Lemma \ref{holball} to conclude that $g|_Y$ is flat and K\"ahler on $B_y(\varepsilon)$.
\end{proof}

\subsection{PK cone metrics singular at hyperplane arrangements}\label{sect:PKhyperplane}

In this section we show that if the singularities of a PK cone metric on \(\C^n\) make a hyperplane arrangement then the stratification given by the singularities of the polyhedral metric agrees with the usual complex stratification. We also show that the Levi-Civita connection is standard in complex coordinates that linearise the Euler vector field.

\subsubsection{{\large Stratification}}

\begin{proposition}\label{prop:PKstrat} 
	Consider $\mathbb C^n$ with a PK cone metric $g$ whose metric singularities form a hyperplane arrangement $\mathcal H$, and let $L\in \mL(\mH)$ be a subspace of dimension $l$. Then 
	\begin{itemize}
		\item[(1)]For any $x\in L^\circ=L\setminus \mH^L$ the tangent cone to $(\mathbb C^n,g)$ at $x$ is isomorphic to $\mathbb C^n$ with a PK metric $g_x$ singular at $\mathcal H_L$. 
		\item[(2)] The restriction of $g$ to $L$ is a PK cone metric. The singularities of $(L,g|_L)$ are given by a sub-arrangement of the induced arrangement $\mH^L$. 
		\item[(3)] $L^\circ$   is an open stratum of metric singularities of complex codimension $n-l$ in $\mathbb C^n$.
	\end{itemize}
\end{proposition}

\begin{proof} 
	(1) The singular set of the tangent cone at \(x\) is arrangement-like, so the first item follows from Corollary \ref{cor:arrlike}.
	
	%The group generated by the action of the Euler field of the tangent cone $C_x$ is contained in $(\mathbb C^*)^k$ for some $k>0$. This action can be linearized by symmetrization (Bryant-Goodwilly). Furthermore, we can symmetrize a map that sends the genuine hyperplane arrangement in $T_x\mathbb C^n$ to the arrangement of singularities $C_x^s$ to conclude that $C_x^s$ is a hyperplane arrangement as well.
	
	(2) Thanks to Lemma \ref{pksubmanifold} (2), to see that $g|_L$ is PK it is enough to show that $L$ is a polyhedral submanifold of $\mathbb C^n$. Note that any point $x\in L$ is a metric singularity of complex codimension at least  $n-l$. Indeed, if it were a metric singularity of complex codimension $n-m$ with $m>l$, by Proposition \ref{stratprop} (2) the arrangement $\mathcal H$ would  a product $\mathbb C^m\times \cal H'$ locally at $x$, which is absurd. By Theorem \ref{singanalitic} (2) the subspace ${(\C^{n})^s_{2n-2l}}$ of singularities of $(\mathbb C^n,g)$ of complex codimension at least $n-l$ is an analytic subvariety of dimension at most $l$ and so $L$ is its irreducible subvariety. Since ${(\mathbb C^{n})^s_{2n-2l}}$ is a polyhedral subspace of $(\mathbb C^n,g)$ by Lemma \ref{subcomplex} (1), we conclude that $L$ is a polyhedral submanifold.  
	
	The second statement is proven by induction by considering a nested sequence of hyperplanes. Indeed, let $H\subset \mathcal H$ be a hyperplane. By Corollary \ref{smoothpoint} the singularities of $g|_H$ are contained in the induced hyperplane arrangement $H_{\mathcal H}$. At the same time, the metric $g|_H$ is PK and so its singularities in $H$ form a hypersuface. So this is a sub-arrangement of $H_{\mathcal H}$, and we can proceed by induction.

	(3) Take $x\in L^\circ$. We saw in (2) that $x$ is a metric singularity of codimension $\ge n-l$. So we only need to show that it is of codimension $\le n-l$. We know from (2) that $g|_L$ is flat on $L^{\circ}$, in particular $g|_L$ is flat in a small neighbourhood of $x$.

	Using (1) we  pass to the tangent cone to $\mathbb C^n$ at $x$, so that the arrangement $\mathcal H$ is replaced by $\mathcal H_L$. Then
	$g|_{L}$ is flat on $L$, and we need to show that $x$ (the vertex of the cone) is a singularity of codimension exactly $n-l$ and not higher. Metric singularities of codimension $>n-l$ form a proper analytic subset of $L$ so we can take a point $y$ on $L$ that is a singularity of codimension $n-l$. Take a open conic neighbourhood $B_y$ centred at this point. Take a parallel vector field $v$ on $L\cap B_y$. This field can be extended to a parallel vector filed $v$ on $B_y$ since $B_y$ is conic. Observe that the inclusion $B_y\setminus \mathcal{H}_L\to \mathbb C^n\setminus \mathcal{H}_L$ is a homotopy equivalence. In particular the corresponding map $\pi_1(B_y\setminus \mathcal{H}_L)\to \pi_1(\mathbb C^n\setminus \mathcal{H}_L)$ is an isomorphism. It follows that the holonomy action of $\pi_1(\mathbb C^n\setminus \mathcal{H}_L)$ on $v$ is trivial, and so we can extend $v$ to a flat field on the whole space $\pi_1(\mathbb C^n\setminus \mathcal{H}_L)$. Finally, $v$ extends to generic points of hypersurface of $\mathcal H_L$ and hence to the whole $\mathbb C^n$. The resulting field is acting by global isometries of $\mathbb C^n$. It follows that $\mathbb C^l$ action on $L$ by parallel translations extends to an isometric $\mathbb C^l$ action on the whole $\mathbb C^n$. Hence indeed $x$ is a singularity of codimension $n-l$.
\end{proof}

\begin{remark}\label{rmk:6points}
	The singularities of \(g|_L\) might be a strict subset of the induced arrangement. 
	
	Associated to a  metric \(g_{FS}\) on \(\CP^1\) with constant curvature \(4\) and cone angles \(2\pi\alpha_i\) at points \(L_i \in \CP^1\) there is a PK cone metric \(g\) on \(\C^2\) with cone angles \(2\pi\alpha_i\) at the complex lines \(L_i\), see \cite[Section 3.3]{Pan}. If the area of \(g_{FS}\) is equal to \(\pi\) then the restriction of \(g\) to any complex line through the origin (in particular \(g|_{L_i}\)) is a flat \(\C\). It is not hard to construct examples of such metrics \(g_{FS}\); e.g. multiply by \(1/4\) the spherical metric constructed in the next Example \ref{ex:6pts}.
\end{remark}

\begin{example}\label{ex:6pts} Here is an example of a sphere of area $4\pi$ with four non-integer cone points. On a round sphere $S^2$ take two meridians that meet at angle $2\pi\theta<\pi$ at the poles and bound a digon $D_{\theta}$ with angles $\pi\theta$. On the equator take a segment of length $\pi$ that doesn't intersect $D_{\theta}$. Cut out $D_{\theta}$ from $S^2$, glue back the remaining shores, cut the sphere along the segment on the equator and glue in $D_{\theta}$.
As a result we get a sphere with four cone points of angles $2\pi(1-\theta), 2\pi(1-\theta), 2\pi(1+\theta), 2\pi(1+\theta)$. 
%Take two distinct segments of the same length \(<\pi\) on \(S^2(1)\). Cut along the first segment and then reglue identifying opposite sides to get three cone points, two less than $2\pi$ and another bigger than $2\pi$. Glue two triangles along the other segment to get another three cone points with angles less than $2\pi$. The triangles can be chosen so that the area of the spherical metric with six non-integer cone points is equal to the area of $S^2(1)$. Note that in this example all six cone points are non-integer.
\end{example}

\subsubsection{{\large Levi-Civita connection is standard}}

\begin{proposition}\label{prop:PKstandard}
	Let \(g\) be a PK cone metric on \(\C^n\) with metric singularities at a hyperplane arrangement \(\mH\). Let \((z_1, \ldots, z_n)\) be complex coordinates which linearise the complex Euler vector field. Then the Levi-Civita connection of \(g\) is standard with respect to the trivialization given by the coordinate frame \(\p_{z_1}, \ldots, \p_{z_n}\).
\end{proposition}

We show first that the Levi-Civita connection of a PK metric on a complex manifold \(X\) has logarithmic singularities at the smooth locus of the complex hypersurface \(X^s\).
To recall the notion of a logarithmic connection, see Definition \ref{def:logcon}.

\begin{lemma}\label{residue}
	Let \(g\) be a PK metric on a complex manifold \(X\). Write \(Y=X^s\) for the complex hypersurface of metric singularities and \(Y^{\circ}\) for the smooth locus of \(Y\). Then the Levi-Civita connection of \(g\) has logarithmic singularities along \(Y^{\circ}\).
\end{lemma}

\begin{proof}
	The Levi-Civita connection of a PK metric on a complex manifold is always a holomorphic connection on the holomorphic tangent bundle of \(X^{\circ}\). The statement about logarithmic singularities is immediate from Corollary \ref{smoothpoint} and the fact that the Levi-Civita connection of \(\C_{\alpha} \times \C^{n-1}\) is logarithmic.
\end{proof}

In the context of Lemma \ref{residue} we have a residues\footnote{See Definition \ref{def:resind} for the definition of residue of a logarithmic connection.} \(A_i\) of \(\nabla\) at the connected components \(Y_i^{\circ}\) of \(Y^{\circ}\).  Note that \(\ker A_i = TY_i^{\circ}\) and  \(\tr A_i = 1-\alpha_i\) where \(2\pi\alpha_i\) is the angle at \(Y_i\). These \(A_i\) are holomorphic sections of \(TX|_{Y_i^{\circ}}\) and we would like to extend them to the whole \(Y_i\).

\begin{lemma}\label{lem:residuesextend}
	Let \(g\) be a PK metric on a complex manifold \(X\). Assume that every irreducible component \(Y_i\) of the metric singular \(Y=X^s\) set is smooth and that \(Y\) is arrangement-like. Then the residues of \(\nabla\) at \(Y_i^{\circ}\) extend across the higher codimension strata as holomorphic sections of \(\End(TX|_{Y_i})\).
\end{lemma}

\begin{proof}
	By Hartogs it is enough to extend in complex codimension one. Since the statement is local, there is no loss of generality in assuming that \(g\) is a cone and that we are in the setting of Proposition \ref{prop:PKstrat}. We let \(L \in \mL(\mH)\) be a codimension two intersection. According to item (3) of Proposition \ref{prop:PKstrat} and Proposition \ref{stratprop} the metric \(g\) is holomorphically isometric around \(x \in L^{\circ}\) to \(\C^2_{PK} \times \C^{n-2}\) where \(\C^2_{PK}\) is an essential PK cone metric on \(\C^2\). By \cite[Theorem 1.7]{Pan} we can find complex coordinates \((z,w)\) on the \(\C^2_{PK}\) factor such that its Euler vector field is given by \(\alpha^{-1}z\p_z + \beta^{-1}w\p_w\) for some positive numbers \(\alpha\) and \(\beta\) and the singular locus is a union of curves of the form \(c_1z^{\alpha}=c_2w^{\beta}\). If \(\alpha=\beta\) then the Levi-Civita connection of \(\C^2_{PK}\) is preserved by scalar multiplication, the singular locus is made of complex lines and the residues are constant along them; therefore they extend holomorphically across \(x\) and we are done. On the other hand, if \(\alpha \neq \beta\) the hypothesis that the singular locus is made of irreducible smooth hypersurfaces implies that \(\C^2_{PK}\) is singular along \(\{z=0\}\cup\{w=0\}\). By \cite[Lemma 3.14]{Pan} \(\C^2_{PK}\) is the product of two \(2\)-cones (in which case we are done) or is the pull-back of a constant metric in \(\C^2\) under the map \((z,w) \mapsto (z^p, z^q)\). In the last case the connection of \(\C^2_{PK}\) is independent of the constant metric on \(\C^2\), so it agrees with the connection of a product of two \(2\)-cones of angles \(2\pi p\) and \(2\pi q\); and we are done. 
\end{proof}

\begin{proof}[Proof of Proposition \ref{prop:PKstandard}]
	It follows from Lemma \ref{lem:residuesextend} that for each \(H \in \mH\) there is a residue \(A_H\), which is a holomorphic matrix-valued function defined on all of \(H\).
	
	On the other hand, in coordinates \((z_1, \ldots, z_n)\) the Levi-Civita connection  \(\nabla\) is invariant under the action
	\[\lambda \cdot (z_1, \ldots, z_n) = (\lambda^{c_1}z_1, \ldots, \lambda^{c_n}z_n) \]
	with \(c_1, \ldots, c_n >0\).
	The entries of \(A_H\) are holomorphic functions invariant under this action, therefore are bounded and hence constant. We conclude that \(A_H\) is a constant matrix for every \(H \in \mH\). Subtracting the standard connection \(\nabla^0=d-\sum_H A_H dh/h\) from \(\nabla\) gives a matrix of holomorphic \(1\)-forms invariant under the above action, hence \(\nabla-\nabla^0 \equiv 0\). 
\end{proof}

\section{Fubini-Study metrics}\label{sect:FS}

In this section we introduce Fubini-Study (FS) metrics as complex links regular PK cones, i.e. the quotient of the unit sphere of the cone by the free isometric \(S^1\)-action generated by the Reeb vector field. We analyse the metric singularities in terms of local models and give a formula for the total volume of complex links and unit spheres of regular PK cones, see Proposition \ref{prop:totalvolume}.

\subsection{Regular PK cones}\label{sect:regPKcone}

\begin{definition}\label{def:regularPK}
	Let  \(X\) be a complex manifold and \(g\) a PK cone metric on it. We say that \(g\) is  \textbf{regular} if there are complex coordinates \((z_1, \ldots, z_n)\) which identify \(X\) with \(\C^n\) and such that 
	\begin{equation}\label{eq:ealpha0}
	e = \frac{1}{\alpha_0} \sum_{i=1}^{n} z_i \frac{\p}{\p z_i}	
	\end{equation}
	for some \(\alpha_0>0\).
\end{definition}

In other words, \((X, g)\) is regular if the numbers \(c_i\) given by Proposition \ref{prop:PKcones} satisfy
\(c_1=\ldots=c_n\).  
The term regular is in accordance with the one used in the Sasakian case. Sasakian manifolds are divided into regular, quasi-regular and irregular types depending on whether the orbits of the Reeb vector field are closed of the same period, just closed or if there is at least one non-closed orbit respectively.
See \cite{Sparks}.

\subsubsection{{\large Linear coordinates}}

\begin{lemma}\label{cor:pkarrang}
	Let \(X\) be a complex manifold and \(g\) a PK cone metric on it with vertex at \(x\). Assume that \(X^s=\cup_{i \in I} Y_i\) is made of smooth hypersurfaces \(Y_i \subset X\) with pairwise transversal intersections at \(x\). Moreover, suppose that \(\{T_xY_i, \,\ i \in I\}\) is an essential and irreducible arrangement of \(T_xX\). 
	Then \(g\) is regular.
\end{lemma}

\begin{proof}
	The \((\C^*)^k\)-action \eqref{action} generated by \(e\) on \(T_xX\) preserves each of the members of the essential irreducible arrangement \(\{T_xY_i, \,\ i \in I\}\), hence we must have \(c_1=\ldots=c_n\).
\end{proof}

\begin{lemma}\label{lem:regPK}
	Let \(g\) be a regular PK cone metric on a complex manifold \(X\) and suppose that the  singular set \(X^s\) is a union of smooth irreducible hypersurfaces. Then
	the coordinates \((z_1, \ldots, z_n)\) of Proposition \ref{prop:PKcones} that linearise the Euler field \(e\) are unique up to linear transformations and identify the singular set \(X^s\) with an arrangement \(\mH\) in \(\C^n\) made of linear hyperplanes. In these coordinates the Levi-Civita connection of \(g\) is standard.
\end{lemma}

\begin{proof}
	A biholomorphism \(F: \C^n \to \C^n\) that commutes with scalar multiplication \(F(\lambda x)=\lambda x\) must be equal to its derivative at the origin and therefore is linear; and a smooth hypersurface invariant under scalar multiplication must be a hyperplane through the origin. The last assertion follows from Proposition \ref{prop:PKstandard}.
\end{proof}

\begin{definition}\label{def:linearcoord}
	With the notation of Definition \ref{def:regularPK}, we say that \((z_1, \ldots, z_n)\) are \textbf{linear  coordinates for} \((X, g)\).
\end{definition}

\begin{corollary}\label{cor:stndLCirr}
	Suppose that \(g\) is a PK cone metric on \(\C^n\) singular at an essential irreducible arrangement. Then its Levi-Civita connection is standard with respect to
	linear coordinates. 
\end{corollary}

\begin{proof}
	Follows from Lemmas \ref{cor:pkarrang} and \ref{lem:regPK}.
\end{proof}

\begin{example}
	The singular set of a regular PK cone metric on \(\C^n\) does not need to be an irreducible arrangement. For example \(\C_{\alpha} \times \ldots \times \C_{\alpha}\) is regular. More generally, the product of any regular PK cone as considered in Definition \ref{def:regularPK} with the 2-cone \(\C_{\alpha_{0}}\) is also regular.
\end{example}

\subsubsection{{\large Homogeneity}}

Let \(g\) be a regular PK cone metric on \(\C^n\). We have linear coordinates in \(\C^n\) so that the complex Euler vector field is
given by Equation \eqref{eq:ealpha0}.

\begin{lemma}\label{lem:scalrmult}
	Standard scalar multiplication \((z_1, \ldots, z_n) \mapsto (\lambda z_1, \ldots, \lambda z_n) \) by \(\lambda \in \C^*\) scales \(g\) by a factor of \(|\lambda|^{2\alpha_0}\) (i.e. distances are scaled by \(|\lambda|^{\alpha_0}\)). In particular, multiplication by complex units preserves the metric.
\end{lemma}

\begin{proof}
	If \(|\lambda|=1\) then Lemma \ref{lem:isometric} implies that scalar multiplication by \(\lambda\) preserves \(g\). Therefore we might assume that \(\lambda\) is a positive real number and let us write \(\lambda = \exp(t/\alpha_0)\) for some \(t \in \R\). The vector field \(e_r\) is equal to the real Euler vector field of \(\C^n\) multiplied by  \((1/\alpha_0)\), so the flow of \(e_r\) at time \(t\) is equal to scalar multiplication by \(\lambda = \exp(t/\alpha_0)\). On the other hand the flow of \(e_r\) at time \(t\) scales the metric \(g\) by a factor of \(\exp(2t)=\lambda^{2\alpha_0}\).
\end{proof}

\begin{lemma}\label{lem:regPKL}
	Let \(L \in \mL(\mH)\) then the induced metric \(g|_L\) is a regular PK cone metric on \(L\).
\end{lemma}

\begin{proof}
	This is a consequence of item (2) in Proposition \ref{prop:PKstrat} together with the fact that the Euler vector field of \(g|_L\) is equal to the restriction of \(e\) to \(L\).
\end{proof}

\begin{lemma}\label{lem:linethrough0}
	Let \(\ell \subset \C^n\) be a complex line through the origin. Then the restriction of \(g\) to \(\ell\) is a totally geodesic \(2\)-cone of total angle \(2\pi\alpha_0\).
\end{lemma}

\begin{proof}
	By Lemma \ref{lem:regPKL}, after restricting we might assume that \(\ell \setminus \{0\}\) is contained in \((\C^n)^{\circ}\).
	The vector field \(e\) is tangent to \(\ell\). I follows from Lemmas \ref{lem:dever} and \ref{lem:deves} that \(\ell\) is mapped under branches of the developing map to complex lines through the origin in Euclidean \(\C^n\), hence \(\ell\) is totally geodesic.
	Identify \(\ell\) with \(\C\) via the restriction of one of the coordinate functions, say \(z_1\), restricted to \(\ell\) and call this coordinate \(t\). 
	The real part \((1/\alpha_0) t \p_t\) dilates the restricted metric and its imaginary part preserves it, it follows that the restriction must be a \(2\)-cone of total angle \(2\pi\alpha_0\).
\end{proof}

\begin{definition}\label{def:angle0}
	We say that \(2\pi \alpha_0\) is the \textbf{angle at the origin} of \(g\).
\end{definition}

\begin{lemma}\label{lem:homogeneityr2}
	There is a continuous function \(\varphi\) defined on \(\C^n \setminus \{0\}\), smooth on \((\C^n)^{\circ}\), invariant under scalar multiplication \((z_1, \ldots, z_n) \mapsto (\lambda z_1, \ldots, \lambda z_n) \) for all \(\lambda \in \C^*\) and such that
	\begin{equation}\label{eq:phi}
	r^2 = |z|^{2\alpha_0} e^{\varphi}
	\end{equation}
	where \(|z|\) denotes the standard Euclidean norm.
\end{lemma}

\begin{proof}
	By Lemma \ref{lem:scalrmult} and since \(r\) measures the distance to the vertex, we deduce that
	\[r(\lambda z) = d( \lambda z, 0) = |\lambda|^{\alpha_0}d(z, 0) = |\lambda|^{\alpha_0}r(z) . \qedhere\]
\end{proof}

\begin{note}
	The function \(\varphi\) in Equation \eqref{eq:phi} is the pull-back of a function on projective space under the standard quotient map \(\C^n\setminus \{0\} \to \CP^{n-1}\).
\end{note}

\subsubsection{{\large Angles and weights}}

Let \(g\) be a regular PK cone metric on \(\C^n\) singular at a hyperplane arrangement \(\mH\) and let \((z_1, \ldots, z_n)\) be linear coordinates.
We have defined \(\alpha_0>0\) so that \(e = (1/\alpha_0) \sum_i z_i \p/\p z_i\) and the restriction of \(g\) to any complex line through \(0\) is a \(2\)-cone of angle \(2\pi\alpha_0\). On the other hand, by Proposition \ref{prop:PKstandard} the Levi-Civita connection \(\nabla\) of \(g\) is standard. 
Recall that the \textbf{angle at} \(H\) is \(2\pi\alpha_H\), where \(g\) is locally holomorphically isometric to \(\C_{\alpha_H} \times \C^{n-1}\) at points in \(H^{\circ}\). Equivalently, \(\alpha_H=1-a_H\) where \(a_H= \tr A_H\) and \(A_H\) is the residue of \(\nabla\) at \(H\). In this section we relate angles of the metric and weights of the connection at higher codimension strata.

\begin{lemma}\label{lem:alpha0}
	Let \((\C^n, g)\) be a regular PK cone metric on \(\C^n\) singular at a hyperplane arrangement \(\mH\) and let \(\nabla\) be its Levi-Civita connection. Suppose that \(\mH\) is essential and irreducible.
	Then \(\alpha_0=1-a_0\) where 
	\[a_0 = \frac{1}{n} \sum_{H \in \mH} a_H \]
	is the weight at the origin of \(\nabla\), as introduced in Definition \ref{def:weights}.
\end{lemma}

\begin{proof}
	Let \(\nabla= d - \sum A_H dh/h\) and \(E = \sum z_i \p / \p z_i\), so
	\[\nabla E = \Id - \sum_{H \in \mH} A_H . \]
	At the same time, since \(\sum_{H \in \mH} A_H\) preserves all hyperplanes in \(\mH\) (Proposition \ref{prop:ftfstcon}) and \(\mH\) is essential and irreducible, \(\sum_{H \in \mH} A_H\) must be a constant multiple of the identity and the constant must equal \(a_0\) as can be seen by taking traces. Therefore \(\nabla E = (1-a_0) \Id\). On the other hand, since
	\(E=\alpha_0 e\) and \(\nabla e = \Id\), we get that
	\(\nabla E = \alpha_0 \Id\) and therefore \(\alpha_0 = 1-a_0\).
\end{proof}

\begin{definition}\label{def:angleatL}
	Let \(\mH\) be an arrangement of linear hyperplanes in \(\C^n\). (We do not assume that \(\mH\) is neither essential nor irreducible.) Let \(g\) be a PK cone metric on \(\C^n\) singular at \(\mH\)
	and let \(L \in \mL_{irr}(\mH)\) be an irreducible subspace of dimension \(d \geq 1\). By Proposition \ref{prop:PKstrat} the metric \(g\) is holomorphically isometric to \(\C^d \times \C^{n-d}_{PK}\)  close to points in \(L^{\circ}\),  where \(\C^{n-d}_{PK}\) is a regular PK thanks to Lemma \ref{cor:pkarrang}. We define the \textbf{angle at} \(L\) of \(g\) to be the angle at the origin of \(\C^{n-d}_{PK}\) and we denote it by \(\alpha_L\). 
\end{definition}

\begin{corollary}\label{cor:weightangle}
	Let \(L \in \mL_{irr}(\mH)\), then \(\alpha_L=1-a_L\) where
	\[a_L = \frac{1}{\codim L} \sum_{H \in \mH_L} a_H \]
	is the weight of \(\nabla\) at \(L\). 
\end{corollary}

\begin{proof}
	Follows from Lemma \ref{lem:alpha0} applied to the transversal PK cone \(\C^{n-d}_{PK}\).
\end{proof}

\begin{lemma}\label{lem:prod}
	Let \(g\) be a PK cone metric on \(\C^n\) singular at a hyperplane arrangement \(\mH\). Moreover, assume that \(\alpha_L  \notin \Z\) for every \(L \in \mL_{irr}(\mH)\). Let \(L\) be an intersection of members in \(\mH\). Then \(g\) splits close to points in \(L^{\circ}\) according to the irreducible decomposition of \(\mH_L\).
\end{lemma}

\begin{proof}
	The Levi-Civita connection \(\nabla\) of \(g\) is (flat torsion free) standard and has non-integer weights, so it satisfies the hypothesis of Theorem \ref{PRODTHM} (note that positivity is automatic since \(\alpha_L>0\)). The statement then follows from the Local Product Decomposition Theorem \ref{thm:locprod} (more precisely Proposition \ref{prop:locaffprod} and Lemma \ref{lem:orthogfact}).
\end{proof}

\subsection{Volume forms of PK cones and local integrability}\label{sect:VolformPKcone}

We show that the positivity assumption in Theorem \ref{PRODTHM} is equivalent to the local integrability (in the usual Lebesgue sense) of the function
\[ \prod_{H \in \mH} |h|^{2\alpha_H-2} . \]
We prove that the Riemannian volume form on the regular part of a PK cone metric on \(\C^n\) singular at a hyperplane arrangement \(\mH\) is up to a constant positive factor equal to
\[\left(\prod_{H \in \mH} |h|^{2\alpha_H-2}\right) dV_{\C^n}\]
where \(dV_{\C^n} = (1/n!) (\sum_i (i/2) dz_i \wedge d\bar{z}_i)^n \) is the Euclidean volume form and \((z_1, \ldots, z_n)\) are linear coordinates in which the Levi-Civita connection of the cone metric is standard. 

\subsubsection{{\large Positivity \(\iff\) local integrability}}

\begin{definition}
	A \textbf{weighted hyperplane arrangement} consists of an arrangement of complex linear hyperplanes \(\mH\) in \(\C^n\) together with \emph{weights} \(a_H \in \C\) indexed by the members of \(\mH\). 
\end{definition}

We re-interpret the positivity condition in Theorem \ref{PRODTHM} in terms of local integrability of a natural measure associated to a weighted hyperplane arrangement with real weights. The discussion that follows holds in the general setting of weighted hyperplane arrangements with real weights and it is metric independent.

Suppose that \(\mH\) is a hyperplane arrangement in \(\C^n\). Choose defining linear equations for the members of \(\mH\), as usual let us denote these by \(h\) where \(H=\{h=0\} \in \mH\).
Moreover, assume that for each \(H \in \mH\) we are given a real number \(a_H\).
For irreducible intersections \(L \in \mL_{irr}(\mH)\) we define
\[ a_L = (\codim L)^{-1} \sum_{H \in \mH_L} a_H  \]
where \(\mH_{L}\) is the set of hyperplanes in the arrangement that contain \(L\). In this setting, we have the next. 

\begin{lemma}\label{lem:positivityvolume}
	The positivity assumption \(a_L<1\) for every \(L \in \mL_{irr}(\mH)\) is equivalent to
	\begin{equation}\label{eq:volumedensity}
	\prod_{H \in \mH} |h|^{-2a_H}
	\end{equation}
	being a locally integrable function with respect to the usual Lebesgue measure.
\end{lemma}

\begin{proof}
	Let \(x \in H^{\circ}\) and take complex coordinates \((x_1, \ldots, x_n)\) centred at \(x\) such that \(H=\{x_1=0\}\). Let \(\rho= |x_1|\). By Fubini's Theorem, \eqref{eq:volumedensity} is 
	locally integrable around \(x\) if and only if
	\[\int_{0}^{1} \rho^{1-2a_H} d\rho < \infty \iff a_H < 1 .\]
	
	We proceed by induction on the codimension of \(L \in \mL(\mH)\).
	By Fubini's Theorem it is enough consider \(L \in \mL_{irr}(\mH)\). Let \(x \in L^{\circ}\) and take complex coordinates centred at \(x\) such that \(L = \{x_1=\ldots=x_c=0\} \) where \(c= \codim L\). Write \(\rho=|(x_1, \ldots, x_c)|\) where \(|\cdot|\) denotes the usual Euclidean norm. Same as before,
	local integrability of \eqref{eq:volumedensity} around \(x\) is equivalent to
	\[ \int_0^1 \rho^{2c-1-2ca_L}d\rho < \infty \iff a_L<1 . \qedhere \]
\end{proof}

\subsubsection{{\large Volume forms of PK cones}}

\begin{notation}
	We let \((z_1, \ldots, z_n)\) be complex linear coordinates in \(\C^n\) and let
	\[\Omega = \sum_{H \in \mH} A_H \frac{dh}{h} \]
	be the connection form of a standard connection \(\nabla=d-\Omega\) in the coordinate frame \(\p/\p z_1, \ldots, \p/\p z_n\).
\end{notation}

\begin{lemma}
	Let \(\Phi= dz_1 \wedge \ldots \wedge dz_n\), then
	\begin{equation}\label{eq:topnabla}
	\nabla \Phi = (\tr \Omega) \Phi = \left(\sum_{H \in \mH} a_H \frac{dh}{h} \right) \Phi 
	\end{equation}
	where \(a_H = \tr A_H\).
\end{lemma}

\begin{proof}
	This is a general fact, the connection form on the cotangent bundle with respect to the dual frame \(dz_1, \ldots, dz_n\) has connection form \(-\Omega^t\) (the negative transpose). The induced connection  on the top exterior power has connection form equal to \(- \tr \Omega\) with respect to the trivializing section \(\Phi\).
\end{proof}

\begin{lemma}\label{lem:volform0}
	Suppose that \(a_H \in \R\) for all \(H \in \mH\). Then the real volume form
	\begin{equation}
	\mu= \left(\prod_{H \in \mH} |h|^{-2a_H}\right) i^{n^2} \Phi \wedge \bar{\Phi}
	\end{equation}
	is a flat section of \(K_{\C^n} \otimes \bar{K}_{\C^n}\).
\end{lemma}

\begin{proof}
	It is enough to check flatness locally.
	On simply connected subsets of \((\C^n)^{\circ}\) we can take branches of \(h^{-a_H}\) and define
	\[s = \left(\prod_{H \in \mH} h^{-a_H}\right) \Phi , \]
	so \(i^{n^2} s \otimes \bar{s} = \mu\). It follows from Equation \ref{eq:topnabla} together with Leibniz's rule that \(\nabla s =0\) and the statement follows.  
\end{proof}

\begin{corollary}\label{lem:volumeform}
	Suppose that the standard connection \(\nabla\) is the Levi-Civita connection of a flat K\"ahler metric on \((\C^n)^{\circ}\) then the volume form \(\omega^n\) of the K\"ahler metric equals \(\mu\) up to a constant factor. In particular, this is the case for the PK cone metrics of Theorem \ref{PRODTHM}.
\end{corollary}

\begin{proof}
	The Riemannian volume form \(\omega^n\) is a parallel section of \(K_{\C^n} \otimes \bar{K}_{\C^n}\) and so is \(\mu\) by Lemma \ref{lem:volform0}; since the bundle has rank one, these two sections must be multiples of each other.
\end{proof}

\subsection{FS metrics as quotients of regular PK cones}\label{sect:FSquotPK}

\begin{setting}\label{set:regPK}
	Let \(g\) be a regular PK cone metric on \(\C^n\) singular at a hyperplane arrangement \(\mH\), so we have coordinates \((z_1, \ldots, z_n)\) in which the members \(H\) of \(\mH\) are linear hyperplanes and such that the Euler vector field is given by
	\[e = \frac{1}{\alpha_0} \sum_{i=1}^{n} z_i \frac{\p}{\p z_i} \]
	for some \(\alpha_0>0\). 
	By Lemma \ref{lem:regPK} these coordinates are unique up to linear transformations.
	%Scalar multiplication by \(\lambda \in \C^*\) scales \(g\) by a factor of \(|\lambda|^{2\alpha_0}\), see Lemma \ref{lem:scalrmult}.
\end{setting}

In all forthcoming subsections we fix a regular PK cone metric on \(\C^n\) as described in Setting \ref{set:regPK}.

\subsubsection{{\large Complex links}}

\begin{notation}
	We denote by \((\CP^{n-1})^{\circ}\) the complement of the projectivized arrangement. So the standard projection from \(\C^n \setminus \{0\}\) to \(\CP^{n-1}\) restricts to a map
	\begin{equation}\label{hopffibration}
	p : S^{\circ} \to (\CP^{n-1})^{\circ} 	
	\end{equation}
	where \(S^{\circ}\) is the regular part of the unit sphere \(S=\{r=1\}\) of the cone.
	Scalar multiplication by complex units gives an action of the circle by isometries of \(S^{\circ}\) (see Lemma \ref{lem:scalrmult}) which makes 
	the map \eqref{hopffibration} into a (trivial if \(\mH\) is non-empty) \(S^1\)-bundle.
\end{notation}

\begin{definition}
	We define the \textbf{Fubini-Study metric} (associated to \(g\)) as the Riemannian metric \(g_{FS}\) on \((\CP^{n-1})^{\circ}\)
	obtained by taking the Riemannian quotient of \(S^{\circ}\) by the circle action given by scalar multiplication.  Equivalently, \(g_{FS}\) is defined so that the map \eqref{hopffibration} is a Riemannian submersion.
\end{definition}

\begin{notation}
	We denote the round unit sphere in \(\C^n\) by \(S^{2n-1}(1)\).
	The Fubini-Study metric \(g_{\CP^{n-1}}\) on \(\CP^{n-1}\) is defined by requiring the Hopf map \(S^{2n-1}(1) \to \CP^{n-1}\) to be a Riemannian submersion. The metric \(g_{\CP^{n-1}}\) is K\"ahler and its sectional curvature ranges between \(1 \leq \mbox{sec} \leq 4\), with sec \(\equiv 4\) for \(2\)-planes invariant under the complex structure.
\end{notation}

\begin{lemma}\label{lem:hopfequivariant}
	\(S^{\circ}\) is locally isometric to \(S^{2n+1}(1)\) via \(S^1\)-equivariant local isometries where \(S^1\) acts by scalar multiplication on both \(S^{\circ}\) and \(S^{2n-1}(1)\).
\end{lemma}

\begin{proof}
	This is a direct consequence of Lemma \ref{lem:deves}.
\end{proof}

\begin{remark}
	The length of a \(S^1\)-orbit is usually not equal to \(2\pi\), in particular
	the local isometries in Lemma \ref{lem:hopfequivariant} are \emph{not} defined on tubular neighbourhoods of such orbits.
\end{remark}

\begin{lemma}
	The fibres of the map \eqref{hopffibration} are geodesic circles of length  \(2\pi \alpha_0\).
\end{lemma}

\begin{proof}
	This follows from Lemmas \ref{hopffibration} and \ref{lem:linethrough0}.
\end{proof}

\begin{lemma}\label{lem:locCP}
	\(g_{FS}\) is locally holomorphically isometric to the standard Fubini-Study metric \(g_{\CP^{n-1}}\). In particular, \(g_{FS}\) is a K\"ahler metric on \((\CP^{n-1})^{\circ}\). 
\end{lemma}

\begin{proof}
	This follows from Lemma \ref{lem:hopfequivariant} and the definition of \(g_{FS}\). More explicitly, we can write local holomorphic isometries in terms of the developing map of \(((\C^n)^{\circ}, g)\) as follows.
	
	Let \(U \subset (\CP^{n-1})^{\circ}\) be a simply connected open set and let \(\tilde{U} = \pi^{-1}(U) \subset (\C^n)^{\circ}\) where \(\pi\) is the standard projection from \(\C^n \setminus \{0\}\) to \(\CP^{n-1}\). Consider the developing map for the flat metric \(g\) restricted to \(\tilde{U}\),
	\begin{equation}\label{eq:dev}
	\dev : \tilde{U} \to \C^n
	\end{equation}
	and write \(\tz= \dev(z)\). The map \eqref{eq:dev} is multivalued, if we go around a circle through the origin \(e^{2\pi i t}z\) with \(0\leq t\leq 1\) then \(\dev\) changes from \(\tz\) at \(t=0\) to \(e^{-2\pi i\alpha_0}\tz\). Since \(\pi_1(\tilde{U}) = \Z\) is generated by such loops,
	we have a well defined holomorphic map 
	\begin{equation}\label{eq:Phi}
	\Phi = \pi \circ \dev \circ \pi^{-1} : U \to \CP^{n-1} . 
	\end{equation}
	Because \(\dev\) is a local holomorphic isometry equivariant with respect to complex dilations of the cone \((\tilde{U},g)\) and  standard scalar multiplication of Euclidean space \(\C^n\), we have
	\(\Phi^*(g_{\CP^{n-1}}) = g_{FS}\).
\end{proof}

\subsubsection{{\large FS K\"ahler form}}

\begin{notation}
	We write \(J\) and \(I\) for the complex structures of \(\C^n\) and \(\CP^{n-1}\). We denote by \(\pi\) the standard projection from \(\C^n \setminus \{0\}\) to \(\CP^{n-1}\).
\end{notation}

\begin{definition}
	On \((\CP^{n-1})^{\circ}\) we define the \textbf{Fubini-Study form}
	\(\omega_{FS}\) as the K\"ahler form of \(g_{FS}\), that is 
	\[\omega_{FS} = g_{FS}(I\cdot, \cdot) . \]
	It is a closed positive \((1,1)\)-form on \((\CP^{n-1})^{\circ}\). 
\end{definition}

\begin{notation}
	Recall that \(g_{\CP^{n-1}}\) is the standard Fubini-Study metric on \(\CP^{n-1}\). If we write \(\omega_{\CP^{n-1}}\) for its K\"ahler form then
	\begin{equation}\label{eq:pullbackCP}
	\pi^*(\omega_{\CP^{n-1}}) = \frac{i}{2} \dd \log |z|^2 .
	\end{equation}
\end{notation}

\begin{notation}
	The K\"ahler \(\omega= g(J \cdot, \cdot)\) of the PK cone is given by
	\[\omega = \frac{i}{2} \dd r^2 \,\ \mbox{ on } (\C^n)^{\circ} , \]
	see Lemma \ref{lem:kahlerform}.
	The function \(r^2\) is smooth on the complement of the arrangement and, since our cone is regular, we can write it as
	\begin{equation}\label{eq:r2}
	r^2 = e^{\varphi} |z|^{2\alpha_0} 
	\end{equation}
	where \(\varphi\) is invariant under scalar multiplication: \(\varphi(\lambda z)=\varphi(z)\) for all \(\lambda \in \C^*\); see Lemma \ref{lem:homogeneityr2}.
\end{notation}

\begin{lemma}
	The identity
	\begin{equation}\label{eq:pullbackomegaFS}
	\pi^* (\omega_{FS}) = \frac{i}{2} \dd \log r^2	
	\end{equation}
	holds on \((\C^n)^{\circ}\).
\end{lemma}

\begin{proof}
	Let \(U \subset (\CP^{n-1})\) be simply connected and let \(\Phi\) be a holomorphic local isometry given by Equation \eqref{eq:Phi}, so \(\omega_{FS} = \Phi^* \omega_{\CP^{n-1}}\) and therefore
	\[\pi^* \omega_{FS} = \dev^* (\pi^* \omega_{\CP^{n-1}}) . \]
	Since \(\dev^*\) commutes with \(\dd\) (because \(\dev\) is holomorphic), Equation \eqref{eq:pullbackomegaFS} follows from Equation \eqref{eq:pullbackCP} together with \(\dev^*|z| = r\).
\end{proof}

\subsubsection{{\large Global potential}}

\begin{notation}
	Without loss of generality we can assume that \(H=\{z_1=0\}\) belongs to the arrangement \(\mH\). We have complex coordinates \((x_1, \ldots, x_n)\) on \(\CP^{n-1} \setminus \{z_1=0\} \cong \C^{n-1}\) given by
	\[x_1=\frac{z_2}{z_1}, \ldots, x_{n-1} = \frac{z_n}{z_1} . \]
\end{notation}

\begin{definition}
	We define a smooth function \(\psi\) on \((\CP^{n-1})^{\circ} \subset \C^{n-1}\) given by 
	\begin{equation}\label{eq:potentialFS}
	\psi(x) = \log \left(1 + |x_1|^2 + \ldots + |x_{n-1}|^2 \right)^{\alpha_0} + \varphi(1, x_1, \ldots, x_{n-1}) .
	\end{equation}
	with \(\varphi\) given by Equation \eqref{eq:r2}. 
\end{definition}

\begin{lemma}
	The identity
	\begin{equation}\label{omegaFS}
	\omega_{FS} = \frac{i}{2} \dd \psi .	
	\end{equation}
	holds on \((\CP^{n-1})^{\circ}\).
\end{lemma}

\begin{proof}
	Since \(\pi\) is a submersion, it is enough to verify that the pull-back of \((i/2)\dd \psi\) satisfies Equation \eqref{eq:pullbackomegaFS}. Equivalently, we claim that \(\pi^* \dd \psi =  \dd \log r^2 \). Recall that \(r^2=|z|^{2\alpha_0}e^{\varphi}\), so
	\[\dd \log r^2 =  \dd \log \left(|z_1|^2 + \ldots + |z_n|^2\right)^{\alpha_0} + \dd \varphi . \]
	On the other hand, since \(\dd \log |z_1|^{2\alpha_0} \equiv 0\) on \(\C^n \setminus \{z_1=0\}\), we obtain
	\begin{equation*}
	\dd \log r^2 = \dd \log \left(1 + |z_2/z_1|^2 + \ldots + |z_n/z_1|^2\right)^{\alpha_0} + \dd \varphi .  
	\end{equation*}
	Moreover, since \(\varphi\) is \(\C^*\)-invariant,  \(\varphi(z_1, \ldots, z_n) = \varphi(1, z_2/z_1, \ldots, z_n/z_1)\) and we conclude that
	\begin{equation*}
	\dd \log r^2 =  \dd (\pi^* \psi) = \pi^* \dd \psi 
	\end{equation*} 
	as we wanted to show.
\end{proof}

\subsubsection{{\large Angular \(1\)-form}}

%In Section \ref{sect:kahlerpotential} we have shown that \(r^2\)  extends as a continuous plurisubharmonic function on \(\C^n\) (taking stricly positive values outside the origin). In particular, \(\omega\) extends as a closed positive current.

\begin{notation}[Action of the complex structure on \(1\)-forms]
	If \(\gamma\) is a \(1\)-form, we denote by \(J\gamma\) the \(1\)-form given by 
	\[J\gamma(X)=\gamma(JX) .\]
\end{notation}

\begin{definition}
	On \((\C^n)^{\circ}\) we introduce the \textbf{angular \(1\)-form} \(\eta\) given by 
	\begin{equation}
	\eta = \frac{Jdr}{r} .
	\end{equation}
	Equivalently, \(\eta = J d \log r\).
\end{definition}

\begin{remark}
	Note that \(d\eta = (-2) \pi^*(\omega_{FS})\), as follows from the identity \(dJd=-2i\dd\) together with Equation \eqref{eq:pullbackomegaFS}.
\end{remark}

\begin{lemma}
	The identity
	\begin{equation}
	\eta (\cdot) = r^{-2} g(e_s, \cdot)
	\end{equation}
	holds. In particular,  \(\eta\) is the metric dual of the Reeb vector field when restricted to \(S^{\circ}\) and
	\begin{equation}
	g_{S^{\circ}} = \eta^2 + p^* g_{FS} 
	\end{equation}
	where \(p\) is the projection map \eqref{hopffibration}.
\end{lemma}

\begin{proof}
	Same as before, by taking local branches of the developing map it is enough to check these identities for the standard Fubini-Study metric.
\end{proof}

\begin{remark}
	We can interpret \(r^{2/\alpha_0}\) as a singular (continuous) Hermitian metric on \(\mO_{\CP^{n-1}}(-1)\) whose curvature form is \((-2/\alpha_0)\omega_{FS}\). From this point of view, \(\alpha_{0}^{-1}\eta\) is a connection \(1\)-form for the \(S^1\)-bundle \eqref{hopffibration}. 
\end{remark}

\subsection{Local models for singularities of FS metrics}\label{sect:locmodsingFS}

Let \(C\) be a PK cone and let \(\rho\) be the distance to its vertex, so the K\"ahler form on the regular part is \(\omega_C = (i/2) \dd \rho^2\). We consider local models for singular FS metrics whose K\"ahler form on the regular part is given by
\begin{equation}\label{eq:localmodel}
\omega_{C,1} = \frac{i}{2} \dd \log (1 + \rho^2) .	
\end{equation}
Since \(\log(1+\rho^2) \approx \rho^2\) when \(\rho^2 \ll 1\) we can expect that \(\omega_{C, 1} \approx \omega_C\) close to the vertex, informally speaking. Indeed, it is not hard to show that \(\omega_C\) is the `tangent cone' of \(\omega_{C,1}\) at the vertex.

\begin{remark}
	More geometrically, the K\"ahler metric defined by Equation \eqref{eq:localmodel} is equal to the pull-back of the standard Fubini-Study metric on \(\C^{\dim C} \subset \CP^{\dim C}\) under the developing map of the cone \(C\).
\end{remark}

\begin{remark}
	Similarly to Equation \eqref{eq:localmodel}, if we set
	\[\omega_{C,-1} = -\frac{i}{2} \dd \log (1-\rho^2) \]
	we obtain a singular complex hyperbolic metric asymptotic to \(\omega_{C}\) at \(\rho =0\).
\end{remark}

In this section we describe our Fubini-Study metric \(g_{FS}\) in terms of local models given by Equation \eqref{eq:localmodel}. In order to do that we introduce some notation.

\begin{notation}\label{not:tangentcone}
	Let \(L\) be a non-zero intersection in \(\mL(\mH)\) and let \(x \in L^{\circ} \cap S\). We recall that \(S=\{r=1\}\) and \(r\) denotes the distance to the vertex of \(g\) located at \(0\);  therefore \(x\) is at distance \(1\) from the origin with respect to \(g\).  Close to \(x\), the cone \(g\) is isometric  to the product
	\begin{equation} \label{eq:tangentcone}
	\C^d \times C^{\perp}	
	\end{equation}
	where \(d= \dim L \geq 1\) and  \(C^{\perp}\) is a PK cone metric on \(\C^{n-d}\). In other words, \eqref{eq:tangentcone} is the \emph{tangent cone} of \(g\) at the point \(x\).
	
	We take complex coordinates \((y_1, \ldots, y_n)\) centred at \(x\) and adapted to the product decomposition \eqref{eq:tangentcone} in the following sense.
	First, we require that
	\[L=\{y_1=\ldots=y_{n-d}=0\} .\]
	Secondly, we ask the K\"ahler potential \(r^2_{\perp}\) of \(C^{\perp}\) to depend only on the variables \(y_1, \ldots, y_{n-d}\). 
	
	In these \(y\)-coordinates, the squared distance to \(x\) is equal to
	\[|\bar{y}|^2 + r_{\perp}^2 \]
	where \(\bar{y} = (y_{n-d+1}, \ldots, y_n)\), and it provides a local potential for the PK cone metric \(g\) close to \(x\).
\end{notation}

\begin{lemma}\label{lem:ycoord}
	After a linear change of coordinates in the last \(d\) variables \(\bar{y} = (y_{n-d+1}, \ldots, y_n) \in \C^d\) that parametrize \(L\)  we have
	\begin{equation} \label{eq:locr2}
	r^2 = | \bar{y} + (0, \ldots, 0, 1)|^2 + r_{\perp}^2  
	\end{equation}
	close to \(x\), where \(r\) is the distance to the vertex of \(g\).  
\end{lemma}

\begin{proof}
	We can perform a linear transformation on \(\bar{y}\) so that the Euler vector field \(e\) of \(g\) equals \(\p_{y_n}\) at \(x\). The vector field \(\p_{y_n}\) is parallel and \(e_x=e-\p_{y_n}\) is the Euler vector field of the tangent cone \eqref{eq:tangentcone} in \(y\)-coordinates centred at \(x\). We have an orthogonal sum
	\[e_x= e_{\C^d} + e_{\perp} \]
	where \(e_{\C^d}\) is the standard Euler field on \(\C^d\) and \(e_{\perp}\) is the Euler field of the PK cone metric on \(\C^{n-d}\). By orthogonality, we have
	\[|e|^2 = |e_x + \p_{y_n}|^2 = |\bar{y} + (0, \ldots, 0, 1)|^2 + |e_{\perp}|^2   . \]
	Equation \eqref{eq:locr2} follows because \(|e|^2=r^2\) and \(|e_{\perp}|^2=r_{\perp}^2\).
\end{proof}

\begin{lemma}\label{lem:FSloc}
	Let \(\bar{x} = \pi(x) \in \CP^{n-1}\) with \(x \in L^{\circ} \cap S\) as in Notation \ref{not:tangentcone}. Take complex coordinates \((y_1, \ldots, y_n)\) centred at \(x\) as in Lemma \ref{lem:ycoord} and define complex coordinates \((y_1, \ldots, y_{n-1})\) in neighbourhood of \(\pi(x)\) by letting
	\begin{equation*}
	(y_1, \ldots, y_{n-1}) \mapsto \pi(y_1, \ldots, y_{n-1},0) .
	\end{equation*}
	Then
	\begin{equation}\label{eq:locpotFS}
	\omega_{FS} = \frac{i}{2} \dd \log (1 + |\hat{y}|^2 + r_{\perp}^2 )
	\end{equation}
	where \(\hat{y} = (y_{n-d+1}, \ldots, y_{n-1})\).
\end{lemma}

\begin{proof}
	The map \(F(y_1, \ldots, y_{n-1}) = (y_1, \ldots, y_{n-1},0)\) defines a local holomorphic section of the projection \(\pi:\C^n \setminus \{0\} \to \CP^{n-1}\), in the sense that \(\pi \circ F = \Id\); therefore 
	\[\omega_{FS} = \frac{i}{2}  F^* \dd \log r^2 = \frac{i}{2} \dd \log (r^2 \circ F) . \]
	Equation \eqref{eq:locpotFS} then follows from Equation \eqref{eq:locr2} together with the identity
	\[| (\bar{y} \circ F) + (0, \ldots, 0, 1)|^2 = |\hat{y}|^2 + 1 . \qedhere \]
\end{proof}

\begin{notation}\label{not:omegaC}
	In the same setting as in Lemma \ref{lem:FSloc} we write
	\begin{equation}\label{eq:rho2}
	\rho^2 = |\hat{y}|^2 + r_{\perp}^2 .
	\end{equation}
	Note that \(\rho^2\) is the K\"ahler potential for the PK cone
	\begin{equation}\label{eq:C}
	C= \C^{d-1} \times C^{\perp} .	
	\end{equation}
	On the other hand, we can re-write Equation \eqref{eq:locpotFS} in the form
	\begin{equation}\label{eq:localmodel2}
	\omega_{FS} = \omega_{C, 1} 	
	\end{equation}
	where \(\omega_{C, 1}\) is given by Equation \eqref{eq:localmodel}.
\end{notation}

\begin{lemma}\label{lem:quasiisom}
	With the same notation as in \ref{not:omegaC}, there is \(M>1\) such that
	\[M^{-1} \omega_C < \omega_{FS} < M \omega_C \]
	close to \(\bar{x}\).
\end{lemma}

\begin{proof}
	We set \(F(t)=\log(1+t)\) and \(\psi = \rho^2\) in the following identity
	\[i \dd F(\psi) = F''(\psi) i \p \psi \wedge \bar{\p} \psi + F'(\psi) i \dd \psi . \]
	Note that \(F'(0)=F''(0)=1\), so we can assume that both \(F'(\rho^2)\) and \(F''(\rho^2)\) belong to the interval \([1/2, 2]\) provided we are close to \(\bar{x}\) so that \(\rho^2\) is sufficiently small. On the other hand
	\[0 \leq i \p \psi \wedge \bar{\p} \psi \leq c(n) |d\psi|_{\omega_C}^2 \omega_C \]
	where \(c(n)>0\) is a dimensional constant; and \(|d\psi|_{\omega_C} = 2\rho \). The lemma follows from these observations.
\end{proof}

\begin{lemma}
	The metric completion of \(g_{FS}\) is \(\CP^{n-1}\).
\end{lemma}

\begin{proof}
	This is a direct consequence of Lemma \ref{lem:quasiisom} and the fact that the metric completion of \(\omega_{C}\) is \(\C^{n-1}\).
\end{proof}

\subsection{Volume forms of FS metrics}\label{sect:volformFS}

Let \(\bar{x} \in \CP^{n-1}\) be a point the projectivized arrangement \(\cup_{H \in \mH} \pi(H)\) as above. Close to \(\bar{x}\) we can write \(\omega_{FS}\) as in Equation \eqref{eq:localmodel}.

\begin{lemma}\label{lem:FSvolumeform}
	Close to \(\bar{x}\) the volume form of \(\omega_{FS}\) is given by
	\begin{equation}\label{FSvolform}
	\omega_{FS}^{n-1} = \frac{1}{(1+\rho^2)^{n}} \omega_C^{n-1} .
	\end{equation}
\end{lemma}

\begin{proof}
	Expanding the expression  \(\omega_{FS} = (i/2) \dd \log (1+ \rho^2)\) and using that \(\omega_C = (i/2) \dd \rho^2\), we obtain
	\begin{equation*}
	\omega_{FS}= - \frac{1}{2(1+\rho^2)^2} i \p \rho^2 \wedge \bar{\p} \rho^2 + \frac{1}{1+\rho^2} \omega_C .
	\end{equation*}
	It follows that
	\[ \omega_{FS}^{n-1} = \frac{1}{(1+\rho^2)^{n-1}} \omega_C^{n-1} - \frac{n-1}{2(1+\rho^2)^n} i \p \rho^2 \wedge \bar{\p} \rho^2 \wedge \omega_C^{n-2} .\]
	Using the identity \eqref{identitydf} we obtain
	\[\omega_{FS}^{n-1} = \left( \frac{1}{(1+\rho^2)^{n-1}} - \frac{\rho^2}{(1+\rho^2)^n} \right) \omega_{C}^{n-1} \]
	and Equation \eqref{FSvolform} follows.
\end{proof}

\begin{lemma}
	Let \(\omega\) be a K\"ahler form on an open set \(U \subset \C^n\) and let \(f\) be a smooth real valued function on \(U\), then
	\begin{equation}\label{identitydf}
	i \p f \wedge \bar{\p} f \wedge \omega^{n-1} = \frac{1}{2n} |df|^2 \omega^n 
	\end{equation}
	where \(|df|\) denotes the norm of \(df\) as measured with respect to \(\omega\).
\end{lemma}

\begin{proof}
	This can be proved by a straightforward computation, choosing local coordinates such that \(\omega = \sum_k i dz_k \wedge d \bar{z}_k\) at a given point, see \cite[Lemma 4.7]{Szek}.
\end{proof}

\begin{corollary}\label{cor:FSL1}
	Let \(y=(y_1, \ldots, y_{n-1})\) be complex coordinates centred at \(\bar{x}\) and write \(\Phi_y = dy_1 \wedge \ldots \wedge dy_{n-1}\), then on the regular part we have
	\[\omega_{FS}^{n-1} = \mu \Phi_y \wedge \bar{\Phi}_y \]
	with \(\mu \in L^1_{loc}\) locally integrable on the whole neighbourhood with respect to ordinary Lebesgue measure. 
\end{corollary}

\begin{proof}
	This is a consequence of Lemma \ref{lem:FSvolumeform} to together with Corollary \ref{lem:volumeform} applied to \(\omega_C\) and Lemma \ref{lem:positivityvolume} for the local integrability statement.
\end{proof}

\begin{remark}\label{rmk:volume}
	It follows from Corollary \ref{cor:FSL1} that we can extend the measure \(\omega_{FS}^{n-1}\) by zero over the projectivized arrangement to obtain an absolutely continuous measure (with respect to the ordinary Lebesgue measure) on \(\CP^{n-1}\). The total volume
	\(\Vol(\omega_{FS})\) of \(g_{FS}\) with respect to this measure is equal to the integral of its Riemannian volume form over its regular part. On the other hand, with the above definition, the quantity \(\Vol(\omega_{FS})\) is equal to the volume of the unit sphere of the PK cone divided by \(2\pi\alpha_0\) and it is clearly a finite number. We note that,
	since \(\CP^{n-1}\) is compact, the statement that \(\Vol(\omega_{FS})< \infty\) also follows from Corollary \ref{cor:FSL1}. 
\end{remark}

\begin{note}
	In coordinates \((y_1, \ldots, y_{n-1})\) centred at \(\bar{x}\) as in Notation \ref{not:omegaC}, it follows from Lemmas \ref{lem:FSvolumeform} and \ref{lem:volumeform} that 
	\begin{equation}
	\omega_{FS}^{n-1} = \exp(-2n\varphi) \left(\prod_{H \in \mH_L} |\ell_H|^{2\alpha_H - s} \right) c(n) \Phi_y \wedge \bar{\Phi}_y
	\end{equation}
	where \(c(n)\) is a dimensional constant, \(\ell_H\) are linear equations in \((y_1, \ldots, y_{n-1})\) for the hyperplanes going through \(\bar{x}\) and \(\varphi = (1/2) \log(1+\rho^2) \) so that \(\omega_{FS}=i \dd \log \varphi\). In particular, the standard formula for the Ricci curvature of a K\"ahler metric \(\mbox{Ric}(\omega) = - i \dd \log \det(\omega)\) gives us \(\mbox{Ric}(\omega_{FS}) = 2n \omega_{FS}\). This is consistent with our normalization for the sectional curvature \(1 \leq \mbox{sec}(g_{FS}) \leq 4\); indeed if
	\[v_1, Iv_1, \ldots, v_{n-1}, Iv_{n-1} \]
	is an orthonormal basis of \(T_p \CP^{n-1}\) then
	\[ \mbox{sec}(v_1, Iv_1) =4, \,\ \mbox{sec}(v_1, v_j) = \mbox{sec}(v_1, Iv_j) = 1 \,\ \mbox{ for } j >1 \]
	and the Ricci curvature equals \(4 + (2n-4) = 2n\).
\end{note}

\subsection{Total volume of links}\label{sect:totalvol}

\begin{note}
	Let us first recall the volumes of the standard complex projective space and the unit sphere. With our normalization \(1 \leq \mbox{sec} \leq 4\) on the sectional curvature of \(g_{\CP^{n-1}}\),	
	its K\"ahler form  satisfies 
	\[\pi^* \omega_{\CP^{n-1}} = (i/2) \dd \log |z|^2 . \] 
	From Chern-Weil it follows that
	\[ \frac{1}{2\pi} [\omega_{\CP^{n-1}}] \in \frac{1}{2} c_1 \left(\mO_{\CP^{n-1}}(1)\right) . \]
	From this, we obtain that the volume of \(g_{\CP^{n-1}}\) is 
	\[\Vol(\CP^{n-1}) = \frac{\pi^{n-1}}{(n-1)!} \]
	and, since the Hopf map has fibres of length \(2\pi\), 
	\[\Vol(S^{2n-1}(1)) = 2 \frac{\pi^n}{(n-1)!} . \]
\end{note}

The next general result from pluri-potential theory will allow us to compute the volume of our singular Fubini-Study metrics.

\begin{proposition}\label{prop:henri}
	Let \(X\) be a compact complex manifold of \(\dim_{\C} X =n\) and let \(Y \subset X\) a complex hypersurface. Suppose that \(\omega\) is a smooth K\"ahler form on \(X\) and let \(u\) be a smooth real function defined on the complement of \(Y\) such that
	\begin{itemize}
		\item[(i)] \(\omega + i \dd u >0\) on \(X \setminus Y\);
		\item[(ii)] \(\sup_{X \setminus Y}|u| < \infty\).
	\end{itemize}
	Then
	\begin{equation}
	\int_{X \setminus Y} (\omega+ i \dd u)^n = \int_X \omega^n .
	\end{equation}
\end{proposition}

\begin{proof}
	This is a consequence of Bedford-Taylor theory together with the fact that bounded psh functions have zero Lelong numbers, see \cite[Chapters 2 and 3]{GZ} and \cite[Section 7.1]{KozNg}.
\end{proof}

We recall the definition of volume of \(g_{FS}\) (see Remark \ref{rmk:volume}) as the integral of its Riemannian volume form over its regular part. With this in mind we have the following.
\begin{proposition}\label{prop:totalvolume}
	The volume of the complex and real links of the PK cone \((\C^n, g)\) are given by
	\begin{equation}
	\Vol(g_{FS}) = \alpha_{0}^{n-1} \Vol(\CP^{n-1}) \hspace{2mm} \mbox{ and } \hspace{2mm} \Vol(S) = \alpha_0^n \Vol(S^{2n-1}(1)) .
	\end{equation}
\end{proposition}

\begin{proof}
	Since \(\Vol(S)=2\pi \alpha_0 \Vol(g_{FS})\), it is enough to show the formula for \(g_{FS}\). Let
	\[Y = \cup_{H \in \mH} \pi(H) \]
	be the union of the projectivized hyperplanes and write \(X=\CP^{n-1}\). By Corollary \ref{cor:FSL1} we know that
	\[\Vol(g_{FS}) = \frac{1}{(n-1)!} \int_{X \setminus Y} \omega_{FS}^{n-1} < \infty . \]
	On the other hand, it follows from Equations \eqref{eq:potentialFS} and \eqref{omegaFS} that
	\[\omega_{FS} = \alpha_0 \omega_{\CP^{n-1}} + i \dd u \]
	where \(u=\varphi(1, x_1, \ldots, x_{n-1})\) (in the notation of Equation \eqref{eq:potentialFS}) is a smooth function on \(X \setminus Y=(\CP^{n-1})^{\circ}\). Moreover, \(u\) extends continuously to the whole \(X\) because \(\varphi\) is the pull-back of a continuous function on \(\CP^{n-1}\). It follows from Proposition \ref{prop:henri} that
	
	\begin{equation}\label{eq:volumepf}
	\Vol(g_{FS}) = \frac{1}{(n-1)!} \int_X (\alpha_0 \omega_{\CP^{n-1}})^{n-1} = \alpha_0^{n-1} \Vol(\CP^{n-1}) . 	
	\end{equation}
	
	Alternatively, instead of appealing to Proposition \ref{prop:henri}, we can argue as follows.  The potential \((1/\alpha_0)\psi\) where \(\psi\) is given by Equation \eqref{eq:potentialFS} belongs to the Lelong class \(\mL^{+}(\C^{n-1})\) (as defined in \cite[pg. 95]{GZ}) and the statement in \cite[Proposition 3.34]{GZ} implies Equation \eqref{eq:volumepf}.
\end{proof}

\subsection{PK cones as lifts of FS metrics}\label{sect:PKliftsFS}

In this section we introduce the notion of a FS metric in \(\CP^{n-1}\) adapted to a projective arrangement (Definition \ref{def:FS}) and show that these metrics lift to PK cones metrics on \(\C^n\) (Lemma \ref{lem:lift}).

\begin{notation}
	Let \(\bmH\) be an arrangement of projective hyperplanes \(\bH \subset \CP^{n-1}\) and write \((\CP^{n-1})^{\circ}\) for the arrangement complement. Similarly to the case of linear arrangements in a vector space, we write \(\bL^{\circ}\) for the smooth part of an intersection:
	\[\bL^{\circ} =  \bL \setminus \bigcup_{\bH \in \bmH \setminus \bmH_{\bL}} \bH \]
	where \(\bmH_{\bL}\) is the set of all \(\bH \in \bmH\) such that \(\bL \subset \bH\).	
\end{notation}

\begin{notation}
	Suppose that for each \(\bH \in \bmH\) we are given a real number \(\alpha_H>0\). We can encode this data into a (formal) real divisor \(\bar{\Delta}=\sum_{\bH \in \bmH} a_H \bH\) where \(a_H=1-\alpha_H\).	
\end{notation}

\begin{definition}\label{def:FS}
	We say that \(\omega_{FS}\) is a \textbf{FS metric} adapted to the pair \((\CP^{n-1}, \bar{\Delta})\) if it satisfies the following:
	\begin{itemize}
		\item[(i)] \(\omega_{FS}\) is a smooth K\"ahler metric on \((\CP^{n-1})^{\circ}\) locally holomorphically isometric to \(\omega_{\CP^{n-1}}\);
		\item[(ii)] for every \(\bar{x} \in \bL^{\circ}\) we can find holomorphic coordinates \((y_1, \ldots, y_{n-1})\) centred at \(\bar{x}\) such that
		\begin{equation}\label{eq:model}
		\omega_{FS} = \frac{i}{2} i \dd \log \left(1+ \hat{r}^2 + \sum_{j> \codim \bL} |y_j|^2 \right)
		\end{equation}
		where \(\hat{r}^2\) is the K\"ahler potential of a PK cone metric on \(\C^{\codim \bL}\) which depends only on the first \(y_1, \ldots, y_{\codim \bL}\) variables.
	\end{itemize}
\end{definition}

\begin{remark}\label{rmk:linkFS}
	Lemma \ref{lem:FSloc} implies that
	the complex link of a regular PK cone metric on \(\C^n\) singular at a hyperplane arrangement is an adapted FS metric as in Definition \ref{def:FS}.
\end{remark}

\begin{lemma}\label{lem:cohomologyFS}
	Let \(\omega_{FS}\) be a Fubini-Study metric adapted to the pair \((\CP^{n-1}, \bar{\Delta})\) as in Definition \ref{def:FS}. Then there is a constant \(\alpha_0>0\) and continuous function \(u: \CP^{n-1} \to \R\) smooth on \((\CP^{n-1})^{\circ}\) such that
	\begin{equation}\label{eq:ddbarlemma}
	\omega_{FS} = \alpha_0 \omega_{\CP^{n-1}} + i \dd u .	
	\end{equation}
\end{lemma}

\begin{proof}
	\(\omega_{FS}\) has continuous K\"ahler potentials and can be considered as a closed positive current, therefore there is \(\alpha_0>0\) such that \([\omega_{FS}]=\alpha_{0} [\omega_{\CP^{n-1}}]\) where the brackets \([\cdot]\) denotes cohomology of currents. There is a distribution such that \(i\dd u = \omega_{FS}-\alpha_{0} \omega_{\CP^{n-1}}\). On the other hand, by Equation \eqref{eq:model} we have continuous local solutions \(\tilde{u}\) of \(i\dd \tilde{u} = \omega_{FS}-\alpha_{0} \omega_{\CP^{n-1}} \); since \(u\) differs from \(\tilde{u}\) by a pluri-harmonic distribution (which are automatically smooth) we conclude that \(u\) is continuous.
\end{proof}

\begin{definition}
	Let \((z_1, \ldots, z_n)\) be linear complex coordinates on \(\C^n\), we define 
	\begin{equation}
	r^2 = |z|^{2\alpha_0}e^u
	\end{equation}
	where \(u\) is as in Lemma \ref{lem:cohomologyFS}, regarded as a \(\C^*\)-invariant function on \(\C^n\setminus \{0\}\).
\end{definition}

\begin{lemma}\label{lem:lift}
	The defined function \(r^2\) is a K\"ahler potential for a PK cone metric on \(\C^n\) singular at the arrangement \(\mH\) with K\"ahler form
	\[\omega = \frac{i}{2} \dd r^2 .\]
\end{lemma}

\begin{proof}
	The identity
	\[ \dd e^f = e^f \p f \wedge \bar{\p} f + e^f \dd f \]
	applied to \(f=\log r^2\) shows that
	\[g=\omega(\cdot, J \cdot) = dr^2 + r^2(\eta^2 + \pi^*g_{FS}) \]
	where \(\eta =Jd\log r \).
	In particular, \(\omega>0\) and \(g\) is a metric cone. Let \(S=\{r=1\}\) and \(S^{\circ}=S\cap (\C^n)^{\circ}\). After introducing a cut on the \(S^1\)-fibres of the projection \(S^{\circ} \to (\CP^{n-1})^{\circ}\) we can lift local isometries between \(g_{FS}\) and \(g_{\CP^{n-1}}\) to local isometries between \(S^{\circ}\) and \(S^{2n-1}(1)\). The local models \eqref{eq:model} for \(\omega_{FS}\) imply that \(g\) is locally a product of PK cones at points in \(\C^n \setminus \{0\}\). It follows that the metric completion of \(g_{S^{\circ}}\) is a polyhedral spherical metric on \(S\) and \(g\) is a PK cone metric on \(\C^n\).
\end{proof}

\begin{remark}
	Lemma \ref{lem:volumeform} tells us that the volume form of \(\omega\) as given by Lemma \ref{lem:lift} equals
	\begin{equation}\label{volform}
	\omega^n = \left( \prod_{H \in \mH} |h|^{-2a_H} \right) c(n) \Phi \wedge \bar{\Phi}	
	\end{equation}
	where \(c(n)\) is a dimensional constant and \(\Phi= dz_1 \wedge \ldots \wedge dz_n\).
	If we pull-back Equation \eqref{volform} by scalar multiplication, by homogeneity we find that \(\alpha_0\) in Equation \eqref{eq:ddbarlemma} equals
	\[\alpha_0 = 1 - \frac{1}{n} \left(\sum_{H \in \mH} a_H \right) . \]
\end{remark}

\section{Irreducible holonomy}\label{sect:irredhol}

The goal of this section is to establish Theorem \ref{thm:irredhol}. Our proof is by induction on the dimension. We consider the \emph{induced} connections at hyperplanes of the arrangement and reduce to the case of \(n=2\). In the case of \(n=2\) the result follows from the correspondence between PK cones and spherical metrics (\cite{Pan}) together with the characterization of co-axial monodromy given in \cite{Dey}. 

Section \ref{sec:triples} contains background material on hyperplane arrangements. Lemma \ref{lem:inducedirred} shows that if \(\mH\) is essential and irreducible then there is one hyperplane \(H_0\) such that the induced arrangement \(\mH^{H_0}\) is also essential and irreducible. Section \ref{sect:indcon} discusses induced connections along the lines of \cite{CHL}. In Section \ref{sect:pfirrhol} we prove Theorem \ref{thm:irredhol}.

\subsection{Deletion-restriction triples}\label{sec:triples}

Let \(\mH\) be a finite collection of linear hyperplanes in a finite dimensional complex vector space \(V\).

\begin{definition}
	Fix \(H_0 \in \mH\), the \textbf{deletion-restriction triple} \((\mH, \mH', \mH'')\) is defined as follows:
	\[\mH'= \mH - \{H_0\} \mbox{ and } \mH''= \{H \cap H_0, \,\ H \in \mH'\} .  \]
	In other words, \(\mH'\) is the arrangement obtained from \(\mH\) after deleting \(H_0\) and \(\mH''\) is the same as the induced arrangement \(\mH^{H_0}\).
\end{definition}

\begin{lemma}\label{lem:rankinducedarrang}
	\(T(\mH'')=T(\mH)\) and \(r(\mH'')=r(\mH)-1\).
\end{lemma}

\begin{proof}
	Since the centre of an arrangement is the intersection of all its hyperplanes, we have
	\(T(\mH'')=T(\mH') \cap H_0=T(\mH)\). Furthermore, the rank of \(\mH''\) is given by
	\begin{align*}
	r(\mH'') = \codim_{H_0} T(\mH'') &= \codim_V T(\mH'') -1 \\
	&= \codim_V T(\mH) -1 = r(\mH)-1 . \qedhere
	\end{align*}
\end{proof}

\begin{notation}
	We denote by \(T, T'\) the respective centres \(T(\mH)\), \(T(\mH')\).
\end{notation}

\begin{definition}
	\(H_0\) is said to be a \textbf{separator} if \(T' \not\subset H_0\), or equivalently if the rank drops after deletion \(r(\mH')<r(\mH)\).
\end{definition}

\begin{lemma}\label{lem:noseparator}
	If \(H_0\) is a separator then \(\mH =  \mH' \uplus \{H_0\}\). In particular, irreducible arrangements with at least two hyperplanes have no separators.
\end{lemma}

\begin{proof}
	Recall that the span of the arrangement \(W(\mH)\) is the  subspace of \(V^*\) spanned by all linear equations defining hyperplanes in \(\mH\).
	If \(H_0\) is a separator then, by definition, \(W(\mH') \subsetneq W(\mH)\). Therefore \(W(\mH) = W(\mH') \oplus \,\ \C \cdot h_0\) and \(\mH = \mH' \uplus \{H_0\}\). 
\end{proof}

\begin{corollary}\label{cor:pan}
	If \((V, \mH)\) is essential, irreducible and \(\dim V \geq 2\) then \(\mH'\) is essential.
\end{corollary}

\begin{proof}
	Since \(\mH\) is irreducible, by Lemma \ref{lem:noseparator} \(H_0\) is not a separator. So \(T' \subset H_0\) and \(T=T'\). Moreover, since \(\mH\) is essential, \(T'=T=\{0\}\).
\end{proof}

\begin{definition}\label{def:projectionpi}
	For \(H \in \mH'\) we set \(\pi(H)=H \cap H_0\).
\end{definition}

\begin{note}
	By definition, \(\pi : \mH' \to \mH''\) is a surjection.
\end{note}

\begin{lemma}
	If \(H_0\) is a separator then \(\pi\) is a bijection.
\end{lemma}

\begin{proof}
	Suppose \(H_1, H_2 \in \mH'\) are distinct hyperplanes such that \(\pi(H_1)=\pi(H_2)\). Then \(H_1 \cap H_2 \subset H_0\), which contradicts \(T' \not\subset H_0\).
\end{proof}

\begin{lemma}\label{lem:sep}
	If \(H_0\) is a separator then \(\mH'/T' \equiv \mH''/T\). In particular \(\mH'\) is irreducible if and only if \(\mH''\) is.
\end{lemma}

\begin{proof}
	By hypothesis \(V=T' + H_0\). If \(v \in V\) then we can write \(v=h+t\) with \(h \in H_0\) and \(t \in T'\). We define a linear map \(F: V/T' \to H_0/T\) by letting \(F[v]=[h]\). If \(v=h_i + t_i\) with \(h_i \in H_0\) and \(t_i \in T'\) for \(i=1,2\) then \((h_1-h_2)+(t_1-t_2)=0\). It follows that \(h_1-h_2\) and \(t_1-t_2\) belong to \(H_0 \cap T'=T\); hence \([h_1]=[h_2]\) and \(F\) is well defined. It is easy to see that \(F\) is a bijection and it clearly takes hyperplanes in \(\mH'/T'\) to hyperplanes in \(\mH''/T\).
\end{proof}

\begin{corollary} \label{cor:separator}
	If \(H_0\) is a separator and \(\mH''\) is irreducible then \(\mH'\) is irreducible and \(\mH=\mH' \uplus \{H_0\}\).
\end{corollary}

\begin{proof}
	Follows from Lemmas \ref{lem:sep} and \ref{lem:noseparator}.
\end{proof}

\begin{figure}
	\begin{center}
		\begin{minipage}{0.48\textwidth}
			\begin{center}
				Case (i)
			\end{center}
			\scalebox{0.7}{
				\begin{tikzpicture}[tdplot_main_coords,font=\sffamily]
				\draw[fill=blue,opacity=0.2] (-3,0,-3) -- (-3,0,3) -- (3,0,3) -- (3,0,-3) -- cycle;
				\draw[fill=blue,opacity=0.2] (0,-3,-3) -- (0,-3,3) -- (0,3,3) -- (0,3,-3) -- cycle;
				\draw[fill=blue,opacity=0.2] (-3,-3,-3) -- (3,3,-3) -- (3,3,3) -- (-3,-3,3) -- cycle;
				\draw[fill=red,opacity=0.2] (-3,-3,0) -- (-3,3,0) -- (3,3,0) -- (3,-3,0) -- cycle;
				\draw[thick](-3,0,0)--(3,0,0);
				\draw[thick] (0,-3,0) -- (0,3,0);
				\draw[thick] (-3,-3,0) -- (3,3,0);
				\draw[dashed] (0,0,-3) -- (0,0,3);
				\node[scale=1.5] at (.7,3.7,1.2) {\(H_0\)};
				\end{tikzpicture}
			}
		\end{minipage}	
		\begin{minipage}{0.48\textwidth}
			\begin{center}
				Case (ii)
			\end{center}
			\scalebox{0.7}{
				\begin{tikzpicture}[tdplot_main_coords,font=\sffamily]
				\draw[fill=blue,opacity=0.2] (-3,-3,3) -- (-3,3,3) -- (3,3,-3) -- (3,-3,-3) -- cycle;
				\draw[fill=blue,opacity=0.2] (-3,-3,3) -- (3,-3,3) -- (3,3,-3) -- (-3,3,-3) -- cycle;
				\draw[fill=blue,opacity=0.2] (3,-3,3) -- (-3,-3,0) -- (-3,3,-3) -- (3,3,0) -- cycle;
				\draw[fill=red,opacity=0.2] (-3,-3,0) -- (-3,3,0) -- (3,3,0) -- (3,-3,0) -- cycle;
				\draw[thick](-3,0,0)--(3,0,0);
				\draw[thick] (0,-3,0) -- (0,3,0);
				\draw[thick] (-3,-3,0) -- (3,3,0);
				\node[scale=1.5] at (.7,3.7,1.2) {\(H_0\)};
				\end{tikzpicture}
			}
		\end{minipage}
	\end{center}
	\caption{The two cases of Proposition \ref{prop:indirredsub}, here \(L''=\{0\}\). The three planes forming \(\mH'\) intersect along a line in Case (i) and at a point in Case (ii).}
	\label{fig:casesiandii}
\end{figure}
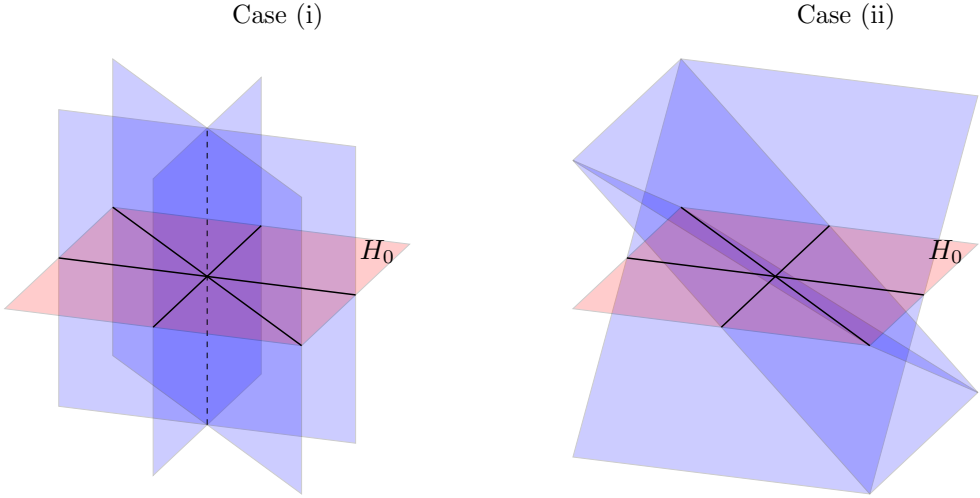

Our next result asserts that if we are given an essential arrangement \(\mH'\) that intersects a hyperplane \(H_0\) (which is not a member of \(\mH'\)) in an irreducible arrangement \(\mH''\) then adding the hyperplane \(H_0\) to the arrangement \(\mH'\) makes \(\mH' \cup \{H_0\}\) irreducible.

\begin{lemma}\label{lem:caseii}
	If \(\mH'\) is essential and \(\mH''\) is irreducible then \(\mH\) is irreducible.
\end{lemma}

\begin{proof}
	Suppose that \(\mH= \mH_1 \uplus \mH_2\) and write \(T_i = T(\mH_i)\) for the centres, so \(T_1 + T_2=V\).
	Since \(\mH\) is essential we have \(T_1 \cap T_2 = \{0\}\), so \(T_1 \oplus T_2 = V\). We want to show that either \(T_1=V\) or \(T_2=V\).  
	
	Let \(T_i''=T_i \cap H_0\). Clearly \(\mH'' = \mH''_{T_1''} \cup \mH''_{T_2''}\)
	and \(T_1'' \oplus T_2'' \subset H_0\). Since \(H_0 \in \mH\) we must have \(H_0 \in \mH_1\), say, so \(T_1'' = T_1\). On the other hand, since \(T_1+T_2=V\) and \(T_1 \subset H_0\), we must have \(T_2 \not\subset H_0\) and \(T_2'' \subsetneq T_2\). We conclude that \(\dim T''_1 = \dim T_1\) and \(\dim T_2'' = \dim T_2 -1\). Since \(\dim T_1 + \dim T_2 = \dim V\), we get that \(\dim T_1'' + \dim T_2 '' = \dim H_0\). Therefore \(T_1'' \oplus T_2'' = H_0 \) and \(\mH'' = \mH''_{T_1''} \uplus \mH''_{T_2''} \).
	
	Since \(\mH''\) is irreducible we must have either of the following two cases:
	\begin{itemize}
		\item \({T}''_2=H_0\). This implies that either \(T_2=H_0\) or \(T_2=V\). If \(T_2= H_0\) then \(\mH_2=\{H_0\}\) but this contradicts \(H_0 \in \mH_1\). We conclude that \(T_2=V\).
		
		\item \({T}''_1=H_0\). This implies that either \(T_1= H_0\) or \(T_1=V\). If \(T_1=H_0\) then \(\mH_1=\{H_0\}\) and \(\mH_2=\mH'\); hence  \(T_2=\{0\}\) (because \(\mH'\) is essential) which contradicts \(T_1+T_2=V\). We deduce that \(T_1=V\).\qedhere
	\end{itemize}
\end{proof}

The next result describes irreducible intersections of \(\mH''\), it is a special case of \cite[Lemma 2.1]{CHL}. 
\begin{proposition}\label{prop:indirredsub}
	Let \(L'' \in \mL_{irr}(\mH'')\) with \(L'' \subsetneq H_0\). Define \(L \in \mL(\mH') \) to be the common intersection of all \(H \in \mH'\) such that \(L'' \subset H\). Then the following dichotomy holds:
	\begin{itemize}
		\item[(i)] \(L'' \subsetneq L\) and the irreducible components of \(L'' \in \mL(\mH)\) are \(L\) and \(H_0\); 
		\item[(ii)] \(L=L'' \in \mL_{irr}(\mH)\).
	\end{itemize}
	In particular, \(L\) is always an irreducible intersection of \(\mH\).
\end{proposition}

\begin{note}
	By definition, \(L \cap H_0= L''\). The two cases \((i)\) and \((ii)\) are illustrated on Figure \ref{fig:casesiandii}.
\end{note}

\begin{proof}
	Case \((i)\). If \(L'' \subsetneq L\) then \(H_0\) is a separator for the arrangement \(\mH_{L''}\). Corollary \ref{cor:separator} implies two things: that \((\mH_{L''})_L = \mH_L\) (see Lemma \ref{lem:HLMisHM}) is irreducible and also that \(\mH_{L''} = \mH_L \uplus \{H_0\}\). In particular,  \(L \in \mL_{irr}(\mH)\) and Case \((i)\) follows.
	
	Case \((ii)\). We want to show that if \(L=L''\) then \(\mH_{L''}\) is irreducible; this follows from Lemma \ref{lem:caseii} applied to the arrangement \(\mH_{L}/L\).	
\end{proof}

The next result is taken from {\cite[Lemma 3.2.4]{OT2}}. For the sake of completeness we include its proof with a slightly simplified presentation.

\begin{proposition}\label{prop:OrlikTerao}
	If \(\mH'\) and \(\mH''\) are reducible then \(\mH\) is reducible.
\end{proposition}

\begin{proof}	
	We are given the following:
	\begin{itemize}
		\item  \(\mH' = \mA_1 \uplus \mA_2\) with \(\mA_1, \mA_2\) non-empty;
		\item \(\mH'' = \mH_1 \uplus \mH_2\) with \(\mH_1, \mH_2\) non-empty.
	\end{itemize}
	Let \(\pi: \mH' \to \mH''\) be as in Definition \ref{def:projectionpi} and  denote
	\[ \mB_1= \pi^{-1}(\mH_1), \qquad \mB_2 = \pi^{-1}(\mH_2) . \]
	We consider the following cases.
	\begin{description}
		\item[Case A:] either \(T(\mA_1) \subset H_0\) or \(T(\mA_2) \subset H_0\).
		\item[Case B:] either \(T(\mB_1) \not\subset H_0\) or \(T(\mB_2) \not\subset H_0\).
	\end{description}
	
	We make the following.
	\begin{claim}\label{claim1}
		If \textbf{Case A} or \textbf{Case B} holds then \(\mH\) is reducible.
	\end{claim}
	\begin{itemize}[leftmargin=*]
		\item[] If \textbf{Case A} holds then either 
		\[\mH=(\mA_1 \cup \{H_0\}) \uplus \mA_2 \qquad \mbox{or} \qquad \mH=(\mA_2 \cup \{H_0\}) \uplus \mA_1 ,\] 
		as follows from \(T(\mA_1)+T(\mA_2)=V\). 
		\item[] If \textbf{Case B} holds then either 
		\[\mH= \mB_1 \uplus (\mB_2 \cup \{H_0\}) \qquad \mbox{or} \qquad \mH= \mB_2 \uplus (\mB_1 \cup \{H_0\}),\] 
		as follows from \(T(\mB_i) \cap H_0=T(\mH_i)\) together with
		\(T(\mH_1) + T(\mH_2) =H_0\). 
	\end{itemize}
	
	We have established Claim \ref{claim1} and proceed to the next.
	
	\begin{claim}\label{claim2}
		At least one of \textbf{Case A} or \textbf{Case B} must happen.
	\end{claim}
	We prove Claim \ref{claim2} by contradiction. Assume that \(T(\mA_i) \not\subset H_0\) and \(T(\mB_i) \subset H_0\). Define
	\[C_1 = T(\mA_1 \cap \mB_1), \hspace{2mm} C_2 = T(\mA_1 \cap \mB_2), \hspace{2mm} C_3 = T(\mA_2 \cap \mB_1), \hspace{2mm} C_4 = T(\mA_2 \cap \mB_2) . \]
	Using \(T(\mA_1)+T(\mA_2)=V\) together with \(C_1 \cap C_2= T(\mA_1)\) and \(C_3 \cap C_4 = T(\mA_2)\), we obtain
	\begin{equation}\label{eq:C1}
	C_1 \cap C_2 + C_3 \cap C_4 = V .
	\end{equation}
	Using \(T(\mB_1)+T(\mB_2)=H_0\) together with \(C_1 \cap C_3= T(\mB_1)\) and \(C_2 \cap C_4 = T(\mB_2)\), we obtain 
	\begin{equation}\label{eq:C2}
	C_1 \cap C_3 + C_2 \cap C_4 = H_0 .
	\end{equation}
	
	Our next assertion is that
	\begin{equation}\label{eq:C3}
	(C_1 + C_2) \cap (C_3 + C_4) = H_0 .
	\end{equation}
	Indeed, the inclusion \(H_0 \subset (C_1 + C_2) \cap (C_3 + C_4)\) follows immediately from Equation \eqref{eq:C2}. Conversely, let \(v \in (C_1 + C_2) \cap (C_3 + C_4)\). Write
	\[v = c_1 + c_2 = c_3 + c_4\]
	with \(c_i \in C_i\). It follows from Equation \eqref{eq:C1} that we can write
	\[c_1 - c_3 = c_4 - c_2 = c_{12} + c_{34} \]
	with \(c_{12} \in C_1 \cap C_2\) and \(c_{34} \in C_3 \cap C_4\). Therefore
	\[ v = (c_1-c_{12}) + (c_{12}+c_2) \]
	with \(c_1-c_{12} = c_{34} + c_3 \in C_1 \cap C_3\) and \(c_{12} + c_2 = c_4 - c_{34}  \in C_2 \cap C_4\). Equation \eqref{eq:C2} implies \(v \in H_0\) and establishes Equation \eqref{eq:C3}.
	
	Equation \eqref{eq:C3} implies that \(H_0 \subset C_1 + C_2\). If \(C_1+C_2=V\) then Equation \eqref{eq:C3} would imply \(C_3 + C_4 = H_0\) contradicting our assumption \(T(\mA_2) = C_3 \cap C_4 \not\subset H_0\). Therefore we must have \(C_1+C_2 = H_0\). Similarly, we must have \(C_3+C_4 = H_0\). In particular, \(C_1, \ldots, C_4\) are all contained in \(H_0\) but this contradicts Equation \eqref{eq:C1}. This finishes the proof of Claim \ref{claim2} and therefore of the lemma.  		
\end{proof}

\begin{definition} \label{def:generalposition}
	A collection of \(n+1\) linear hyperplanes \(\mH=\{H_1, \ldots, H_{n+1}\}\) in a complex vector space of dimension \(n\) is said to be in general position if the intersection of any \(n\)-tuple of distinct members of \(\mH\) is the origin.
\end{definition}

\begin{lemma}\label{lem:isost}
	Any arrangement as in Definition \ref{def:generalposition} is isomorphic to the standard arrangement \(\mH_{st}\) in \(\C^n\) given by
	\[H_1=\{z_1=0\}, \, \ldots, \, H_n=\{z_n=0\}, \, H_{n+1}=\{z_1+\ldots+z_n=0\} . \]
\end{lemma}

\begin{proof}
	We can take linear coordinates \(\tilde{z}_1, \ldots, \tilde{z}_n\) on \(V\) so that \(H_i=\{\tilde{z}_i=0\}\) for \(1 \leq i \leq n\). In these coordinates \(H_{n+1} = \{\sum_{i=1}^{n} c_i \tilde{z}_i = 0\}\) for some \(c_i \in \C\). The general position condition implies \(c_i \neq 0\) for all \(i\). The coordinates \(z_i= c_i \tilde{z}_i\) do the job.
\end{proof}

\begin{lemma}\label{lem:autSt}
	The identity component of the automorphism group of \(\mH_{st}\) is \(\C^*\).
\end{lemma}

\begin{proof}
	A matrix which preserves the coordinate hyperplanes must be diagonal, if it also preserves \(\{z_1+ \ldots + z_n=0\}\) then all diagonal entries must be equal.
\end{proof}

\begin{corollary}\label{cor:starr}
	An arrangement \((V, \mH)\) with \(|\mH|= \dim V +1\) is essential and irreducible if and only if \(\mH \equiv \mH_{st} \).
\end{corollary}

\begin{proof}
	If \(\mH \equiv \mH_{st}\) then \(\mH\) is clearly essential and it is irreducible because of Lemma \ref{lem:autSt}. On the other hand, if \(\mH\) is essential and irreducible and \(|\mH|=\dim V+ 1\) then the hyperplanes must be in general position because if \(\{H_1, \ldots, H_n\}\), say, intersect along a non-zero subspace
	then we would have \(\mH = \{H_1, \ldots, H_n\} \uplus \{H_{n+1}\}\); Lemma \ref{lem:isost} implies \(\mH\equiv \mH_{st}\). 
\end{proof}

\begin{corollary} \label{cor:standardsubarrangement}
	If \((V, \mH)\) has a subset of \(\dim V+1\) hyperplanes in general position then \(\mH\) is essential and irreducible.
\end{corollary}

\begin{proof}
	Lemma \ref{lem:autSt} implies that \(\Aut_0(\mH) = \C^*\).
\end{proof}

\begin{remark}
	The converse to Corollary \ref{cor:standardsubarrangement} does not hold in general, as the following example shows.
\end{remark}

\begin{example}
	Let \(\mH\) be the arrangement of \(n+2\) hyperplanes in \(\C^n\), \(n \geq 4\), given as follows:
	\[H_1=\{z_1=0\}, \ldots, H_{n-1}=\{z_{n-1}=0\}, \qquad H_n=\{z_1+\ldots+z_{n-1}=0\}\]
	\[H_{*}=\{z_n=0\}, \qquad H_0=\{z_1+z_2+z_n=0\} .\]
	It is easy to check that \(\mH\) is essential and irreducible but there is no subset of \(n+1\) hyperplanes in general position.
\end{example}

\begin{lemma}\label{lem:inducedirred}
	If \((V, \mH)\) is essential, irreducible and \(\dim V \geq 2\) then there is \(H_0 \in \mH\) such that the induced arrangement \(\mH''\) is also essential and irreducible.
\end{lemma}

\begin{proof}
	We proceed by induction on the number of hyperplanes \(|\mH| \geq \dim V + 1\). If \(|\mH|= \dim V + 1\) then Corollary \ref{cor:starr} implies that \(\mH\) is isomorphic to the standard arrangement \(\mH_{st}\) of \(n+1\) hyperplanes in general position in \(\C^n\), for which the statement clearly holds.
	
	Suppose that \(|\mH|>\dim V + 1\) and take \(H_0 \in \mH\).
	By Corollary \ref{cor:pan} \(\mH'\) is essential.  If \(H_0\) is such that \(\mH'\) is irreducible then the inductive hypothesis applies to \(\mH'\); so there is \(H \in \mH'\) such that \((\mH')^H\) is essential irreducible and then \((\mH)^H\) is also essential irreducible (see Lemma \ref{lem:ABirred}). Therefore we might assume that \(H_0\) is such that \(\mH'\) is reducible, but in this case \(\mH''\) must be irreducible by Proposition \ref{prop:OrlikTerao}; and \(\mH''\) is automatically essential because \(\mH\) is. 
\end{proof}

\subsection{Induced connection}\label{sect:indcon}

Let \((V, \mH)\) be an arrangement equipped with a flat torsion free standard connection \(\nabla=d-\Omega\) as in Proposition \ref{prop:ftfstcon} i.e., satisfying the Basic Assumptions \ref{ass:basic}. Fix \(H_0 \in \mH\) and recall the deleted and induced arrangements
\[\mH'=\mH \setminus\{H_0\} \hspace{2mm} \mbox{ and } \hspace{2mm} \mH''=\{H \cap H_0, \,\ H \in \mH'\}.\]  
The complement of the induced arrangement \(\mH''\) equals \(H_0^{\circ} = H_0 \setminus \mH' \).

\subsubsection{{\large  The induced connection \(\nabla'\) on \(TV|_{H^{\circ}_0}\) }}

\begin{definition}
	We drop the \(A_{H_0}dh_0/h_0\) term from \(\Omega\) and define
	\begin{equation} \label{eq:omegaprime}
	\Omega' = \sum_{H \in \mH'} A_H \frac{dh}{h} ,
	\end{equation}
	a  one-form with values in \(\End V\) holomorphic on the complement of \(\mH'\). 
\end{definition}

\begin{notation}
	We denote by
	\[\imath_0 : H_0^{\circ} \to V \setminus \mH' \]
	the inclusion map into the complement of the deleted arrangement \(\mH'\). 
\end{notation}

\begin{definition}
	\(\nabla'= d - \imath_0^* \Omega'\) is the induced connection on \(TV|_{H_0^{\circ}}\).
\end{definition} 

\begin{notation}
	We denote hyperplanes in \(\mH''\) by \(H''\), so \(H'' = H \cap H_0\) for some \(H \in \mH'\). If \(h\) is a defining linear equation for \(H\) then \(h''=h|_{H_0}\) is a defining equation for \(H''\). If \(\tilde{H}=\{\tilde{h}=0\}\) is another hyperplane in \(\mH'\) with \(\tilde{H}\cap H_0=H''\) then \(\tilde{h}''= \tilde{h}|_{H_0}\) differs from \(h''\) by multiplication by a constant non-zero factor. In particular
	\[\imath_0^*  \left( \frac{dh}{h} \right) = \imath_0^*  \left( \frac{d\tilde{h}}{\tilde{h}} \right) \]
	and we denote this form unambiguously as \(dh''/h''\).
\end{notation}

\begin{remark}
	\(\nabla'\) is standard and
	Equation \eqref{eq:omegaprime} gives us its connection form: 
	\[\imath_0^*\Omega' = \sum_{H'' \in \mH''} A'_{H''} \frac{dh''}{h''} \hspace{2mm} \mbox{ with } \hspace{2mm} A'_{H''} = \sum_{\substack{H\in \mH'\\ H\cap H_0=H''}} A_H . \]
\end{remark}

Beware that \(A'_{H''} \in \End V\) and \(\nabla'\) is a connection on \(TV|_{H_0^{\circ}}\).

\begin{note}
	\(H''\) is also a codimension two intersection of \(\mL(\mH)\), as such it has a residue \(A_{H''}= \sum_{H \in \mH_{H''}} A_H\) and we see that \(A'_{H''} = A_{H''}-A_{H_0}\). In particular, since \(A_{H_0}\) vanishes on \(H_0\), the restrictions of \(A'_{H''}\) and \(A_{H''}\) to the hyperplane \(H_0\) agree:
	\begin{equation}\label{eq:restrictionAA'agree}
	(A'_{H''})|_{H_0} = (A_{H''})|_{H_0} .
	\end{equation}
\end{note}

\begin{lemma}[cf. Lemma \ref{lem:apendix1}]
	\(\nabla'\) is flat. 
\end{lemma}

\begin{proof}
	Clearly \(d (\imath_0^*\Omega')=0\), we will show that \(\imath_0^* \Omega' \wedge \imath_0^* \Omega' =0\).
	Let \(x \in H_0^{\circ}\) and take holomorphic coordinates \((z_1, \ldots, z_n)\) centred at \(x\) with \(z_1=h_0\). In these coordinates, close to \(x\) we have
	\[\Omega = A_1 \frac{dz_1}{z_1} + A_2 dz_2 + \ldots + A_n dz_n \]
	where \(A_1, \ldots, A_n\) are holomorphic functions with values in \(\End V\). Moreover, we have:
	\begin{itemize}
		\item[(i)] \(A_1 = A_{H_0} + \tilde{A}_1\) with \(\tilde{A}_1\) vanishing along \(H_0\);
		\item[(ii)] \(\Omega' = \tilde{A}_1 dz_1 + A_2 dz_2 + \ldots +A_n dz_n\). 
	\end{itemize}
	Define \(\tilde{\Omega}= A_2 dz_2 + \ldots + A_n dz_n\). It follows from (i) and (ii) that 
	\[\imath_0^*\Omega' = \imath_0^* \tilde{\Omega} = A_2(0, z_2, \ldots, z_n) dz_2 + \ldots + A_n(0, z_2, \ldots, z_n) dz_n . \] 
	On the other hand, we have
	\begin{equation}
	\Omega \wedge \Omega = \left( \sum_{j=2}^{n} [A_1, A_j] \frac{1}{z_1} dz_1 \wedge dz_j \right) + \tilde{\Omega} \wedge \tilde{\Omega}
	\end{equation}
	and \(\tilde{\Omega} \wedge \tilde{\Omega}\) only contains \(dz_i \wedge dz_j\) terms with \(i, j \geq 2\). The equation \(\Omega \wedge \Omega = 0\) implies \(\tilde{\Omega} \wedge \tilde{ \Omega} =0\), hence \(\imath_0^* \Omega' \wedge \imath_0^* \Omega' = \imath_0^* \tilde{\Omega} \wedge \imath_0^*\tilde{\Omega} = 0\) as we wanted to show.
\end{proof}

\begin{note}
	Alternatively, in order to show that \(\nabla'\) is flat, one can check the commutation equations for the residues \(A'_{H''}\), this involves analysing codimension three intersections of \(\mL(\mH)\).
\end{note}

\subsubsection{{\large The induced connection \(\nabla''\) on \(TH_0^{\circ}\)}}

\begin{lemma}\label{lem:indcon0}
	If \(A_{H_0}\) is non-zero then
	\(TH_0^{\circ}\) is a parallel subbundle of \((TV|_{H_0^{\circ}}, \nabla')\).
\end{lemma}

\begin{proof}
	Note that \(TH_0\) is  a parallel subbundle of \((TV|_{H_0}, d)\) where \(d\) is the restricted Euclidean connection. In order to prove the lemma it is sufficient to check that the residues \(A'_{H''} \in \End V\) preserve the subspace \(H_0\).
	Because of Equation \eqref{eq:restrictionAA'agree}, it is enough to show that \(A_{H''}\) preserves \(H_0\). 
	
	By the flatness condition \textbf{(F)} in the Basic Assumptions \ref{ass:basic} we know that \(A_{H''}\) commutes with \(A_{H_0}\). In particular,
	\(A_{H''}\) preserves \(\ker A_{H_0}\).
	On the other hand, the torsion free condition \textbf{(T)} in the Basic Assumptions \ref{ass:basic} implies that \(H_0 \subset \ker A_{H_0}\). Since we are assuming that \(A_{H_0}\) is non-zero, we must have \(H_0=\ker A_{H_0}\) and the lemma follows.
\end{proof}

\begin{definition}\label{def:inducedconnection}
	\(\nabla''=\nabla'|_{TH_0^{\circ}}\) is the induced connection on \(TH_0^{\circ}\).
\end{definition}

\begin{notation}
	We write \(\nabla'' = d - \Omega''\) where \(\Omega''\) is the restriction of \(\imath_0^*\Omega'\) to \(TH_0\), that is
	\[\Omega'' = \sum_{H \in \mH'} (A_H)| _{H_0} \imath_0^* \left(\frac{dh}{h}\right) . \]	
\end{notation}

\begin{remark}\label{rmk:inducedresidues}
	\(\nabla''\) is standard and its connection form can be written as
	\[\Omega''  =  \sum_{H'' \in \mH''} A''_{H''} \frac{dh''}{h''} \hspace{2mm} \mbox{ with } \hspace{2mm} A''_{H''}= (A_{H''})|_{H_0} .  \] 
\end{remark}

\begin{figure}
	\begin{center}
		\scalebox{1}{
			\begin{tikzpicture}[tdplot_main_coords,font=\sffamily]
			\draw[fill=blue,opacity=0.2] (-3,0,-3) -- (-3,0,3) -- (3,0,3) -- (3,0,-3) -- cycle;
			\draw[fill=blue,opacity=0.1] (-3,-3,0) -- (-3,3,0) -- (3,3,0) -- (3,-3,0) -- cycle;
			\draw[thick](-3,0,0)--(3,0,0);
			\node[anchor=south west,align=center] (line) at (3,3,3) {\(H''\)};
			\draw[-latex] (line) to[out=180,in=75] (0,0,0.05);
			\node at (2,3,1) {\(H_0\)};
			\node at (0,0,2.3) {\(H\)};
			
			\node[scale=.9] at (-4,-5,-3.5) {\(A_{H''} = A_H + A_{H_0}\)};
			\node[scale=.9] at (-4,-5.5,-4) {\(A'_{H''} = A_H\)};
			\node[scale=.9] at (-4,-5.15,-4.5) {\(A''_{H''} = (A_{H})|_{H_0}\)};
			\end{tikzpicture}
		}
	\end{center}
	\caption{A reminder for the residues of the connections \(\nabla, \nabla', \nabla''\) at \(H''\). Here \(H''\) is a reducible intersection \(H \pitchfork H_0\).}
\end{figure}
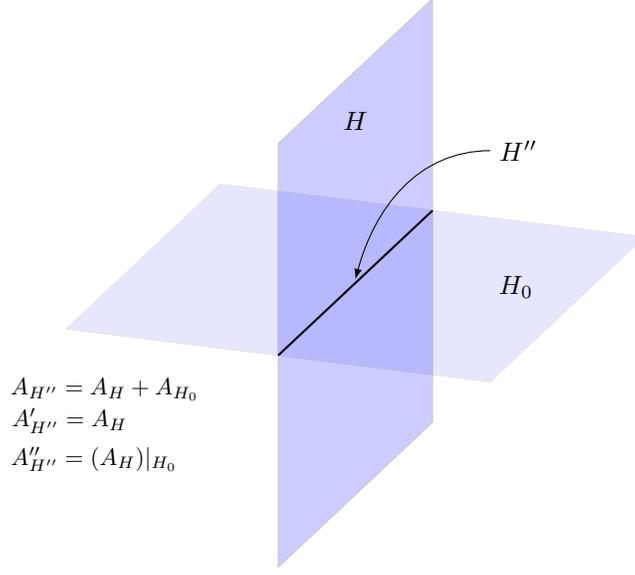

\begin{lemma}\label{lem:restcovder}
	Assume that \(A_{H_0}\) is non-zero.
	Let \(X, Y\) be two holomorphic vector fields defined on an open subset 
	\[U \subset V \setminus \mH' .\] 
	Suppose that \(X, Y\) are tangent to \(H_0\). Denote by \(X''= X|_{H_0}\) and \(Y''=Y|_{H_0}\). Then \(\nabla_X Y\) is holomorphic, tangent to \(H_0\) and
	\[(\nabla_X Y)|_{H_0} = \nabla''_{X''} Y'' . \]
\end{lemma}

\begin{proof}
	Observe that
	\begin{equation*} 
	\nabla_X Y = d_X Y - A_{H_0}(Y) \frac{dh_0(X)}{h_0} - (\mbox{hol.})	
	\end{equation*}
	where \((\mbox{hol.}) = (\imath_X \Omega')(Y)\)  denotes a holomorphic vector field given by contracting \(\Omega'\) with \(X\) and \(Y\). Since \(X, Y\) are tangent to \(H_0\) we have \(X(h_0) =f_X h_0\) and \(Y(h_0) = f_Y h_0\) for some holomorphic functions \(f_X, f_Y\). Since \(A_{H_0} = dh_0(Y) \cdot n_{H_0} = f_Y h_0 n_{H_0}\) we obtain
	\begin{equation}\label{eq:auxcovder} 
	\nabla_X Y = d_X Y - h_0 f_X f_Y n_{H_0} - (\imath_X \Omega')(Y)	
	\end{equation}
	which is clearly holomorphic. If we restrict to \(H_0\) then the middle term of the r.h.s. of Equation \ref{eq:auxcovder} vanishes and \((\nabla_X Y)|_{H_0} = \nabla'_{X''}Y'' = \nabla''_{X''}Y''\). 
\end{proof}

\begin{remark}
	One can use Lemma \ref{lem:restcovder} to define \(\nabla''\).
\end{remark}

\begin{lemma}\label{lem:indconectisftf}
	\(\nabla''\) is flat and torsion free.
\end{lemma}

\begin{proof}
	This is immediate from Lemma \ref{lem:restcovder}, using local extensions. Alternatively, flatness follows since \(\nabla'\) is flat and torsion free follows from \(H'' \subset \ker A''_{H''}\).
\end{proof}

\subsubsection{{\large Residues of the induced connection}}

\begin{lemma}\label{lem:weightsinducedconnection}
	Let \(L'' \in \mL_{irr}(\mH'')\) and let \(L \in \mL_{irr}(\mH)\) be such that \(L \cap H_0 = L''\) with \(L\) given as in Proposition \ref{prop:indirredsub}. Then \(A_L\) preserves \(H_0\) and \(A''_{L''}\) is equal to the restriction of \(A_{L}\) to \(H_0\). 
\end{lemma}

\begin{proof}
	It follows from Definition \ref{def:inducedconnection} that \(A''_{L''}\) equals \(A_{L''}\) restricted to \(H_0\), see Remark \ref{rmk:inducedresidues}. We analyse \(A_{L''}\) by considering the two cases \((i)\) and \((ii)\) of Proposition \ref{prop:indirredsub}. 
	
	Case \((i)\). If \(L''\subsetneq L\) then \(A_{L''}= A_{L}+A_{H_0}\). Since \(\mH_{L''} = \mH_{L} \uplus \{H_0\}\), Lemma \ref{lem:trint} implies that \(A_{L}\) preserves \(H_0\) because \(n_H \in H_0\) for every \(H \in \mH_L\).
	
	Case \((ii)\). If \(L''=L\) then \(A_{L''}=A_L\). Therefore \(H_0 \in \mH_L\) and the flatness condition \textbf{(F)}  implies \([A_{H_0}, A_L]=0\) from which it follows that \(A_L\) preserves \(H_0\).
	
	In any case we see that \(A_L\) preserves \(H_0\) and that \(A''_{L''}\) equals \(A_L\) restricted to \(H_0\).
\end{proof}

\begin{lemma} \label{lem:inducedweights}
	Suppose that the residues \(\{A_H, \,\ H \in \mH\}\) of \(\nabla\) satisfy the Non-Zero Weights Assumptions \ref{ass:nz}. If \(L'' \in \mL_{irr}(\mH'')\) then
	\(a''_{L''}=a_L\) where \(L \in \mL_{irr}(\mH)\) satisfies \(L \cap H_0 = L''\) and is given by Proposition \ref{prop:indirredsub}. In particular,
	the weights \(a''_{H''}\) for \(H'' \in \mL(\mH'')\) are given by
	\begin{equation}\label{eq:aH''}
	a''_{H''} = \begin{cases}
	a_{H} \,\ &\mbox{ if } H \pitchfork H_0 = H'', \\
	a_{H''} \,\ &\mbox{ if } H'' \in \mL_{irr}(\mH) .
	\end{cases}
	\end{equation}
\end{lemma}

\begin{proof}
	By Lemma \ref{lem:weightsinducedconnection}, we have
	\begin{equation}\label{eq:AL''}
	A''_{L''} = (A_L)|_{H_0} . 	
	\end{equation}
	Lemma \ref{lem:Lidec} tells us that
	\[A_L = 0 \cdot \Id_L \oplus \, a_L \cdot \Id_{L^{\perp}} \,\, \mbox{ and } \,\, A''_{L''} = 0 \cdot \Id_{L''} \oplus \, a''_{L''} \cdot \Id_{(L'')^{\perp}}  . \]
	We consider the two cases of Proposition \ref{prop:indirredsub}.
	\begin{itemize}
		\item[(i)] \(L'' \subsetneq L\). Then 
		\(\mH_{L''}=\mH_L \uplus \{H_0\}\) and
		\(n_H \in H_0\) for every \(H \in \mH_L\). It follows that \(L^{\perp} \subset H_0\) and
		\[(A_L)|_{H_0} = 0 \cdot \Id_{L''} \oplus \, a_L \cdot \Id_{L^{\perp}} . \]
		From Equation \eqref{eq:AL''} we conclude that
		\((L'')^{\perp}=L^{\perp}\) and \(a''_{L''}=a_L\).
		\item[(ii)] \(L''=L\). Then \(H_0 \in \mH_L\), so \(n_H \in L^{\perp}\) and \(L^{\perp} \not\subset H_0\). Hence,
		\[A''_{L''} = 0 \cdot \Id_L \oplus \, a_L \cdot \Id_{L^{\perp}\cap H_0} . \]
		From Equation \eqref{eq:AL''} we conclude that
		\((L'')^{\perp}=L^{\perp} \cap H_0\) and \(a''_{L''}=a_L\). \qedhere
	\end{itemize}
\end{proof}

\subsubsection{{\large Linear conditions on the weights}}

Lemma \ref{lem:inducedweights} leads to non-trivial linear equations that the weights of a flat torsion free connection must satisfy. In the next corollaries we collect a couple of such linear conditions which we will use later in Section \ref{sect:Lau} for the case of the braid  arrangement.

\begin{corollary}\label{cor:linear1}
	Let \(L'' \in \mL_{irr}(\mH'')\) and \(L \in \mL_{irr}(\mH)\) be as in Lemma \ref{lem:inducedweights}.
	Then
	\begin{equation}\label{eq:linear1}
	\frac{1}{\codim_V L} \sum_{H \in \mH_L} a_H  = \frac{1}{\codim_{H_0}L''} \sum_{H'' \in \mH''_{L''}} a''_{H''} .
	\end{equation}
\end{corollary}

\begin{proof}
	This is immediate from Lemma \ref{lem:inducedweights}, Equation \eqref{eq:linear1} is equivalent to \(a_L=a''_{L''}\).
\end{proof}

\begin{corollary}\label{cor:linear2}
	Let \((V, \mH)\) be essential and irreducible with \(\dim V =n\). Suppose that the Non-Zero Weights Assumptions \ref{ass:nz} hold. Let \(H_0 \in \mH\) be such that the induced arrangement \(\mH''\) is also essential and irreducible, then
	\begin{equation}\label{eq:linear2}
	\frac{1}{n} \sum_{H \in \mH} a_H = \frac{1}{n-1} \left(\sum_{H \pitchfork H_0} a_H + \sum_{\substack{\codim L =2 \\ L\subset H_0}} a_L \right)
	\end{equation}
	where the first sum on the r.h.s. is over all hyperplanes in \(\mH\) such that \(H \pitchfork H_0\) and the second sum is over all \(L \in \mL_{irr}(\mH)\) of codimension two with \(L \subset H_0\).
\end{corollary}

\begin{proof}
	Apply Corollary \ref{cor:linear1} to \(\{0\} \in \mL_{irr}(\mH) \cap \mL_{irr}(\mH'')\) together with Equation \eqref{eq:aH''}.
\end{proof}

\begin{note}
	If the Non-Zero Weights Assumptions \ref{ass:nz} hold and \(L \in \mL(\mH)\) then there is an induced standard connection on \(TL^{\circ}\) obtained by successive restrictions of \(\nabla\) to hyperplanes. Same as before, the commutation relations on the residues impose linear restrictions on the weights. For a description, in the Dunkl case, see \cite[Lemma 2.18 and Remark 2.19]{CHL}.
\end{note}

\subsection{Proof of Theorem \ref{thm:irredhol}}\label{sect:pfirrhol}

The proof is by induction on the dimension \(n\geq 2\). 

\begin{lemma}
	Theorem \ref{thm:irredhol} holds when \(n=2\).
\end{lemma}

\begin{proof}
	By Theorem \ref{PRODTHM} we have a polyhedral K\"ahler cone metric on \(\C^2\) singular at three or more complex lines; also called PK cones of type \((1,1)\) in \cite[pg. 2208]{Pan}. By \cite[Theorem 1.8]{Pan} there is a corresponding  spherical metric on \(\CP^1\) with cone angles at the points represented by the lines. Moreover, we have a commutative diagram
	\[ \begin{tikzcd}
	\pi_1 \left(\C^2 \setminus \{\text{lines}\}\right) \arrow{r} \arrow[swap]{d}{\Hol(\nabla)} &  \pi_1 \left(\CP^1 \setminus \{\text{lines}\}\right) \arrow[swap]{d}{\rho}    \\
	U(2) \arrow{r} & SO(3)  
	\end{tikzcd}
	\]
	where \(\rho\) is the holonomy of the spherical metric and the map from \(U(2)\) to \(SO(3)=SU(2)/\pm 1\) is  division by \(\sqrt{\det}\). If \(\Hol(\nabla)\) was reducible then the spherical metric would be coaxial and \cite[Theorem 5]{Dey} implies that \(\alpha_H \in \Z\) for some line \(H\) in the arrangement.
\end{proof}

\begin{example}
	Consider the arrangement of three distinct lines in \(\C^2\), also known as \(\mA_2\), endowed with the Lauricella connection \(\nabla^a\) with parameters \(a=(a_1,a_2,a_3)\) as in \eqref{eq:A2matrices}. Take \(a_3=0\) and  \(a_1 \neq 0, a_2 \neq 0, a_1+a_2 \neq 0\). The weights along the three lines \((H_{1,3}, H_{2,3}, H_{1,2})\) are \((a_1,a_2,a_1+a_2)\) and
	\[H_{1,3}^{\perp} = H_{2,3}^{\perp} = H_{1,2} = \C \cdot (1,1) .\] 
	The main diagonal \(\C \cdot (1,1)\) is invariant under all matrices \(A_{i,j}\) and the holonomy of \(\nabla^a\) is reducible. 
	This shows that unitary holonomy is a necessary condition Theorem \ref{thm:irredhol}.  
\end{example}

\begin{lemma}\label{lem:unitaryinducedconect}
	If \(\nabla\) is unitary and \(a_{H_0} \notin \Z\) then  \(\nabla''\) is also unitary.
\end{lemma}

\begin{proof}
	This follows from Corollary \ref{cor:locprod} applied to \(H_0\). Around points \(x \in H_0^{\circ}\) we have complex coordinates in which the metric splits as \(\C_{\alpha_{H_0}} \times \C^{n-1}\). We obtain a Hermitian form on \(H_0 = \{0\} \times \C^{n-1}\) which is parallel for \(\nabla''\).
\end{proof}

We proceed with the proof of Theorem \ref{thm:irredhol}.
Let \(n \geq 3\) and suppose by contradiction that there is some \(p \in (\C^n)^{\circ}\) and a proper vector subspace \(F_p \subset T_p \C^n\) which is invariant under the action of \(\Hol(\nabla, p)\). By parallel translation we get a flat sub-bundle \(F \subset T (\C^n)^{\circ}\). Taking the orthogonal complement we have another flat sub-bundle \(F^{\perp}\) such that \(F \oplus F^{\perp} = T(\C^n)^{\circ}\).

\begin{lemma}\label{lem:irredhol1}
	The sub-bundles \(F\) and \(F^{\perp}\) extend over the generic points of the arrangement \(\cup_{H \in \mH} H^{\circ}\) as holomorphic vector sub-bundles of \(T\C^n\). Moreover, if \(H \in \mH\) and \(x \in H^{\circ}\) then \(F_{x} \oplus F^{\perp}_{x} = T_{x}\C^n\).
\end{lemma}

\begin{proof}
	By Corollary \ref{cor:locprod} there are complex coordinates around \(x\) in which the metric splits as \(\C_{\alpha_H} \times \C^{n-1}\). The parallel sub-bundles \(F, F^{\perp}\) must be constant vector subspaces in these coordinates.
\end{proof}

\begin{lemma}\label{lem:irredhol2}
	Let \(U\) be an open ball of \(\C^n\) and \(Z \subset U\) an analytic subset of complex codimension \(\geq 2\). Let \(E\) be a trivial complex vector bundle over \(U\). Suppose that \(E_1, E_2\) are holomorphic sub-bundles of \(E\) defined on \(U \setminus Z\) and such that \(E_1\oplus E_2 = E\). Then \(E_1, E_2\) extend over the whole \(U\) as holomorphic vector sub-bundles.
\end{lemma}

\begin{proof}
	We have projector maps \(\pi_1, \pi_2 \in H^0(U \setminus Z, \End E)\) defined by 
	\[\ker \pi_1 = E_2, \qquad \pi_1|_{E_1}=\Id_{E_1}, \qquad \ker \pi_2 = E_1, \qquad \pi_2|_{E_2}=\Id_{E_2} . \] 
	By Hartogs we can extend \(\pi_1, \pi_2\) over all of \(U\) and the extensions satisfy \(\pi_1+\pi_2=\Id_E\). If \(x \in Z\) then \(\dim \ker \pi_1 (x) \geq \dim E_2\) and \(\dim \ker \pi_2 (x) \geq \dim E_1\). Since \(\ker \pi_1 (x) \cap \ker \pi_2 (x) = \{0\}\) the dimensions of \(\ker \pi_i\) can not jump-up at \(x\). We conclude that \(\ker \pi_i\) are vector sub-bundles providing the desired extensions. 
\end{proof}

\begin{lemma}\label{lem:irredhol3}
	\(F\) and \(F^{\perp}\) are constant vector subspaces of \(\C^n\).
\end{lemma}

\begin{proof}
	By Lemmas \ref{lem:irredhol1} and \ref{lem:irredhol2} \(F, F^{\perp}\) are holomorphic vector sub-bundles of \(T\C^n\).
	Since the action of \(\C^*\) on \(\C^n\) by multiplication preserves \(\nabla\), it also preserves \(F\) and \(F^{\perp}\). However, any smooth \(\C^*\) invariant foliation on \(\C^n\) is constant.
\end{proof}

We can now establish Theorem \ref{thm:irredhol}.

\begin{proof}[Proof of Theorem \ref{thm:irredhol}]
	Let \(H_0 \in \mH\) be such that the induced arrangement \(\mH''\) is essential and irreducible, see Lemma \ref{lem:inducedirred}. The induced connection \(\nabla''\) is standard, flat and torsion free (see Definition \ref{def:inducedconnection}, Remark \ref{rmk:inducedresidues} and Lemma \ref{lem:indconectisftf}).
	By Lemma \ref{lem:unitaryinducedconect}, the induced connection is also unitary. Moreover, it follows from Lemma \ref{lem:weightsinducedconnection} that the weights of \(\nabla''\) are a subset of \(\{a_L, \,\ L \in \mL_{irr}(\mH)\}\) and therefore are non-integer. 
	
	If either \(F \cap H_0\) or \(F^{\perp} \cap H_0\) is a non-zero proper subspace of \(H_0\) then we get a contradiction by induction. It follows that \(F=\C \cdot n_{H_0}\) or \(F=H_0\). Swapping \(F\) with \(F^{\perp}\) if necessary, we can assume that \(F= \C \cdot n_{H_0}\) and \(F^{\perp}=H_0\).
	
	Suppose there is \(L \in \mL_{irr}(\mH)\) with \(\codim L=2\) and \(L \subset H_0\). Pick  \(x \in L^{\circ}\). Take complex coordinates centred at \(x\) as in Theorem \ref{thm:locprod} in which we have a splitting \[\left(\C^n/L, \, \mH_L/L, \, \nabla^{\C^n/L}, \, \langle \cdot, \cdot \rangle^{\perp}\right) \times \C^{n-2} .\] 
	The subspace \(F\) gives a non-zero one-dimensional parallel sub-bundle of \(\nabla^{\C^n/L}\), which contradicts the \(n=2\) case. We conclude that there are no codimension two irreducible subspaces \(L\subset H_0\), which is to say that \(H_0 \pitchfork H\) for every \(H \in \mH'\). By Lemma \ref{lem:trint} we have \(\C \cdot n_{H_0} \subset \cap_{H \in \mH'} H\), which implies \(\mH= \mH' \uplus \{H_0\}\) and contradicts the irreducibility of \(\mH\).
\end{proof}

To finish the section, we prove Corollary \ref{cor:uniqueinnerprod}.

\begin{proof}[Proof of Corollary \ref{cor:uniqueinnerprod}]
	Let \(\inn\) and \(\inn'\) be the parallel Hermitian inner products on the holomorphic tangent bundle determined by the K\"ahler metrics \(g\) and \(g'\).
	Since \(\inn\) and \(\inn'\) are both parallel with respect to \(\nabla\), it is enough to show that they differ by a scalar multiple at some fixed point \(p \in (\C^n)^{\circ}\). By Theorem \ref{thm:irredhol} the action of \(\Hol_p(\nabla)\) on \(T_p \C^n\) is irreducible and by Schur's lemma \(\inn'_p = \lambda \inn_p \) for some \(\lambda>0\).
\end{proof}

\section{Braid arrangement and the Lauricella connection}\label{sect:Lau}

In this Section we establish Theorem \ref{LAUTHM}.
We begin by introducing the (essential) braid arrangement and the Lauricella connection. 
Our first main result is the uniqueness statement given by
Proposition \ref{prop:UNIQ} which provides a characterization of the Lauricella connection among flat torsion free standard connections with poles at the hyperplanes of the \(\mA_n\)-arrangement. We prove Proposition \ref{prop:UNIQ} by induction on the dimension, making use of induced and quotient connections. Using this characterization, together with the existence Theorem \ref{PRODTHM} and results from Section \ref{sect:PK} on PK metrics, we establish
Theorem \ref{LAUTHM} in Section \ref{sect:pfLau}.

\subsection{The \(\mA_n\)-arrangement}\label{sect:braidarrang}

\begin{definition}\label{def:An}
	Let \((z_1, \ldots, z_n)\) be standard complex coordinates on \(\C^n\). The \textbf{\(\mA_n\)-arrangement} consists of \({n+1\choose2}\) hyperplanes indexed by pairs \(H_{p,q}=H_{q,p}\) with \(p,q\) running from \(1\) to \(n+1\) and \(p \neq q\). The hyperplanes are: \( H_{i,j} = \{z_i=z_j\}\) for \(1 \leq i < j \leq n\) and \(H_{i, n+1} = \{z_i = 0\}\) for \(1 \leq i \leq n\).
\end{definition}

The reason for the notation \(\mA_n\) is that its hyperplanes are the mirrors of the Coxeter group of type \(A_n\). See \cite[Chapter 6]{OT1}.

\begin{definition}
	Let \((\tz_1, \ldots, \tz_{n+1})\) be standard complex coordinates on \(\C^{n+1}\). The \(\textbf{braid arrangement}\) \(\tilde{\mA}_n\) consists of the following hyperplanes
	\begin{equation*}
		\tilde{H}_{i, j} = \{\tz_i=\tz_j\} \mbox{ for } 1\leq i<j\leq n+1 .
	\end{equation*}
\end{definition}

The reason for the name braid arrangement is that the fundamental group of the complement \(\C^{n+1} \setminus \tilde{\mA}_n\) is the pure braid group on \(n+1\) strands. See \cite[pg. 161]{OT1}.

\begin{remark}
	\(\tilde{\mA}_n\) is not essential, its centre is the main diagonal \(\tz_1 = \ldots = \tz_{n+1}\).
	The quotient \(\tilde{\mA}_n/\C \cdot (1, \ldots, 1)\) is identified with  \(\mA_n\), by setting
	\begin{equation*}
		z_i = \tz_i - \tz_{n+1}
	\end{equation*}
	for \(i=1, \ldots, n\). For this reason we also refer to \(\mA_n\) as the \textbf{essential braid arrangement}.
\end{remark}

\begin{notation}
	Given a positive integer number \(N\) we denote by \(\underline{N}\) the set \(\{1, \ldots, N\}\), i.e. 
	\[\underline{N} = \{1, \ldots, N\} . \]
\end{notation}

\begin{notation}
	A partition \(P=\{P_1, \ldots, P_k\}\) of \(\underline{n+1}\) consists of disjoint non-empty subsets \(P_i \subset \underline{n+1}\) such that their union is \(\underline{n+1}\). %\(|S|\) denotes the number of elements of a set \(S\); \(P_i\) is a singleton if \(|P_i|=1\). 
	%We denote the set of all partitions of \(\underline{n+1}\) by \(\mathcal{P}_{n+1}\). 
\end{notation}

\begin{definition}
	Let \(L \in \mL(\mA_n)\). We define \(\sim_{L}\) to be the equivalence relation on \(\underline{n+1}\) \emph{generated} by \(i \sim_{L} j\) if \(L \subset H_{i,j}\). The equivalence classes of \(\sim_{L}\) give us a partition of \(\underline{n+1}\) which we denote by  \(P(L)\).  
\end{definition}

\begin{definition}
	Let \(P=\{P_1, \ldots, P_k\}\) be a partition of \(\underline{n+1}\), we define 
	\[\mH(P)= \{H_{i,j} \in \mA_n, \,\ \{i,j\} \subset P_l \mbox{ for some } l\} .\] 
	The common intersection of all hyperplanes in \(\mH(P)\) is an element of \(\mL(\mA_n)\) which we denote
	by \(L(P) \in \mL(\mA_n) \).
\end{definition}

The next result is a straightforward consequence of the definitions.

\begin{lemma}[{\cite[Proposition 2.9]{OT1}}] \label{lem:interAn}
	The maps \(L \to P(L)\) and \(P \to L(P)\) are inverses of each other and give a natural correspondence between
	\(\mL(\mA_n)\) and partitions of \(\underline{n+1}\).
\end{lemma}

\begin{note}
	The intersections \(\C^n, \{0\} \in \mL(\mA_n)\) correspond to the partitions \(\{\{1\}, \ldots, \{n+1\}\}\) and \(\{\underline{n+1}\}\) respectively. If \(L,M \in \mL(\mA_n)\) correspond to \(P(L)=\{P_1, \ldots, P_k\}\) and \(Q(M)=\{Q_1, \ldots, Q_l\}\) then \(M \subset L\) if and only if for every \(P_i\) there is some \(Q_j\) such that \(P_i \subset Q_j\). 
\end{note}

\begin{remark}
	There are two types of codimension two intersections \(L \in \mL(\mA_n)\)
	\begin{itemize}
		\item irreducible: \((\mA_n)_{L}=\{H_{i, j}, H_{j, k}, H_{i, k}\}\) with \(i<j<k\);
		\item reducible: \((\mA_n)_{L}=\{H_{i, j}, H_{k,l}\}\) with \(\{i, j\} \cap \{k, l\}= \emptyset\).
	\end{itemize}
\end{remark}

\begin{lemma} \label{lem:Anisirred}
	\(\mA_n\) is irreducible.
\end{lemma}

\begin{proof}
	There is a subset of \(n+1\) hyperplanes in general position,  namely:
	\[\{x_1=0\}, \{x_2=x_1\}, \ldots, \{x_n=x_{n-1}\}, \{x_n=0\} . \]
	The lemma follows from Corollary \ref{cor:standardsubarrangement}.
\end{proof}

\begin{remark}\label{rmk:Anquotient}
	Let \(L \in \mL(\mA_n)\) be given by a partition \(P=\{P_1, \ldots, P_k\}\), then
	\[(\mA_n)_L/L \equiv \mA_{|P_1|-1} \times \ldots \times \mA_{|P_k|-1} . \]
	In particular, \(L\) is irreducible if and only if \(|P_i|\geq 2\) for at most one \(i\). See \cite[Corollary 6.28]{OT1} for a more general statement that comprises complex reflection arrangements.
\end{remark}

\begin{lemma}[{\cite[Proposition 6.73]{OT1}}] \label{lem:inducedAn}
	If \(L \in \mL(\mA_n)\) then the induced arrangement \((\mA_n)^L\) is isomorphic to \(\mA_{\dim L}\).
\end{lemma}

\begin{proof}
	It is enough to check it when \(L\) is a hyperplane, say \(L=\{z_n=0\}\). The restriction of coordinates functions \(z_1, \ldots, z_{n-1}\) to \(L\) give us coordinates, which we  denote
	by the same symbol. The hyperplanes of the induced arrangement are \(\{z_i=0\}\) for \(i=1, \ldots, n-1\) and \(\{z_i=z_j\}\) for \(1 \leq i<j \leq n-1\); which agrees with \(\mA_{n-1}\).
\end{proof}

\subsection{The Lauricella connection}\label{sect:Lauconnect}

We introduce the (reduced) Lauricella connection \(\nabla^a\) in Definition \ref{def:redLau} following \cite[Section 2.3]{CHL} with the difference that we impose no restrictions on the complex parameters \(a=(a_1, \ldots, a_{n+1})\).

\begin{notation}
	Let \((\tz_1, \ldots, \tz_{n+1})\) be standard complex coordinates on \(\C^{n+1}\). We fix defining equations for the hyperplanes \(\tH_{i, j} \in \tilde{\mA}_n\) given by 
	\[ \Th_{i, j} = \tz_i - \tz_j \,\ \mbox{ for } \,\ 1\leq i<j \leq n+1 . \]
\end{notation}

\begin{notation}
	We denote the coordinate vectors of \(\C^{n+1}\) by \(\tp_i = \p/\p\tz_i\). 
\end{notation}

\begin{definition}
	Fix \(a_1, \ldots, a_{n+1} \in \C\). For \(1 \leq i<j\leq n+1\) we introduce the following:
	\begin{itemize}
		\item vectors \(n_{i,j} \in \C^{n+1}\) given by
		\[n_{i,j} = a_j \tp_i - a_i \tp_j;\]
		\item endomorphisms \(\tA_{i,j} \in \End \C^{n+1}\) given by
		\[ \tA_{i,j} (v)= d \Th_{i, j}(v) \cdot n_{i,j} .  \]
	\end{itemize}
\end{definition}

\begin{note}
	It is easy to check that
	\[ \tH_{i, j} \subset \ker \tA_{i,j}, \,\ \img \tA_{i, j} \subset \C \cdot n_{i,j}, \,\ \tA_{i,j} n_{i,j} = (a_i + a_j) n_{i,j} . \]
\end{note}

\begin{note}
	The endomorphisms \(\tA_{i, j}\) are represented with respect to the standard basis \(\tp_1, \ldots, \tp_{n+1}\) by the \((n+1)\times(n+1)\) \textbf{Jordan-Pochammer matrices}:
	
	\begin{equation} \label{eq:pochjor0}
	\tA_{i, j} \qquad = \qquad
	\bordermatrix{ 
		& & & &  \underset{\downarrow}{i} &  & & & \underset{\downarrow}{j} \cr 
		& 0 & \cdots & 0 & 0 & 0 & \cdots & 0 & 0 & 0 & \cdots & 0 \cr 
		& \vdots & \ddots & \vdots & \vdots & \vdots & \ddots & \vdots & \vdots & \vdots & \ddots & \vdots \cr
		& 0 & \cdots & 0 & 0 & 0 & \cdots & 0 & 0 & 0 & \cdots & 0 \cr
		i \to & 0 & \cdots & 0 & {\color{red}a_j} & 0 & \cdots & 0 & {\color{red}-a_j} & 0 & \cdots & 0 \cr
		& 0 & \cdots & 0 & 0 & 0 & \cdots & 0 & 0 & 0 & \cdots & 0  \cr
		& \vdots & \ddots & \vdots & \vdots & \vdots & \ddots & \vdots & \vdots & \vdots & \ddots & \vdots \cr
		& 0 & \cdots & 0 & 0 & 0 & \cdots & 0 & 0 & 0 & \cdots & 0  \cr
		j \to & 0 & \cdots & 0 & {\color{red}-a_i} & 0 & \cdots & 0 & {\color{red}a_i} & 0 & \cdots & 0 \cr
		& 0 & \cdots & 0 & 0 & 0 & \cdots & 0 & 0 & 0 & \cdots & 0  \cr
		& \vdots & \ddots & \vdots & \vdots & \vdots & \ddots & \vdots & \vdots & \vdots & \ddots & \vdots \cr
		& 0 & \cdots & 0 & 0 & 0 & \cdots & 0 & 0 & 0 & \cdots & 0
	} .
	\end{equation}
\end{note}

\begin{notation}
	Let \(I \subset \underline{n+1}\) with \(|I| \geq 2\), then it defines \(\tL_{I} \in \mL_{irr}(\tilde{\mA}_n)\) given by
	\[\tL_{I} = \{\tz_i=\tz_j, \,\ i, j \in I\} . \]
	Similarly, we set
	\[\tA_{I} = \sum_{i,j \in I} \tA_{i, j} \] 
\end{notation}

\begin{lemma} \label{lem:LauCom}
	The matrices \(\{\tA_{i, j}, \,\ \tH_{i,j} \in \tilde{\mA}_n\}\) satisfy the Basic Assumptions \ref{ass:basic}. More explicitly, \(H_{i,j} \subset \ker A_{i,j}\) and
	\begin{equation}\label{eq:infbraidrel}
		\begin{aligned}
		[\tA_{i,j}, \tA_{i,k}+ \tA_{j,k}] = 0 \,\, &\mbox{ for } \,\, i<j<k,\\
		[\tA_{i,j}, \tA_{k,l}] = 0 \,\, &\mbox{ for } \,\, \{i,j\} \cap \{k,l\} = \emptyset.
		\end{aligned}
	\end{equation}
\end{lemma}

\begin{proof}
	Take \(\tL \in \mL(\tilde{\mA}_n)\) with \(\codim \tL =2\) and let \(\tH \in \tilde{\mA}_n\) such that \(\tL \subset \tH\). We will show that \([\tA_{\tH}, \tA_{\tL}] = 0\). There are two cases to consider.
	
	\begin{itemize}
		\item\(\tL = \tH_{i,j} \cap \tH_{j,k}\) with \(\{i,j\} \cap \{k,l\} = \emptyset\). In this case \([\tA_{i, j}, \tA_L] = [\tA_{i, j}, \tA_{k,l}]\). Moreover, we have \(n_{i,j} \in \tH_{k,l}\) and \(n_{j,k} \in \tH_{i,j}\), hence
		\[\tA_{i,j} \tA_{k,l} = \tA_{k,l} \tA_{i,j} = 0 . \]
		
		\item \(\tL = \tH_{i, j} \cap \tH_{j,k}\) with \(i<j<k\). In this case 
		\[\tA_{\tL}=\tA_{i,j,k}= \tA_{i, j} + \tA_{i, k} + \tA_{j, k} .\] 
		It is easy to check that
		\[a_k n_{i,j}+a_i n_{j,k} =a_j n_{i,k}\] 
		 and \(\tA_{i,j,k}n_{i,j} = (a_i+a_j+a_k)n_{i,j}\). An easy computation, using \(\tA_{i, j}n_{j,k}=-a_kn_{i,j}\) and \(\tA_{i, j} n_{i,k}=a_k n_{i,j}\), shows
		 \[\tA_{i,j,k} \tA_{i,j} v = \tA_{i, j} \tA_{i,j,k} v = (a_i+a_j+a_k)d\Th_{i, j}(v) \cdot n_{i,j} . \qedhere \] 
	\end{itemize}
\end{proof}

The Equations \eqref{eq:infbraidrel} are also known as infinitesimal braid relations and arise in the study of the KZ equation, see \cite[Chapter 2.1]{Kohno}.

\begin{definition}
	The Lauricella connection is
	\begin{equation*}
	\tilde{\nabla} = d - \sum_{i<j} \tA_{i,j} \frac{d\Th_{ij}}{\Th_{ij}}
	\end{equation*}
	where \(d\) is the Euclidean connection on \(\C^{n+1}\).
\end{definition}

\begin{notation}
	Write \(\omega_{i,j} = d\Th_{i, j}/\Th_{i, j}\). The \textbf{Arnold's relations}:
	\begin{equation}\label{eq:arnrel}
		\omega_{i,j} \wedge \omega_{j,k} + \omega_{j,k} \wedge \omega_{i,k} + \omega_{i,k} \wedge \omega_{i,j} = 0
	\end{equation}
	are easily checked.
\end{notation}

\begin{proposition}
	\(\tilde{\nabla}\) is flat and torsion free. 
\end{proposition}

\begin{proof}
	This follows from Proposition \ref{prop:ftfstcon}, but we also present a direct check.
	
	Write \(\Omega = \sum_{i<j} \tA_{i, j} \omega_{i,j}\). Using that \(\tA_{i, j} \tA_{k,l} =0\) if \(\{i,j\} \cap \{k,l\}=\emptyset\) together with Arnold's relation \eqref{eq:arnrel} we find
	\[\Omega \wedge \Omega = \sum_{i<j<k} \left( [\tA_{i,j,k}, \tA_{i, k}] \omega_{i,j} \wedge\omega_{i,k} + [\tA_{i,j,k}, \tA_{j, k}] \omega_{i,j} \wedge \omega_{j,k} \right) \]
	which vanishes because of Lemma \ref{lem:LauCom}.
\end{proof}

\begin{notation}
	If \(\tL \in \mL_{irr}(\tilde{\mA}_n)\) then we have defined \(\tilde{a}_{\tL}= (\codim \tL)^{-1}  \sum_{\tL \subset \tH} \tilde{a}_{\tH}\) where the sum is over all \(\tH \in \tilde{\mA}_n\) that contain \(\tL\). 
	
	If \(\tH = \tH_{i, j}\) then \(\tilde{a}_{\tH}=a_i + a_j\). We denote \(a_{i,j} = a_i + a_j\) and more generally, if \(I \subset \underline{n+1}\) we set
	\[a_{I} = \sum_{i \in I} a_i . \]
\end{notation}

\begin{lemma} \label{lem:weightsLau}
	If \(\tL = \tL_{I} \in \mL_{irr}(\tilde{\mA}_n)\) with \(I \subset \underline{n+1}\) then \(\tilde{a}_{\tL} = a_{I}\).
\end{lemma}

\begin{proof}
	Note that \(\codim \tL = |I|-1 \), so
	\[\tilde{a}_{\tL} = (|I|-1)^{-1} \sum_{\substack{\{i,j\}\subset I \\ i<j}} (a_i + a_j) = \sum_{i \in I} a_i . \qedhere \]
\end{proof}

\begin{corollary}
	The Non-Zero Weights Assumptions \ref{ass:nz} are equivalent to
	\(\sum_{i \in I}a_i \neq 0\) for every \(I \subset \underline{n+1}\) with \(|I| \geq 2\).
\end{corollary}

\begin{definition}
	We consider the linear subspace
	\begin{equation*}
	V = \{\sum_{i=1}^{n+1} a_i \tz_i = 0\} \subset \C^{n+1}
	\end{equation*}
	and note that \(n_{i,j} \in V\) for all \(i<j\). In particular \(\tA_{i, j}(V) \subset V\).
\end{definition}

\begin{remark}
	If \(\sum_{i=1}^{n+1}a_i \neq 0\) then the hyperplane \(V\) is transversal to the main diagonal \(\C \cdot (1, \ldots,1)\) and the functions 
	\(z_i = (\tz_i - \tz_{n+1})|_V\)
	give linear coordinates on \(V\) such that the induced arrangement \(\{\tH \cap V, \,\  \tH \in \tilde{\mA}_n\}\) agrees with \(\mA_n\).
\end{remark}

\begin{notation}
	We denote
	\[a_0 = \sum_{i=1}^{n+1} a_i .\]
	According to Lemma \ref{lem:weightsLau} \(a_0 = \tilde{a}_{T(\tilde{\mA}_n)}\).
\end{notation}

\begin{lemma}
	Suppose that \(a_0 \neq 0\) and let \(z_i=(\tz_i - \tz_{n+1})|_V\) for \(1\leq i\leq n\) be linear coordinates on \(V\).
	The coordinate vector fields \(\p_i = \p/\p z_i\) on \(V\) are given by
	\[\p_i = \tp_i  - (a_i/a_0) \sum_{k=1}^{n+1} \tp_k \]
\end{lemma}

\begin{proof}
	It is easy to check that \(\p_i \in V\) and \(dz_j(\p_i)=\delta_{ij}\).
\end{proof}

\begin{note}
	If \(a_0 \neq 0\) we can write \(n_{i,j} = \sum_{k=1}^{n} dz_k(n_{i,j}) \cdot \p_k \) to get
	\[n_{i,j} = a_j \p_i - a_i \p_j, \,\ n_{i, n+1} = a_{n+1} \p_i + a_i \sum_{k=1}^{n} \p_k \]
	for \(1 \leq i < j \leq n\).
\end{note}

\begin{definition}
	\(A_{i,j} = \tA_{i, j}|_V \in \End V \).
\end{definition}

\begin{note}
	Assume \(a_0 \neq 0\) and identify \(V\cong \C^n\) by using the coordinates \((z_1, \ldots, z_n)\). It is straightforward to write the matrix that represents \(A_{i,j}\) with respect to the basis \(\p_i\).
	
	If \(i,j \leq n\) then we use \(A_{i,j} (\p_k) = (dz_i-dz_j)(\p_k) \cdot n_{i,j}\) to get the \((n\times n)\)-matrix
	\begin{equation} \label{eq:pochjorRED1}
		A_{i, j} \qquad = \qquad
		\scalebox{.8}{
		\bordermatrix{ 
		& & & &  \underset{\downarrow}{i} &  & & & \underset{\downarrow}{j} \cr 
		& 0 & \cdots & 0 & 0 & 0 & \cdots & 0 & 0 & 0 & \cdots & 0 \cr 
		& \vdots & \ddots & \vdots & \vdots & \vdots & \ddots & \vdots & \vdots & \vdots & \ddots & \vdots \cr
		& 0 & \cdots & 0 & 0 & 0 & \cdots & 0 & 0 & 0 & \cdots & 0 \cr
		i \to & 0 & \cdots & 0 & {a_j} & 0 & \cdots & 0 & {-a_j} & 0 & \cdots & 0 \cr
		& 0 & \cdots & 0 & 0 & 0 & \cdots & 0 & 0 & 0 & \cdots & 0  \cr
		& \vdots & \ddots & \vdots & \vdots & \vdots & \ddots & \vdots & \vdots & \vdots & \ddots & \vdots \cr
		& 0 & \cdots & 0 & 0 & 0 & \cdots & 0 & 0 & 0 & \cdots & 0  \cr
		j \to & 0 & \cdots & 0 & {-a_i} & 0 & \cdots & 0 & {a_i} & 0 & \cdots & 0 \cr
		& 0 & \cdots & 0 & 0 & 0 & \cdots & 0 & 0 & 0 & \cdots & 0  \cr
		& \vdots & \ddots & \vdots & \vdots & \vdots & \ddots & \vdots & \vdots & \vdots & \ddots & \vdots \cr
		& 0 & \cdots & 0 & 0 & 0 & \cdots & 0 & 0 & 0 & \cdots & 0
			}
		} .
	\end{equation}
	equal to the matrix \eqref{eq:pochjor0} with the \((n+1)\)-column and \((n+1)\)-row of zeros deleted.
	
	For \(i\leq n\) we use \(A_{i, n+1} (\p_k) = (dz_i)(\p_k) \cdot n_{i,n+1}\) to get the \((n\times n)\)-matrix
	\begin{equation} \label{eq:pochjorRED2}
	A_{i, n+1} \qquad = \qquad
	\bordermatrix{ 
		& & & &  \underset{\downarrow}{i} \cr 
		& 0 & \cdots & 0 & a_i & 0 & \cdots & 0  \cr 
		& \vdots & \ddots & \vdots & \vdots & \vdots & \ddots & \vdots \cr
		& 0 & \cdots & 0 & a_i & 0 & \cdots & 0  \cr
		i \to & 0 & \cdots & 0 & a_i + a_{n+1} & 0 & \cdots & 0 \cr
		& 0 & \cdots & 0 & a_i & 0 & \cdots & 0 \cr
		& \vdots & \ddots & \vdots & \vdots & \vdots & \ddots & \vdots \cr
		& 0 & \cdots & 0 & a_i & 0 & \cdots & 0  
	} .
	\end{equation}
	The \((n\times n)\)-matrices \eqref{eq:pochjorRED1} and \eqref{eq:pochjorRED2} are called \textbf{reduced Jordan-Pochammer} matrices.
\end{note}

\begin{lemma}\label{lem:LauRedCom}
	The reduced Jordan-Pochammer matrices \(\{A_H, \,\ H \in \mA_n\}\) satisfy the Basic Assumptions \ref{ass:basic}.
\end{lemma}

\begin{proof}
	Straightforward check. Alternatively, this follows from Lemma \ref{lem:LauCom} when \(a_0 \neq 0\) and by continuity also when \(a_0=0\).
\end{proof}

\begin{notation}
	For \(1 \leq i < j \leq n\) we set
	\[\omega_{i,j}= \frac{d(z_i-z_j)}{z_i-z_j}, \qquad \omega_{i,n+1} = \frac{dz_i}{z_i} . \]
\end{notation}

\begin{definition}\label{def:redLau}
	The \textbf{reduced Lauricella connection} is defined by
	\begin{equation}\label{eq:redLaucon}
		\nabla^a= d - \sum_{1 \leq i < j \leq n+1} A_{i,j} \omega_{i,j} 
	\end{equation} 
	where \(d\) is the Euclidean connection on \(\C^n\).
\end{definition}

\begin{notation}
	We have defined a family of connections parametrized by \(a=(a_1, \ldots, a_{n+1}) \in \C^{n+1}\). We write \(\nabla^{a} = d - \Omega^{a}\) where 
	\[\Omega^a =\sum_{1 \leq i < j \leq n+1} A_{i,j} \omega_{i,j}\]
	is the connection form. Note that the family  is a complex vector space and
	\[\Omega^{a+\lambda b} = \Omega^a + \lambda \Omega^b  \]
	for \(a,b \in \C^{n+1}\) and \(\lambda \in \C\).
\end{notation}

\begin{lemma}
	\(\nabla^a\) is flat and torsion free.
\end{lemma}

\begin{proof}
	Follows from Lemma \ref{lem:LauRedCom}.
\end{proof}

\begin{lemma} \label{lem:weightsRedLau}
	Let \(L \in \mL_{irr}({\mA}_n)\) correspond to \(I \subset \underline{n+1}\). Then the weight of the reduced Lauricella connection \(\nabla^a\) at \(L\) equals 
	\[a_{L} = \sum_{i \in I} a_i . \]
\end{lemma}

\begin{proof}
	Same as Lemma \ref{lem:weightsLau}.
\end{proof}

\begin{example}
	For \(\mA_1\) we obtain the standard connection on \(\C\) given by
	\[\nabla^a = d - \frac{a_1+a_2}{z} dz . \]
\end{example}

\begin{example}
	In the \(\mA_2\) case we have \(\nabla^a = d -\Omega^a\) with
\begin{equation}\label{eq:A2matrices} 
	\Omega^a = 
	\begin{pmatrix}
	a_1+a_3 & 0  \\
	a_1 & 0  
	\end{pmatrix}
	\frac{dz_1}{z_1}
	+
	\begin{pmatrix}
	0 &  a_2 \\
	0 &  a_2+a_3  
	\end{pmatrix}
	\frac{dz_2}{z_2} 
	+
	\begin{pmatrix}
	a_2 & -a_2  \\
	-a_1 & a_1  
	\end{pmatrix}
	\frac{d(z_1-z_2)}{z_1-z_2} .
\end{equation}
\end{example}

\subsection{Characterization of \(\nabla^a\)}\label{sect:charactLau}

In this section we characterize the Lauricella family as the unique flat torsion free connections with simple poles at the members of \(\mA_n\). We haven't found such a statement in the literature, except for the Dunkl case \cite[Proposition 2.29]{CHL}. Here it is the statement.

\begin{proposition} \label{prop:UNIQ}
	Let \(n\geq 2\) and suppose that \(\{B_H, \,\ H \in \mA_n\} \subset \End \C^n\) is such that
	\begin{equation} \label{eq:Bconnection}
		\nabla = d - \sum_{H\in \mA_n} B_H \frac{dh}{h}
	\end{equation}
	is flat and torsion free. Moreover, assume that the weights
	\[b_L= (\codim L)^{-1} \sum_{H \in (\mA_n)_L} \tr B_H\]
	are \(\neq 0\) for every \(L \in \mL_{irr}(\mA_n)\). Then there is a unique \(a \in \C^{n+1}\) such that \(\nabla\) is equal to the reduced Lauricella connection \(\nabla^a\) given by Equation \eqref{eq:redLaucon}.
\end{proposition}

\begin{note}
	We can restate the hypothesis of Proposition \ref{prop:UNIQ} more shortly by saying that \(\{B_H, \,\ H \in \mA_n\}\) satisfies the Non-Zero Weights Assumptions \ref{ass:nz}.
\end{note}

	Note that irreducible intersections of \(\mA_n\) correspond to subsets of \(\underline{n+1}\), see Remark \ref{rmk:Anquotient}.

\begin{notation}\label{not:bij}
	We write \(B_{i,j} = B_{H_{i,j}}\) and  \(b_{i,j} = b_{H_{i,j}} = \tr B_{i,j}\). 
	 If \(L \in \mL_{irr}(\mA_n)\) corresponds to \(I=\{i_1, \ldots, i_k\} \subset \underline{n+1}\), we denote
	\[b_{i_1, \ldots, i_k}= b_L = \frac{1}{k-1} \sum_{\substack{\{i,j\}\subset I \\ i<j}} b_{i,j} .\]
\end{notation}

\begin{lemma}\label{lem:compateq}
	Let \(n\geq2\) and let \(S: \C^{n+1} \to \C^{{n+1\choose 2}}\) be the linear map given by
	\[S(a_1, \ldots, a_{n+1}) = (a_i+a_j)_{i<j} . \]
	Then the following holds.
	\begin{itemize}
		\item[(i)] \(S\) is injective.
		\item[(ii)] If \(n=2\) then \(S: \C^3 \to \C^3\) is surjective. 
		\item[(iii)] If \(n \geq 3\) then \((b_{i,j})_{i<j}\) is in the image of \(S\) if and only if the compatibility equations
		\begin{equation} \label{eq:compat}
		b_{i,j} + b_{k,l} = b_{i,k} + b_{j,l} 
		\end{equation}
		hold for any quadruple \(\{i,j,k,l\} \subset \underline{n+1}\) of distinct indices.
	\end{itemize}
\end{lemma}

\begin{notation}
	In Equation \eqref{eq:compat} we allow the indices \(\{i,j,k,l\}\) run over all its \(24\) permutations, using that \(b_{q,p}= b_{p,q}\) in case \(q>p\). 
	For each quadruple \(\{i,j,k,l\}\) we obtain in total two linearly dependent relations:
	\[b_{i,j} + b_{k,l} = b_{i,k} + b_{j,l} = b_{i,l} + b_{j,k} \,\ \mbox{ for } \,\ i < j < k < l; \]
	which correspond to the three different ways of splitting a quadruple into two pairs. 
\end{notation}

\begin{proof}
	Item \((i)\).
	Suppose \(S(a_1, \ldots, a_{n+1})=0\) and let \(i<j<k\) with \(\{i,j,k\}\subset \underline{n+1}\). Then \(a_i+a_j=0\) together with \(a_i+a_k=0\) imply that \(a_j=a_k=-a_i\) and \(a_j+a_k=0\) gives us \(a_i=a_j=a_k=0\).
	
	Item \((ii)\). If \(n=2\) the domain and target of the linear map \(S\) have the same dimension. Since \(S\) is injective, it must be an isomorphism. It is also useful to write a formula for its inverse. Given \((b_{i,j})\), the solution to \(a_i+a_j=b_{i,j}\) is given by 
	\begin{equation}\label{eq:sol}
	a_i = \frac{1}{2} \left(b_{i,j} + b_{i,k} - b_{j,k} \right) 
	\end{equation}
	where \(\{i,j,k\} = \{1, 2, 3\}\).
	
	Item \((iii)\).
	If we can solve \(b_{i,j}=a_i+a_j\) then Equation \eqref{eq:compat} clearly holds, since both sides of the equation are equal to \(a_i+a_j+a_k+a_l\). On the other hand, 
	let \((b_{i,j})_{i<j}\) be given and suppose that Equation \eqref{eq:compat} is satisfied. Fix \(i \in \underline{n+1}\) and define \(a_i\) by Equation \eqref{eq:sol},
	where \(j,k\) are such that \(\{i,j,k\} \subset \underline{n+1}\) is a triple of distinct indices.
	The compatibility conditions \eqref{eq:compat} imply that the value of \(a_i\) given by \eqref{eq:sol} is independent of choices \(\{i,j',k'\}\) and \(a_i+a_j=b_{i,j}\) by construction. 
\end{proof}

\begin{lemma}\label{lem:uniq2}
	Proposition \ref{prop:UNIQ} holds when \(n=2\).
\end{lemma}

\begin{proof}
	Write \(b_{i,j} = \tr B_{i,j}\) as in Notation \ref{not:bij}.
	The torsion free condition \(\ker B_H=H\) implies 
	\[ 
	B_{1,3} = 
	\begin{pmatrix}
	b_{1,3} & 0 \\
	t_1 & 0
	\end{pmatrix}, \,\ 
	B_{2,3} = 
	\begin{pmatrix}
	0 & t_2 \\
	0 & b_{2,3}
	\end{pmatrix}, \,\ 
	B_{1,2} = 
	\begin{pmatrix}
	t_3 & -t_3 \\
	-t_4 & t_4
	\end{pmatrix} 
	 \]
	 for some \(t_1, \ldots, t_4 \in \C\) with 
	 \(t_3+t_4=b_{1, 2}\).
	 The flatness condition implies that the sum \(B_{1,3}+B_{2,3}+B_{1,2}\) commutes with each of its terms. Therefore it must preserve each of the three lines of the \(\mA_2\) arrangement and so must be a constant multiple of the identity. We get that
	 \[b_{1,3} + t_3 = b_{2,3} + t_4, \,\ t_2=t_3, \,\ t_1=t_4 . \]
	 In total we have four equations for \(t_1, \ldots, t_4\) which uniquely determine their values.
	 
	 To match with the Lauricella connection we take \((a_1,a_2,a_3)\) such that \(a_i+a_j = b_{i,j}\). We obtain  the equations
	 \begin{equation*}
	 	\begin{cases}
	 	t_3 + t_4 = a_1 + a_2 \\
	 	t_3 - t_4 = a_2 - a_1 
	 	\end{cases}
	 	,
	 	\qquad
	 	\begin{cases}
	 	t_1=t_4 \\
	 	t_2 = t_3
	 	\end{cases}
	 \end{equation*}
	 which give us \(t_1=t_4=a_1\) and \(t_2=t_3=a_2\), agreeing with Equation \eqref{eq:A2matrices}.
\end{proof}

\begin{figure}
	\centering
	\scalebox{0.8}{
		\begin{tikzpicture}
		
		\draw (-5,-1) to (5,-1);
		\draw[green, thick] (4.5,-1.75) to (-0.5,5.75);
		\draw[] (-4.5,-1.75) to (0.5,5.75);
		\draw (0,-2) to (0,6);
		\draw (-5,-1.5) to (3,2.5);
		\draw (5,-1.5) to (-3,2.5);
		
		\node[scale=1.2, red, thick] at (.8, 5) {\((i,j,k)\)};
		\node[scale=1, blue, thick] at (0, 6.2) {\((j,k)\)};
		\node[scale=1, blue, thick] at (-.9, 5.9) {\((i,j)\)};
		\node[scale=1, blue, thick] at (.9, 5.9) {\((i,k)\)};
		
		\node[scale=1.2, red, thick] at (-3.5, -1.4) {\((i,k,l)\)};
		\node[scale=1, blue, thick] at (-5.4, -1) {\((i,l)\)};
		\node[scale=1, blue, thick] at (-5.3, -1.5) {\((k,l)\)};
		\node[scale=1, blue, thick] at (-4.8, -1.9) {\((i,k)\)};

		\node[scale=1.2, red, thick] at (3.5, -1.4) {\((i,j,l)\)};
		\node[scale=1, blue, thick] at (5.4, -1) {\((i,l)\)};
		\node[scale=1, blue, thick] at (5.3, -1.5) {\((j,l)\)};
		\node[scale=1, blue, thick] at (4.8, -1.9) {\((i,j)\)};

		%\node[scale=1.5] at (-.3,.5) {\(p_l\)};
		
		\node[scale=1, blue, thick] at (-3.5,2.5) {\((j,l)\)};
		\node[scale=1, blue, thick] at (3.5,2.5) {\((k,l)\)};
		%\node[scale=1.5] at (.5,-2) {\(H_{jk}\)};
		
		%\draw (0,-1) circle [radius=.15];
		%\draw (-2,2) circle [radius=.15];
		%\draw (2,2) circle [radius=.15];
		\end{tikzpicture}
	}
	\caption{The projectivized \(\mA_3\)-arrangement consists of six lines in \(\CP^2\), it is known as the complete quadrilateral. The hyperplane \(H_{i,j}\) is highlighted in green.}
	\label{fig:completequad}
\end{figure}
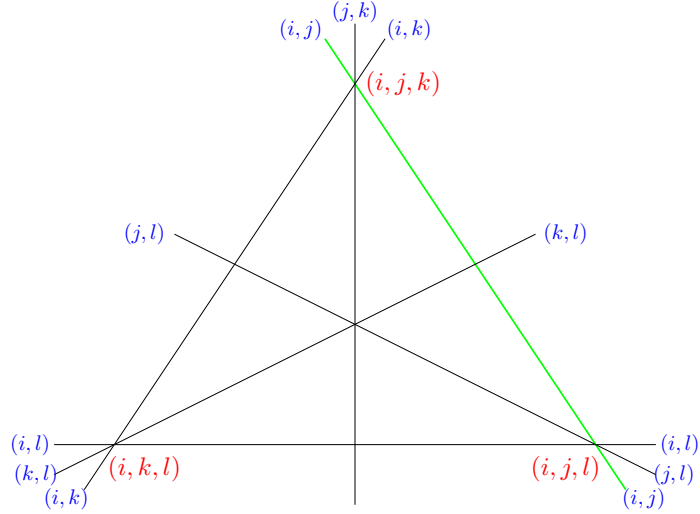

\begin{lemma}\label{lem:completequad}
	In the setting of Proposition \ref{prop:UNIQ} with \(n=3\) there is \((a_1, \ldots, a_4) \in \C^4\) such that \(b_{i,j}=a_i+a_j\).
\end{lemma}

\begin{proof}
	Let \(\nabla\) be the connection given by Equation \eqref{eq:Bconnection}.
	Fix \(H= H_{i,j} \in \mA_3\) and consider the induced connection \(\nabla'' =\nabla^H\) (see Section \ref{sect:indcon}) obtained by restricting \(\nabla\) to the arrangement \((\mA_3)^H \equiv \mA_2\). Let \(k,l\) be complementary indices so that \(\{i,j,k,l\}=\{1,2,3,4\}\). The weights of \(\nabla^H\) on \((\mA_3)^H\) are 
	\begin{equation*} %\label{eq:weightsij} 
		w= \{ b_{i,j,k}, \,\ b_{k,l} , \,\ b_{i,j,l}\} 
	\end{equation*}
	as follows from Equation \eqref{eq:aH''}, see Figure \ref{fig:completequad}. Similarly, if we take \(H'=H_{i,k}\), the weights of the induced connection \(\nabla^{H'}\) on \((\mA_3)^{H'}\) are 
	\begin{equation*} %\label{eq:weightsik}  
	w' = \{ b_{i,j,k}, \,\ b_{j,l} , \,\ b_{i,k,l} \} . 
	\end{equation*}
	
	Note that \(\{0\}\) is an irreducible intersection of \((\mA_3)^H\) and its weight with respect to the induced connection \(\nabla^H\) is equal to the sum of the elements in \(w\) divided by two. Since \(\{0\}\) is also an irreducible intersection of \(\mA_3\), its weight can also be computed as \((1/3)\sum_{i<j}b_{i,j}\) (see Corollary \ref{cor:linear2}). The same reasoning applies to \((\mA_3)^{H'}\). We conclude that the sum of the elements in \(w\) equals the sum of the elements in \(w'\). After cancelling the common term \(b_{i,j,k}\), this gives us
	\begin{equation*}
		b_{k,l} + \frac{b_{i,j} + b_{i,l} + b_{j,l}}{2} =
		b_{j,l} + \frac{b_{i,k} + b_{i,l} + b_{k,l}}{2}
	\end{equation*} 
	which easily rearranges to Equation \eqref{eq:compat}. The statement follows from Lemma \ref{lem:compateq}.
\end{proof}

\begin{lemma}
	In the setting of Proposition \ref{prop:UNIQ} with \(n\geq 4\) there is \((a_1, \ldots, a_{n+1}) \in \C^{n+1}\) such that \(b_{i,j}=a_i+a_j\).
\end{lemma}

\begin{proof}
	Let \(I=\{i,j,k,l\} \subset \underline{n+1}\) and consider the corresponding \(L \in \mL_{irr}(\mA_n)\) given by the intersection of all hyperplanes in
	\[(\mA_n)_L = \{ H_{p,q}, \,\ \{p,q\} \subset I \} .\]
	Note that \(\codim_{\C^n}L=3\) and \((\mA_n)_L/L \equiv \mA_3\). The quotient connection \(\nabla^{\C^n/L}\) (see Definition \ref{def:quotconnect}) has weights \(b_{i,j}\) along \(H_{i,j}/L\). By Lemma \ref{lem:completequad} there are \(a_i, a_j, a_k, a_l\) such that \(b_{p,q}=a_p+a_q\) for any pair \(\{p,q\} \subset \{i,j,k,l\}\). In particular, the relation \(b_{i,j}+b_{k,l}=b_{i,k}+b_{j,l}\) holds.  The statement follows from Lemma \ref{lem:compateq}-\((iii)\).
\end{proof}

The following lemma finishes the proof of Proposition \ref{prop:UNIQ}.

\begin{lemma}
	Let \(B_{i,j}\) be as in Proposition \ref{prop:UNIQ}. Let \(a \in \C^{n+1}\) be such that \(\tr B_{i,j} =a_i+a_j\). Let \(A_{i,j}\) be the reduced Jordan-Pochammer matrices \eqref{eq:pochjorRED1} and \eqref{eq:pochjorRED2} of the Lauricella connection \(\nabla^a\). Then \(B_{i,j}=A_{i,j}\).
\end{lemma}

\begin{proof}
	We know that \(\ker B_{i,j} = \ker A_{i,j}\). The matrices \(B_{i,j}\) and \(A_{i,j}\) have one dimensional \((a_i+a_j)\)-eigenspaces \(\C \cdot N_{i,j}\) and \(\C \cdot n_{i,j}\). We will show that \(\C \cdot N_{i,j} = \C \cdot n_{i,j}\) by induction, the case \(n=2\) being settled in Lemma \ref{lem:uniq2}.
	
	Assume \(n\geq 3\). Fix \(i,j\) and take \(\{k,l\} \subset \underline{n+1}\) such that \(\{i,j\} \cap \{k,l\} = \emptyset\).
	Write \(H= H_{k,l}\). By Lemma \ref{lem:trint} both \(N_{i,j}\) and \(n_{i,j}\) belong to \(H\). Let \(\mH\) be the induced arrangement \((\mA_n)^{H}\), so \(\mH \equiv A_{n-1}\). The induced connections \(\nabla^H\) and \((\nabla^a)^H\) have the same weights on members of \(\mH\) because the weights of induced connections only depend on weights of the given connections and not on the normal directions, see Lemma \ref{lem:inducedweights}.
	By induction \(\nabla^H\) is a Lauricella connection, as well as \((\nabla^a)^H\). Since Lauricella connections are uniquely determined by its weights at hyperplanes, see Lemma \ref{lem:compateq}-(i) and Equation \ref{eq:redLaucon}, we conclude that \(\nabla^H=(\nabla^a)^H\).
	The respective residues of \(\nabla^H\) and \((\nabla^a)^H\) at the hyperplane \(H_{i,j} \cap H \in \mH\) are equal to the restrictions \(B_{i,j}|_{H}\) and \(A_{i,j}|_{H}\), as follows from Remark \ref{rmk:inducedresidues} and the fact that \(H_{i,j} \pitchfork H\). Therefore, since the restrictions of the connections \(\nabla\) and \(\nabla^a\) to \(H\) agree, we must have \(B_{i,j}|_{H}=A_{i,j}|_{H}\). In particular,
	\[B_{i,j} n_{i,j} = (a_i + a_j) n_{i,j} .\]
	Since \(n_{i,j}\neq 0\) we conclude that \(\C \cdot n_{i,j} = \C \cdot N_{i,j}\). 
\end{proof}

\subsection{Flat Hermitian form}\label{sect:flathermform}

The fundamental group of \(\C^n \setminus \mA_n\) is the pure braid group on \(n+1\) strands, denoted by \(P_{n+1}\). We write
\begin{equation}\label{eq:gasrep}
	\Hol(a): P_{n+1} \to GL(n, \C) 	
\end{equation}
for the holonomy representation of the reduced Lauricella connection \(\nabla^a\) with parameters \(a=(a_1, \ldots, a_{n+1})\). Let's fix a base point \(p \in (\C^n)^{\circ}\),  generators for the fundamental group \(\pi_1((\C^n)^{\circ}, p) =  P_{n+1}\) and a suitable basis of \(T_p \C^n\). The homomorphism \ref{eq:gasrep} can be written in terms of explicit holonomy matrices and is identified with the `reduced Gassner representation' specialized at \(\exp(2\pi i a_i)\); see \cite[Chapter 8]{OT2} and also \cite{KM} for the case of purely imaginary parameters \(a_i\).

\begin{fact}\label{fact1}
	If \(a- a' \in \Z^{n+1}\) then \(\Hol(a)\) and \(\Hol(a')\) are conjugate.
\end{fact}
Fact \ref{fact1}
follows from the matrix representation (with respect to a suitable basis of \(T_p\C^n\) depending on \(a\)) presented in \cite[pg. 76]{OT2}. The matrix entries only depend on the values \(\exp(2\pi ia_i)\) - which the authors denote by \(c_i\) .

\begin{fact}[{\cite{MS}}]\label{fact2}
	The holonomy representation \eqref{eq:gasrep} is irreducible
	if and only if \(a_i \notin \Z\) for all \(i=1, \ldots, n+1\) and \(a_1+\ldots +a_{n+1} \notin \Z\).
\end{fact}

 \begin{notation}
 	We introduce
 	\[a_{n+2} = 2 - \sum_{i=1}^{n+1}a_i . \]
 	We will also denote \(a_{\infty}=a_{n+2}\).
 \end{notation}

\begin{fact}\label{fact3}
	If \(a_i \in \R\) for all \(i\) then \eqref{eq:gasrep} leaves invariant a natural Hermitian form \(\inn = \inn_a\) with the following properties.
	\begin{itemize}
		\item[(i)] The kernel of \(\inn\) has dimension equal to the number of integer values in the sequence \(a_1, \ldots, a_{n+2}\).
		\item[(ii)] The Hermitian form \(\inn\) has signature \((p,q)\)
		with
		\begin{equation}\label{eq:sgn}
			p = -1 + \sum_{i=1}^{n+2}\{a_i\}, \qquad q = -1 + \sum_{i=1}^{n+2}\{1-a_i\} . 	
		\end{equation}
	\end{itemize}
\end{fact}

\begin{note}
	In this article the fractional part of a real number \(a\) is defined by \(\{a\} = a - \floor{a}\). It is the \(\Z\)-periodic extension of the identity function \(a \mapsto a\) for \(0 \leq a <1\), for example \(\{-0.7\}=0.3\). Note that in Equation \eqref{eq:sgn} we have included the weight at infinity \(a_{n+2}\) in the sum so the term \(\sum_{i=1}^{n+2}\{a_i\}\) is an integer.
\end{note}

Fact \ref{fact3} is a consequence of \cite[Lemmas 3.21 and 3.22]{CHL}. Notice that \cite[Lemma 3.21]{CHL} requires \(0\leq a_i<1\) for all \(i=1, \ldots, n+2\) and an analytic continuation argument is used to deal with the case when \(a_{n+2} \notin (0,1)\), see also \cite[Section 1.6]{Loo}. Subtracting an integer vector from \(a \in \R^{n+1}\) we can assume that \(0 \leq a_i<1\) for \(i=1, \ldots, n+1\) and the case of arbitrary real parameters \(a_i\) follows from Fact \ref{fact1}. 

Equation \eqref{eq:sgn} for the signature also appears in \cite[Corollary 2.21]{DM}. In particular, if none of \(a_1, \ldots, a_{n+2}\) are integers then \(\inn\) is non-degenerate and unique up to a constant factor because of Fact \ref{fact2}.

\subsubsection{{\large Picturing Equation \eqref{eq:sgn} for the signature}}

The expression for \(p\) given by Equation \eqref{eq:sgn} defines an integer-valued function
\begin{align*}
\R^{n+1} &\to \Z \\
a\hspace{4mm} &\mapsto p(a) .
\end{align*}
Here \(a=(a_1, \ldots, a_{n+1})\) and we have eliminated the dependence on \(a_{n+2}\) by using that \(a_{n+2} = 2 - \sum_{i=1}^{n+1}a_i\). This way \(p\) is periodic
\begin{equation}\label{eq:periodic}
	p\left(a+\Z^{n+1}\right)=p(a) 
\end{equation}
and we can easily analyse its values on the interior of the unit cube \(Q=(0,1)^{n+1}\) where \(0<a_i<1\) for all \(1\leq i \leq n+1\). If we let 
\[s=\sum_{i=1}^{n+1}a_i\] 
then \(0<s<n+1\) on \(Q\) and Equation \eqref{eq:sgn} reads as follows
\begin{equation}\label{eq:pq}
	p = \floor*{s} \in \{0, \ldots, n\}, \qquad q = n-p .	
\end{equation}

\begin{definition}\label{def:Tpm}
	We introduce the open simplexes
	\[T_- = \{s<1\} \cap Q, \qquad T_+ = \{s>n\} \cap Q  .   \]	
\end{definition}

\begin{figure}
	\scalebox{0.7}{
	\includegraphics[]{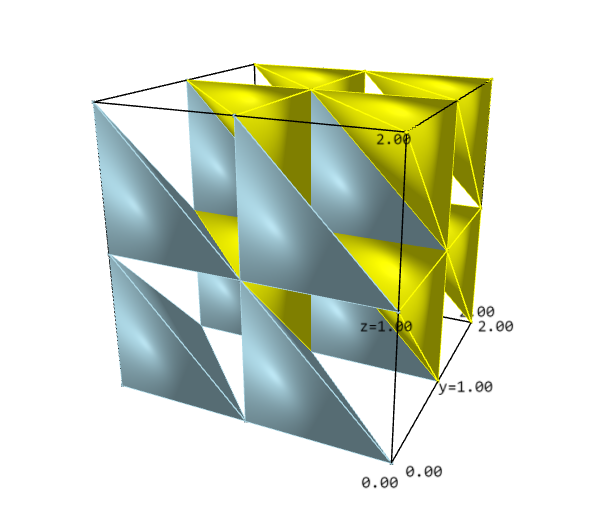}
	}
	\centering
	\caption{The simplices \(T_{-} + \Z^{n+1}\) and \(T_{+} + \Z^{n+1}\) illustrated in different colours.}
\end{figure}

\begin{corollary}\label{cor:sgn}
	Let \(a=(a_1, \ldots, a_{n+1}) \in \R^{n+1}\) and set
	\[\{a\}= \left(\{a_1\}, \ldots, \{a_{n+1}\}\right).\]
	Then the Hermitian form \(\inn_a\) in Fact \ref{fact3} is non-degenerate and definite if and only if one of the following holds:
	\begin{itemize}
		\item either \(\{a\} \in T_-\), in which case \(\inn_a<0\);
		\item or \(\{a\} \in T_+\), in which case \(\inn_a >0\).
	\end{itemize} 
\end{corollary}

\begin{proof}
	Follows from Equations \ref{eq:pq} and \eqref{eq:periodic} together with Definition \ref{def:Tpm}.
\end{proof}

\subsection{Proof of Theorem \ref{LAUTHM}}\label{sect:pfLau}
 
In this section we show several results that settle Theorem \ref{LAUTHM}, for the record the proof is as follows.

\begin{proof}[Proof of Theorem \ref{LAUTHM}]
	Existence follows from Proposition \ref{prop:existence}. Uniqueness follows from Proposition \ref{prop:uniq}. The fact that the three numerical conditions \eqref{noninteger}, \eqref{positivity} and \eqref{signature} must be obeyed by any such PK cone metric follows from Proposition \ref{prop:nonexistence}.
\end{proof}

\subsubsection{{\large Preliminary numerical lemmas}}

\begin{lemma}\label{lem:hereditary}
	The conditions \eqref{noninteger}, \eqref{positivity} and \eqref{signature} are hereditary to subsets. More precisely, if  \(\alpha_1, \ldots, \alpha_{n+1}\) satisfy \eqref{noninteger}, \eqref{positivity}, \eqref{signature} and \(I \subset \{1, \ldots, n+1\}\) then \(\{\alpha_i, \,\ i \in I\}\) satisfy \eqref{noninteger}, \eqref{positivity}, \eqref{signature}.
\end{lemma}

\begin{proof}
	Enough to check it for \eqref{signature}. If \(\sum_{i=1}^{n+1} \{\alpha_i\} < 1\) then, since \(\{\alpha_i\} \geq 0\) for every \(i\), we have
	\[\sum_{i \in I} \{\alpha_i\} \leq \sum_{i=1}^{n+1} \{\alpha_i\} < 1 . \]
	If \(\sum_{i=1}^{n+1} \{\alpha_i\} > n\) then, since \(\{\alpha_i\} \leq 1\) for every \(i\), we have
	\[n < \sum_{i=1}^{n+1} \{\alpha_i\} \leq \sum_{i \in I} \{\alpha_i\} + n + 1 - |I| . \]
	Rearranging, we obtain \(\sum_{i \in I} \{\alpha_i\} > |I|-1\).
\end{proof}

\begin{lemma}\label{lem:P}
	Suppose that \(\alpha_1, \ldots, \alpha_{n+1}\) satisfy \eqref{noninteger}, \eqref{positivity} and \eqref{signature}. Then for every \(I \subset \{1, \ldots, n+1\}\) with \(|I| \geq 2\) we have
	\[ \sum_{i \in I} \alpha_i > |I| -1 . \]
\end{lemma}

\begin{proof}
	The proof is by induction on \(|I|\), it holds for \(|I|=2\) by \eqref{positivity}. Let \(\alpha_{max} = \max \{\alpha_i, \,\ i \in I\}\). If \(\alpha_{max} \geq 1\) then
	\begin{align*}
	\sum_{i \in I} \alpha_i &= \left( \sum_{i \in I\setminus\{max\}} \alpha_i \right) + \alpha_{max} \\
	& > (|I|-2) + 1 = |I|-1 .	
	\end{align*}
	Otherwise \(\alpha_i<1\) for every \(i \in I\). For every \(j \in I\) we take \(i \in I\) with \(i\neq j\), since \(\alpha_i+\alpha_j>1\) and \(\alpha_i<1\) we conclude that \(\alpha_j>0\). Therefore \(0<\alpha_i<1\) for every \(i \in I\), i.e. \(\alpha_i=\{\alpha_i\}\) for every \(i \in I\). If we take a pair \(i, j \in I\) with \(i \neq j\) then
	\[\{\alpha_i\} + \{\alpha_j\} = \alpha_i + \alpha_j > 1 .\]
	The signature condition \eqref{signature} can only hold with \(\sum_{i=1}^{n+1} \{\alpha_i\} > n\). By hereditary
	\[\sum_{i \in I} \alpha_i = \sum_{i \in I} \{\alpha_i\} > |I|-1 .  \qedhere\] 
\end{proof}

\begin{lemma}\label{lem:NI}
	Suppose that \(\alpha_1, \ldots, \alpha_{n+1}\) satisfy \eqref{noninteger}, \eqref{positivity} and \eqref{signature}. Then
	for every \(I \subset \{1, \ldots, n+1\}\) with \(|I| \geq 2\) we have
	\[ \sum_{i \in I} \alpha_i \notin \Z . \]
\end{lemma}

\begin{proof}
	Let \(I \subset \{1, \ldots, n+1\}\) with \(|I| \geq 2\).
	By \eqref{noninteger} we have \(\{\alpha_i\}>0\) for every \(i \in I\). If \(\sum_{i \in I} \alpha_i \in \Z\) then \(\sum_{i \in I} \{\alpha_i\}\) must be a \emph{positive} integer. In particular, \(\sum_{i \in I} \{\alpha_i\} \geq 1\). By hereditary, i.e. Lemma \ref{lem:hereditary}, condition \eqref{signature} implies that \(\sum_{i \in I}\{\alpha_i\}>|I|-1\). Since \(\sum_{i \in I} \{\alpha_i\}\) is an integer we must have \(\sum_{i \in I} \{\alpha_i\} \geq |I|\) but this is impossible because \(\{\alpha_i\}<1\) for each \(i\).
\end{proof}

\subsubsection{{\large Existence and Uniqueness}}

Let \(\alpha_1, \ldots, \alpha_{n+1} \in \R\). We introduce the following conditions:
\begin{itemize}
	\item non-integer: for every \(1 \leq i \leq n+1\)
	\begin{equation} \label{noninteger}
	\tag{N.I.}
	\alpha_i \notin \Z ;
	\end{equation}
	
	\item positivity: for every \(1 \leq i < j \leq n+1\)
	\begin{equation} \label{positivity}
	\tag{P}
	\alpha_i + \alpha_j > 1 ;
	\end{equation}
	
	\item signature:
	\begin{equation} \label{signature}
	\tag{S}
	\sum_{i=1}^{n+1} \{\alpha_i\} < 1, \mbox{ or } 	\sum_{i=1}^{n+1} \{\alpha_i\} > n .
	\end{equation}
\end{itemize}	

\begin{remark}
	Condition \eqref{positivity} implies that there is at most one \(1 \leq i \leq n+1\) such that \(\alpha_i<0\).
\end{remark}

\begin{proposition}[Existence]\label{prop:existence}
	Suppose that \(\alpha_1, \ldots, \alpha_{n+1}\) satisfy \eqref{noninteger}, \eqref{positivity} and \eqref{signature}.
	Then there is a PK cone metric on \(\C^n\) with cone angles \(2\pi(\alpha_i + \alpha_j -1)\) along the hyperplanes \(H_{i,j} \in \mA_n\) of the essential braid arrangement.
\end{proposition}

\begin{proof}
	Consider the Lauricella connection \(\nabla^a\)
	with parameters \(a=(a_1, \ldots, a_{n+1})\) where \(a_i=1-\alpha_i\). Equation \eqref{signature} together with Corollary \ref{cor:sgn} imply that \(\nabla^a\) has unitary holonomy. It remains to verify the non-integer and positivity assumptions of Theorem \ref{PRODTHM}.
	
	Let \(L \in \mL_{irr}(\mA_n)\), by Lemma \ref{lem:weightsRedLau} its weight equals 
	\[a_L = \sum_{i \in I} a_i \hspace{2mm} \mbox{ for some } \hspace{2mm} I \subset \{1, \ldots, n+1\} . \]
	Lemmas \ref{lem:NI} and \ref{lem:P}  imply that the non-integer and positivity assumptions of Theorem \ref{PRODTHM} are satisfied. Hence, the existence of the PK cone metric follows from Theorem \ref{PRODTHM}. 
	The statement about the cone angles follows from the equation \(a_{i,j} = a_i + a_j\) together with Corollary \ref{cor:locprod} applied to \(H=H_{i,j}\).  
\end{proof}

\begin{proposition}[Uniqueness]\label{prop:uniq}
	Let \(g\) and \(g'\) be two PK cone metrics on \(\C^n\) with the same cone angles \(2\pi\beta_{i,j}\) along the hyperplanes \(H_{i,j} \in \mA_n\). Suppose that there are \((\alpha_1, \ldots, \alpha_{n+1}) \in \R^{n+1}\) satisfying \eqref{noninteger}, \eqref{positivity} and \eqref{signature} such that
	\[\beta_{i,j} = \alpha_i + \alpha_j -1 . \]
	Then \(g' = \lambda g\) for some constant factor \(\lambda>0\).
\end{proposition}

\begin{proof}
	Consider the Levi-Civita connections \(\nabla\) and \(\nabla'\) of \(g\) and \(g'\). By Proposition \ref{prop:PKstandard}, the flat torsion free connections \(\nabla\) and \(\nabla'\) are standard. Their weights are given by \(a_{i,j} = 1-\beta_{i,j}\) at \(H_{i,j} \in \mA_n\). If we let \(a_i= 1 - \alpha_i\) then \(a_{i,j}=a_i + a_j\). It follows from Lemmas \ref{lem:weightsRedLau} and \ref{lem:NI} that \(\nabla\) and \(\nabla'\) satisfy the Non-Zero Weights Assumptions \ref{ass:nz}. Proposition \ref{prop:UNIQ} implies that
	\[ \nabla = \nabla' = \nabla^a\]
	where \(\nabla^a\) is the the reduced Lauricella connection with parameters \(a=(a_1, \ldots, a_n)\). Theorem \ref{thm:irredhol} (or alternatively Fact \ref{fact2}) implies that \(\Hol(\nabla^a)\) is irreducible and therefore by Corollary \ref{cor:uniqueinnerprod} we must have \(g'=\lambda g\) for some \(\lambda>0\).
\end{proof}

\subsubsection{{\large Non-existence and holes}}

\begin{proposition}\label{prop:nonexistence}
	Let \(g\) be a PK cone metric on \(\C^n\) singular at the \(\mA_n\)-arrangement. Suppose that \(g\) has non-integer cone angles \(\beta_L\) at all \(L \in \mL_{irr}(\mA_n)\). Then there are uniquely determined real numbers \(\alpha_1, \ldots, \alpha_{n+1}\) such that \(\beta_{H_{i,j}} = \alpha_i + \alpha_j -1\). Moreover, the numerical conditions \eqref{noninteger}, \eqref{positivity} and \eqref{signature} are satisfied.
\end{proposition}

\begin{proof}
	Let \(\nabla\) be the Levi-Civita connection of \(g\). It follows from Proposition \ref{prop:PKstandard} that \(\nabla\) is standard. We will verify next that it satisfies the positivity and non-integer conditions of Theorem \ref{PRODTHM}. 
	
	Since \(g\) is a polyhedral metric on \(\C^n\) its volume form is clearly locally integrable\footnote{Because compact sets of \(\C^n\) have finite measure with respect to \(g\).}, it follows from Lemmas \ref{lem:volumeform} and \ref{lem:positivityvolume} that \(\nabla\) satisfies the positivity assumption \(1-b_L>0\) for every \(L \in \mL_{irr}(\mA_n)\). On the other hand, by Corollary \ref{cor:weightangle}, the weights \(b_L\) of \(\nabla\) are given in terms of the angles \(\beta_L\) by \(b_L=1-\beta_L\). By assumption, \(b_L \notin \Z\).
	We conclude that \(\nabla\) satisfies all hypothesis of Theorem \ref{PRODTHM}.
	
	The residues \(B_{i,j}\) of \(\nabla\) at \(H_{i,j} \in \mA_n\) satisfy the Non-Zero Weights Assumptions \ref{ass:nz} because the weights \(b_L\) are non-integer and in particular \(\neq 0\). It follows from Proposition \ref{prop:UNIQ} that there is a uniquely determined \(a=(a_1, \ldots, a_{n+1}) \in \C^{n+1}\) such that \(\nabla\) equals the reduced Lauricella connection \(\nabla^a\) with parameters \(a\). In particular, \(b_{H_{i,j}} = a_i + a_j\). Since \(b_{H_{i,j}}\) are real, we must have \(a_i \in \R\) for all \(i\). Since
	\(\beta_{H_{i,j}}=1-b_{H_{i,j}}\) we obtain
	\begin{equation}\label{eq:angleadition}
		\beta_{H_{i,j}}=\alpha_i+\alpha_j-1
	\end{equation}
	where \(\alpha_i = 1-a_i\).
	
	By Theorem \ref{thm:irredhol}, the holonomy of \(\nabla\) is irreducible. By Fact \ref{fact2} in Section \ref{sect:flathermform} we have \(\alpha_i \notin \Z\) for all \(i\), so the non-integer condition \ref{noninteger} is satisfied. Since \(\beta_{H_{i,j}}>0\), it follows from Equation \eqref{eq:angleadition} that the positivity condition \eqref{positivity} is satisfied. Finally, since the holonomy of \(\nabla\) is irreducible, any parallel Hermitian form must be a scalar multiple of the one given by \(g\); in particular it must be definite. It follows from Fact \ref{fact3} in Section \ref{sect:flathermform} that the signature condition \eqref{signature} must be satisfied.
\end{proof}

The rest of the section makes connection with the classical case of spherical metrics on \(\CP^1\) with three cone points by giving an alternative proof of
the following.

\begin{proposition}\label{prop:holes}
	Let \(\nabla^a\) be the reduced Lauricella connection on \(\C^n\) with real parameters \(a=(a_1, \ldots, a_{n+1})\) satisfying the positivity and non-integer assumptions
	\begin{equation}
		\sum_{i \in I} a_i < 1  \,\ \mbox{ and } \,\ \sum_{i \in I} a_i \notin \Z 
	\end{equation}
	for every \(I \subset \underline{n+1}\).
	If \(\nabla^a\) has unitary holonomy then \(\sum_{i=1}^{n+1} \{a_i\}\) must be either \(<1\) or \(>n\). 
\end{proposition}

Proposition \ref{prop:holes} follows from
Theorem \ref{PRODTHM} and Proposition \ref{prop:nonexistence}. We will provide an alternative proof without using Facts \ref{fact2} and \ref{fact3} from Section \ref{sect:flathermform}.
We argue by induction on \(n\), taking the case \(n=2\) of spherical triangles as granted. The main idea is to look at induced metrics and complex links of singularities. The next case after three points on the sphere is the complete quadrilateral in \(\CP^2\). We can then either restrict a FS metric to any of the six lines or look at the complex link of any of the four triple points. This gives spherical metrics with three cone points and therefore angle restrictions on the original metric on \(\CP^2\). The same reasoning applies inductively.
\newline

\noindent \textbf{Spherical metrics on \(\CP^1\) with three cone points.}
We denote
\[\alpha_i = 1- a_i \]
where \(a_i\) are as in Proposition \ref{prop:holes}.
For \((i,j,k)\) cyclic permutations of \((1,2,3)\) we write
\begin{equation} \label{eq:betamap}
	\beta_i = \alpha_j + \alpha_k -1 .
\end{equation}

The affine map \(\R^3 \to \R^3\) given by \eqref{eq:betamap} doubles volume and has the following properties:
\begin{enumerate}
	\item It is equivariant with respect to the translation actions of \(\mathbb{Z}^3\) on \(\alpha\) and \(\mathbb{Z}^3_e\) (integer vectors such that the sum of their components is even) on \(\beta\); it maps the centre of cubes lattice \((1/2,1/2,1/2) + \mathbb{Z}^3\) to the even integer lattice \(\mathbb{Z}^3_e\). 
	\item It maps the tetrahedron \(T_-=\{\alpha_1+\alpha_2+\alpha_3>2, \,\ 0<\alpha_i<1\}\) to the open convex hull of \(\{(1,0,0), (0,1,0), (0,0,1), (1,1,1)\}\) and the three translates of \(T_+=\{\alpha_1+\alpha_2+\alpha_3<1, \,\ 0<\alpha_i\}\) by \((1,1,0), (0,1,1), (1,0,1)\) to the three reflections of \(\beta(T)\) across the \(\beta_i=1\) faces of the unit cube.
	\item \(O = [0,1]^3 \setminus (T_+ \cup T_-)\) is a solid octahedron, six faces are right triangles left out from the cube and the other two faces are equilateral triangles of length side \(\sqrt{2}\). The transformation \eqref{eq:betamap} maps \(\mathbb{Z}^3 + O\) to regular octahedrons centred at the even lattice points.
\end{enumerate}

Given the above statements, the content of Proposition \ref{prop:holes} for \(n=2\) follows from the correspondence between PK cone metrics on \(\C^2\) singular along three complex lines and spherical cone metrics on \(\CP^1\) with three cone points together with
the well-known angle constraints for the existence of such spherical metrics (see \cite{UY}, \cite{Ere} and \cite[Section 3.1.2]{MP}).
\newline

\noindent
\textbf{Propagation of holes to higher dimensions.}
Assume \(n \geq 3\), the case \(n=2\) is taken for granted.

\begin{lemma}\label{lem:2cases}
	We have the following two cases:
	\begin{equation}\label{eq:CaseA}
	\tag{A}
	\sum_{j \neq i} \{a_j\} < 1 \hspace{2mm} \mbox{for all } i=1, \ldots, n + 1
	\end{equation}
	or
	\begin{equation*}\label{eq:CaseB} 
	\tag{B}
	n -1 < \sum_{j \neq i} \{a_j\} \hspace{2mm} \mbox{for all } i=1, \ldots, n + 1 .
	\end{equation*}
\end{lemma}

\begin{proof}
	Let \(\ell_i \in \mL_{irr}(\mA_n)\) be the complex line in \(\C^n\) given by
	\[\ell_i = \bigcap_{j,k \neq i} H_{j,k} , \]
	it is an irreducible intersection of \(\mA_n\). The quotient arrangement \(\mA_n/\ell_i\) is isomorphic to \(\mA_{n-1}\) and the quotient connection \(\nabla^{\C^n/\ell_i}\) agrees with the Lauricella connection with parameters 
	\[a_1, \ldots, a_{i-1}, a_{i+1}, \ldots, a_{n+1} .\] 
	The quotient connection is also unitary, see Remarks \ref{rmk:localizationunitary} and \ref{rmk:localizationconection}. By induction, either
	\begin{equation}\label{induction}
		\sum_{j \neq i} \{a_j\} < 1 \,\ \mbox{ or } \,\ n-1 < \sum_{j \neq i} \{a_j\} .	
	\end{equation}
	On the other hand, if \(i \neq j\) then
	\begin{equation}\label{elementary}
		\left| \sum_{k \neq i} \{a_k\} - \sum_{k \neq j} \{a_k\} \right| < 1	.
	\end{equation}
	Since \(n \geq 3\), the statement of the lemma follows from Equation \eqref{induction} and \eqref{elementary}.
\end{proof}

\begin{proof}[Proof of Proposition \ref{prop:holes}]
	For \(i < j\) we let \(\sigma_{ij}\) be the complex plane in \(\C^n\) spanned by \(\ell_i\) and \(\ell_j\). The plane \(\sigma_{ij}\) belongs to \(\mL_{irr}(\mA_n)\), it is the common intersection of all hyperplanes \(H_{k,l}\) with \(\{k,l\} \cap \{i,j\}=\emptyset\). The induced arrangement \(\mA_n^{\sigma_{ij}}\) is isomorphic to \(\mA_2\) and the induced connection is Lauricella with parameters
	\[\tilde{a}_1 = a_i, \hspace{2mm} \tilde{a}_2=a_j, \hspace{2mm} \tilde{a}_3 = \sum_{k \neq i, j} a_k  . \]
	The induced connection is unitary, see Lemma \ref{lem:unitaryinducedconect}. 
	Therefore, from the \(n=2\) case we know that
	\begin{equation}\label{eq:3angles}
	\{\tilde{a}_1\} + \{\tilde{a}_2\} + \{\tilde{a}_3\} < 1 \hspace{2mm} \mbox{or } \{\tilde{a}_1\} + \{\tilde{a}_2\} + \{\tilde{a}_3\} > 2 .
	\end{equation}
	Next, we consider the two cases of Lemma \ref{lem:2cases}.
	
	\textbf{Case \eqref{eq:CaseA}:} \(\sum_{i \in I} \{a_i\} < 1\) whenever \(|I|=n\). In this case \(\{\tilde{a}_3\} = \sum_{k \neq i, j}\{a_k\}\) and 
	\[\{\tilde{a}_1\} + \{\tilde{a}_2\} + \{\tilde{a}_3\} = \sum_i \{\alpha_i\} . \]
	It follows from  \(\sum_i \{a_i\} < 2\) and \eqref{eq:3angles} 
	that we must have \(\sum_i \{a_i\} < 1\).

	\textbf{Case \eqref{eq:CaseB}:} \(\sum_{i \in I} \{a_i\} > n-1\) whenever \(|I|=n\). In this case \(\sum_{k \neq i, j}\{a_k\} > n - 2\) and 
	\[\{\tilde{a}_1\} + \{\tilde{a}_2\} + \{\tilde{a}_3\} = \sum_i \{a_i\} -n + 2 . \]
	It follows from  \(\sum_i \{\alpha_i\} > n-1\) that \(\{\tilde{a}_1\} + \{\tilde{a}_2\} + \{\tilde{a}_3\} > 1\) and \eqref{eq:3angles} therefore implies that \(\{\tilde{a}_1\} + \{\tilde{a}_2\} + \{\tilde{a}_3\} > 2\), equivalently \(\sum_i \{a_i\} > n\).	
\end{proof}

\begin{remark}
	The volume of the standard \((n+1)\)-simplex \(\{a_i>0, \,\, \sum_i a_i < 1\} \subset \R^{n+1}\) is equal to \(1/(n+1)!\) and the volume of the non-existence hole in the unit cube is \(1-2/(n+1)!\). We see that as the dimension increases the existence region becomes more scarce.
\end{remark}

\subsection{FS metrics singular at the  \(\mA_n\)-arrangement }\label{sect:FSLau}

Recall that we have defined adapted FS metrics on \(\CP^{n-1}\) as the complex links of regular PK cone metrics on \(\C^n\) singular at a linear arrangement, see Definition \ref{def:FS}. We can restate our results in terms of FS metrics as follows.

\begin{theorem}\label{thm:FSLau}
	Let \(\alpha_1, \ldots, \alpha_{n+1}\) be real numbers satisfying the numerical conditions \ref{positivity}, \ref{noninteger} and \ref{signature}. Then there is a unique FS metric adapted to the pair
	\begin{equation}\label{eq:FSpair}
		\left(\CP^{n-1}, \sum_{i<j} (1-\beta_{i,j})\bar{H}_{i,j} \right)
	\end{equation}
	where \(\bar{H}_{i,j}\) are the hyperplanes of the projectivized \(\mA_n\) arrangement and \(\beta_{i,j}=\alpha_i+\alpha_j-1\).
	Moreover, any FS metric adapted to the pair \eqref{eq:FSpair} with non-integer angles at all irreducible intersections must be equal to one given by the above family.
\end{theorem}

\begin{proof}
	Follows from Theorem \ref{LAUTHM} together with Lemma \ref{lem:lift} and Remark \ref{rmk:linkFS}.
\end{proof}

\subsubsection{{\large Volumes, limits and log K-stability}}

In this section we restrict to angles in the interval \((0,2\pi)\). 
Let \(T\) be the standard \((n+1)\)-simplex
\begin{equation}\label{eq:T}
T = \{a_i>0, \,\ \sum_i a_i < 1  \} \subset \R^{n+1} .	
\end{equation}
For every \(a=(a_1, \ldots, a_{n+1}) \in T\) we have a FS metric \(g_a\) on \(\CP^{n-1}\) with cone angles \(2\pi(1-a_i-a_j)\) along the hyperplanes \(\bar{H}_{i,j}\) of the projectivized \(\mA_n\) arrangement. 

\begin{lemma}\label{lem:volFSa}
	The total volume of the FS metric \(g_a\) is given by
	\begin{equation}
		\Vol(g_a) = (1-a_0)^{n-1} \Vol(\CP^{n-1}) 
	\end{equation}
	where \(a_0 = \sum_i a_i\).
\end{lemma}

\begin{proof}
	Follows from Proposition \ref{prop:totalvolume}.
\end{proof}

The `most degenerate points' \(p_i \in \CP^{n-1}\) (for \(1\leq i \leq n+1\)) of the projectivized \(\mA_n\) arrangement are given by the intersection of all hyperplanes \(\bar{H}_{j,k}\) with \(i \notin \{j,k\}\).	
The tangent cone of \(g_a\) at \(p_i\) is a PK cone metric on \(\C^{n-1}\) singular at the hyperplanes of the \(\mA_{n-1}\) arrangement given by parameters 
\begin{equation}\label{parameters}
		(a_1, \ldots, a_{i-1}, a_{i+1}, \ldots, a_{n+1}) .
\end{equation}
We denote it by \(C_{p_i}(g_a)\).
Write
\[\hat{a}_{i} = \sum_{j\neq i} a_j \,\ \mbox{ and } \,\ \hat{\alpha}_{i} = 1-\hat{a}_{i} . \] 
The product
\[C_{p_i}(g_a) \times \C_{\hat{\alpha}_{i}} \]
is a regular PK cone metric on \(\C^n\) and we denote its complex link by \(g_{a,i}\).
	
The FS metric \(g_{a,i}\) on \(\CP^{n-1}\) is singular at the hyperplanes of the projective completion \footnote{The projective completion of a hyperplane arrangement \(\mH \subset \C^{n-1}\) is obtained by taking its closure under the inclusion  \(\C^{n-1} \subset \CP^{n-1}\) together with the hyperplane at infinity \(\CP^{n-1} \setminus \C^{n-1}\).} of \(\mA_{n-1}\)  with cone angles determined by the parameters \eqref{parameters} and cone angle \(2\pi \hat{\alpha}_{i}\) along the hyperplane at infinity. The metric \(g_a\) agrees with \(g_{a,i}\) in a neighbourhood of \(p_i\) (indeed this agrees with the description of the singularities of \(g_a\) in terms of local models given in Section \ref{sect:locmodsingFS}).

\begin{lemma}\label{lem:volrugby}
	The total volume of the FS metric \(g_{a,i}\) on \(\CP^{n-1}\) is given by
	\begin{equation}
	\Vol(g_{a,i}) = (1- \hat{a}_{i})^{n-1} \Vol(\CP^{n-1}) .
	\end{equation}
\end{lemma}

\begin{proof}
	Again, this follows from Proposition \ref{prop:totalvolume}.
\end{proof}

\noindent \textbf{Geometric interpretation.} Note that 
\[\frac{\Vol(g_a)}{\Vol(g_{a,i})} = \left(\frac{1-a_0}{1-\hat{a}_i}\right)^{n-1} < 1 \,\ \iff \,\ a_i>0 .  \]
We can re-write the equations defining the faces of \(T\) in terms of the volume inequalities:
\begin{align*}
	\sum_i a_i < 1 &\iff \,\,\, \Vol(g_a) > 0  , \\
	a_i>0 &\iff \Vol(g_a) < \Vol(g_{a,i})  .
\end{align*}

\noindent \textbf{Limits.}
Let \(a(t)\) be a path in \(T\) for \(t \in (0, \epsilon)\) and suppose that \(a(0) = \lim_{t \to 0} a(t) \in \p \, T\). We ask what is the behaviour of the FS metrics \(g_{a(t)}\) as \(t \to 0\). We address this question at an informal level, without specifying the sense of convergence and details. In any case, we always have a natural limit Lauricella connection \(\nabla^{a(0)}\) endowed with a
(possibly degenerate) parallel Hermitian form.

For simplicity, we assume that \(a(0)\) belongs to the interior of a face of \(T\) - e.g. this avoids discussion of integer angles. There are two cases to consider.

\begin{itemize}
	\item \(a(0)\) belongs to the interior of the face \(\{\sum_i a_i = 1\}\). It follows from Lemma \ref{lem:volFSa} that \(\Vol(g_{a(t)}) \to 0\) as \(t\to 0\). On the other hand, \cite{CHL} show that there is a flat K\"ahler metric on the arrangement complement, by Theorem \ref{PRODTHM} its metric completion is a polyhedral metric on \(\CP^{n-1}\). We can normalize it to have volume \(1\) and denote it by \(g_{PK}\). The natural expectation is that
	\[ \lim_{t \to 0} \lambda_t^{-1} \cdot g_{a(t)} = g_{PK} . \]
	with \(\lambda_t=\Vol(g_{a(t)})^{1/(n-1)}\). 
	
	\item \(a(0)\) belongs to the interior of the face \(\{a_i=0\}\). In this case by removing the \(0\) entry from \(a(0)\) we obtain a point in the \(n\)-simplex and a corresponding FS metric on \(\CP^{n-2}\), its projective lift to \(\CP^{n-1}\) (i.e. a metric of the kind \(g_{a,i}\) described before) is the natural candidate for the limit as \(t \to 0\).
\end{itemize}

\noindent \textbf{Log K-stability .}
Let \(a=(a_1, \ldots, a_{n+1}) \in \R^{n+1}\) be such that
\(0<a_i+a_j<1\) for every pair \(i<j\).
Write \(\Delta_a = \sum_{i<j} (a_i+a_j) \bar{H}_{i,j}\) thought as an \(\R\)-divisor in \(\CP^{n-1}\).

\begin{lemma}
	The pair \((\CP^{n-1}, \Delta_a)\) is log Fano if and only if \(\sum_i a_i < 1\).
\end{lemma}

\begin{proof}
	This is an immediate consequence of the identity
	\[ \sum_{1 \leq i < j \leq n+1} (a_i+a_j) < n \sum_{i=1}^{n+1}a_i . \qedhere \] 
\end{proof}

Recall that \(T\) is the standard \((n+1)\)-simplex given by Equation \eqref{eq:T}. Our final observation is the next.

\begin{lemma}
	The log Fano pair \((\CP^{n-1}, \Delta_a)\) is log K-stable if and only if \(a \in T\).
\end{lemma}

\begin{proof}
	According to \cite[Theorem 1.5]{Fuj} the pair \((\CP^{n-1}, \Delta_a)\) is log K-stable if and only if the log Calabi-Yau pair \((\CP^{n-1}, c \cdot \Delta_a)\) with \(c= \left(\sum_i a_i\right)^{-1}\) is klt. Let us write \((\p_1 T)^{\circ}\) for the interior of the face of \(T\) where the coordinates add up to one. It follows from Lemma \ref{lem:positivityvolume} that \((\CP^{n-1}, c \cdot \Delta_a)\) is klt if and only if \(c \cdot a \in (\p_1 T)^{\circ}\). 
	Equivalently, \((\CP^{n-1}, \Delta_a)\) is log K-stable if and only if \(a\) belongs to the interior of a segment connecting \(0\) with \((\p_1 T)^{\circ}\). The set of all such \(a\) is precisely \(T\).
\end{proof}

\appendix
\section{Logarithmic connections}\label{app:logconn}

A standard reference on the topic is \cite{Deligne}. For a more analytic approach, as presented here, see \cite{NY}.

\subsection{Residue and induced connection}

Let \(X\) be a complex manifold, \(D \subset X\) a \emph{smooth} complex hypersurface and \(E\) a holomorphic vector bundle over \(X\).
We write \(n = \dim X\) and \(r = \rk E\).

\begin{definition}\label{def:logcon}
	Let \(\nabla\) be a holomorphic connection on \(E\) over \(X \setminus D\). We say that \(\nabla\) is a \textbf{logarithmic connection} (or has a logarithmic singularity along \(D\)) if for every \(x \in D\) there are holomorphic coordinates \((z_1, \ldots, z_n)\) centred at \(x\) with \(D=\{z_1=0\}\) and a frame of holomorphic sections of \(E\) in which we have \(\nabla=d-\Omega\) with
	\begin{equation}\label{eq:logconect}
		\Omega = A_1 \frac{dz_1}{z_1} + A_2 dz_2 + \ldots + A_n dz_n
	\end{equation}
	and \(A_1, \ldots, A_n\) are holomorphic matrix-valued functions. 
\end{definition}

\begin{remark}
	Definition \ref{def:logcon} does not depend on the choices of neither coordinates nor frame.
\end{remark}

\begin{proof}
	If we change \(z_1 \mapsto z_1 \varphi\) where \(\varphi\) is a holomorphic non-vanishing function then 
	\begin{equation}\label{eq:logderiviative}
		 \frac{dz_1}{z_1} \mapsto \frac{dz_1}{z_1} + \hol
	\end{equation}
	with \(\hol = d\varphi/\varphi\).
	
	If we change frames \((s_1, \ldots, s_r) \mapsto (s_1, \ldots, s_r) \cdot G\) where \(G\) is a holomorphic \(GL(r, \C)\)-valued function then 
	\begin{equation}\label{eq:changegauge}
		\Omega \mapsto G\Omega G^{-1} + \hol
	\end{equation}
	with \(\hol=(dG)G^{-1}\).
	
	The remark follows from Equations \eqref{eq:logderiviative} and \eqref{eq:changegauge}.
\end{proof}

\begin{definition}\label{def:resind}
	Let \(\nabla\) be a logarithmic connection as in Equation \eqref{eq:logconect}. 
	\begin{itemize}
		\item The \textbf{residue} of \(\nabla\) is defined as 
		\[\RES(\nabla) = A_1(0, z_2, \ldots, z_n) .\] 
		It is a holomorphic section of \(\End E|_D\).
		\item The \textbf{induced connection} is defined as \(\nabla'=d-\Omega'\) with
		\begin{equation*}
		\Omega' = A_2(0,z_2, \ldots, z_n)dz_2 + \ldots + A_n(0,z_2, \ldots, z_n)dz_n .
		\end{equation*}
		It is a holomorphic connection on \(E|_D\).
	\end{itemize}	
\end{definition}

\begin{note}
	Same as before, the independence of \(\RES(\nabla)\) and \(\nabla'\) on the coordinates and frame used in Definition \ref{def:resind} follows from Equations \eqref{eq:logderiviative} and \eqref{eq:changegauge}. Alternatively, it is possible to give more intrinsic definitions as follows.
	
	Notice that if \(s\) is a holomorphic section of \(E\) and \(v\) is a vector field tangent to \(D\), then the section \(\nabla_v s\) (which in principle is only defined on \(X \setminus D\)) extends holomorphically over \(D\). Consider now a vector \(s_x\) in the fibre \(E_x\) over \(x \in D\). Extend \(s_x\) to a holomorphic section \(s\) of \(E\) defined in a neighbourhood of \(x\), we check that
	\[\RES(\nabla)(s_x) = (\nabla_{z_1\p_{z_1}} s)(x) .\]
	If \(v_x \in T_xD\) and \(\tilde{s}\) is a holomorphic section of \(E|_D\) then we can extend \(v_x\) to a holomorphic vector field tangent to \(D\) defined in a neighbourhood of \(x\) and we can extend \(\tilde{s}\) to a holomorphic section of \(E\) around \(x\). We check that
	\[ \nabla'_{v_x} \tilde{s} = (\nabla_v s)(x) . \]
	 
\end{note}

\begin{example}\label{ex:limithol}
	Let \(E = \mO^{\oplus r} \) be the trivial vector bundle over the disc \(D \subset \C\). A logarithmic connection with a singularity at the origin is given by 
	\[\nabla = d - \frac{A}{z} dz \]
	with \(A\) a holomorphic matrix-valued function. The residue of \(\nabla\) is the constant matrix \(A(0)\),
	\[\RES(\nabla) = A(0) . \]
	For each \(x \in D^*\) we have a holonomy matrix \(T_x \in GL(r, \C)\) given by parallel translation along the loop \(xe^{2\pi i t}\) for \(0\leq t \leq 1\). A direct computation shows that
	\[ \lim_{x \to 0} T_x = \exp(2\pi i \RES(\nabla)) , \]
	see \cite[Th\'eor\`eme 1.17]{Deligne}. However, one must be aware that \(T_x\) for \(x \in D^*\) might not be conjugate to \(T_0 = \lim_{x \to 0} T_x\); see Example \ref{ex:del}.
\end{example}

\subsection{Flat and torsion free conditions}

\begin{lemma}\label{lem:apendix1}
	Let \(\nabla\) be a \emph{flat} logarithmic connection. Then the following holds:
	\begin{itemize}
		\item[(i)] the induced connection \(\nabla'\) is flat;
		\item[(ii)] the residue \(\RES(\nabla)\) is a parallel section with respect to the connection on \(\End E|_D\) defined by \(\nabla'\).
	\end{itemize}
\end{lemma}

\begin{proof}
	In a local trivialization \(\nabla=d-\Omega\), and the flatness condition means that \(d\Omega=\Omega\wedge\Omega\). If we write \(\Omega\) as in Equation \eqref{eq:logconect} we obtain
	\begin{equation}\label{eq:pflem1}
		\Omega \wedge \Omega = \tilde{\Omega} \wedge \tilde{\Omega} + \sum_{j \geq 2}[A_1, A_j] \frac{dz_1 \wedge dz_j}{z_1} ,
	\end{equation}
	with \(\tilde{\Omega} = \sum_{j=2}^{n} A_j dz_j\); and
	\begin{equation}\label{eq:pflem2}
		d\Omega = \tilde{d} \tilde{\Omega} + \sum_{j \geq 2} \left( z_1 \frac{\p A_j}{\p z_1} -\frac{\p A_1}{\p z_j} \right) \frac{dz_1 \wedge dz_j}{z_1} 
	\end{equation}
	where \(\tilde{d} = \sum_{j=2}^{n} (\p/\p z_j) dz_j\). Since \(d\Omega= \Omega \wedge \Omega\), the r.h.s. of Equations \eqref{eq:pflem1} and \eqref{eq:pflem2} must match each other term by term. Setting \(z_1=0\) we obtain
	\begin{equation}\label{eq:pflem3}
		d \Omega' = \Omega' \wedge \Omega' \hspace{2mm} \mbox{ and } \hspace{2mm} \frac{\p \bar{A}_1}{\p z_j} = [\bar{A}_j, \bar{A}_1] \hspace{1mm} \mbox{ for } j=2, \ldots, n . 	
	\end{equation}
	Here \(\Omega' = \tilde{\Omega}(0, z_2, \ldots, z_n)\) is as in Definition \ref{def:resind} and \(\bar{A}_i = A_i(0, z_2, \ldots, z_n)\). The first equation in \eqref{eq:pflem3} implies that \(\nabla'\) is flat, the second set of equations imply that \(\RES(\nabla)\) is parallel; see Note \ref{not:end}.
\end{proof}

\begin{remark}
	It follows from the second item in Lemma \ref{lem:apendix1} that \(\RES(\nabla)\) varies in the same conjugacy class along \(D\), in particular its eigenvalues are constant.
\end{remark}

\begin{note}\label{not:end}
	Let \(\nabla\) be a connection on a vector bundle \(E\) written as \(\nabla=d-\Omega\) with respect to some frame defined on some open set \(U\). Sections of \(E\) over \(U\) are represented by vector-valued functions \(v=(v_1, \ldots, v_r)\) and \(\nabla v = dv - \Omega v\).
	The connection on \(\End E\) induced by \(\nabla\) is uniquely defined by imposing the Leibniz rule. More precisely, a section of \(\End E\) over \(U\) is represented by a matrix-valued function \(M\) and we require that
	\begin{equation}\label{eq:Leibniz}
		\nabla (Mv) = (\nabla M) v + M(\nabla v) .
	\end{equation}
	Expanding the terms \(\nabla(Mv)\) and \(\nabla v\), it follows that
	\[\nabla M = dM - [\Omega, M] . \]
\end{note}

\begin{note}\label{not:gaugeflat}	 
	 Equation \eqref{eq:Leibniz} implies that \(M\) is a gauge automorphism, i.e. it commutes with \(\nabla\), precisely when \(\nabla M =0\). The same discussion applies when we have two bundles with connections \((E_1, \nabla_1)\) and \((E_2, \nabla_2)\). There is an induced connection on \(\Hom(E, F)\) defined by the Leibniz rule and gauge transformations correspond to flat sections.  
\end{note}

\begin{lemma}[cf. {\cite[Lemma 1.7]{CHL}}]\label{lem:apendixTDflat}
	Suppose \(\nabla\) is a logarithmic connection on \(TX\). If \(\nabla\) is torsion free then \(TD \subset \ker \RES(\nabla)\). In particular, if \(\nabla\) is flat torsion free and \(\RES(\nabla)\) is non-zero then \(TD\) is a flat sub-bundle of \((TX|_D, \nabla')\) hence the restriction of \(\nabla'\) to \(TD\) defines a flat torsion free connection.
\end{lemma}
 
\begin{proof}
	Recall that the torsion of \(\nabla\) is given by
	\[T(X,Y) = \nabla_X Y - \nabla_Y X - [X, Y] . \]
	Write \(\nabla = d -\Omega\) as in Equation \eqref{eq:logconect} using the trivialization of \(TX\) given by the coordinate vector fields \(\p_{z_1}, \ldots, \p_{z_n}\).
	Replace \(\nabla=d-\Omega\) in the equation for \(T\) to obtain
	\[T(X,Y) = (\imath_Y \Omega)(X) - (\imath_X \Omega)(Y) , \]
	where \(\imath_X\) denotes contraction with the vector \(X\).  Take \(X=z_1\p_{z_1}\) and \(Y=\p_{z_j}\) with \(j>1\), so
	\[\imath_X \Omega = A_1, \qquad \imath_Y \Omega = A_j . \]
	Since \(T(X,Y) \equiv 0\) we get that \(A_1(\p_{z_j}) - A_j(z_1 \p_{z_1})\) vanishes identically, restricting to \(z_1=0\) we obtain \(\RES(\nabla)(\p_{z_j})=0\).
	
	For the second part, note that if the residue is non-zero then \(TD= \ker \RES(\nabla)\), by Lemma \ref{lem:apendix1} \(\RES(\nabla)\) is parallel, hence \(TD\) must be a flat sub-bundle.
\end{proof}

\subsection{Normal Form Theorem}

The next result is well-known. 

\begin{theorem}[Normal Form Theorem]\label{thm:normalform}
	Let \(\nabla = d - \Omega\)  be a flat logarithmic connection on the trivial bundle over the polydisc \(D^n\) given by Equation \eqref{eq:logconect}. Let \(A\) be the constant matrix given by the residue of \(\nabla\) at the origin, that is \(A=A_1(0)\). Moreover, suppose that the following \emph{non-resonance condition} holds:
	\begin{itemize}
		\item no two eigenvalues of \(A\) differ by a non-zero integer. 
	\end{itemize}
	Then \(\nabla\) is holomorphically gauge equivalent to the standard connection \(\nabla^0 =d-\Omega^0\) with \(\Omega^0=Adz_1/z_1\).
\end{theorem}

The proof of Theorem \ref{thm:normalform} is a consequence of the next few lemmas and it is given at the end of this section.

\begin{lemma}\label{lem:conjhol}
	Let \(\nabla\) and \(\nabla^0\) be as in Theorem \ref{thm:normalform}. Then the holonomies of \(\nabla\) and \(\nabla^0\) are conjugate.
\end{lemma}

\begin{proof}
	 We pull-back \(\nabla\) to the punctured disc via the inclusion \(D^* \times \{0\} \subset D^* \times D^{n-1} \) to obtain a `non-resonant Fuchsian system' which is holomorphically equivalent to the Euler system on \(D^{*}\) defined by \(Adz_1/z_1\); see \cite[Theorem 16.16]{Yak}.
	 In particular, the holonomies of \(\nabla\) and \(\nabla^0\) along the standard loop encircling \(\{z_1=0\}\) are conjugate. Since the fundamental group of \(D^* \times D^{n-1}\) is generated by such loops, the statement follows. 
\end{proof}

\begin{example}[{\cite[pg. 54]{Deligne}}]\label{ex:del}
	An example of a resonant Fuchsian system of rank two is given by the following
	\[\nabla \binom{u}{v} = d \binom{u}{v} - \begin{pmatrix}
	0 & 0 \\
	0 & 1
	\end{pmatrix}  \frac{dx}{x}
	\binom{u}{v} - \begin{pmatrix}
	0 & 0 \\
	dx & 0
	\end{pmatrix}
	\binom{u}{v} .   \]
	The equations for a parallel section \(\nabla s= 0\) 	with \(s=(u,v)\) are
	\[du=0, \,\ dv = \left(\frac{v}{x} + u\right)dx .\]
	The solutions are given by \(u = a\), \(v=a x \log x + bx\) with \(a,b \in \C\). The holonomy is
	\[T_x = \begin{pmatrix}
	1 & 2\pi i x \\
	0 & 1
	\end{pmatrix} . \]
	We see that \(T_0\) is not conjugate to \(T_x\) for \(x\neq0\).	
\end{example}

The proofs of Lemmas \ref{lem:Gmero} and \ref{lem:Ghol} that follow are taken from Malgrange's exposition in \cite{Borel}.

\begin{lemma}\label{lem:Gmero}
	Let \(\nabla\) and \(\nabla^0\) be as in Theorem \ref{thm:normalform}. Then any holomorphic gauge transformation \(G\) between \(\nabla\) and \(\nabla^0\) defined on \(D^* \times D^{n-1}\) is meromorphic on \(D^n\).
\end{lemma}

\begin{proof}
	The two connection forms are
	\[\Omega = \frac{M_1}{z_1} dz_1 + \sum_{i>1} M_i dz_i, \hspace{3mm} \Omega^0 = \frac{A}{z_1} dz_1 \]
	with \(M_1=A + O(|z|)\). The \(dz_1\) component of the equation
	\(dG = G \Omega^0 - \Omega G\) implies
	\[z_1 \frac{\p G}{\p z_1} = G A - M_1 G . \]
	Therefore
	\[|z_1| \left\| \frac{\p G}{\p z_1} \right\| \leq C \|G\| \]
	for some \(C>0\), which integrates to give the bound \(\|G\|=O(|z_1|^{-N})\) for some \(N>0\). This implies that \(G\) extends meromorphically.
\end{proof}

\begin{note}
	In the context of Lemma \ref{lem:Gmero}, \(G\) is a flat section of \(\End E\) endowed with the logarithmic connection induced by \(\nabla\) and \(\nabla^0\). It is a general fact that flat sections of logarithmic connections have `moderate growth', see \cite{Deligne}, from which it follows that \(G\) must be meromorphic.
\end{note}

\begin{lemma}\label{lem:Ghol}
	Let \(\nabla\) and \(\nabla^0\) be as in Theorem \ref{thm:normalform}. Then any holomorphic gauge transformation \(G\) between \(\nabla\) and \(\nabla^0\) defined on \(D^* \times D^{n-1}\) is holomorphic on \(D^n\).
\end{lemma}

\begin{proof}
	By Lemma \ref{lem:Gmero} we can take \(G=\sum_k G_k z_1^k\) to be the Laurent series of \(G\) with respect to \(z_1\) and coefficients \(G_k=G_k(z_2, \ldots, z_n)\) given by  holomorphic functions in \(z_i\) for \(i>1\).
	Let \( p = \min_{k \in \Z} \{k : \,\ G_k \neq  0\}\). The equation \(dG=G \Omega^0 - \Omega G\) implies that
	\[p G_p = G_p A - A G_p \]	
	We rewrite this as
	\[(A + p \Id) G_p = G_p A . \]
	Now we use the following fact:
	\begin{itemize}
		\item If \(B, C, S \in M(r \times r, \C)\) satisfy \(BS=SC\) and \(S\neq 0\) then \(B\) and \(C\) have a common eigenvalue.
	\end{itemize}
	We apply the above with \(B= A + p \Id\), \(C=A\) and \(S=G_p\). It follows that there are two eigenvalues \(\lambda_1, \lambda_2\) of \(A\) with \(\lambda_2=\lambda_1 + p\).
	Since \(p \in \Z\), our hypothesis imply that the only way this can happen is when \(p=0\). It follows that \(G\) is holomorphic.
\end{proof}

We can now establish the Normal Form Theorem.

\begin{proof}[Proof of Theorem \ref{thm:normalform}]
	By Lemma \ref{lem:conjhol} there is a holomorphic gauge transformation \(G\) between \(\nabla\) and \(\nabla^0\) defined on \(D^*\times D^{n-1}\). By Lemma \ref{lem:Ghol} the gauge transformation \(G\) extends holomorphically to the whole \(D^n\).
\end{proof}

\bibliographystyle{alpha}
\bibliography{BIBLIO}

\Address

\end{document}